\documentclass[11pt,reqno]{amsart}

\usepackage{booktabs}                                 
\usepackage[english]{babel}
\usepackage{latexsym}
\usepackage[utf8]{inputenc}
\usepackage{babel}
\usepackage{fancyhdr}
\usepackage{longtable}
\usepackage{mathrsfs}
\usepackage{geometry}
\usepackage{comment}
\usepackage{textcomp}
\usepackage{caption}
\usepackage{lmodern}
\usepackage{mathrsfs}
 \usepackage[T1]{fontenc}
\usepackage[all,cmtip]{xy}
\usepackage{epsfig}
\usepackage{amsmath,amsfonts,amssymb,amsthm}
\usepackage{graphics}
\usepackage{graphicx}
\usepackage{verbatim}
\usepackage{dsfont}
\usepackage{enumerate}  
\usepackage{pdflscape}
\usepackage{longtable}
\usepackage{booktabs}
\usepackage{colortbl}

\usepackage{tikz-cd}
\usepackage{ytableau}

\usepackage[shortlabels]{enumitem}
\usepackage[breaklinks]{hyperref} 
\hypersetup{
	colorlinks,
	linkcolor={red},
	citecolor={blue},
	urlcolor={red},
}
\usepackage{stmaryrd}
\usepackage{multirow}
\usepackage{longtable}

\usepackage{xcolor,colortbl}
\definecolor{lavender}{rgb}{0.9, 0.9, 0.98}

\usepackage[disable]
{todonotes}

\newcommand{\red}[1]{{\color{red}{#1}}}

\newcommand{\Q}{\mathbb{Q}}

\newcommand{\mQ}{\mathcal{Q}}

\newcommand{\mR}{\mathcal{R}}

\newcommand{\mF}{\mathcal{F}}

\newcommand{\PP}{\mathbb{P}}

\newcommand{\mE}{\mathcal{E}}
\newcommand{\mU}{\mathcal{U}}
\newcommand{\mL}{\mathcal{L}}

\newcommand{\mZ}{\mathscr{Z}}
\newcommand{\of}{\mathcal{O}}

\newcommand{\W}{\bigwedge}

\newcommand{\bbGr}{\mathbb{G}\mathrm{r}}
\newcommand{\git}{\mathbin{
  \mathchoice{/\mkern-6mu/}
    {/\mkern-6mu/}
    {/\mkern-5mu/}
    {/\mkern-5mu/}}}

\DeclareMathOperator{\Spec}{Spec}

\DeclareMathOperator{\Sym}{Sym}

\DeclareMathOperator{\Gr}{Gr}
\DeclareMathOperator{\Fl}{Fl}

\DeclareMathOperator{\SGr}{SGr}
\DeclareMathOperator{\Hom}{Hom}

\DeclareMathOperator{\Bl}{Bl}

\DeclareMathOperator{\DP}{dP}

\newcommand{\fanoid}[1]{
[\#{#1}]
}

\newtheorem{thm}{Theorem}[section]
\newtheorem{corollary}[thm]{Corollary}

\newtheorem{lemma}[thm]{Lemma}
\newtheorem{proposition}[thm]{Proposition}

\newtheorem*{aim*}{Aim of this paper}

\theoremstyle{definition}
\newtheorem{definition}[thm]{Definition}
\newtheorem{rmk}[thm]{Remark}
\newtheorem{ex}[thm]{Example}

\newtheorem{alg}[thm]{Algorithm}

\usepackage{thmtools}
\declaretheoremstyle[
spaceabove=1.5ex, spacebelow=1.5ex,
headfont=\bf,
notefont=\mdseries, notebraces={(}{)},
bodyfont=\normalfont,
headpunct=.,
numberwithin=,
postheadhook=\leavevmode%
qed={$\bullet$}
]{mystyle}
\declaretheorem[style=mystyle]{fano}

\usepackage[capitalise,nameinlink]{cleveref}
\crefname{mystyle}{Fano}{Fanos}
\crefname{claim}{Claim}{Claims}
\crefname{rmk}{Remark}{Remarks}
\crefname{workhyp}{WH}{WH}
\crefname{thm}{Theorem}{Theorems}
\crefname{proposition}{Proposition}{Propositions}
\crefname{app}{Appendix}{Appendices}

\crefname{lemma}{Lemma}{Lemmas}
\crefname{alg}{Algorithm}{Algorithms}
\crefname{ex}{Example}{Examples}
\makeatletter    
\def\l@subsection{\@tocline{1}{0,2pt}{2pc}{8mm}{\ \ }} 
\makeatletter    
\def\l@section{\@tocline{1}{0,2pt}{2pc}{8mm}{\ \ }} 

\newcommand{\crefpart}[2]{%
  \hyperref[#2]{\namecref{#1}~\labelcref*{#1}~(\ref*{#2})}%
}

\author{Enrico Fatighenti}
\address{Dipartimento di Matematica \\
Universit\`a di Bologna\\
Piazza di Porta San Donato 5\\
40127 Bologna, Italy}
\email[E.~Fatighenti]{enricofatighenti6@gmail.com}

\author{Fabio Tanturri}
\address{Dipartimento di Matematica \\
Dipartimento di Eccellenza 2023-2027\\
Universit\`a di Genova\\
Via Dodecaneso 35\\
16146 Genova, Italy}
\email{tanturri@dima.unige.it}

\author{Federico Tufo}
\address{Dipartimento di Matematica \\
Universit\`a di Bologna\\
Piazza di Porta San Donato 5\\
40127 Bologna, Italy}
\email[F.~Tufo]{federico.tufo2@unibo.it}

\@namedef{subjclassname@2020}{\textup{2020} Mathematics Subject Classification}
\subjclass[2020]{Primary 14J30 · 14J45; Secondary 14M15 ·
14E30}
 
\title[On the geometry of some quiver zero loci Fano fourfolds]{On the geometry of some quiver zero loci Fano fourfolds}

\usepackage{libertine}
\usepackage{libertinust1math} 

\usepackage{float}
\floatstyle{plaintop}
\restylefloat{table}

\begin{document}
\begin{abstract}
    In this paper, we study 170 families of quiver flag zero loci Fano fourfolds as described by Kalashnikov. 
    We interpret those manifolds as zero loci of sections of homogeneous vector bundles in homogeneous varieties, and we give a birational and biregular description of all 170 families.
\end{abstract}
\maketitle

\section{Introduction}

Fano varieties and their classification are one of the most studied topics in algebraic geometry. 
It is well known that Fano manifolds are \emph{bounded}, thanks to a famous and motivating result by Koll\'ar, Miyaoka, and Mori \cite{kmm}: 
in each dimension, there is a finite number of deformation classes of Fano manifolds. 
Up to dimension three, we have a total of 116 families of Fano manifolds, but from dimension four onwards the problem is widely open, 
 although we have some results in higher index and Picard rank, see \cite{ag5,wisniewski1991fano}. 
When the dimension is low enough, we understand the full picture: 
there is in fact only one Fano manifold in dimension one, the projective line. In dimension two there are ten families of Fano manifolds, the Del Pezzo surfaces, already known since \cite{DP}.
The classification of Fano threefolds started with the work of Fano himself in \cite{fano}, and was carried on by Iskovskikh \cite{isko77,isko78}. 
 In particular, Iskovskikh completed the classification of the seventeen prime Fano threefolds, i.e., the ones with Picard rank equal to one. Later on, Mori and Mukai classified all 105 families of Fano threefolds using the birational Mori's theory of extremal rays, see \cite{morimukai,MME}.
Mukai's work did not stop with the classification. In fact, in \cite{mukai}, he revisited the work of Iskovskikh using the biregular vector bundle method. In this way, he could describe the prime ones as zero loci of appropriate vector bundles, and almost all of them as linear sections in homogeneous or quasi-homogeneous varieties. A very recent update on the vector bundle method was given in \cite{bayer2024mukai}, with a complete proof of Mukai’s Theorem on the existence of certain exceptional vector bundles on prime Fano threefolds.

A natural question arises: can all the 105 Fano threefolds be described this way?
A first answer can be found in \cite{corti}. There, all the families of Fano threefolds are described as zero loci of a section of homogeneous vector bundles over GIT quotients $V\git G $ where $G$ is a product of groups $GL_n(\mathbb{C})$, and $V$ is a representation of $G$. 
This kind of description is particularly efficient for computing quantum periods. From this program, the Fano threefolds can be written as complete intersections in toric varieties or subvarieties in products of Grassmannians.
Another answer is in \cite{DFT}. The descriptions were reworked once again, and the threefolds were rewritten as zero loci of sections of homogeneous vector bundles (section zero loci for short) over products of Grassmannians in the spirit of the original Mori--Mukai classification. This method, together with the establishment of a homogeneous-to-birational dictionary, suggests an approach for the classification of Fano manifolds of higher dimensions. In particular, in \cite{BFMT} the authors moved the first steps of this program in dimension four.\\

In the fourfold case, the program highlighted in \cite{coates2mirror} was developed in a series of works using the techniques of \cite{corti}, see for example \cite{CGKS, coates2017laurent,kalashnikov}. In particular, in \cite{CGKS} the authors constructed a big database of known Fano fourfolds. This includes the known and classified 35 families of Fano fourfolds with index $\iota_X>1$, see \cite{ag5}, and also many other fourfolds obtained in \cite{batyrev1999classification, obro2007algorithm,sato2000toward,strangeway2014quantum}.
Using mirror symmetry they were able to compute the quantum period of all toric complete intersections, products of lower-dimension Fano manifolds, and certain complete intersections in projective bundles. In \cite{kalashnikov}, on the other hand, the method was developed in a slightly different way. The aim was to produce a list of Fano fourfolds constructed as zero loci of global sections of certain homogeneous vector bundles over quiver flag manifolds. In this way, the database produced in \cite{kalashnikov} contains the non-toric ones in \cite{CGKS} and has 141 Fano fourfolds with a new quantum period.

In some sense, the above technique can be seen as a partially, yet far-reaching, generalization of the work of K{\"u}chle \cite{kuchle}. In this work, he classified 18 Fano fourfolds of index one obtained as zero loci of sections of completely reducible homogeneous vector bundles over a single Grassmannian, which is not a projective space, with two of them later identified by Manivel, see \cite{manivel2015fano}. One of the main advantages of this method is that the computation of invariants such as Hodge numbers and volume relies on combinatorial data, hence they can be calculated with the aid of computer algebra software. For this reason, as explained in \cite{DFT}, the next step is to consider zero loci of sections of completely reducible homogeneous vector bundles over products of Grassmannians. This is exactly what was done in dimension three in \cite{DFT}, and then in dimension four in \cite{BFMT}. In the latter, the authors built a database of at least 634 Fano fourfolds and extracted the so-called $K3$-type ones; they found 64 deformation families of Fano fourfolds of $K3$-type, which were thoroughly described birationally and biregularly. The exploration of the rest of the database has barely started, see for example \cite{bernardara2024even,manivel2023four} for the first few extra results. It will be also interesting to compare this database with other lists of Fano described by birational terms, see for example \cite{casagrande2023lefschetz}.\\

Both of the previously described methods provide databases. The one in \cite{kalashnikov} has the aim to study the quantum period of Fano fourfolds. On the other hand, the one used in \cite{BFMT} is particularly suited to compute invariants such as Hodge numbers, volumes, and allows a relatively easy birational description. The two databases are not contained in each other, for example, the fourfold \red{c2} in \cite{kuchle} is in the second database but not the first. Vice versa, the fourfold $\red{K_{743}}$ is in the database built in \cite{kalashnikov}, but not in the one used in \cite{BFMT}.
This paper grew out of the natural question of how to compare the two databases and build a bridge between the two approaches. We also want to understand a systematic way to pass from the quiver language to the biregular setting. We plan to address these questions in the forthcoming Ph.D.\ thesis of the third author. In some sense, this article can be seen as a piece of evidence for our future works on the topic.\\

In this paper, we will translate the language of section zero loci in quiver flag manifolds into the language of section zero loci in products of Grassmannians. We point out how the process can be quite tricky, and yet reveal interesting details on the geometry of these fourfolds. As an example, we point out one of the Fano described in \cite{coates2017laurent}, which was understood by Casagrande using a very natural yet interesting birational description. The translation process is explained in \cref{sec:ZlqToZlg}, and it is a direct application of what is written in \cite{CRAW11}. We give also an algorithmic recipe for the translation in \cref{QuivToGrass}, thanks to which the process can be automatized using as input the adjacency matrix of the quiver flag manifold, the dimension vector, and the representation of the homogeneous bundles.

One of our main progresses is that we are able to distinguish certain Fano fourfolds admitting a blow-down to another Fano fourfold, and those who do not. Following Mukai's terminology, we call the latter \emph{primitives}. In particular, in this paper, we study the 170 Fano fourfolds mentioned above on a case-by-case analysis, which should be considered a benchmark for future works on the subject. Our idea was in fact to gain an understanding, as complete as possible, of the network of birational maps between Fano fourfolds. 

As a recap of our results, which are described in detail in \cref{sec:results}, we find that most of the 170 families can be described as blow-ups of smooth Fano fourfolds of lower Picard rank in smooth centers. Still, there are several families whose geometry is more interesting and complicated. These families deserve further study and attention, which we hope to carry out in the future. 

\subsection*{Notation} Throughout the paper we will work over $\mathbb{C}$.
$V_n$ will denote a $n$-dimensional vector space, and $\Gr(k,n)$ the Grassmannian parametrizing its $k$-dimensional subspaces, and similarly for the Flag variety  $\Fl(k_1,k_2,...,k_t,n)$. We will denote with $\mU$ and $\mQ$ the tautological and quotient bundle of $\Gr(k,n)$, respectively of rank $k$ and $n-k$. We will follow the convention in which $\det \mQ \cong \of(1)$. If $\alpha$ is a partition, we will denote with $\Sigma_{\alpha} \mU$, $\Sigma_{\alpha} \mQ$, $\Sigma_{\alpha}V_n$ the Schur functor applied to $\mU$, respectively $\mQ$, $V_n$, associated to $\alpha$. Similarly, on the Flag variety, $\mU_i$, $\mQ_i$ will denote the tautological and quotients pulled back from the base Grassmannians, and $\mR_i$ the intermediate tautological defined by the sequence $0\to\mU_{i-1}\to\mU_i\to\mR_i\to 0$. Moreover, also Grassmann bundles have a relative Euler sequence, and we will denote all the relative bundles indexed with $\mR$.
We will also adopt the following notation: if $X$ is a complex variety and $\mathcal{F}$ is a vector bundle on $X$ we denote the zero locus of the generic section of $\mathcal{F}$ as $\mZ(\mathcal{F})\subset X$. If $X$ is $\PP^{n+1}$ and $\mathcal{F}=\of(2)$, hence an n-dimensional quadric, we call it $Q_n$. Given $\varphi$, a morphism between vector bundles on $X$, we denote $D_k(\varphi)\subset X$, the $k$-th degeneracy locus, i.e,\ the locus of points of $X$ where the morphism $\varphi$ has rank (at most) $k$. Moreover, if $X=\Gr(k,n)$ and $\mathcal{F}=\of(1)^{\oplus k(n-k)-d}$, we denote the $d$-fold $\mZ(\mathcal{F})\subset X$ as $X_{n}^d$. Finally, we refer to \red{$K_n$} as the $n$-th Fano fourfold in \cite[Table 3]{kalashnikov}.

\subsubsection*{Acknowledgements}
We would like to thank Pieter Belmans, Marcello Bernardara, Cinzia Casagrande, Elana Kalashnikov, Alexander Kuznetsov, and Laurent Manivel for fruitful conversations and comments.\\
The authors acknowledge the support of the European Union - NextGenerationEU under the National Recovery and Resilience Plan (PNRR) - Mission 4 Education and research - Component 2 From research to business - Investment 1.1 Notice Prin 2022 - DD N. 104 del 2/2/2022, from title "Symplectic varieties: their interplay with Fano manifolds and derived categories", proposal code 2022PEKYBJ – CUP J53D23003840006. The authors are all members of INDAM-GNSAGA.\\
The second author is partially supported by PRIN2020 research grant ``2020KKWT53”, by the MIUR Excellence Department Project awarded to Dipartimento di Matematica, Università di Genova, CUP D33C23001110001, and by the Curiosity Driven 2021 Project \emph{Varieties with trivial or negative canonical bundle and the birational geometry of moduli spaces of curves: a constructive approach} - PNR DM 737/2021, funded by the European Union - NextGenerationEU.\\
The third author wants to thank the Institut de Mathématiques de Toulouse for the hospitality during the last months of this work.

\section{Results of this paper}\label{sec:results}

\subsection{Main results} This paper wants to systematically analyze the families first described by Kalashnikov in \cite{kalashnikov}. 
In the aforementioned paper, she describes 141 families of Fano fourfolds with a new quantum period (out of a database of 749 families) as zero loci in quiver flag manifolds that could not be found as toric complete intersections. In the same database, she describes other 29 families which were not found in toric quiver flag varieties (possibly appearing in another way as a toric complete intersection).
 We have focused on these 170 families and translated them into section zero loci in products of Grassmannians, using the machinery described in \cref{sec:ZlqToZlg}. We can summarize their birational description in the following result:

\begin{thm}\label[thm]{thm:main}
   The 170 families of Fano fourfolds mentioned above can be subdivided into subsets as follows:
    \begin{itemize}
        \item 9 have Picard rank $\rho=1$ and are classical, appearing for example in \cite{ag5,kuchle};
        \item 124 are blow-ups of smooth Fano fourfolds (from the list of the 749) in smooth centers. Of these, 19 also appear in \cite{BFMT} and \cref{F.2.12} is described in \cite{Kuz19};
        \item 8 are products of lower dimensional Fano manifolds;
        \item \cref{F.2.43}, \cref{F.2.48} and \cref{F.3.50} are projective bundles over Fano threefolds;
        \item 8 are conic bundles, of which one appears  in \cite{BFMT};
        \item \cref{F.2.1} and \cref{F.2.35} are stratified projective bundles over Fano threefolds, described in \cite{lang};
        \item 16 are small resolutions of singular fourfolds.

    \end{itemize}
\end{thm}

The above statement needs to be put in context. First of all, this subdivision is obtained by studying the canonical projections on the factors of the ambient spaces. When we say (for example) that 124 families are blow-ups, 11 are conic or projective bundles, we mean that for 124 families at least one canonical projection realizes the fourfold as blow-up of a smooth Fano in a smooth center. In the group of the 11 families that are listed as projective or conic bundle, none of the canonical projections is a blow-down, and at least one of them has the desired structure (and the same for the rest of the subsets).

For the Fano fourfolds that can be described as blow-ups of other Fano fourfolds, we decided to keep only that description to emphasize the birational links between the Fano fourfolds. The 16 Fano fourfolds described only as small resolution are:
\begin{itemize}
    \item \cref{F.2.4},\ \cref{F.2.14},\ \cref{F.2.15},\ \cref{F.2.37},\ \cref{F.2.50},\ \cref{F.2.51},\ \cref{F.2.60},\ \cref{F.2.73},\ \cref{F.3.12},\ \cref{F.3.14},\  \cref{F.3.43} from the list of 141 families.
    \item \cref{F.3.11}, \ \cref{F.3.17}, \ \cref{F.3.20}, \ \cref{F.3.34}, \ \cref{F.4.5} in the extra 29 families.
\end{itemize}
 Looking at the projections to the factors of the ambient spaces of these 16 families of fourfolds we were able to describe them only as small resolutions of singular fourfolds. Thus these families have a more interesting and complicated geometry and deserve a deeper study.   

During the systematic study of the 170 families of fourfolds, we were also able to give some information about their rationality. In particular, we obtained that out of the 170, 129 are rational fourfolds. Unfortunately, we are not able to say more about the rationality of the remaining ones.

We refer to \cref{sec:tab} for the summarizing table with all the information about the deformation invariants and the rationality of these 170 families of Fano fourfolds.

\subsection{Plan of the paper} The paper is divided as follows.
 In \cref{workingTool} we summarize the basic tools to study zero loci of global sections of homogeneous vector bundles in a product of Grassmannians. In particular, we recall the basic techniques to compute Hodge numbers for such varieties and some results from degeneracy loci of morphisms between vector bundles. We end the section with some Lemmas which are shortcuts to describe biregularly section zero loci in products of Grassmannians. In \cref{sec:examples} we give three detailed examples of the main methods we have applied to study the 170 Fano fourfolds. In \cref{sec:pic1}, \cref{sec:pic2}, \cref{sec:pic3} and  \cref{sec:pic4}  we give the biregular description of the Fano fourfolds in the 170 with, respectively, Picard rank one, two, three and four. A summarizing table in provided in \cref{sec:tab}. Finally, in \cref{sec:ZlqToZlg}, authored by the last author and E.\ Kalashnikov, we explain a method to translate the language of zero loci of global sections of homogeneous vector bundles over quiver flag varieties into the language of zero loci of global sections of homogeneous vector bundles over products of Grassmannians, which allows us to perform all the computations explained in \cref{workingTool}.

\section{Working Tools}\label{workingTool}

In this section, we will recall some useful results that help to speed up the birational description of zero loci in products of Grassmannians, some of them are very classical and can be found extensively in literature, e.g., \cite{WJ}, \cite{DFT}, or \cite{BFMT}.

\subsection{Computing Hodge Numbers}
In \cite{DFT,BFMT} it is shown an automatized method to compute the Hodge numbers for zero loci of a general global section of globally generated completely reducible homogeneous vector bundle on products of Grassmannians. Summarizing what is explained in \cite[Section 3]{DFT} it suffices to use a combination of the Koszul complex and the cotangent sequence to obtain an appropriate resolution of the $k-th$ exterior power of the cotangent bundle. In particular, if we have $M=\mZ(\mathcal{F})\subset X$, from the co-normal sequence, we will obtain:
\[
0\to\Sym^{j}\mathcal{F}^{\vee}_{|M}\to(Sym^{j-1}\mathcal{F}^{\vee}\otimes\Omega_X)_{|M}\to...\to(\Sym^{j-k}\mathcal{F}^{\vee}\otimes\Omega_X^{k})_{|M}\to...\to\Omega_{X|M}^{j}\to\Omega_{M}^{j}\to 0
\]
and each term can be resolved by a Koszul complex of the following type:
\[
0\to\bigwedge^{r}\mathcal{F}^{\vee}\otimes\Sym^{j-k}\mathcal{F}^{\vee}\otimes\Omega_{X}^{K}\to...\to\mathcal{F}^{\vee}\otimes\Sym^{j-k}\mathcal{F}^{\vee}\otimes\Omega_X^{k}
\]
And since we are dealing with globally generated completely reducible homogeneous vector bundles, we can apply Borel--Bott--Weil and Littlewood--Richardson Theorems to compute the dimension of the cohomology groups of each resolution. Even if those computations can become quickly cumbersome, we can use some computer algebra software, such as \cite{M2} to automatize and speed up computations.\\
Other important invariants that are computed in this paper are $h^0(-K)$, $(-K)^4$ and $\chi(T_M)$. Since we are dealing with products of Grassmannians we know how to integrate and then Hirzebruch--Riemann--Roch Theorem yields a way to compute $\chi(\mathcal{E})$ for any vector bundles $\mathcal{E}$ with prescribed Chern classes. Small changes can be made to compute those invariants for $M=\mZ(\mathcal{F})\subset X$, and, again, in concrete the computations for the examples are made using a computer algebra routine such as \cite{Schubert2}.

\subsection{Degeneracy Loci} In our setting we will consider zero loci of global sections of globally generated homogenous vector bundles in products of Grassmannians, hence objects of the form $\mZ(\mathcal{F})\subset\prod_{i=1}^{m}\Gr(k_i,n_i)$, with $\mathcal{F}=\bigoplus_{j=1}^s\mathcal{F}_{1,j}\boxtimes...\boxtimes\mathcal{F}_{m,j}$, each factor $\mathcal{F}_{i,j}$ being an irreducible homogeneous vector bundle over $\Gr(k_i,n_i)$. We will be able to study them by looking at the projections on the factors of $\prod_{i=1}^{m}\Gr(k_i,n_i)$, thanks to the following:

\begin{proposition}\label[proposition]{thm:degloc} Let $\Gr_1=\Gr(k_1,V_{n_1})$ and $\Gr_2=\Gr(k_2,V_{n_2})$ be Grassmannians, let $\mathcal{E}_1$ and $\mathcal{E}_2$ be globally generated irreducible homogeneous vector bundles on $\Gr_1$, $\Gr_2$, respectively. Set $\mathcal{E}_1\boxtimes\mathcal{E}_2$ with $H^0(\Gr_1\times\Gr_2,\mathcal{E}_1\boxtimes\mathcal{E}_2)=\Sigma_{\alpha_1}V_{n_1}\otimes\Sigma_{\alpha_2}V_{n_2}$. 
Let $M=\mZ(\mathcal{E}_1\boxtimes\mathcal{E}_2)\subset \Gr_1\times \Gr_2$, then the projection $\pi_1:M\to \Gr_1$ induces one of the following birational relations between $M$ and $\Gr_1$:
\begin{itemize}
    \item if $rk\mathcal{E}_1 \, rk\mathcal{E}_2 \leq k_2(n_2-k_2)$, the generic fiber of $\pi_1$ is $Z_2=\mZ(\mathcal{E}_2^{\oplus rk\mathcal{E}_1})\subset \Gr_2$. Moreover, the exceptional loci of $\pi_1$, if not empty, are the nested loci where the map
    \[
        \varphi:\mathcal{E}_1^{\vee}\to\Sigma_{\alpha_2}V_{n_2}\otimes\of_{\Gr_1}
    \]
    which is induced by the section of $\mathcal{E}_1\boxtimes\mathcal{E}_2$ degenerates to a $min\{rk \mathcal{E}_1, h^0(\Gr_2,\mathcal{E}_2)\}-k$ map, in other words $D_{min\{rk \mathcal{E}_1,h^0(\Gr_2,\mathcal{E}_2)\}-k}(\varphi)$, for $k>0$;
    \item if $rk\mathcal{E}_1 \, rk\mathcal{E}_2 > k_2(n_2-k_2)$, the generic fiber of $\pi_1$ is empty. Moreover, exists a $l \leq min\{rk \mathcal{E}_1,h^0(\Gr_2,\mathcal{E}_2)\}-1$ such that over $D_{l}(\varphi)\subset X$ the fiber is given by the zero loci 
$Z_2=\mZ(\mathcal{E}_2^{\oplus l})\subset \Gr_2$, and each further exceptional locus, if not empty, is given by $D_{l-k}(\varphi)$, $k>0$, with fiber $Z_{2,k}=\mZ(\mathcal{E}_2^{\oplus l-k})\subset \Gr_2$

\end{itemize}
\end{proposition}
\begin{proof}
\begin{itemize}
\item Let us start with the case $rk\mathcal{E}_1 \, rk\mathcal{E}_2\leq k_2(n_2-k_2)$. By the K\"unneth Theorem a global section $\sigma$ in $ H^0(\Gr_1\times\Gr_2,\mathcal{E}_1\boxtimes\mathcal{E}_2)$ can be seen as
\[
\sigma = \pi_1^* \sigma_1 \otimes \pi_2^* \sigma_2 \in H^0(\Gr_1,\mathcal{E}_1) \otimes  H^0(\Gr_2,\mathcal{E}_2) 
\]
and by definition the evaluation satisfies $\sigma(\Pi_1,\Pi_2)=\sigma_1(\Pi_1) \otimes \sigma_2(\Pi_2)$. By adjunction one has
\[
H^0(\Gr_1\times\Gr_2,\mathcal{E}_1\boxtimes\mathcal{E}_2)=\Hom(\pi_1^*\mathcal{E}_1^{\vee},\pi_2^*\mathcal{E}_2)=\Hom(\mathcal{E}_1^{\vee},H^0(\Gr_2,\mathcal{E}_2))
\]
so that $\sigma$ can be regarded as a morphism 
\[
\varphi:\mathcal{E}_1^{\vee}\to\Sigma_{\alpha_2}V_{n_2}\otimes\of_{\Gr_1}.
\]
For any $\Pi_1 \in \Gr_1$, its preimage via $\pi_1$ is given by all points $(\Pi_1,\Pi_2)$ such that all elements of the image of $\varphi(\Pi_1)$ vanish in $\Pi_2$, when regarded as sections of $\mathcal{E}_2$. Over a general $\Pi_1$ we have $rk \mathcal{E}_1$ sections of $\mathcal{E}_2$, which give a non-empty fiber if $rk\mathcal{E}_1 \, rk\mathcal{E}_2\leq k_2(n_2-k_2)$. The exceptional locus is given by the locus where the fibers have higher dimension, and this is exactly where the image of $\varphi$ does not have maximal dimension. This concludes the proof of the first part.
\item The second part is proved essentially as the first one, with the only difference being that the projection $\pi_1$ is not surjective, because the generic fiber is empty. This implies that to obtain the image of $\pi_1$ we have to consider where the morphism $\varphi$ loses enough rank to obtain at least a finite number of points as generic fiber.
We can consider the first degeneracy loci $D \subset \Gr_1$, for which the generic fiber is non-empty and smooth. We can replicate the procedure of the first part, by taking into account the smaller degeneracy loci. 
\qedhere
\end{itemize}

\end{proof}

Another useful tool often used with the above \cref{thm:degloc} is the following Lemma is a consequence of the one proved by Kuznetsov in \cite[Lemma 2.1]{kuznetsovKuchle}:
\begin{lemma}\label[lemma]{thm:degpr}
    Let $\phi:\mathcal{E}\to\mathcal{F}$ be a morphism of vector bundles of ranks $rk\mathcal{F}=f<rk\mathcal{E}=e$ on a Cohen--Macaulay scheme $\mathcal{S}$. Consider the projectivization $p:\PP_S(\mathcal{E})\to S$, then $\phi$ gives a global section of the vector bundle $p^*\mathcal{F}\otimes\of_{\mR}(1)$. Moreover, let $k$ be the greatest integer such that $D_k(\varphi)$ is not empty. Then $D_k(\varphi)\cong\mZ(p^*\mathcal{F}\otimes\of_{\mR}(1))\subset\PP_S(\mathcal{E})$.
\end{lemma}

\begin{rmk}
    Note that \cref{thm:degloc} cannot easily be generalized. In particular, all hypotheses are important. For example, let us consider the threefold $X=\mZ(\mQ_{\PP^3}\boxtimes\mU_{\Gr(2,6)}^{\vee}\oplus\of(0,1)^{\oplus 2})\subset\PP^3\times\Gr(2,6)$.
    We can also write $X$ as $\mZ(\mQ_{\PP^3}\boxtimes\mU^{\vee}_{S_2\Gr(2,6)})\subset\PP^3\times \SGr_2\Gr(2,6)$,
    with $\SGr_2\Gr(2,6)$ the bisymplectic Grassmannian in $\Gr(2,6)$.
    If we apply \cref{thm:degloc} to $X$ written in this way forgetting about the hypothesis on the ambient space, we would get an isomorphism $X$ is $\PP^3$, which is not the case. 
    
    The correct way to apply \cref{thm:degloc} is the following: 
    consider $X'=\mZ(\mQ_{\PP^3}\boxtimes\mU^{\vee}_{\Gr(2,6)})\subset\PP^3\times\Gr(2,6)$. By the first point of \cref{thm:degloc} we get the projection to $\PP^3$ has a generic fiber equal to $\PP^2$. Moreover, the first degeneracy locus is empty. We consider now $\mZ(\of(0,1)^{\oplus 2})\subset X'$. We show that the fiber of the projection to $\PP^3$ is generically a point but becomes a $\PP^1$ along a $\PP^1$. This happens because the bundle $\mQ_{\PP^3}\boxtimes\mU^{\vee}_{\Gr(2,6)}$ induces a map
    \[
    \varphi:V_6\otimes\of_{\PP^3}\to\mQ_{\PP^3}.
    \]
    Since $\varphi$ does not have a degeneracy locus we get the rank three bundle $\mathcal{K}=ker\varphi$, which can be verified to be isomorphic to $\of_{\PP^3}(-1)\oplus\of_{\PP^3}^{\oplus 2}$. By taking the second exterior power of the exact sequence
    \[
0\to\mathcal{K}\xrightarrow{i} V_6\otimes\of_{\PP^3}\xrightarrow{\varphi}\mQ_{\PP^3},
    \] 
    we get the induced injective map
    \[
    \phi:\bigwedge^2\mathcal{K}\to\bigwedge^2 V_6\otimes\of_{\PP^3}.
    \]
    Recall that  $\sigma\in H^0(\PP^3\times\Gr(2,6),\of(0,1)^{\oplus 2})$, induces a map 
    \[
    \tau:\bigwedge^2 V_6\otimes\of_{\PP^3}\to V_2\otimes\of_{\PP^3},
    \]
    obtained by considering the two linearly independent sections.
    Since we are working on $X'$ we have a further map: 
    \[
    \Phi:\bigwedge^2\mathcal{K}\xrightarrow{\bigwedge^2i}\bigwedge^2 V_6\otimes\of_{\PP^3}\xrightarrow{\tau}V_2\otimes\of_{\PP^3}
    \]
    which exists because we have two bundles of the form $\mathcal{E}_1\boxtimes\mathcal{E}_2$. Thus $D_1(\tau_{|X'})=D_1(\Phi)$ and corresponds to the exceptional locus of the projection to $\PP^3$. Since $\mathcal{K}=\of_{\PP^3}(-1)\oplus\of^{\oplus 2}$, then $\bigwedge^2\mathcal{K}=\of(-1)^{\oplus 2}\oplus\of$. By \cref{thm:degpr}, $D_1(\Phi)=\mZ(\of(1,1)^{\oplus 2}\oplus\of(0,1))\subset\PP^1\times\PP^3$, which is $\PP^1$. This concludes the description of $X$ as $\Bl_{\PP^1}\PP^3$. 
\end{rmk}

In this way, we can describe the variety $M$ in \cref{thm:degloc} using the stratification of the degeneracy loci, and over each stratum, we have control of the fiber of the projection. Examples of this phenomenon includes stratified projective bundles (for example, blow ups), where we have local triviality on each stratum. Notice how this is not the only possible case, as one can easily obtain conic bundles as an easy example. Of course the above proposition is merely a general guideline: to understand the geometric picture requires a case-by-case analysis properly.

The point now is to study degeneracy loci of morphisms $\varphi$ between vector bundles on $\Gr(k,n)$, hence some well known tools (see \cite{WJ}) come in handy:

\begin{thm}\label{thm:ddl} Let $X$ be a smooth projective variety, let $\mathcal{E}$ and $\mathcal{F}$ vector bundles on $X$ such that $\mathcal{E}\otimes\mathcal{F}$ is globally generated, and let $\varphi\in H^0(X,\mathcal{E}\otimes\mathcal{F})$ a generic global section. If $D_k(\varphi)\neq\varnothing$, then the dimension of $D_k(\varphi)$ is $m_k=dim\ X -(rk\mathcal{E}-k)(rk\mathcal{F}-k)$.
\end{thm}

\begin{thm}\label{thm:en} 
Let $\varphi:\mathcal{E}\to\mathcal{F}$ a morphism between vector bundles on $M$. Suppose $e=rk\mathcal{E}>f=rk\mathcal{F}$, and suppose $codim D_f(\varphi)=e-f+1$. Then the $\of_M$-module $\of_{D_f(\varphi)}$ admits a locally free resolution given by the Eagon--Northcott complex
\[
0\to\bigwedge^e\mathcal{E}\otimes\Sym^{e-f}\mathcal{F}^{\vee}\otimes det\mathcal{F}^{\vee}\to...\to\bigwedge^{f+1}\mathcal{E}\otimes\mathcal{F}^{\vee}\otimes det\mathcal{F}^{\vee}\to\bigwedge^{f}\mathcal{E}\otimes det\mathcal{F}^{\vee}\to\of_{M}\to\of_{D_{f-1}}(\varphi)\to 0
\]
\end{thm}
For example, we can compute the Hilbert polynomial of $D_{f}(\varphi)$ using this resolution if we deal with a degeneracy locus of dimension lower or equal to two.
The combination of these results makes easier the biregular description of the variety $M$. Indeed, since we are dealing with varieties of dimension 4 the computation of invariants of the degeneracy loci often leads to the complete identification of the exceptional locus of the projection.

\begin{rmk}
    The non-emptiness condition in Proposition \ref{thm:degloc} is fundamental. As an example, we can consider the zero locus $\mZ(\mQ \boxtimes\mU^{\vee})\subset \PP^4 \times \Gr(2,6)$. A direct inspection, or directly \cref{flagCor}, shows how this fourfold is simply $\PP^4$. However, a careless application of Proposition \ref{thm:degloc} would point out to the first projection $\pi_1$ to $\PP^4$ as only birational, with $D_3(\varphi)$ as exceptional locus of dimension 1. On the other hand, we can show directly how this degeneracy locus of the map $\varphi: \mQ^{\vee} \to \of^{\oplus 6}$ is empty.

From the linear algebra viewpoint, a global section of $\mQ \boxtimes\mU^{\vee}$ is given by $\alpha \in V_6^{\vee} \otimes V_5$, or $\alpha: V_6 \to V_5$, which is of maximal rank by genericity. We denote its 1-dimensional kernel by $K$. The zero locus of $\alpha$ corresponds to the pairs $(l, \Pi) \in \PP^4 \times \Gr(2, 6)$ with $\alpha(\Pi) \subset l$. For a fixed $l$, the fiber is thus $\Pi=\langle l, K \rangle$, hence a point. The first degeneracy loci $D_3(\varphi)$ corresponds to the pairs $(l, \Pi)$ with $\alpha(\Pi)=0$. But this cannot happen, since a generic $\alpha$ has a 1-dimensional kernel. A less direct and quicker proof of this fact can be found also by computing the invariants of $D_3(\varphi)$ by using the Eagon--Northcott complex as in \ref{thm:en}.
\end{rmk}

\subsection{Birational tricks} Here we recollect some classical results and variations that naturally arise when we deal with zero loci in products of Grassmannians. 

Throughout the paper, the concept of Grassmann bundle is central, since $\Gr(k_1,n_1)\times\Gr(k_2,n_2)$ can be also seen as $\bbGr_{\Gr(k_1,n_1)}(k_2,\of_{\Gr(k_1,n_1)}^{\oplus n_2})$. Thus, if we have a vector bundle $\mathcal{D}$ on $\Gr(k_1,n_1)$, the behavior of varieties like $\mZ(\mathcal{D}\boxtimes\mU^{\vee}_{\Gr(k_2,n_2)})\subset\Gr(k_1,n_1)\times\Gr(k_2,n_2)$ is important itself:
 
\begin{thm}\label[thm]{thm:seq} 
Let $\mathcal{F}$, $\mathcal{E}$ and $\mathcal{D}$ vector bundles on $\Gr(k,n)$, such that: 
\[
0\to\mathcal{F}\to\mathcal{E}\to\mathcal{D}\to0
\]
is a short exact sequence of vector bundles. Let $X=\bbGr_{\Gr(k,n)}(h,\mathcal{F})$, then $X=\mZ(\mU^{\vee}_{\mR}\boxtimes\mathcal{D})\subset\bbGr_{\Gr(k,n)}(h,\mathcal{E})$.
\end{thm}

The next lemma on the other hand is in some sense more general and gives a precise description of the geometry of $M$, a complete proof of this can be found in \cite[Lemma 3.1]{BFMT}:

\begin{lemma}\label[lemma]{lm:blowPPtX} 
Let $X$ be a smooth projective variety with a globally generated line bundle $\mathcal{L}$, and consider $Y=\mZ(\mQ_{\PP^n}\boxtimes\mL)\subset\PP^n\times X$. Then $Y=\Bl_Z$, where $Z=\mZ(\mL^{\oplus n+1})\subset X$.
\end{lemma}
These kinds of results are often referred to as Cayley tricks and a sort of extension of those in the Grassmannian world are the following, which are proved in \cite[Lemma 2.2]{DFT} and \cite[Prop. 3.3]{BFMT}:

\begin{lemma}\label[lemma]{lm:blowFlagGr} Let us consider $\Gr_1=\Gr(k_1,n_1)$, $\Gr_2=\Gr(k_2,n_2)$, $M=\mZ(\mQ_{\Gr(k_1,n_1)}\boxtimes\mU^{\vee}_{\Gr(k_2,n_2)})\subset\Gr_1\times\Gr_2$, then we have: 
\begin{enumerate}
    \item\label{item_1} if $n_2=n_1$ and $k_2<k_1$ then $M=\Fl(k_2,k_1,n)$;
    \item\label{item_2} if $k_1=k_2$ and $n_1=n_2-1$ then $M=\Bl_{\Gr(k_1-1,n_1)}\Gr_2$, where the center of the blow-up $\Gr(k_1-1,n_1)$ is identified with $\mZ(\mQ_{\Gr_2})\subset\Gr_2$;
    \item\label{item_3} if $k_2=k_1$ and $n_2=n_1+2$ then outside a point $p$, $M$ is isomorphic to $\Bl_{\tilde{D}}\Gr_2$, with $\tilde{D}\cong\mZ(\mQ_{\PP^{n-1}}\boxtimes\mU^{\vee}_{\Gr_2})\subset\PP^{n-1}\times\Gr_2$
\end{enumerate}
    
\end{lemma}

The first point of the previous lemma is the key to obtaining a chain of results that we summarize in the following:

\begin{corollary}\label[corollary]{flagCor} Let us consider $\Gr_1=\Gr(k_1,n_1)$, $\Gr_2=\Gr(k_2,n_2)$, $M=\mZ(\mQ_{\Gr(k_1,n_1)}\boxtimes\mU^{\vee}_{\Gr(k_2,n_2)})\subset\Gr_1\times\Gr_2$, then we have:
    \begin{enumerate}
        \item\label{cor:flag_1} if $k_2<k_1<n_1$ and $n_2=n_1-1$ then $M=\mZ(\mU^{\vee}_1)\subset\Fl(k_2,k_1,n_1)$;
        \item\label{cor:flag_2} if $k_2<k_1<n_1$ and $n_1=n_2-1$ then $M=\mZ(\mQ_2)\subset\Fl(k_2,k_1+1,n_2)$;
        \item\label{cor:flag_3} if $k_2\geq k_1$ and $n_2=n_1+c$ for every $c\geq 0$ then $M=\bbGr_{\Gr_1}(k_2,\mU_{\Gr_1}\oplus\of_{\Gr_1}^{\oplus c})$.
    \end{enumerate}
\end{corollary}
\begin{proof} The proof of these results is very straightforward:
    \begin{enumerate}
        \item recall that $\Gr(k,n) \supset\mZ(\mU^{\vee}_{\Gr(k,n)}) \cong \Gr(k,n-1)$, hence $\Gr(k_1,n-1) \times \Gr(k_2,n) \supset \mZ(\mU^{\vee}_{\Gr(k_1,n-1)} \boxtimes\mQ_{\Gr(k_2,n)})\cong\mZ(\mU^{\vee}_{\Gr(k_1,n)}\oplus\mU^{\vee}_{\Gr(k_1,n)}\boxtimes\mQ_{\Gr(k_2,n)})\subset\Gr(k_1,n)\times\Gr(k_2,n)$. Now we apply \crefpart{lm:blowFlagGr}{item_1}
    and we are done;
        \item  Recall that $\Gr(k,n) \supset\mZ(\mQ_{\Gr(k,n)}) \cong \Gr(k-1,n-1)$, hence arguing as in \crefpart{flagCor}{item_1} we are done; 
        \item Note that $\Gr(k_1,n+c)\times\Gr(k_2,n)$ can be written as $\bbGr_{\Gr(k_2,n)}(k_1,\of_{\Gr(k_2,n)}^{\oplus n+c})$, hence we can rewrite $\mZ(\mU^{\vee}_{\Gr(k_1,n+c)}\boxtimes \mQ_{\Gr(k_2,n)})\subset\Gr(k_1,n+c) \times \Gr(k_2,n)$ as $\mZ(\mU^{\vee}_{\mR}\boxtimes\mQ_{\Gr(k_2,n)})\subset\bbGr_{\Gr(k_2,n)}(k_1,\of_{\Gr(k_2,n)}^{\oplus n+c})$, and by applying \cref{thm:seq} we complete the proof. \qedhere
    \end{enumerate}
\end{proof}

The last result of this flavor we apply frequently throughout the paper is \cite[Lemma 2.5]{DFT} 
\begin{lemma}\label[lemma]{lm:flag2.5}
    Let $\Fl(k_1,k_2,n)$ be a two-step flag. We have the following identifications:
    \begin{itemize}
    \item $\Fl(k_1,k_2,n)\cong\bbGr_{\Gr(k_2,n)}(k_1,\mU)\cong\bbGr_{\Gr(k_1,n)}(k_2-k_1,\mQ(-1))$;
    \item $\mZ(\mQ_2)\subset\Fl(k_1,k_2,n)\cong\bbGr_{\Gr(k_2-1,n-1)}(k_1,\mU_{\Gr(k_2-1,n-1)}\oplus\of)$
    \item $\mZ(\mU^{\vee}_1)\subset\Fl(k_1,k_2,n)\cong\bbGr_{\Gr(k_1,n-1)}(k_2-k_1,\mQ_{\Gr(k_1,n-1)}\oplus\of(-1))$
    \end{itemize}
\end{lemma}

\subsection{Conic Bundles}
Some of the Fano fourfolds we have studied have a natural structure of conic bundles. Hence to describe them in the most detailed way we analyzed also their determinant locus, for this, we refer to \cite[Lemma 3.5]{BFMT}:

\begin{lemma}\label[lemma]{lm:conicbdl}
Let $f:X\to B$ be a conic bundle, given by a line bundle $\mathcal{K}$, a rank three bundle $\mathcal{E}$ on $B$ and a bilinear map $\mathcal{K}^{\vee}\to\Sym^2\mathcal{E}^{\vee}$. Then the degeneracy loci of the conic bundle are given by the discriminant divisor $\Delta$ and its codimension two singular locus $\Delta_{sing}$. In particular, let $k:=c_1(\mathcal{K})$ and $c_i:=c_i(\mathcal{E})$, then 
\[
[\Delta]=2c_1+3k\ and\ [\Delta_{sing}]=4(k^3+2k^2c_1+kc_1^2+kc_2+c_1c_2-c_3)
\]
\end{lemma}

\section{A few worked examples}\label{sec:examples}

In this section, we want to give four examples of the 170 varieties we studied. These are chosen to explain in detail the different methods applied throughout the systematic description of these manifolds.
\begin{ex}\label[ex]{ex:degLoc}
\vspace{0.2cm}
\textbf{Fano 2-1}
\begin{itemize}
\item $X=\mZ(\of(1,1)\oplus\of(0,1)^{\oplus 3})\subset\PP^2\times\Gr(2,5)$;
\item Invariants: $h^0(-K)=39, \ (-K)^4=160,\ \chi(T_X)=-9,\ h^{1,1}=2,\ h^{3,1}=0,  \ h^{2,2}=7$;
\item Description: $X$ is a $\PP^1$-bundle on $X_5^3$ which degenerates to a $\PP^2$-bundle over five points. 
\end{itemize}

\subsubsection*{Identification} 

In this example we write down all the explicit computations involving degeneracy loci, which are the most frequent in the paper.

Let us consider the Fano fourfold $\red{K_{508}}$. We can describe it as a $\PP^1$-bundle over $X_5^3$, whose fibers jump to a $\PP^2$ on five points. In \cite{kalashnikov} this Fano fourfold corresponds to the following datum:
\[
     A=\begin{bmatrix}
        0 & 3 & 0\\
        0 & 0 & 5\\
        0 & 0 & 0
    \end{bmatrix},\ \ r=[1,1,2],\ \ E=(\Sym^2\blacksquare,\bigwedge^2\square)^{\oplus 3}\oplus(\blacksquare,\bigwedge^2\square)
\]
Now this means that we have a quiver of this form:
\[
\begin{tikzcd}
1 \arrow[r, "3"] & 1 \arrow[r, "5"] & 2
\end{tikzcd}
\]
Applying \cref{QuivToGrass} it is clear that the left part of the quiver corresponds to $\PP^2$ while the second part defines a Grassmann bundle over that $\PP^2$ as $\bbGr_{\PP^2}(2,\of(-1)^{\oplus 5})$. 
In this way $\red{K_{508}}$ can be written also as $X=\mZ(\of(-1)\boxtimes\of_{\mR}(1)\oplus\of(-2)\boxtimes\of_{\mR}(1)^{\oplus 3}\subset\bbGr_{\PP^2}(2,\of(-1)^{\oplus 5})$. Now twisting $\of(-1)^{\oplus 5}$ with $\of(1)$, we obtain that $X=\mZ(\of(1,1)\oplus\of(0,1)^{\oplus 3})\subset\PP^2\times\Gr(2,5)$.

We are in the correct setting to use the combination of Koszul, co-normal sequence, Borel--Bott--Weil, and Littlewood--Richardson Theorems to compute Hodge numbers. Since we are dealing with Fano fourfolds, the only interesting Hodge numbers are $h^{0,0}(X)$, $h^{1,1}(X)$, $h^{2,1}(X)$, $h^{3,1}(X)$ and $h^{2,2}(X)$. Using Borel--Bott--Weil and Littlewood--Richardson Theorems with some diagram chasing we can compute the following numbers:
\[
h^{0,0}=1,\ h^{1,1}=2,\ h^{2,1}=0,\ h^{3,1}=0,\ h^{2,2}=7.
\]
The other important invariants to compute are $h^0(-K)$, $vol(X)=(-K)^4$ and $\chi(T_X)$. For the first, we simply notice that using the adjunction formula, Koszul complex and Borel--Bott--Weil, we obtain that in this case $h^0(-K)=39$. For the volume, we can use the Hirzebruch--Riemann--Roch Theorem and Chern classes and get $vol(X)=160$. To compute $\chi(T_X)$ we either use Hirzebruch--Riemann--Roch or, since we are dealing with Fano fourfolds, we can use Kodaira vanishing, and we simply get that $\chi(T_X)=h^0(X, T_X)-h^1(X, T_X)$. In this way, we obtain $\chi(T_X)=-9$.
Those computations can become quite cumbersome very quickly, for this reason, and for the large number of computations we have to make we intensively use computer algebra software as \cite{M2}.

We can now proceed with the birational description of $X$. Let us consider the natural projection $\pi_2$ on $\Gr(2,5)$
\[
    \pi_{2|X}:X\to\Gr(2,5).
\]
Note that the fiber of $\pi_{2|X}$ is generically empty, this is because in the description of $X$ as zero locus of section in product of Grassmannians, the ambient $\PP^2\times\Gr(2,5)$ is cut by three generic sections of $\of(0,1)$, which is $\pi_2^*\of_{\Gr(2,5)}(1)$. Hence $X$ can be also written as $\mZ(\of(1,1))\subset\PP^2\times X_5^3$, with $X_5^3=\mZ(\of(1)^{\oplus 3})\subset\Gr(2,5)$ the Fano threefold \red{1-15} in \cite[Table 1]{DFT}. The projection $\pi_2$ can be rewritten as 
\[
    \pi_{2|X}:X\to X_5^3.
\]
 For each point $p\in X_5^3$ the generic fiber $\pi_{2|X}^{-1}(p)$ is $\PP^2$ cut by a section of $\of_{\PP^2}(1)\otimes\pi_2^*\of_{\Gr(2,5)}(1)(p)$, i.e., $\of_{\PP^2}(1)\otimes\mathbb{C}=\of_{\PP^2}(1)$. Thus via the projection $\pi_2$, $X$ is generically a $\PP^1$-fibration over the Fano threefold $X_3^3$. We could achieve the same description also by rewriting $\PP^2\times\Gr(2,5)$ as $\PP_{\Gr(2,5)}(\of^{\oplus 3})$, so $X=\mZ(\of_{\mR}(1)\boxtimes\of_{\Gr(2,5)}(1)_{X_5^3})\subset\PP_{X_5^3}(\of_{X_5^3}^{\oplus 3})$. Twisting $\of_{X_5^3}^{\oplus 3}$ with $\of_{X_5^3}(-1)$ we get $X\subset\mZ(\of_{\mR}(1))\subset\PP_{X_5^3}(\of_{X_5^3}(-1)^{\oplus 3})$, which is generically a $\PP^2$ bundle over $X_5^3$ whose generic fiber is cut by a section of the relative bundle $\of_{\mR}(1)$, obtaining again the previous description.

 This description is made with the genericity assumption, hence it is a birational description of $X$. Until this point we have that $X$ is birational to a $\PP^1$-fibration over $X_5^3$, namely $\hat{X}=\mZ(\of_{\mR}(1))\subset\PP_{\Gr(2,5)}(\of(-1)^{\oplus 3})$, via the restriction of the natural projection $\pi_2$. But we can say more about the exceptional locus of $\pi_{2|X}$. Note that $X=\mZ(\sigma)\subset\PP^2\times X_5^3$, where $\sigma\in H^0(\PP^2\times X_5^3,\of(1,1))=V_3^{\vee}\otimes \bigwedge^2 V_5^{\vee}$. Moreover, when we look for the exceptional locus of $\pi_{2|X}$ we are studying for which $p\in X_5^3$, $\sigma(l,p)=0$ for all $l\in\PP^2$. This fact can be translated into the language of degeneracy loci. Indeed, $\sigma\in V_3^{\vee}\otimes \bigwedge^2 V_5^{\vee}$, so can be written also as $v\otimes(w_1\wedge w_2)$. The condition $\sigma(l,p)=0$ for all $l\in \PP^2$ is equivalent of saying that  we are looking for $p\in X_5^3$, such that $v\otimes(w_1\wedge w_2)(p)=0$, but $w_1\wedge w_2$ is a global section of $\of_{X_5^3}(1)$, while $v$ can be seen as a global section of $V_3^{\vee}\otimes\of_{X_5^3}$, since can not vanish on points of $X_5^3$. Thus the problem can be translated into finding points of $X_5^3$ for which the global section $\gamma\in H^0(X_5^3,V_3^{\vee}\otimes\of_{1})$ is zero or where the morphism
 \[
    \varphi:\of(-1)\to V_3^{\vee}\otimes\of
 \]
 is of rank 0. This is exactly a degeneracy locus of a morphism between vector bundles on $X_5^3$, in particular, we get that the exceptional locus of $\pi_{2|X}$ is equivalent to $D_0(\varphi)$. By \cref{thm:ddl} $D_0(\varphi)$ is a finite number of points and we can apply Eagon--Northcott Theorem to compute the exact number:
 \[
 0\to\of_{X_5^3}(-2)\to\of_{X_5^3}(-2)^{\oplus 2}\to\of_{X_5^3}(-1)^{\oplus 3}\to\of_{X_5^3}\to\of_{D_0(\varphi)}\to 0.
 \]
 By the fact that the alternating sum of the Euler characteristic of vector bundles in an exact sequence is equal to zero, we get that
 \[
 \chi(\of_{D_0(\varphi)})=\chi(\of_{X_5^3})-3\chi(\of_{X_5^3}(-1))+\chi(\of_{X_5^3}(-2)).
 \]
The right-hand side is easy to compute since $\chi(\of_{X_5^3})=4$ and the others can be obtained by some diagram chasing combined with the usual pair of Borel--Bott--Weil and Littlewood--Richardson Theorems. In particular $\chi(\of_{X_5^3}(-1))=0$ and $\chi(\of_{X_5^3}(-2)=1$, so $D_0(\varphi)=\{5\ points\}$. In this way, we have the full biregular description of $X$ as a $\PP^1$-fibration over $X_5^3$, which jumps to a $\PP^2$ over five points.

To find the exceptional locus we could have reasoned in another way, which is not always fruitful, but sometimes makes computations easier. Instead of using the Eagon--Northcott complex we could have observed that the degeneracy locus of $\varphi$ is $\mZ(\of(1)^{\oplus 3})\subset\PP_{X_5^3}(\of)$, which is simply $X_5^3\cap H\cap H'\cap H''=\{5\ points\}$, with $H,\ H',\ H''$ three generic hyperplanes. This completes the most used method to give the biregular description throughout the paper.

We can say something about the rationality of this manifold. If we consider the natural projection $\pi_1$ to $\PP^2$ and we restrict it to $X$, then it is easy to see that it is surjective and the generic fiber is $X_5^3$ cut by a generic section of $\of_{\Gr(2,5)}(1)$, which is a $\DP_5$. This means that $\pi_{1|X}$ induce a $\DP_5$ fibration on $\PP^2$, hence $X$ is birational to a $\DP_5$-fibration, and by \cite{isko}, this suffices to conclude that $X$ is rational. \\
Moreover, as we have already seen $X$ can be written as $\mZ(\of_{X_5^3}\boxtimes\of_{\mR}(1))\subset\PP_{X_5^3}(\of^{\oplus 2})$. If we consider the sequence of sheaves:
\[
0\to\mathcal{K}\to\of_{X_5^3}^{\oplus 3}\to\of_{X_5^3}(1)
\]
and we apply \cref{thm:seq}, we get that $X$ can be seen as $\PP_{X_5^3}(\mathcal{K})$. If we dualize the previous sequence, we obtain $\mathcal{E}=\mathcal{K}^{\vee}$, defined by
\[
0\to\of_{X_5^3}(-1)\to\of_{X_5^3}^{\oplus 3}\to\mathcal{E}\to 0
\]
and this is the model of Fano fourfold described in \cite[Theorem 8.2.3]{lang}.
\end{ex}

\begin{ex}\label[ex]{ex:conicBdl}
\textbf{Fano 2.13} 
\begin{itemize}
    \item $X=\mZ(\mQ_{\Gr(3,5)}(1,0)\oplus \of(0,1)^{\oplus 3}\oplus \of(2,0))\subset\PP^4\times\Gr(3,5)$;
    \item Invariants: $h^0(-K)=30, \ (-K)^4=116,\ \chi(T_X)=-11, \ h^{1,1}=2,\ h^{3,1}=0,\ h^{2,2}=10$;
    \item Description: The projection $\pi_1: X \to Q_3 $ is a conic bundle with discriminant a sextic surface in $Q_3$, singular in 40 points. The second projection $\pi: X \to X_5^3$ is a conic bundle discriminant a $K3$ surface singular in eight points.
\end{itemize}

\subsubsection*{Identification} 

In this example, we write the computation for the conic bundles described in the paper.

Let us consider the Fano fourfold $\red{K_{582}}$ we can describe it as a conic bundle over $X_5^3$, with discriminant a K3 surface singular in eight points. In \cite{kalashnikov} this fourfold is given as the datum of:
    \[
         A=\begin{bmatrix}
        0 & 0 & 5\\
        0 & 0 & 0\\
        0 & 1 & 0
    \end{bmatrix},\ \ r=[1,1,3],\ \ E=(\varnothing,\bigwedge^3\square)^{\oplus 3}\oplus(\Sym^2\square,\varnothing).
    \]
This means that we have a quiver of the following form:
\[
\begin{tikzcd}
1 \arrow[r, "5"] & 1 \arrow[r, "1"] & 3
\end{tikzcd}
\]
Applying \cref{QuivToGrass} we obtain that the left part of the quiver corresponds to $\Gr(2,5)$ and the second part defines a projective bundle over that $\Gr(2,5)$ as $\PP_{\Gr(3,5)}(\mU)$. Hence \red{$K_{582}$} can be rewritten as $X=\mZ(\of_{\mR}(2)\oplus\of_{\Gr(3,5)}(1)^{\oplus 3})\subset\PP_{\Gr(3,5)}(\mU)$. If we use \cref{thm:seq} with the sequence: 
\[
    0\to\mQ\to V_5\to\mU\to 0
\]
then $X=\mZ(\mQ_{\Gr(3,5)}(1,0)\oplus\of(2,0)\oplus\of(0,1)^{\oplus 3})\subset\PP^4\times\Gr(3,5)$. Dualizing $\Gr(3,5)$ we get $X=\mZ(\mU^{\vee}_{\Gr(2,5)}(1,0)\oplus\of(2,0)\oplus\of(0,1)^{\oplus 3})\subset\PP^4\times\Gr(2,5)$, and applying the same strategy as the previous examples we get the following invariants:
\[
    h^0(-K)=30,\ (-K)^{4}=116,\ \chi(T_X)=-11,\ h^{0,0}=1,\ h^{1,1}=2,\ h^{2,1}=0,\ h^{3,1}=0,\ h^{2,2}=10. 
\]
To give the birational description let us consider $X=\mZ(\of_{\mR}(2)\oplus\of_{\Gr(3,5)}(1)^{\oplus 3})\subset\PP_{\Gr(3,5)}(\mU)$. The three sections of $\of(0,1)$ cut on $\Gr(3,5)$ a Fano threefold $X_5^3$, while the section $\of_{\mR}(2)$ cuts on the fibers of the $\PP^2$ fibration a conic, hence $X$ is a conic bundle on $X_5^3$. We want to say more about the discriminant locus $\Delta$ of this conic bundle, and to do so we apply \cref{lm:conicbdl}. The rank three bundle in this case is $\mU_{\Gr(3,5)}$ while the line bundle is the trivial one. By computation of Chern classes, we obtain:
\[
[\Delta]=2c_1(\mU_{\Gr(3,5)|X_5^3})+3c_1(\of_{\Gr(3,5)|X_5^3})=2c_1(\mU_{\Gr(3,5)|X_5^3})=2H
\]
Which is a surface $S$ of degree two inside $X_5^3$, which means that $S=\mZ(\of(2)\oplus\of(1)^{\oplus 3})\subset\Gr(2,5)$, hence a $K3$ surface, singular along
\[
[\Delta_{sing}]=4(c_1(\of_{\Gr(3,5)|X_5^3})^3+2c_1(\of_{\Gr(3,5)|X_5^3})^2c_1(\mU_{\Gr(3,5)|X_5^3})+c_1(\of_{\Gr(3,5)|X_5^3})c_1(\mU_{\Gr(3,5)|X_5^3})^2+\]
\[+c_1(\of_{\Gr(3,5)|X_5^3})c_2(\mU_{\Gr(3,5)|X_5^3})+c_1(\mU_{\Gr(3,5)|X_5^3})c_2(\mU_{\Gr(3,5)|X_5^3})-c_3(\mU_{\Gr(3,5)|X_5^3}))=\]
\[=4(c_1(\mU_{\Gr(3,5)|X_5^3})c_2(\mU_{\Gr(3,5)|X_5^3})-c_3(\mU_{\Gr(3,5)|X_5^3}))=8H^2
\]
So in the end we have that $X$ is a conic bundle over $X_5^3$ with discriminant $\Delta$ a $K3$ surface singular in eight points.

 The other projection gives another conic bundle over $Q_3$. In particular $X$ can be rewritten as $\mZ(\of_{\mR}(1)^{\oplus 3})\subset\bbGr_{Q_3}(2,\mQ_{\PP^4|Q_3}(-1))$, hence in this case the rank three bundle $\mathcal{E}$ is obtained as the kernel bundle:
\[
0\to\mathcal{E}\to\bigwedge^2 Q^{\vee}(1)\to\of_{Q_3}^{\oplus 3}. 
\]
Here $\mathcal{K}=\of_{Q_3}$ hence using \cref{lm:conicbdl} we get that the discriminant $\Delta$ is a sextic surface in $Q_3$, singular in $\Delta_{sing}=\{40\ points\}$.

\end{ex}

\begin{ex}\label[ex]{ex:mixQuiv}
\textbf{Fano 3.24} 
\begin{itemize}
    \item $X=\mZ(\mQ_{\PP^2}\boxtimes\mU^{\vee}_{\Gr(2,5)}\oplus\mQ_{\Gr(2,5)}(1,0,0)\oplus\mU^{\vee}_{\Gr(2,5)}(1,0,0))\subset\PP^5\times\PP^2\times\Gr(2,5)$;
    \item Invariants: $h^0(-K)=53, \ (-K)^4=235,\ \chi(T_X)=-1,\ h^{1,1}=3,\ h^{3,1}=0,  \ h^{2,2}=7$;
    \item Description: $X=\Bl_{\DP_6}Y$, with $Y$ the Fano \red{6} in \cite[Table 3]{kalashnikov}.
\end{itemize}

\subsubsection*{Identification} 

In this last example, we want to show another useful method, which combines techniques from quiver flag manifolds and degeneracy loci methods.

Let us consider the Fano fourfold $\red{K_{212}}$. We will describe it as $\Bl_{\DP6}Y$, with $Y$ the Fano fourfold $\PP_{\PP^2}(\of\oplus\of(-1)^{\oplus 2})$. In \cite{kalashnikov}, $\red{K_{212}}$ is described as:
\[
      A=\begin{bmatrix}
        0 & 1 & 3 & 2\\
        0 & 0 & 0 & 0\\
        0 & 0 & 0 & 1\\
        0 & 1 & 0 & 0
    \end{bmatrix},\ \ r=[1,1,1,2],\ \ E=(\square,\varnothing,\square).
\]
Then this is associated with the quiver of the following form:
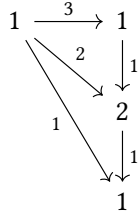
\[
    \begin{tikzcd}
1 \arrow[r, "3"] \arrow[rd, "2"] \arrow[rdd, "1"'] & 1 \arrow[d, "1"] \\
                                                   & 2 \arrow[d, "1"] \\
                                                   & 1               
\end{tikzcd}
\]
Applying \cref{QuivToGrass} we can see that the upper arrow gives a $\PP^2$, while the first triangle gives a Grassmann bundle on that $\PP^2$, $\bbGr_{\PP^2}(2,\of_{\PP^2}(-1)\oplus\of_{\PP^2}^{\oplus 2})$. The lowest triangle induces a projective bundle on the above Grassmann bundle, in particular, \red{$K_{212}$} can be rewritten as $X=\mZ(\mU^{\vee}_{\mR1}\boxtimes\of_{\mR_2}(1))\subset\PP_{\bbGr_{\PP^2}(2,\of_{\PP^2}(-1)\oplus\of_{\PP^2}^{\oplus 2})}(\mU_{\mR_1}\oplus\of)$, with the subscripts $\mR_1$ and $\mR_2$ referred to the relative bundles of Grassmann bundle and the projective bundle respectively.

Starting from the Grassmann bundle we can apply \cref{thm:seq} using the Euler:
\[
0\to\of_{\PP^2}(-1)\to V_3\to\mQ_{\PP^2}\to 0
\]
In this way, $X$ can be rewritten as $\mZ(\mQ_{\PP^2}\boxtimes\mU^{\vee}_{\Gr(2,5)}\oplus\mU^{\vee}_{\Gr(2,5)}\boxtimes\of_{\mR_2}(1))\subset\PP^2\times\PP_{\Gr(2,5)}(\mU_{\Gr(2,5)}\oplus\of)$. Now we can apply the Euler sequence for Grassmannians on the second factor of the ambient space:
\[
    0\to\mU_{\Gr(2,5)}\to V_5\to\mQ_{\Gr(2,5)}\to 0
\]
obtaining in the end $X=\mZ(\mQ_{\PP^2}\boxtimes\mU^{\vee}_{\Gr(2,5)}\oplus\mQ_{\Gr(2,5)}(1,0,0)\oplus\mU^{\vee}_{\Gr(2,5)}(1,0,0))\subset\PP^5\times\PP^2\times\Gr(2,5)$.

We can now use the usual routine to compute:
\[
    h^0(-K)=53,\ (-K)^{4}=235,\ \chi(T_X)=-1,\ h^{0,0}=1,\ h^{1,1}=3,\ h^{2,1}=0,\ h^{3,1}=0,\ h^{2,2}=7.
\]

We want to give now the birational description of $X$. Note that if we project $X$ via:
\[
    \pi_{2,3|X}:X\to\PP^2\times\Gr(2,5)
\]
the image of this map is $Y=\mZ(\mQ_{\PP^2}\boxtimes\mU^{\vee}_{\Gr(2,5)})\subset\PP^2\times\Gr(2,5)$ and for the generic $y\in Y$, $\pi_{2,3}^{-1}(y)$ is $\PP^5$ by $\mZ(\of_{\PP^5}\otimes\pi_{2,3}^*(\pi_{3}^{*}\mQ_{\Gr(2,5)})(y))$ and $\mZ(\of_{\PP^5}\otimes\pi_{2,3}^*(\pi_{3}^{*}\mU_{\Gr(2,5)}^{\vee})(y)))$. Since $\pi_{3}^*\mQ_{\Gr(2,5)}$ is a rank three bundle over $Y$ while $\pi_3^*\mU^{\vee}_{\Gr(2,5)}$ is a rank two bundle over $Y$, then the generic fiber $\pi_{2,3}^{-1}(y)$ is $\PP^5$ cut by five sections of $\of_{\PP^5}(1)$, hence a point. With this argument, we have shown that $X$ is birational, via $\pi_{2,3|X}$, to $Y=\mZ(\mQ_{\PP^2}\boxtimes\mU^{\vee}_{\Gr(2,5)})\subset\PP^2\times\Gr(2,5)$. Now $Y$ can be rewritten as $\bbGr_{\PP^2}(2,\of(-1)\oplus\of^{\oplus 2})$, dualizing the fiber of this Grassmannian bundle, we get that $Y=\PP_{\PP^2}(\of(1)\oplus\of^{\oplus 2})$, now we can twist with $\of(-1)$ and obtain $\PP_{\PP^2}(\of\oplus\of(-1)^{\oplus 2})$.

Recall that, using \cref{QuivToGrass} in the opposite direction, we get that $Y$ is the quiver flag variety obtained from the datum of
\[
    A=\begin{bmatrix}
        0 & 1 & 3 \\
        0 & 0 & 0 \\
        0 & 2 & 0 
    \end{bmatrix},\ \ r=[1,1,1].
\]
which corresponds to \red{$K_{6}$}. In this way, we have recovered that $X$ is birational to an already known Fano fourfold. Note that $Y=\PP_{\PP^2}(\of\oplus\of(-1)^{\oplus 2})$ firstly appears as the fifth case of the main theorem in \cite{SW2}. In particular, $Y$ can be described as the blow-up of the cone in $\PP^6$ over $\PP^1\times\PP^2\subset\PP^5$ along its vertex.

To complete the biregular description it suffices to compute the exceptional locus of the map:
\[
    \pi_{2,3|X}:X\to Y.
\]
We proceed as in \cref{ex:degLoc} to compute the invariants for the exceptional locus, and, using \cref{thm:ddl} we obtain that it is a surface while using \cref{thm:degloc} we get it can be described as $S=\mZ(\mQ_{\Gr(2,5)}(1,0,0)\oplus\mU^{\vee}_{\Gr(2,5)}(1,0,0))\subset\PP^3\times Y$. This lets us compute the canonical bundle $K_S$ via adjunction formula and the other usual invariants are:
\[
    K_S=\of(-1,-1)_{|S},\ h^0(-K)=7,\ (-K)^{2}=6,\ h^{0,0}=1,\ h^{1,1}=4.
\]

Now the invariants are the same as the ones of a $\DP_6$, and the anticanonical is ample, hence by classification we are done. Since over $S$ the fibers are a $\PP^1$, then $X=\Bl_{\DP_6}Y$.
\begin{rmk}
This Fano fourfold has the same invariants as the Fano fourfold \red{$X^5_{0,1}$} in \cite[Table 2]{secci}. Recall that \red{$X^5_{0,1}$} can be also written as $\Bl_{\DP_5}\PP_{X_5^3}(\of_{X_5^3}\oplus\of_{X_5^3}(-1))$, while we have described \red{$K_{508}$} as $\Bl_{\DP_6}\PP_{\PP^2}(\of_{\PP^2}\oplus\of_{\PP^2}(-1)^{\oplus 2})$. These fourfolds are likely isomorphic, but we could not find a proof for it.
\end{rmk}

\end{ex}
\section{Fano Varieties with Picard Rank 1}\label{sec:pic1}
\renewcommand\thefano{1--\arabic{fano}}
\setcounter{fano}{0}

\begin{fano}
\fanoid{$K_{742}$}
\label[mystyle]{F.1.1}
$X=\mZ(\of(1)\oplus\of(3))\subset\Gr(2,5)$
\subsubsection*{Invariants} $h^0(-K)=9, \ (-K)^4=15,\ \chi(T_X)=-109, \ h^{1,1}=1,\ h^{3,1}=41,  \ h^{2,2}=232$.
\subsubsection*{Description} $X$ is the Fano fourfold \red{b1} in \cite{kuchle}.
\end{fano} \vspace{5mm}

\begin{fano}
\fanoid{$K_{740}$}
\label[mystyle]{F.1.2}
$X=\mZ(\of(2)^{\oplus 2})\subset\Gr(2,5)$
\subsubsection*{Invariants} $h^0(-K)=10, \ (-K)^4=20,\ \chi(T_X)=-72,\ h^{1,1}=1,\ h^{3,1}=20,  \ h^{2,2}=132$.
\subsubsection*{Description} $X$ is the Fano \red{b2} in \cite{kuchle}.
\end{fano} \vspace{5mm}

\begin{fano}
\fanoid{$K_{715}$}
\label[mystyle]{F.1.3}
$X=\mZ(\mQ(1))\subset\Gr(2,6)$
\subsubsection*{Invariants} $h^0(-K)=15, \ (-K)^4=42,\ \chi(T_X)=-42,\ h^{1,1}=1,\ h^{3,1}=6,  \ h^{2,2}=57$.
\subsubsection*{Description} $X$ is the Fano \red{b3} in \cite{kuchle}.
\end{fano} \vspace{5mm}

\begin{fano}
\fanoid{$K_{730}$}
\label[mystyle]{F.1.4}
$X=\mZ(\mU^{\vee}(1)\oplus\of(1)^{\oplus 2})\subset\Gr(2,6)$
\subsubsection*{Invariants} $h^0(-K)=13, \ (-K)^4=33,\ \chi(T_X)=-48,\ h^{1,1}=1,\ h^{3,1}=8,  \ h^{2,2}=70$.
\subsubsection*{Description} $X$ is the Fano \red{b5} in \cite{kuchle}.
\end{fano} \vspace{5mm}

\begin{fano}
\fanoid{$K_{734}$}
\label[mystyle]{F.1.5}
$X=\mZ(\of(1)^{\oplus 3}\oplus\of(2))\subset\Gr(2,6)$
\subsubsection*{Invariants} $h^0(-K)=12, \ (-K)^4=28,\ \chi(T_X)=-63,\ h^{1,1}=1,\ h^{3,1}=15,  \ h^{2,2}=106$.
\subsubsection*{Description} $X$ is the Fano \red{b6} in \cite{kuchle}.
\end{fano} \vspace{5mm}

\begin{fano}
\fanoid{$K_{24}$}
\label[mystyle]{F.1.6}
$X=\mZ(\of(1)^{\oplus 2})\subset\Gr(2,5)$
\subsubsection*{Invariants} $h^0(-K)=85, \ (-K)^4=405,\ \chi(T_X)=8,\ h^{1,1}=1,\ h^{3,1}=0,  \ h^{2,2}=2$.
\subsubsection*{Description} $X$ is the Fano $X_5^4$.
\end{fano} \vspace{5mm}

\begin{fano}
\fanoid{$K_{365}$}
\label[mystyle]{F.1.7}
$\mZ(\of(1)^{\oplus 4})\subset\Gr(2,6)$
\subsubsection*{Invariants} $h^0(-K)=51, \ (-K)^4=224,\ \chi(T_X)=-9,\ h^{1,1}=1,\ h^{3,1}=0,  \ h^{2,2}=8$.
\subsubsection*{Description} $X$ is the Fano $X_6^4$.
\end{fano} \vspace{5mm}

\begin{fano}
\fanoid{$K_{443}$}
$\mZ(\mU^{\vee}(1))\subset\Gr(2,5)$
\label[mystyle]{F.1.8}
\subsubsection*{Invariants} $h^0(-K)=45, \ (-K)^4=192,\ \chi(T_X)=-15,\ h^{1,1}=1,\ h^{3,1}=0,  \ h^{2,2}=12$.
\subsubsection*{Description} $X$ is the Fano $V_{12}^{4}$ in \cite{CGKS}.
\end{fano} \vspace{5mm}

\begin{fano}
\fanoid{$K_{496}$}
\label[mystyle]{F.1.9}
$\mZ(\of(1)\oplus\of(2))\subset\Gr(2,5)$
\subsubsection*{Invariants}  $h^0(-K)=39, \ (-K)^4=160,\ \chi(T_X)=-24,\ h^{1,1}=1,\ h^{3,1}=1,  \ h^{2,2}=22$.
\subsubsection*{Description} $X$ is the Fano fourfold \red{$X_{10}$} in \cite[Table 3]{BFMT}.
\end{fano} \vspace{5mm}

\section{Fano Varieties with Picard Rank 2}\label{sec:pic2}

\renewcommand\thefano{2--\arabic{fano}}
\setcounter{fano}{0}

\begin{fano}
\fanoid{$K_{508}$}
\label[mystyle]{F.2.1}
$X=\mZ(\of(1,1)\oplus\of(0,1)^{\oplus 3})\subset\PP^2\times\Gr(2,5)$
\subsubsection*{Invariants} $h^0(-K)=39, \ (-K)^4=160,\ \chi(T_X)=-9,\ h^{1,1}=2,\ h^{3,1}=0,  \ h^{2,2}=7$.
\subsubsection*{Description} $X$ is a $\PP^1$-bundle on $X_5^3$ which degenerates to a $\PP^2$-bundle over five points. 
\subsubsection*{Identification} 
    See \cref{ex:degLoc}.

\end{fano} \vspace{5mm}

\begin{fano}
\fanoid{$K_{509}$}
\label[mystyle]{F.2.2}
$\mZ(\mQ_{\Gr(2,4)}\boxtimes\mU^{\vee}_{\Gr(2,5)}\oplus\of(0,1)\oplus\of(1,1))\subset\Gr(2,4)\times\Gr(2,5)$
\subsubsection*{Invariants} $h^0(-K)=33, \ (-K)^4=130,\ \chi(T_X)=-22,\ h^{1,1}=2,\ h^{3,1}=1,  \ h^{2,2}=23$.
\subsubsection*{Description} $X$ is the Fano fourfold \red{GM-22} in \cite[Table 5]{BFMT}.
\end{fano} \vspace{5mm}

\begin{fano}
\fanoid{$K_{517}$}
\label[mystyle]{F.2.3}
$X=\mZ(\mQ_{\Gr(2,5)}(1,0)\oplus\mU^{\vee}_{\Gr(2,5)}(1,0)\oplus\of(0,1)^{\oplus 2})\subset\PP^5\times\Gr(2,5)$
\subsubsection*{Invariants}  $h^0(-K)=33, \ (-K)^4=131,\ \chi(T_X)=-13, \ h^{1,1}=2, \ h^{3,1}=0, \ h^{2,2}=12$.
\subsubsection*{Description} $X=\Bl_{S}X_5^4$, with $S=\Bl_{9pts}\PP^2$.
\subsubsection*{Identification} The restriction of the projection $\pi$ to $\Gr(2,5)$ gives a birational map between $X$ and $X_5^4$, since the generic fiber is a fixed codimension five linear section of $\PP^5$. To find where the dimension of the fiber increases we want to apply \cref{thm:degloc} and use the strategy explained in \cref{ex:degLoc}. Recall that $X$ is the zero locus of $\sigma_1\in V_6^{\vee}\otimes V_5$, $\sigma_2\in V_6^{\vee}\otimes V_5$ and $\omega_1,\ \omega_2\in\bigwedge^{2}V_5^{\vee}$. If we consider only $\sigma_1$ and $\sigma_2$ and we project to $\Gr(2,5)$ we are fixing a plane $\Pi\in V_5$ and we look for $l\in V_6$ such that the relations $\sigma_1(l,\Pi)=0$ and $\sigma_2(l,\Pi)=0$ hold simultaneously. This fact is equivalent to studying the map \[\varphi:\mQ_{\Gr(2,5)|X_5^4}\oplus\mU_{\Gr(2,5)|X_5^4}\to\of_{X_5^4}^{\oplus 6}\]
The exceptional locus of $\pi$ is where the map $\varphi$ degenerates to a rank-4 map. We can apply \cref{thm:en} to compute the invariants. We can also notice that $D_4(\varphi)$ can be described as $S=\mZ(\Q_{\Gr(2,5)}(1,0)\oplus\mU^{\vee}_{\Gr(2,5)}(1,0)\oplus\of(0,1)^{\oplus 2})\subset\PP^3\times\Gr(2,5)$. This is a surface $S$ with $e(S)=12$, $(-K_S)^2=0$, $h^{1,0}=h^{2,0}=0$, and $K_S=\of(1,-1)_{|S}$. In particular, $K_S^2=0$ and $h^0(S,K_S^{\otimes 2})=0$, hence by classification, it is the $\Bl_{9pts}\PP^2$.\newline
Since $X$ is birational to $X_5^4$, then it is rational.
\end{fano} \vspace{5mm}

\begin{fano}
\fanoid{$K_{527}$}
\label[mystyle]{F.2.4}
$X=\mZ(\mQ_{\PP^2}\boxtimes\mU^{\vee}_{\Gr(2,7)}\oplus\of(0,1)^{\oplus 4})\subset\PP^2\times\Gr(2,7)$
\subsubsection*{Invariants} $h^0(-K)=30, \ (-K)^4=114,\ \chi(T_X)=-20, \ h^{1,1}=2,\ h^{3,1}=0, \ h^{2,2}=17$.
\subsubsection*{Description} $X$ is a small resolution of $Y$, a fourfold singular in two points with $e(Y)=21$ and $deg(Y)=28$.
\subsubsection*{Identification}This fourfold is the zero locus of $\sigma\in V_3\otimes V_7^{\vee}$ and $\omega_1,\ \omega_2,\ \omega_3,\ \omega_4\in\bigwedge^2 V_7^{\vee}$. We study The projection $\pi$ to $\Gr(2,7)$ of $\mZ(\mQ_{\PP^2}\boxtimes\mU^{\vee}_{\Gr(2,7)})$. Once we project to $\Gr(2,7)$, we look for lines $l$ in $V_3$ such that, for a point $ \Pi \in\Gr(2,7)$ the relation $\sigma(l,\Pi)=0$ holds. Equivalently, with the same technique in \cref{ex:degLoc}, we apply \cref{thm:degloc} and study where the map \[\varphi:\mU^{\vee}_{\Gr(2,7)}\to \of_{\Gr(2,7)}^{\oplus 3}\] degenerates. In particular, using \cref{thm:en,thm:ddl}, $\varphi$ degenerates to a rank 1 map on an eightfold $Z=D_{1}(\varphi)$ of degree 28, which is singular on a fourfold $K=D_{0}(\varphi)\subset Z$ of degree 2. To complete the picture, we need to further cut with the extra linear sections. In this way, we obtain that $\pi$ is a birational map between X and a fourfold $Y=Z\cap H_1\cap H_2\cap H_3\cap H_4$ with $e(Y)=21$, degree 28 and singular in two points, such that it contracts two $\PP^2$'s to the 2 singular points of $Y$. Hence, $\pi$ is the map associated to a small resolution of $Y$.\newline
Note that the projection to $\PP^2$ gives a $\DP_5$ fibration, since the generic fiber is $\Gr(2,5)\cong\mZ(\mU^{\vee\oplus 2})\subset\Gr(2,7)$, cut with $4$ hyperplane sections, hence by \cite{isko}, $X$ is rational.
\end{fano} \vspace{5mm}

\begin{fano}
\fanoid{$K_{558}$}
\label[mystyle]{F.2.5}
$X=\mZ(\mQ_{\Gr(2,4)}\boxtimes\mU^{\vee}_{\Gr(2,5)}\oplus\of(1,0)\oplus\of(0,2))\subset\Gr(2,4)\times\Gr(2,5)$
\subsubsection*{Invariants} $h^0(-K)=30, \ (-K)^4=114,\ \chi(T_X)=-24, \ h^{1,1}=2,\ h^{3,1}=1, \ h^{2,2}=24$.
\subsubsection*{Description} $X$ is the Fano fourfold \red{GM-21} in \cite[Table 5]{BFMT}.
\end{fano} \vspace{5mm}

\begin{fano}
\fanoid{$K_{559}$}
\label[mystyle]{F.2.6}
$X=\mZ(\mQ_{\Gr(2,5)}(1,0)\oplus\of(0,1)^{\oplus 2}\oplus\of(1,1))\subset\PP^4\times\Gr(2,5)$
\subsubsection*{Invariants} $h^0(-K)=29, \ (-K)^4=110,\ \chi(T_X)=-21, \ h^{1,1}=2,\ h^{3,1}=1, \ h^{2,2}=22$.
\subsubsection*{Description} $X$ is the Fano fourfold in \cite[Table 6]{BFMT}.
\end{fano} \vspace{5mm}

\begin{fano}
\fanoid{$K_{566}$}
$X=\mZ(\mQ_{\Gr(2,5)}(1,0)\oplus\of(0,1)\oplus\of(0,2))\subset\PP^3\times\Gr(2,5)$
\label[mystyle]{F.2.7}
\subsubsection*{Invariants} $h^0(-K)=26, \ (-K)^4=94,\ \chi(T_X)=-28, \ h^{1,1}=2,\ h^{3,1}=1, \ h^{2,2}=28$.
\subsubsection*{Description} $X$ is the Fano fourfold \red{GM-20} in \cite[Table 5]{BFMT}.
\end{fano} \vspace{5mm}

\begin{fano}
\fanoid{$K_{547}$}
\label[mystyle]{F.2.8}
$X=\mZ(\mQ_{\Gr(2,4)}(0,1)\oplus\mU^{\vee}_{\Gr(2,4)}(0,1)^{\oplus 2})\subset\Gr(2,4)\times\PP^6$
\subsubsection*{Invariants} $h^0(-K)=30, \ (-K)^4=116,\ \chi(T_X)=-16, \ h^{1,1}=2,\ h^{3,1}=0, \ h^{2,2}=15$.
\subsubsection*{Description} $X=\Bl_S\Gr(2,4)$, with $S=\Bl_{12pts}\PP^2$ 
\subsubsection*{Identification} 
The projection $\pi$ to $\Gr(2,4)$ gives a birational map between $X$ and $\Gr(2,4)$ since the fiber is a codimension six linear subspace of $\PP^6$. To understand the exceptional locus of $\pi$ we note that $X$ is the zero locus of a section $\sigma\in V_7^{\vee}\otimes(V_4\oplus (V_4^{\vee})^{\oplus 2})$, in particular, we are looking for lines $l\in V_7$ such that, fixed a plane $\Pi\in V_4$, the relation $\sigma(l,\Pi)=0$ holds. So if we project to $\Gr(2,4)$ and apply \cref{thm:degloc}, this is equivalent to studying the map 
\[\varphi:\mQ^{\vee}_{\Gr(2,4)}\oplus\mU^{\oplus 2}_{\Gr(2,4)}\to\of_{\Gr(2,4)}^{\oplus 7}.\] Hence to understand where the projection $\pi$ is not a isomorphism it suffices to study $D_3(\varphi)$, which using \cref{thm:ddl,thm:en} is a surface $S$ that can be described as $\mZ(\mQ_{\Gr(2,4)}(0,1)^{\oplus 2}\oplus\mU^{\vee}_{\Gr(2,4)}(0,1))\subset\Gr(2,4)\times\PP^4$. In particular, $S$ has $e(S)=15$, $deg(S)=9$, $K_S=\of(-1,1)_{|S}$, $h^{1,0}=h^{2,0}=0$ and $h^0(S,K_S^{\otimes 2})=0$, so $S\cong\Bl_{12pts}\PP^2$.\newline
Note that since $X$ is birational to $\Gr(2,4)$ in particular is rational.
\end{fano} \vspace{5mm}

\begin{fano}
\fanoid{$K_{549}$}
\label[mystyle]{F.2.9}
$X=\mZ(\mQ_{\Gr(2,5)}(1,0)\oplus\of(0,1)^{\oplus 3}\oplus\of(2,0))\subset\PP^5\times\Gr(2,5)$
\subsubsection*{Invariants} $h^0(-K)=32, \ (-K)^4=126,\ \chi(T_X)=-11, \ h^{1,1}=2,\ h^{3,1}=0, \ h^{2,2}=10$.
\subsubsection*{Description} $X$ is a conic bundle which discriminates on a $K3$ surface $S$ singular in eight points.
\subsubsection*{Identification} Using \crefpart{flagCor}{cor:flag_2} we can rewrite $X$ as $\mZ(\of(0,1)^{\oplus 3}\oplus\of(2,0)\oplus\mQ_2)\subset\Fl(1,3,6)$. Now, using \cref{lm:flag2.5}, this can be rewritten as $\mZ(\of_{\mR}(2))\subset\PP_{X_5^3}(\mU_{\Gr(2,5)|X_5^3}\oplus\of_{X_5^3})$, which is a conic bundle over $X_5^3$. Using \cref{lm:conicbdl} with $\mathcal{E}=\mU^{\vee}\oplus\of$ and $\mathcal{K}=\of$, we get that this conic bundle has as discriminant locus a degree 2 surface $S\subset X^3_5$, with eight singular points, which is a $K3$ surface singular in 8 points. \newline
We analyze the other projection. Let us consider first $X'=\mZ(\mQ_{\Gr(2,5)}(1,0)\oplus\of(0,1)^{\oplus 3})\subset \PP^5\times\Gr(2,5)$. $X'$ is birational to $\PP^5$ via the first projection. If we cut $\PP^5$ with $\mZ(\of(2))$ we get that $X$ is birational to $Q_4$. To study the exceptional locus note that $X'$ is defined as the vanishing of $\alpha\in V_6^{\vee}\otimes V_5$ and $\beta\in \bigwedge^3 V_5\otimes V3$. By \cref{thm:degloc} we can understand the degeneracy loci by studying the morphisms:
\[
    \varphi:V_5^{\vee}\otimes\of_{Q_4}\to\of_{Q_4}(1);
\]
\[
    \phi: \bigwedge^3 V_5^{\vee}\otimes\of_{Q_4}\to \of_{Q_4}^{\oplus 3}.
\]
By \cref{thm:ddl,thm:en} we get that $D_0(\varphi)=\{p\}$ and the fiber of $\pi_1$ over it is $\PP^1\times\PP^1$. Moreover $\varphi$ induces the sequence 
\[
0\to\mathcal{K}\xrightarrow{i} V_5^{\vee}\otimes\of_{Q_4}\xrightarrow{\varphi}\of_{Q_4}(1)
\] 
and if we do the third exterior power of the sequence we get a map 
\[
\bigwedge^3 i:\bigwedge^3\mathcal{K}\to\bigwedge^3 V_5
\]
which we can compose to $\phi$ and obtain 
\[
\Phi:\bigwedge^3\mathcal{K}\to \of_{Q_4}^{\oplus 3}.
\]
In particular, this gives that $D_{2}(\phi_{|\alpha=0})=D_2(\Phi)$, which is a $S=\DP_2$. Since generically on $p\in S$, $\pi_1^{-1}(p)=\PP^1$ we have that $X$ is birational to $Q_4$ and, outside a point, $X$ is $\Bl_{\DP_2}Q_4$. The final picture will be $X$ a small resolution of $Y=\Bl_SQ_4$, where $S$ is a $\DP_2$ singular in a point. \newline
This gives also the rationality of $X$.
\end{fano} \vspace{5mm}

\begin{fano}
\fanoid{$K_{552}$}
\label[mystyle]{F.2.10}
$X=\mZ(\of(0,1)^{\oplus 2}\oplus\mU^{\vee}_{\Gr(2,5)}(1,0)^{\oplus 2})\subset\PP^4\times\Gr(2,5)$
\subsubsection*{Invariants} $h^0(-K)=30, \ (-K)^4=116,\ \chi(T_X)=-14, \ h^{1,1}=2,\ h^{3,1}=0, \ h^{2,2}=13$.
\subsubsection*{Description} $X=\Bl_{\Bl_{10pts}\PP^2}X_5^4$.
\subsubsection*{Identification} The projection $\pi$ to $\Gr(2,5)$ gives a birational map between $X$ and $X_5^4$, since the generic fiber is $\PP^4$ cut with $4$ hyperplane sections. Now to understand where the fiber degenerates recall that $X$ is the zero locus of $\sigma\in V_5\otimes (V_5^{\vee\oplus 2})$ and $\omega_1,\omega_2\in\bigwedge^2V_5$, hence, by \cref{thm:degloc}, we can study the projection of $X$ to $X_5^4$ by studying the map \[\varphi_:\mU_{\Gr(2,5)|X_5^4}^{\oplus 2}\to\of_{X_5^4}^{\oplus 5}\] In particular, $\pi$ degenerates along $D_1(\varphi)$. Using \cref{thm:ddl,thm:en} we get that $D_1(\varphi)$ can be described as $S=\mZ(\mQ_{\Gr(2,5)}(1,0)^{\oplus 2})\subset\PP^4\times X_5^4$. $S$ has $e(S)=13$, $deg(S)=11$, $K_S=\of(1,-1)_{|S}$, $(-K_S)^2=1$, $h^{1,0}=h^{2,0}=0$, and $h^0(S,K_S^{\otimes 2})=0$, hence using the classification of surfaces we get that $S$ is $\Bl_{10pts}\PP^2$.\newline
Note that since $X$ is birational to $X_5^4$, in particular, it is rational.
\end{fano} \vspace{5mm}

\begin{fano}
\fanoid{$K_{25}$}
\label[mystyle]{F.2.11}
$X=\mZ(\mQ_{\Gr(2,4)} \boxtimes \mU^{\vee}_{\Gr(2,5)} \oplus \of(0,1)^{\oplus 2}) \subset \Gr(2,4) \times \Gr(2,5)$
\subsubsection*{Invariants} $h^0(-K)=70, \ (-K)^4=325,\ \chi(T_X)=4,\ h^{1,1}=2,\ h^{3,1}=0,  \ h^{2,2}=4$.
\subsubsection*{Description} $X=\Bl_{\PP^1}X_5^4$.
\subsubsection*{Identification} Note that if we consider $X'=\mZ(\mQ_{\Gr(2,4)}\boxtimes\mU^{\vee}_{\Gr(2,5)})\subset\Gr(2,4)\times\Gr(2,5)$ we can apply \crefpart{lm:blowFlagGr}{item_2} and obtain $X'=\Bl_{D'}\Gr(2,5)$, with $D'=\PP^3=\mZ(\mQ)\subset\Gr(2,5)$. Since $X=X'\cap H_1\cap H_2$, with $H_1$, $H_2$, generic hyperplane sections, we obtain that $X=\Bl_DX_5^4$, with $D=D'\cap H_1\cap H_2$ which is $\PP^1$. \newline
Note that since $X$ is birational to $X_5^4$, in particular, it is rational.
\end{fano} \vspace{5mm}

\begin{fano}
\fanoid{$K_{554}$}
\label[mystyle]{F.2.12}
$X=\mZ(\of(1)\boxtimes Sym^2\mU^{\vee}_{\Gr(2,4)}) \subset \PP^3 \times \Gr(2,4)$
\subsubsection*{Invariants} $h^0(-K)=24, \ (-K)^4=90,\ \chi(T_X)=-9, \ h^{1,1}=2, \ h^{3,1}=0,\ h^{2,2}=12$.
\subsubsection*{Description} $X=\Bl_{S}\Gr(2,4)$, with $S$ an Enriques surface.
\subsubsection*{Identification} The projection $\pi$ to $\Gr(2,4)$ gives a birational map between $X$ and $\Gr(2,4)$. In fact $\Sym^2\mU^{\vee}_{\Gr(2,4)}$ has rank 3, so the generic fiber is a codimension three linear subspace of $\PP^3$. To understand the exceptional locus of $\pi$ recall that $X$ is the zero locus of $\sigma\in V_4\otimes\Sym^2V_4^{\vee}$. Hence if we project to $\Gr(2,4)$, we are looking for lines $l$ in $\PP^3$ such that, once we fix a plane, $\Pi$ in $V_4$, $\sigma(l,\Pi)=0$. Equivalently, by \cref{thm:degloc} we can study \[\varphi:\Sym^2\mU_{\Gr(2,4)}\to\of_{\Gr(2,4)}^{\oplus 4}.\] The fiber $\pi$ has higher dimension where $\varphi$ degenerates, hence we want to understand $D_2(\varphi)$. Using \cref{thm:ddl,thm:en} we get that $D_2(\varphi)$ is a surface $S$ with $e(S)=12$ and $deg(S)=10$. Note that $X$ is described in\cite[Theorem 2]{Kuz19}, hence $S$ is an Enriques surface. \newline
Note that since $X$ is birational to $\Gr(2,4)$, in particular, it is rational.

\end{fano} \vspace{5mm}

\begin{fano}
\fanoid{$K_{582}$}
\label[mystyle]{F.2.13}
$X=\mZ(\mQ_{\Gr(3,5)}(1,0)\oplus \of(0,1)^{\oplus 3}\oplus \of(2,0))\subset\PP^4\times\Gr(3,5)$
\subsubsection*{Invariants} $h^0(-K)=30, \ (-K)^4=116,\ \chi(T_X)=-11, \ h^{1,1}=2,\ h^{3,1}=0,\ h^{2,2}=10$.
\subsubsection*{Description} The projection $\pi_1: X \to Q_3 $ is a conic bundle with discriminant a sextic surface in $Q_3$, singular in 40 points. The second projection $\pi: X \to X_5^3$ is a conic bundle discriminant a $K3$ surface singular in eight points.

\subsubsection*{Identification} 
See \cref{ex:conicBdl}.

\end{fano} \vspace{5mm}

\begin{fano}
\fanoid{$K_{583}$}
\label[mystyle]{F.2.14}
$\mZ(\mU^{\vee}_{\Gr_1(2,4)}(0,1)\oplus\mU^{\vee}_{\Gr_2(2,4)}(1,0)) \subset \Gr_1(2,4) \times \Gr_2(2,4)$.
\subsubsection*{Invariants} $h^0(-K)=28, \ (-K)^4=106,\ \chi(T_X)=-16, \ h^{1,1}=2,\ h^{3,1}=0, \ h^{2,2}=15$.
\subsubsection*{Description}
Out of two points $X$ is isomorphic to $\Bl_{\DP_3}\Gr(2,4)$ via the projection $\pi$ on one of the $\Gr(2,4)$. On those two points, the fiber of $\pi$ is $\PP^1\times\PP^1$.
\subsubsection*{Identification} Let us note that $X$ is defined by $\alpha_1\in V_4^{\vee}\otimes\bigwedge^2W_4^{\vee}$ and $\alpha_2\in\bigwedge^2V_4^{\vee}\otimes W_4^{\vee}$. We analyze The projection $\pi$ to $\Gr(2,W_4)=\Gr_2(2,4)$, and this, by \cref{thm:degloc}, is equivalent to study where the two following maps degenerate:
\[
    \varphi:V_4\otimes\of_{\Gr_2(2,4)}\to\of_{\Gr_2(2,4)}(1),
\]
\[
    \phi:\bigwedge^2V_4\otimes\of_{\Gr_2(2,4)}\to\mU^{\vee}_{\Gr_2(2,4)}.
\]
By \cref{thm:ddl,thm:en} we obtain that $D=D_0(\varphi)={2pts}$ and the fiber of $\pi$ on the points is a $Q_2$. Note that the first vector bundles of $\varphi$ and $\phi$ are related by the following sequences:
\[
    0\to\mathcal{K}\to V_4\otimes\of_{\Gr_2(2,4)}\xrightarrow{\varphi}\of_{\Gr_2(2,4)}(1)\to \mathcal{F}\to 0
\]
with $\mathcal{K}=ker(\varphi)$ and $\mathcal{F}=coker(\varphi)$. Now, we can make the second exterior power of the above sequence, and obtain:
\[
    0\to\bigwedge^2\mathcal{K}\xrightarrow{\tau}\bigwedge^2V_4\otimes\of_{\Gr_2(2,4)}\to\mathcal{K}(1)\to\mathcal{K}\otimes\mathcal{F}\to 0.
\]
Hence we get
\[
    \Phi:\bigwedge^2\mathcal{K}\xrightarrow{\phi\circ\tau}\mU^{\vee}_{\Gr_2(2,4)}.
\]
So $D'=D_1(\phi_{|D})=D_1(\Phi)$, and using \cref{thm:ddl,thm:en}, then we obtain that $D'=\DP_3$. Hence out of two points $X$ is isomorphic to $\Bl_{\DP_3}\Gr(2,4)$ via $\pi$. On those two points, the fiber of $\pi$ is $Q_2$. Thus $X$ is a small resolution of $Y=\Bl_{S}\Gr(2,4)$, with $S$ a $\DP_3$ singular in two points.\\ Since $X$ is birational to $\Gr(2,4)$, it is rational.

\end{fano} \vspace{5mm}

\begin{fano}
\fanoid{$K_{589}$}
\label[mystyle]{F.2.15}
$X=\mZ(\mQ_{\PP^2}(0,1)^{\oplus 2}) \subset \PP^2 \times \Gr(2,5)$
\subsubsection*{Invariants} $h^0(-K)=24, \ (-K)^4=85,\ \chi(T_X)=-24, \ h^{1,1}=2,\ h^{3,1}=0, \ h^{2,2}=22$.
\subsubsection*{Description} $X$ is a small resolution of a fourfold $Y$ singular in five points with $e(Y)=23$ and $deg(Y)=15$.
\subsubsection*{Identification} Note that $X$ is the zero locus of $\sigma\in V_3\otimes(\bigwedge^2V_5^{\vee})^{\oplus 2}$. Once we project to $\Gr(2,5)$ via $\pi$, we are looking at the lines $l$ in $\PP^2$ such that, fixed the plane $\Pi\in\Gr(2,5)$, the relation $\sigma(l,\Pi)=0$ holds. By \cref{thm:degloc}, this is equivalent to study the map \[\varphi:\of_{\Gr(2,5)}(-1)^{\oplus 2}\to\of_{\Gr(2,5)}^{\oplus 3}\] In particular, $X$ is birational to $Y=D_1(\varphi)$, since The projection $\pi$ to $D_1(\varphi)\subset\Gr(2,5)$ is $\PP^2$ cut by a section of $\mQ_{\PP^2}$, hence a point. Moreover, $\pi$ degenerates on $D_2(\varphi)$, which is also the singular locus of $Y$, and the fiber over $D_2(\varphi)\subset D_1(\varphi)$ is $\PP^2$. Now using \cref{thm:ddl,thm:en} we have that $Y$ is a fourfold with $e(Y)=23$ and $deg(Y)=15$, while $D_2(\varphi)$ consists of 5 points. In this way, $X$ is a small resolution of the singular fourfold $Y$.\newline
Note that the projection to $\PP^2$ is a $\DP_5$ fibration, indeed, the generic fiber is $\Gr(2,5)$ cut by 4 hyperplane sections, using \cite{isko}, we have that $X$ is rational.
\end{fano} \vspace{5mm}

\begin{fano}
\fanoid{$K_{579}$}
$X=\mZ(\mQ_{\Gr(2,4)}(1,0)^{\oplus 2}\oplus\of(1,1))\subset\PP^5\times\Gr(2,4)$
\label[mystyle]{F.2.16}
\subsubsection*{Invariants}$h^0(-K)=27, \ (-K)^4=101,\ \chi(T_X)=-21, \ h^{1,1}=2,\ h^{3,1}=1, \ h^{2,2}=23$.
\subsubsection*{Description} $X$ is the Fano fourfold \red{M-9} in \cite[Table 4]{BFMT}.
\end{fano} \vspace{5mm}

\begin{fano}
\fanoid{$K_{597}$}
\label[mystyle]{F.2.17}
$X=\mZ(\mQ_{\Gr(2,4)}(1,0)\oplus\mU^{\vee}_{\Gr(2,4)}(1,0)\oplus\of(1,1))\subset\PP^5\times\Gr(2,4)$
\subsubsection*{Invariants} $h^0(-K)=27, \ (-K)^4=101,\ \chi(T_X)=-23, \ h^{1,1}=2,\ h^{3,1}=1, \ h^{2,2}=24$.
\subsubsection*{Description} $X$ is the Fano fourfold \red{K3-33} in \cite[Table 6]{BFMT}.
\end{fano} \vspace{5mm}

\begin{fano}
\fanoid{$K_{600}$}
\label[mystyle]{F.2.18}
$X=\mZ(\mU^{\vee}_{\Gr(2,4)}(1,0)\oplus\mQ_{\Gr(2,4)}(1,0)\oplus\of(0,1) \oplus\of(2,0))\subset \PP^6 \times \Gr(2,4)$
\subsubsection*{Invariants} $h^0(-K)=27, \ (-K)^4=100,\ \chi(T_X)=-20, \ h^{1,1}=2, \ h^{3,1}=0, \ h^{2,2}=18$.
\subsubsection*{Description} $X=\Bl_{\Bl_{9pts}\PP^2}Q_5\cap Q_5'$.
\subsubsection*{Identification} We apply \cref{thm:seq} and rewrite $X$ as $\mZ(\of_{\mR}(1)\oplus\of_{\mR}(2))\subset\PP_{Q_3}(\mU_{\Gr(2,4)|Q_3}\oplus\mQ_{\Gr(2,4)|Q_3}^{\vee})$. Since the generic fiber of the projection $\pi$ to $Q_3$ is $\PP^3$ cut by a linear subspace of codimension one and a quadric, we get that $X$ is a conic bundle over $Q_3$. Here the rank three bundle that induces the conic fibration is not explicit. To recover it we notice that $\of_{\mR}(1)$ induces a morphism:
\[
    \varphi:\mU_{\Gr(2,4)|Q_3}\oplus\mQ_{\Gr(2,4)|Q_3}\to\of_{Q_3};
\]
hence an exact sequence of vector bundles on $Q_3$:
\[
    0\to\mE\to\mU_{\Gr(2,4)|Q_3}\oplus\mQ_{\Gr(2,4)|Q_3}\to\of_{Q_3}\to 0
\]
with $\mE$ a rank three bundle completely identified by the Chern classes of $\mU_{\Gr(2,4)|Q_3}\oplus\mQ_{\Gr(2,4)|Q_3}$ and $\of_{Q_3}$.
By applying \cref{thm:seq} then we obtain $X=\mZ(\of_{\mR}(2))\subset\PP_{Q_3}(\mE)$. We are in the setting to apply \cref{lm:conicbdl} with $\mathcal{K}=\of_{Q_3}$, so we obtain that $\Delta$ is a quartic surface and $\Delta_{sing}=12$.\\
We study the projection $\pi_1$ to $\PP^6$. Recall that $X$ is obtained as the vanishing of the sections 
$\alpha\in V_6^{\vee}\otimes V_4^{\vee}$, 
$\beta\in V_6^{\vee}\otimes V_4$,
$\omega\in\Sym^2V_6^{\vee}$ and $\tau\in\bigwedge^2V_4$.
Let us first consider $\mZ(\alpha)\cap\mZ(\beta)\subset\PP^6\times\Gr(2,4)$. So if we want to understand the projection to $\PP^6$, by \cref{thm:degloc}, we can study the morphisms:
\[
\varphi:\of(-1)\to V_4^{\vee},
\]
\[
\phi:\of(-1)\to V_4
\]
Hence if we dualize $\phi$ and we compose it to $\varphi$, we get:
\[
\Phi:\of(-1)\xrightarrow{\varphi} V_4^{\vee}\xrightarrow{\phi^{\vee}}\of(1)
\]
Thus $\Phi(l)=0$ for $l\in\PP^6$ is equivalent to studying the image of the projection $\pi_1$. If we twist the bundles we can rewrite $\Phi$ as :
\[
\Phi':\of(-2)\to\of
\]
and $\Phi(l)=0$ if and only if $\Phi'(l)=0$, which is $Q_5\subset\PP^6$. To understand the fiber of $\pi_1$ over $Q_5$, recall that the fiber is obtained by looking at $\varphi(l)\subset V_4^{\vee}$. Since $\phi^{\vee}(\varphi(l))=0$ we have that $\varphi(l)\in\mathcal{K}_l=Ker\phi^{\vee}$, and the fiber over $l$ is $\PP(\mathcal{K}_l/\varphi(l))$. Generically $dim \mathcal{K}_l=3$ and $dim \varphi(l)=1$, hence generically the fiber over $Q_5$ we have a $\PP^1$. In particular, if we cut $\PP^6$ with $Q_5'=\mZ(\omega)$ and the fiber with $\mZ(\tau)$, then we have that the image of $\pi_1$ on $\PP^6$ is $Q_5\cap Q_5'$ and the fiber is a point. In this way we got that $\pi_1$ induces a birational map between $X$ and the intersection of two quadrics in $\PP^6$. To understand the exceptional locus of $\pi_1$ we can recall that, by \cite[Theorem 1.4]{ottaviani}, $\mU^{\vee}_{\Gr(2,4)|Q_3}=\mQ_{\Gr(2,4)|Q_3}$ is the spinor bundle on $Q_3$, so we can rewrite $X$ as $\mZ(\mU^{\vee}_{\Gr(2,4)|Q_3}(1,0)^{\oplus 2}\oplus\of(2,0))\subset\PP^6\times Q_3$. By \cref{thm:degloc}, we study the map:
\[
\theta:\of(-1)^{\oplus 2}\to V_4^{\vee}.
\]
The exceptional locus of $\pi_1$ is $D_1(\theta)$. By \cref{thm:ddl,thm:en} we have that $D_1(\theta)$ is a surface $S$ with $h^{0,0}=1$ and $h^{1,1}=10$, which is $\Bl_{9pts}$. Hence $X=\Bl_{\Bl_{9pts}\PP^2}Q_5\cap Q_5'$.
\end{fano} \vspace{5mm}

\begin{fano}
\fanoid{$K_{603}$}
\label[mystyle]{F.2.19}
$X=\mZ(\mU^{\vee}_{\Gr(2,6)}(1,0)\oplus\of(0,1)^{\oplus 4}) \subset \PP^2 \times \Gr(2,6)$
\subsubsection*{Invariants} $h^0(-K)=27, \ (-K)^4=100,\ \chi(T_X)=-18,\ h^{1,1}=2, \ h^{3,1}=0, \ h^{2,2}=16$.
\subsubsection*{Description} $X=\Bl_{S}X_6^4$ with $S$ a $\DP_2$.
\subsubsection*{Identification} $X$ and $X_6^4$ are birational since the generic fiber of The projection $\pi$ to $X_6^4\subset\Gr(2,6)$ is a codimension two linear subspace in $\PP^2$. To understand the exceptional locus of $\pi$, recall that $X$ is zero locus of $\sigma\in V_3\otimes V_6^{\vee}$ and $\omega_1,\ \omega_2,\omega_3,\omega_4\in \bigwedge^2V_6^{\vee}$. If we start by studying the zero locus of $\sigma$ and we project to $\Gr(2,6)$, we are looking at the lines $l$ in $\PP^3$ such that the relation $\sigma(l,\Pi)=0$ holds fixed a plane $\Pi$ in $\Gr(2,6)$. We can restrict ourselves to $X_6^4\subset\Gr(2,6)$ and the problem, by \cref{thm:degloc}, is equivalent to studying the map \[\varphi:\mU_{\Gr(2,6)|X_6^4}\to\of_{X_6^4}^{\oplus 3}\] In particular the fiber of $\pi$ is a $\PP^2$ where $\varphi$ degenerates, hence over $D_1(\varphi)$. Now using \cref{thm:ddl,thm:en} we have that $D_1(\varphi_{|X_6^4})$ is a surface which can be described as $S=\mZ(\mQ_{\Gr(2,6)}(1,0))\subset\PP^2\times X_6^4$ with $e(S)=10$, $K_S=\of(1,-1)_{|S}$, $(-K_S)^2=2$, $h^{1,0}=h^{2,0}=0$ and $h^0(S,K_S^{\otimes 2})=0$, hence, by classification $S=\DP_2$. In the end $X=\Bl_{\DP_2}X_6^4$.\newline
Note that since $X$ is birational to $X_6^4$ is, in particular, rational.
\end{fano} \vspace{5mm}

\begin{fano}
\fanoid{$K_{604}$}
\label[mystyle]{F.2.20}
$X=\mZ(\mQ_{\Gr(3,5)}(1,0)\oplus\of(1,1)\oplus\of(0,1)^{\oplus 2})\subset\PP^3\times\Gr(3,5)$
\subsubsection*{Invariants} $h^0(-K)=26, \ (-K)^4=95,\ \chi(T_X)=-22,\ h^{1,1}=2, \ h^{3,1}=1, \ h^{2,2}=23$.
\subsubsection*{Description} $X$ is the Fano fourfold \red{K3-34} in \cite[Table 6]{BFMT}.
\end{fano} \vspace{5mm}

\begin{fano}
\fanoid{$K_{29}$}
\label[mystyle]{F.2.21}
$X=\mZ(\mQ_{\Gr(2,4)}(1,0)^{\oplus 2})\subset\PP^4\times\Gr(2,4)$
\subsubsection*{Invariants} $h^0(-K)=60, \ (-K)^4=273,\ \chi(T_X)=3, \ h^{1,1}=2,\ h^{3,1}=0, \ h^{2,2}=3$.
\subsubsection*{Description} $X=\Bl_{\PP^2}\Gr(2,4)$.
\subsubsection*{Identification} The projection $\pi$ to $\Gr(2,4)$ is birational, since the generic fiber is a codimension four linear subspace in $\PP^4$. To understand the exceptional locus of $\pi$ recall that $X$ is the zero locus of $\sigma\in V_5\otimes(V_4^{\vee})^{\oplus 2}$. Using the strategy explained in \cref{ex:degLoc}, by \cref{thm:degloc}, the problem is equivalent to studying the map \[\varphi:\mU_{\Gr(2,4)}^{\oplus 2}\to\of_{\Gr(2,4)}^{\oplus 5}.\] The exceptional locus of $\pi$ corresponds to the locus where $\varphi$ becomes a rank 3 map, hence $D_3(\varphi)$. By \cref{thm:ddl,thm:en} has the same invariants of the surface $S=\mZ(\mQ_{\Gr(2,4)}(1,0)^{\oplus 2})\subset\PP^2\times\Gr(2,4)$. $S$ has $e(S)=3$, $K_S=\of(1,-2)_{|S}$, $(-K_S)^2=9$, $h^{1,0}=h^{2,0}=0$ and $h^0(S,K_S^{\otimes 2})=0$, hence by classification it is $\PP^2$. In this way $X=\Bl_{\PP^2}\Gr(2,4)$.\newline
Since $X$ is birational to $\Gr(2,4)$ then it is rational.
\end{fano} \vspace{5mm}

\begin{fano}
\fanoid{$K_{612}$}
\label[mystyle]{F.2.22}
$X=\mZ(\mQ(0,2))\subset \PP^2\times\PP^4$
\subsubsection*{Invariants} $h^0(-K)=30, \ (-K)^4=113,\ \chi(T_X)=-12, \ h^{1,1}=2,\ h^{2,1}=5, \ h^{3,1}=0 \ h^{2,2}=3$.
\subsubsection*{Description} $X=\Bl_{Q_2\cap Q_2'\cap Q_2''}\PP^4$.
\subsubsection*{Identification} We can simply apply \cref{lm:blowPPtX} and obtain that $X=\Bl_{Q_2\cap Q_2'\cap Q_2''}\PP^4$.\\ $X$ is rational by definition.
\end{fano} \vspace{5mm}

\begin{fano}
\fanoid{$K_{614}$}
\label[mystyle]{F.2.23}
$X=\mZ(\of(1,1)\oplus\bigwedge^3\mQ_{\PP^4}(0,1))\subset\PP^4\times\PP^5$
\subsubsection*{Invariants} $h^0(-K)=24, \ (-K)^4=86,\ \chi(T_X)=-24, \ h^{1,1}=2, \ h^{3,1}=1 \ h^{2,2}=26$.
\subsubsection*{Description} $X$ is the Fano fourfold \red{C-7} in \cite[Table 4]{BFMT}.
\end{fano} \vspace{5mm}

\begin{fano}
\fanoid{$K_{616}$}
\label[mystyle]{F.2.24}
$X=\mZ(\mQ_{\Gr(2,4)}(1,0)\oplus\of(0,2)\oplus\of(2,0) \subset \PP^4 \times \Gr(2,4)$
\subsubsection*{Invariants} $h^0(-K)=26, \ (-K)^4=96,\ \chi(T_X)=-13, \ h^{1,1}=2, \ h^{2,1}=2,\ h^{3,1}=0, \ h^{2,2}=10$.
\subsubsection*{Description} $X$, is a conic fibration over $Z=\mZ(\of(2))\subset\Gr(2,4)$ which discriminates on a $K3$ surface singular in 8 points.
\subsubsection*{Identification} Applying \crefpart{flagCor}{cor:flag_2} and \cite[Lemma 5]{DFT} we get $X=\mZ(\of_{\mR}(2)))\subset\PP_Z(\mU_{\Gr(2,4)|Z}\oplus\of_Z)$, with $Z=\mZ(\of(2))\subset\Gr(2,4)$, which is the Fano threefold \red{1-14} in \cite[Table 1]{DFT}. Applying \cref{lm:conicbdl}, with $\mathcal{E}=\mU^{\vee}\oplus\of$ and $\mathcal{K}=\of$, we get that the discriminant locus of the conic bundle is a surface of degree 2 in $Z$ singular in 8 points, which means, by adjunction formula, a $K3$ surface singular in 8 points.

\end{fano} \vspace{5mm}

\begin{fano}
\fanoid{$K_{617}$}
\label[mystyle]{F.2.25}
$X=\mZ(\mU^{\vee}_{\Gr(2,5)}(0,1) \oplus \mU^{\vee}_{\Gr(2,5)}(1,0)) \subset \PP^2 \times \Gr(2,5).$
\subsubsection*{Invariants} $h^0(-K)=25, \ (-K)^4=90,\ \chi(T_X)=-21, \ h^{1,1}=2, \ h^{3,1}=0,\ h^{2,2}=19$.
\subsubsection*{Description} $X=\Bl_{\DP_3}Y$, with $Y$ the \cref{F.1.8}.
\subsubsection*{Identification} The projection $\pi$ to $\Gr(2,5)$ gives a birational map between $X$ and $Y=\mZ(\mU^{\vee}_{\Gr(2,5)}(1))\subset\Gr(2,5)$, the \cref{F.1.8}, since the generic fiber is $\PP^2$ cut with two hyperplane sections. To understand the exceptional locus of $\pi$ it suffices, by \cref{thm:degloc}, to study where the map \[\varphi:\mU_{\Gr(2,5)|Y}\to \of_{\Gr(2,5)|Y}^{\oplus 3}\] degenerates to a rank-1 map, i.e.,   we have to study $D_3(\varphi)$. Using \cref{thm:ddl,thm:en} we get that $D_1(\varphi_{|Y})$ is a surface $S$ that can be described as $\mZ(\mU^{\vee}_{\Gr(2,5)}(1,0))\subset\PP^1\times Y$. $S$ is such that $e(S)=9$, $\mathcal{K}_{S}=\of(1,-1)_{|S}$, $(-K_S)^2=3$, $h^{1,0}=h^{2,0}=0$ and $h^0(S,K_S^{\otimes 2})=0$, hence by classification $S=\DP_3$.\newline Note that the projection to $\PP^2$ is a $\DP_5$ fibration, indeed the generic fiber is $\Gr(2,5)$ cut by a section of $\mU^{\vee}(1)$ and a section of $\mU^{\vee}$, which is a rational surface of degree 5, and hence a $\DP_5$.
Now using \cite{isko}, we get that $X$ is rational.
\end{fano} \vspace{5mm}

\begin{fano}
\fanoid{$K_{32}$}
\label[mystyle]{F.2.26}
$X=\mZ(\mQ_{\PP^4_1}^{\vee}(1,1))\subset\PP^4_1\times\PP^4_2$
\subsubsection*{Invariants} $h^0(-K)=51, \ (-K)^4=225,\ \chi(T_X)=-1,\ h^{1,1}=2,\ h^{2,1}=1 ,\ h^{3,1}=0, \ h^{2,2}=3$.
\subsubsection*{Description} $X=\Bl_{S}\PP^4$, with $S=\PP_E(\mF)$ a $\PP^1$-bundle over a quintic elliptic curve.
\subsubsection*{Identification} The projection $\pi$ gives a birational map between $X$ and $\PP^4_1$, indeed, the generic fiber is a codimension four linear subspace of $\PP^4_2$. To understand the exceptional locus of $\pi$, we apply \cref{thm:degloc} and look where \[\varphi:\mQ_{\PP^4}(-1)\to\of_{\PP^4}^{\oplus 5}\] degenerates to a rank-3 map, hence we want to study $D_3(\varphi)$. Using \cref{thm:ddl,thm:en} we get that $D_3(\varphi)$ is a surface $S$ that can be described as $\mZ(\of_{\mR}(1)^{\oplus 5})\subset\PP_{\PP^4_1}(\mQ(-1))$. Applying \cite[Lemma 5]{DFT} we obtain that $S$ is equivalent also to $\mZ(\of(0,1)^{\oplus 5})\subset\Fl(1,2,5)$. Using \cref{lm:flag2.5} we can rewrite $S$ as $\PP_E(\mU|_E)$, where $E$ is a quintic elliptic curve obtained as a linear complete intersection in $\Gr(2,5)$.
\newline
Note that $X$ is rational by definition.
\end{fano} \vspace{5mm}

\begin{fano}
\fanoid{$K_{109}$}
\label[mystyle]{F.2.27}
$X=\mZ(\mQ_{\Gr(2,5)}(1,0)\oplus\of(0,1)^{\oplus 3})\subset\PP^4\times\Gr(2,5)$
\subsubsection*{Invariants} $h^0(-K)=60, \ (-K)^4=272,\ \chi(T_X)=3,\ h^{1,1}=2,\ h^{3,1}=0, \ h^{2,2}=2$.
\subsubsection*{Description} $X=\Bl_{\nu_2(\PP^2)}\PP^4$
\subsubsection*{Identification} We can apply \crefpart{lm:blowFlagGr}{item_1} and \cref{lm:flag2.5}, to get the identification $X=\PP_{X_5^3}(\mU_{\Gr(2,5)|X_5^3})$.\\
We study the projection $\pi_1$ to $\PP^4$. Generically the fiber is a point, hence we get that $\pi_1$ induces a birational morphism between $X$ and $\PP^4$. Using the combination of \cref{thm:degloc},\cref{thm:ddl}, and \cref{thm:en}, we get that the exceptional locus of $\pi_1$ is a surface $S$ that can be described as $\mZ(\mQ_{\PP^4}^{\vee}(1,1))\subset\PP^2\times\PP^4$. In particular, $S$ is isomorphic to $\PP^2$ embedded in $\PP^4$ in a non-degenerate way, hence it can be identified with the projection of $\nu_2(\PP^2)$. The fibers of $\pi_1$ along this $\PP^2$ are a $\PP^1$, hence $X=\Bl_S\PP^4$.
\end{fano} \vspace{5mm}

\begin{fano}
\fanoid{$K_{101}$}
\label[mystyle]{F.2.28}
$X=\mZ(\mQ_{\Gr(2,4)} \boxtimes \mU^{\vee}_{\Gr(2,5)} \oplus \of(0,1) \oplus \of(1,0)) \subset \Gr(2,4) \times \Gr(2,5)$
\subsubsection*{Invariants} $h^0(-K)=75, \ (-K)^4=352,\ \chi(T_X)=7,\ h^{1,1}=2,\ h^{3,1}=0,  \ h^{2,2}=3$.
\subsubsection*{Description} $X=\Bl_{\PP^2}X_5^4$.
\subsubsection*{Identification} Let us first recall that, by \crefpart{lm:blowFlagGr}{item_2}, $\mZ(\mQ_{\Gr(2,4)}\boxtimes\mU^{\vee}_{\Gr(2,5)})\subset\Gr(2,4)\times\Gr(2,5)$ is $\Bl_{\PP^3}\Gr(2,5)$, with $\PP^3=\mZ(\mQ)\subset\Gr(2,5)$. Now the remaining sections cut $\Gr(2,5)$ as hyperplane sections, but $\mZ(\of(1,0)\oplus\mQ_{\Gr(2,4)} \boxtimes \mU^{\vee}_{\Gr(2,5)})\subset\Gr(2,4) \times \Gr(2,5)$ contains the $\PP^3$ in $\Bl_{\PP^3}\Gr(2,5)$, while $\mZ(\of(0,1)\oplus\mQ_{\Gr(2,4)} \boxtimes \mU^{\vee}_{\Gr(2,5)})\subset\Gr(2,4) \times \Gr(2,5)$ cuts both the $\PP^3$ and $\Gr(2,5)$, hence the final picture is $X=\Bl_{\PP^2}X_5^4$.\newline
Note that since $X$ is birational to $X_5^4$ it is, in particular, rational.

\end{fano} \vspace{5mm}

\begin{fano}
\fanoid{$K_{102}$}
\label[mystyle]{F.2.29}
$X=\mZ(\mQ_{\Gr(2,5)}(1,0) \oplus \of(0,1)^{\oplus 2}) \subset \PP^3 \times \Gr(2,5)$
\subsubsection*{Invariants}$h^0(-K)=69, (-K)^4=320,\ \chi(T_X)=4,\ h^{1,1}=2,\ h^{3,1}=0, \ h^{2,2}=4$.
\subsubsection*{Description} $X=\Bl_{Q_2}X_5^4$.
\subsubsection*{Identification} 

 $X$ is defined by $\alpha_1, \alpha_2 \in \W^2V_5^{\vee}$ and by $\beta \in V_4^{\vee}\otimes V_5= \mathrm{Hom}(V_4, V_5)$. We analyze the projection to $\Gr(2,5)$, which is the locus of isotropic 2-planes for both $\alpha_i$, i.e., $X_5^4$. The fiber of $\pi$ over a generic point $p$ of $X_5^4$ is the zero locus on $\PP^3$ of a three-dimensional space of linear sections. Hence, the projection to $X_5^4$ is generically birational.
 
 By \cref{thm:degloc}, to understand the exceptional locus, we need to analyze the locus where the (generically rank three) morphism
 \[\varphi:\mQ_{\Gr(2,5)}|_{X_5^4}\to\of_{{X_5^4}} \otimes V_4 \]

 degenerates to a rank 2 morphism, which is by definition $D_2(\varphi_{|X_5^4})$. Notice also how, by dimension reasons, there is no further degeneration.
Now using \cref{thm:ddl,thm:en} we have that $D_2(\varphi)$ is a surface which can be described as $\mZ(\of(1)^{\oplus 2})\subset\Gr(2,4)$, which is $Q_2$, hence $X=\Bl_{Q_2}X_5^4$.\\
In particular, $X$ is rational.

\end{fano} \vspace{5mm}

\begin{fano}
\fanoid{$K_{625}$}
\label[mystyle]{F.2.30}
$\mZ(\of(1,1)^{\oplus 2} \oplus \mU^{\vee}_{\Gr(2,4)}(1,0)) \subset \PP^4 \times \Gr(2,4)$
\subsubsection*{Invariants} $h^0(-K)=24, \ (-K)^4=85,\ \chi(T_X)=-28,\ h^{1,1}=2, \ h^{3,1}=2, \ h^{2,2}=32$.
\subsubsection*{Description} $X=\Bl_{\mathcal{E}}\Gr(2,4)$, with $\mathcal{E}$ the elliptic fibration with base $\PP^1$ and generic fiber an elliptic curve $E$ of degree five.
\subsubsection*{Identification} The projection $\pi$ to $\Gr(2,4)$ is a birational map between the latter and $X$. Moreover, using the self-duality of $\Gr(2,4)$, we can rewrite $X$ as $\mZ(\of(1,1)^{\oplus 2} \oplus \mQ_{\Gr(2,4)}(1,0)) \subset \PP^4 \times \Gr(2,4)$, and by applying \crefpart{flagCor}{cor:flag_2} and \cref{lm:flag2.5}, we have that $X=\mZ(\of(1,1)^{\oplus 2})\subset\PP_{\Gr(2,4)}(\mU_{\Gr(2,4)}\oplus\of_{\Gr(2,4)})$. If we twist the bundle $\mU_{\Gr(2,4)}\oplus\of_{\Gr(2,4)}$ with $\of_{\Gr(2,4)}(-1)$, we get $X=\mZ(\of_{\mR}(1)^{\oplus 2})\subset\PP_{\Gr(2,4)}(\mU_{\Gr(2,4)}(-1)\oplus\of_{\Gr(2,4)}(-1))$. Now, to understand the exceptional locus of $\pi$ we can apply the usual combination of \cref{thm:degloc}, \cref{thm:ddl,thm:en}, obtaining that it degenerates on a surface $S$ that can be described as $\mZ(\mU^{\vee}(1,1)\oplus\of(1,1))\subset\PP^1\times\Gr(2,4)$. If we project $S$ on $\PP^1$ we have that the generic fiber is the curve $E=\mZ(\mU^{\vee}(1)\oplus\of(1))\subset\Gr(2,4)$ which is such that $\mathcal{K}_E=\of_{E}$, hence by classification is an elliptic curve of degree five. This means that $S$ is an elliptic fibration $\mathcal{E}$, and $X=\Bl_{\mathcal{E}}\Gr(2,4)$.\newline
In particular, since $X$ is birational to $\Gr(2,4)$, it is rational.
\end{fano} \vspace{5mm}

\begin{fano}
\fanoid{$K_{626}$}
\label[mystyle]{F.2.31}
$X=\mZ(\mQ_{\Gr(2,4)}(1,0)\oplus\of(0,1)\oplus\of(1,1)\oplus\of(2,0))\subset\PP^5\times\Gr(2,4)$
\subsubsection*{Invariants} $h^0(-K)=25, \ (-K)^4=90,\ \chi(T_X)=-23,\ h^{1,1}=2, \ h^{3,1}=1, \ h^{2,2}=24$.
\subsubsection*{Description} $X$ is the Fano fourfold \red{K3-23} in \cite[Table 6]{BFMT}.
\end{fano} \vspace{5mm}

\begin{fano}
\fanoid{$K_{628}$}
\label[mystyle]{F.2.32}
$X=\mZ(\bigwedge^2\mQ_{\PP^3}(0,1)\oplus\of(1,1)\oplus\of(0,2))\subset\PP^3\times\PP^6$
\subsubsection*{Invariants} $h^0(-K)=23, \ (-K)^4=80,\ \chi(T_X)=-27,\ h^{1,1}=2, \ h^{3,1}=1, \ h^{2,2}=28$.
\subsubsection*{Description} $X$ is the Fano fourfold \red{K3-31} in \cite[Table 6]{BFMT}.
\end{fano} \vspace{5mm}

\begin{fano}
\fanoid{$K_{629}$}
\label[mystyle]{F.2.33}
$X=\mZ(\mQ_{\PP^2}(0,1)\oplus\of(1,1)\oplus\of(0,1))\subset\PP^2\times\Gr(2,5)$
\subsubsection*{Invariants} $h^0(-K)=23, \ (-K)^4=80,\ \chi(T_X)=-26,\ h^{1,1}=2, \ h^{3,1}=1, \ h^{2,2}=27$
\subsubsection*{Description}{Description} $X$ is the Fano fourfold \red{GM-19} in \cite[Table 5]{BFMT}.
\end{fano} \vspace{5mm}

\begin{fano}
\fanoid{$K_{115}$}
\label[mystyle]{F.2.34}
$X=\mZ(\mU^{\vee}_{\Gr(2,4)}(1,0)\oplus\mQ_{\Gr(2,4)}(1,0) )\subset \PP^4 \times \Gr(2,4)$
\subsubsection*{Invariants}  $h^0(-K)=60, \ (-K)^4=272,\ \chi(T_X)=1, \ h^{1,1}=2,\ h^{3,1}=0, \ h^{2,2}=4$.
\subsubsection*{Description} $X=\Bl_{Q_2}\Gr(2,4)$
\subsubsection*{Identification} The projection to $\Gr(2,4)$ gives a birational map between the latter and $X$. In particular, by \cref{thm:degloc}, the fiber becomes a $\PP^1$ where the map \[\varphi:\mU_{\Gr(2,4)}\oplus\mQ^{\vee}_{\Gr(2,4)}\to\of_{\Gr(2,4)}^{\oplus 5}\] degenerates to a rank-3 map, i.e., on $D_3(\varphi)$. Using \cref{thm:ddl,thm:en} we have that $D_3(\varphi)$ is a surface $S$ that can be described as $\mZ(\mQ_{\Gr(2,4)}(1,0)\oplus\mU^{\vee}_{\Gr(2,4)}(1,0))\subset\PP^2\times\Gr(2,4)$, in particular, $S$ is such that $e(S)=4$, $K_S=\of(1,-2)_{|S}$, $(-K_S)^2=8$, $h^{1,0}=h^{2,0}=0$ and $h^0(S,K_S^{\otimes 2})=0$, $S$ is a $\DP_8$. Since $S$ is birational to $\PP^1\times\PP^1$ then it is a $Q_2$.
\newline
Note that since $X$ is birational to $\Gr(2,4)$, it is in particular rational.
\end{fano} \vspace{5mm}

\begin{fano}
\fanoid{$K_{126}$}
\label[mystyle]{F.2.35}
$\mZ(\mQ_{\Gr(2,4)}(1,0)\oplus\of(1,1)\oplus\of(0,1))\subset\PP^4\times\Gr(2,4)$
\subsubsection*{Invariants} $h^0(-K)=51, \ (-K)^4=224,\ \chi(T_X)=-5,\ h^{1,1}=2,\ h^{3,1}=0, \ h^{2,2}=7$.
\subsubsection*{Description} $X$ is generically a $\PP^1$-bundle with fiber jumping to a $\PP^2$ over five points. $X$ is a small resolution of $Y=\Bl_{S}\PP^4$ with $S$ a $\DP_7$ singular in a point.
\subsubsection*{Identification} We can apply \crefpart{flagCor}{cor:flag_2} and \cref{lm:flag2.5} to write $X$ as $\mZ(\of(1,1))\subset\PP_{Q_3}(\mU_{\Gr(2,4)|Q_3}\oplus\of_{Q_3})$. Now if we twist $\mU_{\Gr(2,4)|Q_3}\oplus\of_{Q_3}$ with $\of_{\Gr(2,4)}(-1)_{|Q_3}$ we get that $X$ is the $\PP^1$-bundle $\mZ(\of_{\mR}(1))\subset\PP_{Q_3}(\mU_{\Gr(2,4)|Q_3}(-1)\oplus\of_{Q_3}(-1))$. Applying \cref{thm:degloc}, \cref{thm:ddl,thm:en} we get that the fiber of natural projection $\pi$ associated with the $\PP^1$-bundle $X$ degenerates to a $\PP^2$ over five points. Moreover, $X$ can be seen as $\PP_{Q_3}(\of_{Q_3}(-1)\oplus\mU_{\Gr(2,4)|Q_3}(-1))$, as we have seen above. $\mU_{\Gr(2,4)|Q_3}$ is the dual of the spinor bundle on $Q_3$, hence $X=\mZ(\of_{\mR}(1))\subset\PP_{Q_3}(\of_{Q_3}(-1)\oplus\mathcal{S}^{\vee}(-1))$. Now we consider the sequence: 
\[
0\to\mathcal{K}\to\of_{Q_3}(-1)\oplus\mathcal{S}^{\vee}(-1)\to\of_{Q_3}\to 0
\]
and we define $\mathcal{E}=\mathcal{K}^{\vee}(-1)$. If we dualize the above sequence and we apply \cref{thm:seq},, the $X=\PP_{Q_3}(\mathcal{E}^{\vee})$, which is a model for the Fano fourfold described in \cite[Theorem 7.2.12]{lang}.\\
We study the projection $\pi_1$ on $\PP^4$. The fiber generically is a point, hence $\pi$ induces a birational map between $X$ and $\PP^4$. Moreover, $X$ is where $\alpha\in V_5^{\vee}\otimes V_4$ and $\beta\in (V_5^{\vee}\oplus V_1)\otimes V_4$ vanish. Thus, to understand where $\pi_1$ degenerates we can apply \cref{thm:degloc} and  obtain that we need to study the morphisms:
\[
\varphi:V_4^{\vee}\otimes\of_{\PP^4}\to\of_{\PP^4}(-1),
\]
\[
\phi:\bigwedge^2V_4^{\vee}\otimes\of_{\PP^4}\to\of_{\PP^4}(-1)\oplus\of_{\PP^4}.
\]
In this way we get that $D_0(\varphi)=\{p\}$ and $\pi_1^{-1}(p)$ is $\PP^1\times\PP^1$. Note that $\varphi$ induces the sequence:
\[
0\to\mathcal{K}\xrightarrow{i}V_4^{\vee}\otimes\of_{\PP^4}\xrightarrow{\varphi}\of_{\PP^4}(-1)
\]
and if we do the second exterior power of the above sequence we get a map:
\[
    \bigwedge^2 i:\bigwedge^2\mathcal{K}\to \bigwedge^2 V_4^{\vee}.
\]
If we consider $\Phi=\phi\circ\bigwedge^2i$ we obtain a morphism:
\[
\Phi: \bigwedge^2\mathcal{K}\to \of_{\PP^4}(-1)\oplus\of_{\PP^4}
\]
and $D_1(\phi_{|\alpha=0})=D_1(\Phi)$, which is a surface $S=\DP_7$. The generic fiber of $\pi_1^{-1}$ over $S$ is a $\PP^1$, which means that outside a point $X$ is $\Bl_{\DP_7}\PP^4$. In the end $X$ is a small resolution of $Y=\Bl_{S}\PP^4$ with $S$ a $\DP_7$ singular in a point.
\end{fano} \vspace{5mm}

\begin{fano}
\fanoid{$K_{124}$}
\label[mystyle]{F.2.36}
$X=\mZ(\mU^{\vee}_{\Gr(2,4)}(1,0)\oplus\of(1,1) )\subset \PP^3 \times \Gr(2,4)$
\subsubsection*{Invariants} $h^0(-K)=54, \ (-K)^4=240,\ \chi(T_X)=-4, \ h^{1,1}=2, \ h^{3,1}=0, \ h^{2,2}=7$.
\subsubsection*{Description} $X=\Bl_{S}\Gr(2,4)$, with $S=\DP_5$.
\subsubsection*{Identification} Recall that if we apply \crefpart{lm:blowFlagGr}{item_1}, the self-duality of $\Gr(2,4)$ and \cref{lm:flag2.5}, we can rewrite $X$ as $\mZ(\of(1,1))\subset\PP_{\Gr(2,4)}(\mU_{\Gr(2,4)})$. We can twist $\mU_{\Gr(2,4)}$ and get that $X$ is $\mZ(\of_{\mR}(1))\subset\PP_{\Gr(2,4)}(\mU_{\Gr(2,4)}(-1))$. $X$ is a $\PP^1$-bundle cut with a relative hyperplane section, hence it is birational to the base, which is $\Gr(2,4)$. Moreover, applying \cref{thm:degloc}, \cref{thm:ddl,thm:en}, we can compute the exceptional locus of The projection $\pi$ to $\Gr(2,4)$, which is a surface $S$ that is which can be described as $\mZ(\mU^{\vee}(1))\subset\Gr(2,4)$, i.e., a $\DP_5$.\newline
In this way we get that $X=\Bl_{\DP_5}\Gr(2,4)$, and in particular $X$ is rational.
\end{fano} \vspace{5mm}

\begin{fano}
\fanoid{$K_{642}$}
\label[mystyle]{F.2.37}
$X=\mZ(\mQ_{\PP^2}\boxtimes\mU^{\vee}_{\Gr(2,6)}\oplus\of(0,1)\oplus\of(0,2))\subset\PP^2\times\Gr(2,6)$
\subsubsection*{Invariants} $h^0(-K)=24, \ (-K)^4=82,\ \chi(T_X)=-38, \ h^{1,1}=2,  \ h^{3,1}=3, \ h^{2,2}=42$.
\subsubsection*{Description} $X$ is a small resolution of $Y$, a fourfold singular in two points with $e(Y)=52$.
\subsubsection*{Identification} $X$ is the zero locus of $\sigma\in V_3\otimes V_6^{\vee}$ intersecting a 1-dimensional linear subspace of $\bigwedge^2 V_6^{\vee}$ and a quadric in $\bigwedge^2 V_6^{\vee}$. Let us first study $X'$, the zero locus of $\sigma$. If we project to $\Gr(2,6)$ we are fixing a plane $\Pi\in\Gr(2,6)$ and we are looking for the lines $l\in\PP^2$ such that the relation $\sigma(l,\Pi)=0$ holds. By \cref{thm:degloc} for a generic $\Pi$ there are no such lines $l$, hence we have to restrict ourselves to some $Y'\subset\Gr(2,6)$, which coincides to the locus where the map \[\varphi:\mU_{\Gr(2,6)}\to \of^{\oplus 3}_{\Gr(2,6)}\] becomes a rank one map, i.e., $D_1(\varphi)$. Using \cref{thm:ddl} we have that $Y'$ is a sixfold, and it is singular on $D_0(\varphi)$, which is not empty, indeed, using \cref{thm:ddl}, we have that this is a surface $S$ with $e(S)=3$ and $deg(S)=1$. Hence $X'$ is birational to $Y'$ via The projection $\pi$ to $\Gr(2,6)$, whose fiber degenerates to a $\PP^2$ along $S\subset Y'$. Now $X$ is obtained by cutting $X'$ with a hyperplane $H$ and a quadric $Q$, hence $Y$ is a fourfold with $e(Y)=52$ singular on 2 points obtained as $S\cap H\cap Q$. So the final picture of $X$ is a small resolution of $Y$ along its singular locus.\\ 
The other projection is a $\DP_4$-fibration since the fiber over a generic point of $\PP^2$ is $\Gr(2,6)$ cut with two sections of $\mU^{\vee}_{\Gr(2,6)}$, a linear subspace of codimension one and a quadric subspace of codimension one.
\end{fano} \vspace{5mm}

\begin{fano}
\fanoid{$K_{653}$}
\label[mystyle]{F.2.38}
$X=\mZ(\of(0,1)^{\oplus 2} \oplus \of(1,1)^{\oplus 2}) \subset \PP^2 \times \Gr(2,5)$.
\subsubsection*{Invariants} $h^0(-K)=22, \ (-K)^4=75,\ \chi(T_X)=-28, \ h^{1,1}=2, \ h^{3,1}=4, \ h^{2,2}=32$.
\subsubsection*{Description} $X=\Bl_{\mathcal{E}}X_5^4$ with $\mathcal{E}$ elliptic fibration over $\PP^1$ with generic fiber a quintic elliptic curve $E$.
\subsubsection*{Identification}Let us note that $X$ can be seen also as $\mZ(\of(1,1)^{\oplus 2})\subset\PP_{X_5^4}(\of_{X_5^4}^{\oplus 3})$ and twisting $\of_{X_5^4}^{\oplus 3}$ with $\of_{X_5^4}(-1)$ we can rewrite $X$ as $\mZ(\of_{\mR}(1)^{\oplus 2})\subset\PP_{X_5^4}(\of_{X_5^4}(-1)^{\oplus 3})$, which is a $\PP^2$-bundle over $X_5^4$ cut by 2 relative hyperplane sections, hence it is generically isomorphic to $X_5^4$. By \cref{thm:degloc}, to understand the exceptional locus of the natural projection $\pi$ on $X_5^4$, we can study where the map \[\varphi:\of_{\Gr(2,5)}(-1)^{\oplus 2}\to\of_{\Gr(2,5)}^{\oplus 3}\] restricted to $X_5^4$ degenerates to a rank one map, i.e., $D_1(\varphi_{X_5^4})$. Using \cref{thm:ddl,thm:en} we get that $D_1(\varphi)$ is a surface $S$ that can be described as $\mZ(\of(1,1)^{\oplus 3})\subset\PP^1\times X_5^4$. The projection to $\PP^1$ of $S$ gives an elliptic fibration $\mathcal{E}$ with fiber the elliptic curve $E=\mZ(\of(1)^{\oplus 2})\subset X_5^4$, which is quintic.\newline
Hence $X=\Bl_{\mathcal{E}}X_5^4$, and this gives also its rationality.
\end{fano} \vspace{5mm} 

\begin{fano}
\fanoid{$K_{663}$}
\label[mystyle]{F.2.39}
$X=\mZ(\of(2,0)\oplus\mQ_{\Gr(2,5)}(0,1)) \subset \PP^2 \times \Gr(2,5)$
\subsubsection*{Invariants} $h^0(-K)=30, \ (-K)^4=112,\ \chi(T_X)=-12, \ h^{1,1}=2, \ h^{2,1}=5,\ h^{3,1}=0, \ h^{2,2}=2$.
\subsubsection*{Description} $X$ is $Q_1\times Y$, with $Y$ the Fano threefold  \red{1-7} in \cite[Table 1]{DFT}.
\subsubsection*{Identification} Since the bundles cut each factor of the product separately we get that this is simply the product between $Q_1$ and $Y=\mZ(\mQ(1))\subset\Gr(2,5)$, which is a Fano threefold since $\mathcal{K}_Y=\of_{\Gr(2,5)}(-1)_{|Y}$. Moreover, $Y$ has $e(Y)=-6$, $vol(Y)=14$, $h^0(-\mathcal{K}_Y)=10$ and $h^{2,1}=5$, using the classification of the Fano threefolds this tells us that $Y$ is the Fano threefold \red{1-7} in \cite[Table 1]{DFT}, which is also $\mZ(\of(1)^{\oplus 5})\subset\Gr(2,6)$.
\end{fano} \vspace{5mm}

\begin{fano}
\fanoid{$K_{664}$}
\label[mystyle]{F.2.40}
$X=\mZ(\mQ_{\Gr(2,4)}(1,0)\oplus\of(0,1)\oplus(2,1))\subset\PP^4\times\Gr(2,4)$
\subsubsection*{Invariants}$h^0(-K)=21, \ (-K)^4=69,\ \chi(T_X)=-40, \ h^{1,1}=2, \ h^{3,1}=5, \ h^{2,2}=52$.
\subsubsection*{Description} $X$ is a conic bundle which discriminates on a surface $S\subset Q_3$ of degree 10 with 40 singular points. $X$ is a small resolution of $Y=\Bl_{S}\PP^4$ with $S$ a surface singular in a point with $h^{0,0}=1$, $h^{1,0}=0$, $h^{2,0}=5$ and $h^{1,1}=48$
\subsubsection*{Identification} Applying \crefpart{flagCor}{cor:flag_2} and \cref{lm:flag2.5} we can rewrite $X$ as $\mZ(\of(1,2))\subset\PP_{Q_3}(\mU_{\Gr(2,4)|Q_3}\oplus\of_{\Gr(2,4)})$. Now twisting $\mU_{\Gr(2,4)|Q_3}\oplus\of_{\Gr(2,4)}$ with $\of_{Q_3}(-1)$ we get $X=\mZ(\of_{\mR}(2))\subset\PP_{Q_3}(\mU_{\Gr(2,4)}(-1)_{|Q_3}\oplus\of_{Q_3}(-1))$ which is a conic bundle on $Q_3$. Using \cref{lm:conicbdl} with $\mathcal{E}=\mU^{\vee}\oplus\of$ and $\mathcal{K}=\of(1)$ we have that the discriminant locus is a surface $S$ in $Q_3$ of degree 5 with 40 singular points.\newline
Note that the projection to $\PP^4$ gives a birational map since the generic fiber is $\Gr(2,4)$ cut with $\mQ$, hence $\PP^2$, and then we have 2 hyperplane sections cutting $\PP^2$, so $X$ is rational.\\
The projection $\pi_1$ to $\PP^4$ gives a birational map between the latter and $X$. Recall that $X$ is where $\alpha\in V_5^{\vee}\otimes V_4$, $\beta\in(\Sym^2V_5^{\vee}\oplus V_1)\otimes\bigwedge^2 V_4$. Thus to understand the exceptional locus of $\pi_1$, by \cref{thm:degloc}, is equivalent to study where the following morphisms degenerate:
\[
\varphi:V_4^{\vee}\otimes\of\to\of(-1),
\]
\[
\phi:\bigwedge^{2}V_4^{\vee}\otimes\of\to\of(-2)\oplus\of.
\]
Using \cref{thm:ddl,thm:en} we get that $D_0(\varphi)=\{p\}$ and $\pi_1^{-1}(p)$ is $\PP^1\times\PP^1$. Moreover, $\varphi$ induces the sequence:
\[
0\to\mathcal{K}\xrightarrow{i} V_4^{\vee}\otimes\of\xrightarrow{\varphi}\of(-1) 
\]
and making the second exterior power of it we obtain the map:
\[
\bigwedge^2i: \bigwedge^2\mathcal{K}\to\bigwedge^{2}V_4^{\vee}.
\]
If we consider $\phi\circ\bigwedge^2i$ we get the morphism:
\[
\Phi: \bigwedge^2\mathcal{K}\to \of(-2)\oplus\of.
\]
Note that $D_1(\phi_{|\alpha=0})=D_1(\Phi)$, and using \cref{thm:ddl,thm:en} we get that it is a surface $S$ with $h^{0,0}=1$, $h^{1,0}=0$, $h^{2,0}=5$ and $h^{1,1}=48$. Then $X$ outside a point is $\Bl_S\PP^4$ and over a point we have a $\PP^1\times\PP^1$. This gives the interpretation of $X$ as a small resolution of $\Bl_{S'}\PP^4$ with $S'$ singular in a point and whose resolution is $S$.
\end{fano} \vspace{5mm}

\begin{fano}
\fanoid{$K_{665}$}
\label[mystyle]{F.2.41}
$\mZ(\mQ_{\Gr(2,4)}(1,0) \oplus \of(1,2)) \subset \PP^3 \times \Gr(2,4)$.
\subsubsection*{Invariants} $h^0(-K)=20, \ (-K)^4=64,\ \chi(T_X)=-44, \ h^{1,1}=2, \ h^{3,1}=6, \ h^{2,2}=59$
\subsubsection*{Description} $X=\Bl_S\Gr(2,4)$, with $S$ a canonical surface with $e(S)=71$, $K_S=\of(1)_S$, and $(-K_S)^2=13$.
\subsubsection*{Identification} We can apply \crefpart{lm:blowFlagGr}{item_1} and rewrite $X$ as $\mZ(\of(2,1))\subset\PP_{\Gr(2,4)}(\mU_{\Gr(2,4)})$. Now we can twist $\mU_{\Gr(2,4)}$ with $\of(-2)$ and obtain $X=\mZ(\of_{\mR}(1))\subset\PP_{\Gr(2,4)}(\mU_{\Gr(2,4)}(-2))$, hence we have a $\PP^1$-bundle over $\Gr(2,4)$ cut with a relative hyperplane section, hence $X$ is birational to $\Gr(2,4)$. To understand the exceptional locus of The projection $\pi$ to $\Gr(2,4)$ it suffices to apply the usual combination of \cref{thm:degloc}, \cref{thm:ddl,thm:en}, in this way, we get that it is a surface $S$ which can be described as $\mZ(\mU^{\vee}(2))\subset\Gr(2,4)$. Now $S$ is such that $e(S)=71$, $K_S=\of(1)_S$, $(-K_S)^2=13$, $h^{1,0}=0$, $h^{2,0}=6$ and $h^{1,1}=57$.\newline
Hence $X=\Bl_S\Gr(2,4)$, and in particular is rational. 
\end{fano} \vspace{5mm}

\begin{fano}
\fanoid{$K_{669}$}
\label[mystyle]{F.2.42}
$X=\mZ(\of(2,1) \oplus \of(0,1)^{\oplus 3}) \subset \PP^2 \times \Gr(2,5)$.
\subsubsection*{Invariants}  $h^0(-K)=21, \ (-K)^4=70,\ \chi(T_X)=-30, \ h^{1,1}=2 , \ h^{3,1}=3, \ h^{2,2}=37$.
\subsubsection*{Description} $X$ is a conic bundle on $X_5^3$ which discriminates along a degree three surface singular in three points.
\subsubsection*{Identification} We can rewrite $X$ as $\mZ(\of(1,2))\subset\PP_{X_5^3}(\of_{X_5^3}^{\oplus 3})$, now we can twist $\of_{X_5^3}^{\oplus 3}$ with $\of_{X_5^3}(-1)$ we get $X=\mZ(\of_{\mR}(2))\subset\PP_{X_5^3}(\of_{X_5^3}(-1){\oplus 3})$. Applying \cref{lm:conicbdl} with $\mathcal{E}=\of^{\oplus 3}$ and $\mathcal{K}=\of(1)$, we have that the discriminant locus of the conic bundle is a degree 3 surface singular in 20 points. \newline
Note that the projection to $\PP^2$ gives a $\DP_5$-fibration since the fiber is $\Gr(2,5)$ cut with 4 hyperplane sections. This, by \cite{isko}, gives the rationality of $X$. 

\end{fano} \vspace{5mm}

\begin{fano}
\fanoid{$K_{136}$}
\label[mystyle]{F.2.43}
$\mZ(\mU^{\vee}_{\Gr(2,4)}(1,0)\oplus\of(0,1))\subset\PP^3\times\Gr(2,4)$
\subsubsection*{Invariants} $h^0(-K)=81, \ (-K)^4=384,\ \chi(T_X)=10,\ h^{1,1}=2,\ h^{3,1}=0,  \ h^{2,2}=2$.
\subsubsection*{Description} $X=\PP_{Q_3}(\mU_{\Gr(2,4)|Q_3})$.
\subsubsection*{Identification} We simply apply the self-duality of $\Gr(2,4)$,  \crefpart{lm:blowFlagGr}{item_1} and \cref{lm:flag2.5} and obtain that $X$ is $\PP_{Q_3}(\mU_{\Gr(2,4)|Q_3})$.
\end{fano} \vspace{5mm}

\begin{fano}
\fanoid{$K_{723}$}
\label[mystyle]{F.2.44}
$X=\mZ(\of(0,1)^{\oplus 2} \oplus \of(1,2)) \subset \PP^1 \times \Gr(2,5).$
\subsubsection*{Invariants}  $h^0(-K)=16, \ (-K)^4=45,\ \chi(T_X)=-50, \ h^{1,1}=2,\ h^{3,1}=8, \ h^{2,2}=72$.
\subsubsection*{Description} $X=\Bl_SX_5^4$, with $S=Q_3\cap Q_3'\subset X_5^4$.
\subsubsection*{Identification} We simply apply \cref{lm:blowPPtX}, and obtain that $X=\Bl_SX_5^4$, with $S=\mZ(\of(2)^{\oplus 2})\subset X_5^4$, in this way we have that $e(S)=88$, $K_S=\of_{\Gr(2,5)}(1)_{|S}$, $(-K_S)^2=20$, $h^{1,0}=0$, $h^{2,0}=8$. Since $X$ is birational to $X_5^4$, it is rational.

\end{fano} \vspace{5mm}

\begin{fano}
\fanoid{$K_{219}$}
\label[mystyle]{F.2.45}
$X=\mZ(\mU^{\vee}_{\Gr(2,5)}(1,0)\oplus\of(0,1)^{\oplus 3})\subset\PP^3\times\Gr(2,5)$
\subsubsection*{Invariants} $h^0(-K)=51, \ (-K)^4=224,\ \chi(T_X)=-1,\ h^{1,1}=2,\ h^{3,1}=0, \ h^{2,2}=3$.
\subsubsection*{Description} $X$ a $\PP^1$-bundle whose fiber jumps to a $\PP^2$ over one point.
\subsubsection*{Identification} We can write $X$ as $\mZ(\of(0,1)\oplus \mU^{\vee}_{\Gr(2,5)}(0,1))\subset\PP_{X_5^3}(\of_{X_5^3}^{\oplus 5})$, now using \cref{thm:seq} we have that $X=\mZ(\of_{\mR}(1))\subset\PP_{X_5^3}(\mQ_{X_5^3}^{\vee})$, which is a $\PP^1$-bundle over $X_5^3$. To understand where the fiber of The projection $\pi$ to $X_5^3$ becomes a $\PP^2$ we can apply \cref{thm:degloc}, \cref{thm:ddl,thm:en} and get that $\pi$ has fiber a $\PP^2$ on one point of $X_5^3$.\\
The projection $\pi_1$ to $\PP^3$ is a conic bundle over the latter. The generic fiber of $\pi_1$ is $\Gr(2,5)$ cut by a section of $\mU^{\vee}_{\Gr(2,5)}$ and a codimension three linear subspace, hence a conic. By \cref{thm:seq} we can rewrite $X$ as $\mZ(\of_{\mR}(1)^{\oplus 3})\subset\bbGr_{\PP^3}(2,\mQ_{\PP^3}^{\vee}\oplus\of_{\PP^3})$, which gives the morphism:
\[
    \varphi:\W^2(\mQ^{\vee}\oplus\of)\to\of^{\oplus 3}
\]
and if we call $\mathcal{E}=Ker\varphi$, then it is the rank three bundle that defines the conic bundle. By \cref{lm:conicbdl} with $\mathcal{K}=\of$, then we get that the conic bundle degenerates on $\Delta$ which is a \red{sextic surface in $\PP^3$} with \red{40} singular points.
\end{fano} \vspace{5mm}

\begin{fano}
\fanoid{$K_{195}$}
\label[mystyle]{F.2.46}
$X=\mZ(\mU_{\Gr(2,5)}^{\vee}(1,0)\oplus\of(0,1)^{\oplus 2})\subset\PP^2\times\Gr(2,5)$
\subsubsection*{Invariants} $h^0(-K)=63, \ (-K)^4=288,\ \chi(T_X)=2, \ h^{1,1}=2,\ h^{3,1}=0, \ h^{2,2}=4$.
\subsubsection*{Description} $X=\Bl_{Q_2}X_5^4$.
\subsubsection*{Identification} The projection $\pi$ to $\Gr(2,5)$ gives a birational map between $X$ and $X_5^4$ since the fiber is $\PP^2$ cut with a codimension two linear subspace. To understand where the fiber has greater dimension we use \cref{thm:degloc}. We apply the combination of \cref{thm:degloc}, \cref{thm:ddl,thm:en} and we get that the locus we are looking for is a surface $S$ that can be described as $\mZ(\mQ_{\Gr(2,5)}(1,0))\subset\PP^1\times X_5^4$. $S$ is such that $e(S)=4$, $K_S=\of(1,-2)_{|S}$, $(-K_S)^2=8$, $h^{1,0}=h^{2,0}=0$ and $h^0(S,K_S^{\otimes 2})=0$ hence, by the classification, it is a $\DP_8$. In particular, note that if we project to $\PP^1$ we have that the generic fiber is $\Gr(2,5)$ cut with a section of $\mQ_{\Gr(2,5)}$ and a codimension two linear subspace, hence $\PP^1$. This implies that $S$ is birational to $\PP^1\times\PP^1$, but since it is a $\DP_8$, it is $\PP^1\times\PP^1$. \newline
In the end $X=\Bl_{S}X_5^4$, so, in particular, is rational.
\end{fano} \vspace{5mm}

\begin{fano}
\fanoid{$K_{230}$}
\label[mystyle]{F.2.47}
$X=\mZ(\mQ_{\Gr(2,4)}(1,0)\oplus\of(0,1)\oplus\of(2,0))\subset\PP^4\times\Gr(2,4)$
\subsubsection*{Invariants} $h^0(-K)=54, \ (-K)^4=240,\ \chi(T_X)=1,\ h^{1,1}=2 ,\ h^{3,1}=0, \ h^{2,2}=2$.
\subsubsection*{Description} $X$ is a conic bundle that discriminates on a degree-two surface $S\subset Q_3$ singular in 4 points.
\subsubsection*{Identification} Using \crefpart{flagCor}{cor:flag_2} we can rewrite $X$ as $\mZ(\of_{\mR}(2))\subset\PP_{Q_3}(\mU^{\vee}_{\Gr(2,4)|Q_3}\oplus\of_{Q_3})$, which is a conic bundle whose discriminant, by \cref{lm:conicbdl} with $\mathcal{E}=\mU^{\vee}\oplus\of$ and $\mathcal{K}=\of$, is a degree-2 surface singular in 4 points.\\
The other projection $\pi_2$ over $Q_3\subset\PP^4$, has a generic fiber $\PP^1$, since is $\Gr(2,4)$ cut by one global section of $\mQ_{\Gr(2,4)}$ and a global section of $\of_{\Gr(2,4)}(1)$. Using \cref{thm:degloc,thm:ddl} we obtain that the fiber does not jump in dimension, hence $X$ is a $\PP^1$-bundle on $Q_3$. To recover the rank two bundle on $Q_3$ we are projectivizing note that by self-duality we can rewrite $X$ as $\mZ(\mU^{\vee}_{\Gr(2,4)}(1,0)\oplus\of(2,0)\oplus\of(0,1))\subset\PP^4\times\Gr(2,4)$. We can consider the map:
\[
\varphi:V_4\otimes\of_{Q_3}\to\of_{Q_3}(1).
\]
By \cref{thm:ddl} $\varphi$ does not degenerate, and we can define $\mathcal{K}'=ker\varphi$ as a rank-3 vector bundle over $Q_3$, since the dimension of fiber does not jump anywhere. In particular, $\mathcal{K}'$ fits the exact sequence 
\[
\mathcal{K}'\to V_4\otimes\of_{Q_3}\to\of_{Q_3}(1), 
\]
hence by \cref{thm:seq} $X=\mZ(\of_{\mR}(1))\subset\bbGr_{Q_3}(2,\mathcal{K}')$. Using the Pl\"ucker embedding we get $X=\mZ(\of_{\mR}(1))\subset\PP_{Q_3}(\bigwedge^2\mathcal{K}')$. Arguing as before, we can write a sequence
\[
0\to\mathcal{K}\to\bigwedge^2\mathcal{K}'\xrightarrow{\phi}\of_{Q_3}.
\]
In this way $X$ can be written as $\PP_{Q_3}(\mathcal{K})$.

\end{fano} \vspace{5mm}

\begin{fano}
\fanoid{$K_{238}$}
\label[mystyle]{F.2.48}
$X=\mZ(\mU_{\Gr(2,4)}^{\vee}(1,0)\oplus\of(0,2))\subset\PP^3\times\Gr(2,4)$
\subsubsection*{Invariants} $h^0(-K)=45, \ (-K)^4=192,\ \chi(T_X)=-4, \ h^{1,1}=2,\ h^{2,1}=2,\ h^{3,1}=0, \ h^{2,2}=2$.
\subsubsection*{Description} $X=\PP_Y(\mU_{\Gr(2,4)|Y})$, with $Y$ the Fano threefold \red{1-14} in \cite[Table 1]{DFT}.
\subsubsection*{Identification} We apply \crefpart{lm:blowFlagGr}{item_1} and \cref{lm:flag2.5} to rewrite $X$ as $\PP_Y(\mU_{\Gr(2,4)|Y})$, with $Y=\mZ(\of(2))\subset\Gr(2,4)$, which is also $\mZ(\of(2)^{\oplus 2})\subset\PP^5$, hence the Fano threefold \red{1-14} in \cite[Table 1]{DFT}.\\
The projection $\pi_1$ on $\PP^3$ is a conic bundle. In fact the generic fiber of $\pi_1$ is $\Gr(2,4)$ cut by a global section of $\mU^{\vee}_{\Gr(2,4)}$ and one of $\of_{\Gr(2,4)}(2)$. Using the self-duality of $\Gr(2,4)$ and \cref{lm:blowFlagGr} we write $X=\mZ(\of_{\mR}(2))\subset\PP_{\PP^3}(\mQ_{\PP^3}(-1))$. Using \cref{lm:conicbdl} with $\mathcal{E}=\mQ_{\PP^3}^{\vee}(1)$ and $\mathcal{K}=\of_{\PP^3}$ we get that the discriminant $\Delta$ is a degree four surface in $\PP^3$, hence a $K3$ surface, singular in $\Delta_{sing}=\{8\ points\}$.
\end{fano} \vspace{5mm}

\begin{fano}
\fanoid{$K_{679}$}
\label[mystyle]{F.2.49}
$X=\mZ(\mU_{\Gr(2,4)}^{\vee}(1,1)) \subset \PP^2 \times \Gr(2,4)$
\subsubsection*{Invariants} $h^0(-K)= 18, \ (-K)^4=56,\ \chi(T_X)=-36, \ h^{1,1}=2, \ h^{3,1}=4, \ h^{2,2}=47$.
\subsubsection*{Description} $X=\Bl_S\Gr(2,4)$, with $S$ a surface with $e(S)=55$, $K_S=\of_{\PP^3\times\Gr(2,4)}(1,0)_{|S}$ and $(-K_S)^2=5$.
\subsubsection*{Identification} The projection $\pi$ to $\Gr(2,4)$ is a birational map between the latter and $X$. Indeed we can rewrite $X$ as $\mZ(\mU^{\vee}_{\Gr(2,4)}(1,1))\subset\PP_{\Gr(2,4)}(\of_{\Gr(2,4)}^{\oplus 3})$ and twist $\of_{\Gr(2,4)}^{\oplus 3}$ with $\of_{\Gr(2,4)}(-1)$. We get $X=\mZ(\mU^{\vee}_{\Gr(2,4)}(1,0))\subset\PP_{\Gr(2,4)}(\of_{\Gr(2,4)}(-1)^{\oplus 3})$, which is a $\PP^2$-bundle cut by a codimension two linear subspace in the fibers. Thus the projection $\pi$ gives a birational map. To understand the exceptional locus of $\pi$ we apply \cref{thm:degloc}, \cref{thm:ddl,thm:en}. We get that the locus we are looking for is a surface $S\subset\Gr(2,4)$ which can be described as $\mZ(\mQ_{\Gr(2,4)}(1,0)\oplus\of(1,1)^{\oplus 3})\subset\PP^3\times\Gr(2,4)$. $S$ has $e(S)=55$, $K_S=\of_{\PP^3\times\Gr(2,4)}(1,0)_{|S}$, $(-K_S)^2=5$, $h^0(S,-K_S)=10$, $h^{1,0}=0$ and $h^{2,0}=4$.\newline
Note that since $X=\Bl_S\Gr(2,4)$, it is rational.

\end{fano} \vspace{5mm}

\begin{fano}
\fanoid{$K_{253}$}
\label[mystyle]{F.2.50}
$X=\mZ(\mQ_{\PP^4}\boxtimes\mU^{\vee}_{\Gr(2,8)}\oplus\mU^{\vee}_{\Gr(2,8)}(1,0)\oplus\of(0,1)^{\oplus 2})\subset\PP^4\times\Gr(2,8)$
\subsubsection*{Invariants} $h^0(-K)=42, \ (-K)^4=177,\ \chi(T_X)=-11, \ h^{1,1}=2,\ h^{3,1}=0, \ h^{2,2}=11$.
\subsubsection*{Description} $X$ is a small resolution of $Y=\Bl_{S}\PP^4$, with $S$ a $\DP_6$ singular in two points.
\subsubsection*{Identification} If we consider $\pi:X\to\PP^4$, then the generic fiber is $\Gr(2,8)$ cut by 5 sections of $\mU^{\vee}$ and 2 hyperplane section, hence it is a point.
To understand the exceptional loci we have to consider the two maps between vector bundles obtained by applying \cref{thm:degloc} on $\pi$:
\[
    \varphi:V_8\otimes\of_{\PP^4}\to\mQ_{\PP^4}\oplus\of_{\PP^4}(1),
\]
\[
    \phi:\bigwedge^2V_8\otimes\of_{\PP^4}\to\of_{\PP^4}^{\oplus 2}
\]
By \cref{thm:ddl,thm:en} we obtain that $D_4(\varphi)={2pts}$ and the fiber of $\pi$ over it is a $Q_2$. To understand $D_1(\phi)$ let us notice that $\varphi$ and $\phi$ are related in the following way:
\[
    0\to\mathcal{K}\to V_8\otimes\of_{\PP^4}\xrightarrow{\varphi}\mQ_{\PP^4}\oplus\of_{\PP^4}(1)\to\mathcal{F}\to 0
\]
with $\mathcal{K}=ker(\varphi)$ and $\mathcal{F}=coker(\varphi)$.
Now we can take the second exterior power of the sequence and obtain the map 
\[
    \tau:\bigwedge^2\mathcal{K}\to\bigwedge^2V_8
\]
composing $\tau$ and $\phi$ we get
\[
    \Phi:\bigwedge^2\mathcal{K}\to\of_{\PP^4}^{\oplus 2}
\]
hence in the end we get that $D'=D_1(\phi_{|D})=D_1(\Phi)$, which is, by \cref{thm:ddl,thm:en}, a $\DP_6$. The final picture is: $X-{2pts}=\Bl_{\DP_6}\PP^4$ and over the two points there is a $Q_2$.
The projection to $\Gr(2,8)$ is a small resolution of a singular fourfold $Y$ of degree 28.

\end{fano} \vspace{5mm}

\begin{fano}
\fanoid{$K_{142}$} 
\label[mystyle]{F.2.51}
$X=\mZ(\mQ_{\PP^2}\boxtimes\mU_{\Gr(2,6)}^{\vee}\oplus\of(0,1)^{\oplus 2})\subset\PP^2\times\Gr(2,6)$
\subsubsection*{Invariants} $h^0(-K)=75, \ (-K)^4=352,\ \chi(T_X)=6,\ h^{1,1}=2,\ h^{3,1}=0,  \ h^{2,2}=4$.
\subsubsection*{Description}  $X$ is a small resolution of $Y$, a fourfold singular in a point with $e(Y)=9$.
\subsubsection*{Identification} $X$ is the zero locus of $\sigma\in V_3\otimes V_6^{\vee}$ and $\omega_1,\omega_2\in\bigwedge^2V_6^{\vee}$. We first study $X'=\mZ(\sigma)$. If we project to $\Gr(2,6)$ by \cref{thm:degloc} we are studying where the map \[\varphi:\mU_{\Gr(2,6)}\to\of_{\Gr(2,6)}^{\oplus 3}\] degenerates to a rank-one map. We are looking for $D_1(\varphi)$. Using \cref{thm:ddl}, we get that $D_1(\varphi)$ is a sixfold, and since $D_0(\varphi)\subset Y'$ is not empty, we have that it is singular along $D_0(\varphi)$, which is a surface $S$. Cutting with the 2 remaining linear sections, we obtain that $X$ is birational to $Y=Y'\cap H_1\cap H_2$, a fourfold, singular in a point $p=S\cap H_1\cap H_2$. The projection $\pi$ to $\Gr(2,6)$ is the birational map between $X$ and $Y$, and, over the singular point $p\in Y$, the fiber is a $\PP^2$. $X$ is a small resolution of $Y$ along its singular point.\\
The projection $\pi_1$ to $\PP^2$ gives a quadric bundle. Indeed the generic fiber of $\pi_1$ is $\Gr(2,6)$ cut with two global sections of $\mU^{\vee}_{\Gr(2,6)}$ and a linear subspace of codimension two.

\end{fano} \vspace{5mm}

\begin{fano}
\fanoid{$K_{686}$}
\label[mystyle]{F.2.52}
$X=\mZ(\mU^{\vee}_{\Gr(2,5)}(0,1)\oplus\of(0,1)\oplus\of(2,0) \subset \PP^2 \times \Gr(2,5)$
\subsubsection*{Invariants}  $h^0(-K)=27, \ (-K)^4=96,\ \chi(T_X)=-15,\ h^{1,1}=2, \ h^{2,1}=7,\ h^{3,1}=0, \ h^{2,2}=2$.
\subsubsection*{Description} $X$ is $Q_1\times Y$, with Y the Fano threefold \red{1-6} in \cite[Table 1]{DFT}.
\subsubsection*{Identification} The bundles cut the factors as $X=Q_1\times Y$, with $Y=\mZ(\of(1)\oplus\mU_{\Gr(2,5)}(1))\subset\Gr(2,5)$, which is exactly the Fano threefold \red{1-6} in \cite[Table 1]{DFT}. 
\end{fano} \vspace{5mm}

\begin{fano}
\fanoid{$K_{689}$}
\label[mystyle]{F.2.53}
$\mZ(\mQ_{\PP^2}(0,2)\oplus\of(1,1))\subset\PP^2\times\PP^5$
\subsubsection*{Invariants} $h^0(-K)=17, \ (-K)^4=51,\ \chi(T_X)=-36,\ h^{1,1}=2,\ h^{3,1}=2,\ h^{2,2}=41$.
\subsubsection*{Description} $X$ is the Fano fourfold \red{A-67} in \cite[Table 8]{BFMT}
\end{fano} \vspace{5mm}

\begin{fano}
\fanoid{$K_{699}$}
\label[mystyle]{F.2.54}
$X=\mZ(\mU^{\vee}_{\Gr(2,5)}(0,1)\oplus\of(1,1))\subset\PP^1\times\Gr(2,5)$
\subsubsection*{Invariants} $h^0(-K)=19, \ (-K)^4=60,\ \chi(T_X)=-31,\ h^{1,1}=2,\ h^{3,1}=1,\ h^{2,2}=32$.
\subsubsection*{Description} $X$ is the Fano \red{K3-24} in \cite[Table 6]{BFMT}.
\end{fano} \vspace{5mm}

\begin{fano}
\fanoid{$K_{288}$}
\label[mystyle]{F.2.55}
$X=\mZ(\of(1,1)\oplus(0,1)^{\oplus 2})\subset\PP^1\times\Gr(2,5)$
\subsubsection*{Invariants} $h^0(-K)=54, \ (-K)^4=240,\ \chi(T_X)=-4,\ h^{1,1}=2 ,\ h^{3,1}=0, \ h^{2,2}=7$.
\subsubsection*{Description} $X=\Bl_{\DP_5}X_5^4$.
\subsubsection*{Identification} $X$ can be rewritten as $\mZ(\of(1,1))\subset\PP^1\times X_5^4$. In this way, we simply apply \cref{lm:blowPPtX} and obtain that $X=\Bl_SX_5^4$, with $S=X_5^4\cap H_1\cap H_2$ and this is a $\DP_5$.\newline
Since $X=\Bl_{\DP_5}X_5^4$ is in particular rational.
\end{fano} \vspace{5mm}

\begin{fano}
\fanoid{$K_{464}$}
\label[mystyle]{F.2.56}
$X=\mZ(\mU^{\vee}_{\Gr(2,5)}\boxtimes\mQ_{\Gr(3,5)}\oplus\of(1,0)^{\oplus 2}\oplus\of(0,1)^{\oplus 2})\subset\Gr(2,5)\times\Gr(3,5)$
\subsubsection*{Invariants}  $h^0(-K)=39, \ (-K)^4=162,\ \chi(T_X)=-8, \ h^{1,1}=2,\ h^{3,1}=0, \ h^{2,2}=8$.
\subsubsection*{Description} $X=\Bl_{\DP_6}X_5^4$.
\subsubsection*{Identification} Using \crefpart{lm:blowFlagGr}{item_1} and \cref{lm:flag2.5} we rewrite $X$ as $\mZ(\of_{\mR}(1)^{\oplus 2})\subset\PP_{X_5^4}(\mQ_{\Gr(2,5)}(-1)_{|X_5^4})$. In this way, the natural projection to $X_5^4$ gives a birational map between the latter and $X$, indeed $X$ is a $\PP^2$-bundle over $X_5^4$ with generic fiber cut by two relative hyperplane sections. Note that applying \cref{thm:degloc}, \cref{thm:ddl,thm:en} we get that the fiber degenerates to a $\PP^1$ along a surface $S$ that can be described as $\mZ(\bigwedge^2\mQ_{\Gr(2,5)}(1,0)\oplus\of(0,1)^{\oplus 2})\subset\PP^1\times\Gr(2,5)$. $S$ is such that $e(S)=8$, $\mathcal{K}_{S}=\of(1,-1)_{|S}$, $(-K_S)^2=4$, $h^{1,0}=h^{2,0}=0$ and $h^0(S,K_S^{\otimes 2})=0$, using the classification, we get that $X$ is a $\DP_4$.\newline
Since $X=\Bl_{\DP_4}X_5^4$ it is rational.
\end{fano} \vspace{5mm}

\begin{fano}
\fanoid{$K_{350}$}
\label[mystyle]{F.2.57}
$X=\mZ(\mQ_{\Gr(2,4)}\boxtimes\mU^{\vee}_{\Gr(3,6)}\oplus\of(0,1)^{\oplus 3})\subset\Gr(2,4)\times\Gr(3,6)$
\subsubsection*{Invariants} $h^0(-K)=42, \ (-K)^4=177,\ \chi(T_X)=-12, \ h^{1,1}=2,\ h^{3,1}=0, \ h^{2,2}=12$.
\subsubsection*{Description} $X=\Bl_{\Bl_{9pts} \PP^2}\Gr(2,4)$.
\subsubsection*{Identification} We use the self-duality of both $\Gr(2,4)$ and $\Gr(3,6)$ to write $X=\mZ(\mU^{\vee}_{\Gr(2,4)}\boxtimes\mQ_{\Gr(3,6)}\oplus\of(0,1)^{\oplus 3})\subset\Gr(2,4)\times\Gr(3,6)$. Using \crefpart{flagCor}{cor:flag_2} we obtain $X=\mZ(\mU^{\vee}_{1}\oplus\of(0,1)^{\oplus 3})\subset\Fl(2,3,6)$, and by \cref{lm:flag2.5} we have $X=\mZ(\of_{\mR}(1)^{\oplus 3})\subset\PP_{\Gr(2,4)}(\mQ_{\Gr(2,4)}(-1)\oplus\of(-1)^{\oplus 2})$. This is a $\PP^3$-bundle over $\Gr(2,4)$ with the generic fiber cut with 3 relative-hyperplane sections, so the natural projection to $\Gr(2,4)$ gives a birational map between the latter and $X$. To understand where the fiber degenerates we use \cref{thm:degloc}, \cref{thm:ddl,thm:en} and obtain that the locus we are looking for is a surface $S$ that can be described as $\mZ(\mQ_{\Gr(2,4)}(1,0) \oplus\of(1,1)^{\oplus 2})\subset\PP^2\times\Gr(2,4)$, which has $e(S)=12$, $\mathcal{K}_{S}=\of(1,-1)_{|S}$, $\mathcal{K}_{S}^2=0$, $h^{1,0}=h^{2,0}=0$ and $h^0(S,K_S^{\otimes 2})=0$. Using the classification, we get that $S$ is the blow-up of $\PP^2$ in 9 points.\newline
Note that since $X=\Bl_S\Gr(2,4)$ is in particular rational.
\end{fano} \vspace{5mm}

\begin{fano}
\fanoid{$K_{361}$}
\label[mystyle]{F.2.58}
$X=\mZ(\mQ_{\PP^3}\boxtimes\mU^{\vee}_{\Gr(2,7)}\oplus\of(0,1)^{\oplus 2}\oplus\of(1,1))\subset\PP^3\times\Gr(2,7)$
\subsubsection*{Invariants}$h^0(-K)=33, \ (-K)^4=129,\ \chi(T_X)=-23, \ h^{1,1}=2,\ h^{3,1}=1, \ h^{2,2}=23$.
\subsubsection*{Description} $X$ is the Fano fourfold \red{R-62} in \cite[Table 7]{BFMT}.
\end{fano} \vspace{5mm}

\begin{fano}
\fanoid{$K_{364}$}
\label[mystyle]{F.2.59}
$X=\mZ(\of(2,0)\oplus\of(0,1)^{\oplus 3})\subset\PP^2\times\Gr(2,5)$
\subsubsection*{Invariants} $h^0(-K)=69, \ (-K)^4=320,\ \chi(T_X)=6, \ h^{1,1}=2,\ h^{3,1}=0, \ h^{2,2}=2$.
\subsubsection*{Description} $X$ is $Q_1\times Y$ with $Y$ the Fano threefold \red{1-15} in \cite[Table 1]{DFT}
\subsubsection*{Identification} The bundles cut the factors of the ambient space separately, hence $X$ is $Q_1\times Y$ with $Y=\mZ(\of(1)^{\oplus 3})\subset\Gr(2,5)$, hence the Fano threefold \red{1-15} in \cite[Table 1]{DFT}.
\end{fano} \vspace{5mm}

\begin{fano}
\fanoid{$K_{374}$}
\label[mystyle]{F.2.60}
$\mZ(\mQ_{\PP^6}\boxtimes\mU^{\vee}_{\Gr(2,9)}\oplus\mU^{\vee}_{\Gr(2,9)}(1,0)^{\oplus 2})\subset\PP^6\times\Gr(2,9)$
\subsubsection*{Invariants} $h^0(-K)=45, \ (-K)^4=193,\ \chi(T_X)=-14,\ h^{1,1}=2,\ h^{3,1}=0, \ h^{2,2}=15$.
\subsubsection*{Description} $X$ is a small resolution of $Y$, a fourfold with $e(Y)=17$, $deg(Y)=5$ and singular in four points.
\subsubsection*{Identification} $X$ is the zero locus of $\sigma\in (V_7\oplus V_7^{\vee}\oplus V_7^{\vee})\otimes V_9^{\vee}$. Once we project to $\PP^6$ we are looking for the planes $\Pi\in\Gr(2,9)$ such that fixed a line $l\in\PP^6$ the relation $\sigma(l,\Pi)=0$ holds. This, \cref{thm:degloc}, is equivalent to study the map \[\varphi:\mQ^{\vee}_{\PP^6}\oplus\of(-1)^{\oplus 2}\to\of_{\PP^6}^{\oplus 9}\] in particular $X$, for dimensional reason, is birational to the fourfold $Y=D_7(\varphi)$, for \cref{thm:en} we get that $e(Y)=17$ and $deg(Y)=5$. Moreover, since $D_6(\varphi)$ is not empty, we have that $Y$ is singular in $D_6(\varphi)$, which consists, by \cref{thm:ddl,thm:en}, of four points. Over those four points, the projection to $\PP^2$ is a $\PP^2$, hence $X$ is a small resolution of $Y$.
\end{fano} \vspace{5mm}

\begin{fano}
\fanoid{$K_{385}$}
\label[mystyle]{F.2.61}
$X=\mZ(\mQ_{\Gr(2,5)} \boxtimes \mU^{\vee}_{\Gr(2,6)}\oplus \of(1,0)^{\oplus 2} \oplus \of(0,1)^{\oplus 2}) \subset \Gr(2,5) \times \Gr(2,6)$
\subsubsection*{Invariants} $h^0(-K)=45, \ (-K)^4=193,\ \chi(T_X)=-8,\ h^{1,1}=2,\ h^{3,1}=0, \ h^{2,2}=9$.
\subsubsection*{Description} $X=\Bl_{\PP^2}X_6^{4}$.
\subsubsection*{Identification} Let $X'$ be $\mZ(\mQ_{\Gr(2,5)} \boxtimes \mU^{\vee}_{\Gr(2,6)})\subset\Gr(2,5)\times\Gr(2,6)$, this is by \crefpart{lm:blowFlagGr}{item_2} $\Bl_{\PP^4}\Gr(2,6)$. We consider the bundles $\of(1,0)^{\oplus 2}$ and $\of(0,1)^{\oplus 2}$, which correspond to four hyperplane sections cutting $\Gr(2,6)$ with the first two not intersecting the exceptional locus $\PP^4=\mZ(\mQ)\subset\Gr(2,6)$. In this way $X$ is obtained from $X'$ as $\Bl_{\PP^4\cap H_3\cap H_4}\Gr(2,6)\cap H_1\cap H_2\cap H_3\cap H_4$, which is $\Bl_{\PP^2}X_6^4$.\newline
Note that for the above description, $X$ is rational.
\end{fano} \vspace{5mm}

\begin{fano}
\fanoid{$K_{15}$}
\label[mystyle]{F.2.62}
$\mZ(\mQ_{\PP^4}\boxtimes\mU^{\vee}_{\Gr(2,7)}\oplus\mU^{\vee}_{\Gr(2,7)}(1,0))\subset\PP^4\times\Gr(2,7)$
\subsubsection*{Invariants} $h^0(-K)=90, \ (-K)^4=433,\ \chi(T_X)=13,\ h^{1,1}=2,\ h^{3,1}=0,  \ h^{2,2}=3$.
\subsubsection*{Description} $X=\Bl_{C}\PP^4$, with $C$ a rational curve.
\subsubsection*{Identification} The projection $\pi$ to $\PP^4$ is a birational map between the latter and $X$. To understand the exceptional locus of the map we apply the usual combination of \cref{thm:degloc}, \cref{thm:ddl,thm:en}, and we obtain the locus where the fiber of $\pi$ degenerates to a $\PP^2$ is a rational curve $C$. Computing its degree we get that it is a conic.

\end{fano} \vspace{5mm}

\begin{fano}
\fanoid{$K_{407}$}
\label[mystyle]{F.2.63}
$\mZ(\mQ_{\PP^5}\boxtimes\mU^{\vee}_{\Gr(2,8)}\oplus\mU_{\Gr(2,5)^{\vee}(1,0)}\oplus\of(1,1))\subset\PP^5\times\Gr(2,8)$
\subsubsection*{Invariants} $h^0(-K)=39, \ (-K)^4=161,\ \chi(T_X)=-21,\ h^{1,1}=2,\ h^{3,1}=1,  \ h^{2,2}=23$.
\subsubsection*{Description} $X$ is the Fano fourfold \red{C-5} in \cite[Table 4]{BFMT}
\end{fano} \vspace{5mm}

\begin{fano}
\fanoid{$K_{473}$}
\label[mystyle]{F.2.64}
$\mZ(\mQ_{\Gr(2,4)}\boxtimes\mU^{\vee}_{\Gr(2,5)}\oplus\mU^{\vee}_{\Gr(2,4)}(0,1))\subset\Gr(2,4)\times\Gr(2,5)$
\subsubsection*{Invariants} $h^0(-K)=36, \ (-K)^4=146,\ \chi(T_X)=-15, \ h^{1,1}=2, \ h^{3,1}=0, \ h^{2,2}=14$.
\subsubsection*{Description} $X=\Bl_{\PP^2}Y$ with $Y$ the Fano fourfold \cref{F.1.8}.
\subsubsection*{Identification} Let us first consider $X'=\mZ(\mQ_{\Gr(2,4)}\boxtimes\mU^{\vee}_{\Gr(2,5)})\subset\Gr(2,4)\times\Gr(2,5)$, using \crefpart{lm:blowFlagGr}{item_2}, we obtain that $X'=\Bl_{\PP^3}\Gr(2,5)$. For the remaining section, notice that cutting $X'$ with the zero locus of a section of$\mU^{\vee}_{\Gr(2,4)}(0,1)$ is equivalent to cut $\Gr(2,5)$ with $\mZ(\mU^{\vee}_{\Gr(2,5)}(1))$, and $\PP^3=\mZ(\mQ)\subset\Gr(2,5)$ with a hyperplane section, since $\mZ(\mU^{\vee}_{\Gr(2,5)})$ contains $\PP^3$. In this way $X=\Bl_{\PP^2}Y$ with $Y=\mZ(\mU^{\vee}(1))\subset\Gr(2,5)$ which is \cref{F.1.8}.
\end{fano} \vspace{5mm}

\begin{fano}
\fanoid{$K_{439}$}
\label[mystyle]{F.2.65}
$X=\mZ(\mQ_{\Gr(2,6)}(1,0)\oplus\of(0,1)^{\oplus 4}) \subset \PP^4 \times \Gr(2,6)$
\subsubsection*{Invariants} $h^0(-K)=35, \ (-K)^4=141,\ \chi(T_X)=-14, \ h^{1,1}=2, \ h^{3,1}=0, \ h^{2,2}=13$.
\subsubsection*{Description} $X=\Bl_{\DP_5}X_6^4$.
\subsubsection*{Identification} The projection $\pi$ to $\Gr(2,6)$ gives a birational map between $X$ and $X_6^4$.
To understand the exceptional locus of $\pi$ let us note that $X$ is the zero locus of $\sigma\in V_5^{\vee}\otimes V_6$ and $\omega_1,\omega_2,\omega_3,\omega_4\in\bigwedge^2 V_6^{\vee}$. 
For now, let us study $X'=\mZ(\sigma)$. Once we project to $\Gr(2,6)$ this is equivalent to find the lines $l\in\PP^4$ such that, fixed a plane $\Pi\in\Gr(2,6)$, the relation $\sigma(l,\Pi)=0$ holds.
Which means, \cref{thm:degloc}, that we can instead study the map \[\varphi:\mQ^{\vee}_{\Gr(2,6)|X_6^4}\to\of_{X_6^4}^{\oplus 5}\] So, the exceptional locus of $\pi$ coincides to the locus where the map $\varphi$ restricted to $X_6^4$ becomes a rank-3 map, i.e., $D_1(\varphi)$. Using \cref{thm:ddl,thm:en} we have that  $D_1(\varphi_{X_6^4})$ is a surface $S$ with $e(S)=7$, $K_S=\of_{\Gr(2,6)}(-1)_{|S}$, $(-K_S)^2=5$, $h^{1,0}=h^{2,0}=0$ and $h^0(S,K_S^{\otimes 2})=0$, hence using the classification we have that $S$ is $\DP_5$.\newline
Note that $X$ is birational to $X_6^4$, hence it is rational.
\end{fano} \vspace{5mm}

\begin{fano}
\fanoid{$K_{717}$}
\label[mystyle]{F.2.66}
$X=\mZ(\of(0,1) \oplus \of(0,2)  \oplus \of(1,1)) \subset \PP^1 \times \Gr(2,5)$
\subsubsection*{Invariants} $h^0(-K)=17, \ (-K)^4=50,\ \chi(T_X)=-38, \ h^{1,1}=2, \ h^{3,1}=2, \ h^{2,2}=42$.
\subsubsection*{Description} $X$ is the Fano fourfold \red{A-69} \cite[Table 8]{BFMT}.
\end{fano} \vspace{5mm}

\begin{fano}
\fanoid{$K_{474}$}
\label[mystyle]{F.2.67}
$X=\mZ(\mQ_{\PP^2}(0,1)\oplus\of(0,1)^{\oplus 2})\subset\PP^2\times\Gr(2,5)$
\subsubsection*{Invariants} $h^0(-K)=40, \ (-K)^4=165,\ \chi(T_X)=-7, \ h^{1,1}=2,\ h^{2,1}=1,\ h^{3,1}=0, \ h^{2,2}=4$.
\subsubsection*{Description} $X=\Bl_EX_5^4$, with $E$ a quintic elliptic curve.
\subsubsection*{Identification} Note that $X$ can be rewritten as $\mZ(\mQ_{\PP^2}(0,1))\subset\PP^2\times X_5^4$. Using \cref{lm:blowPPtX} $X=\Bl_C X_5^4$, with $C=\mZ(\of(1)^{\oplus 3})\subset X_5^4$, which is an elliptic curve $E$ using the adjunction formula and the classification of curves.\newline
Since $X$ is birational to $X_5^4$, it is in particular rational.
\end{fano} \vspace{5mm}

\begin{fano}
\fanoid{$K_{483}$}
\label[mystyle]{F.2.68}
$\mZ(\mQ_{\Gr(2,4)}\boxtimes\mU_{\Gr(2,6)}^{\vee}\oplus\of(1,0)\oplus\of(0,1)^{\oplus 3})\subset\Gr(2,4)\times\Gr(2,6)$
\subsubsection*{Invariants} $h^0(-K)=36, \ (-K)^4=146,\ \chi(T_X)=-11, \ h^{1,1}=2, \ h^{3,1}=0, \ h^{2,2}=10$.
\subsubsection*{Description} $X=\Bl_{C}X_6^4$ with $C$ a rational curve.
\subsubsection*{Identification} Let $X'=\mZ(\mQ_{\Gr(2,4)}\boxtimes\mU_{\Gr(2,6)}^{\vee})\subset\Gr(2,4)\times\Gr(2,6)$, then by \crefpart{lm:blowFlagGr}{item_3} we have that, outside a point $p$, $X'=\Bl_{\tilde{D}}\Gr(2,6)$ and $\tilde{D}=\mZ(\mQ_{\PP^3}\boxtimes\mU^{\vee}_{\Gr(2,6)})\subset\PP^3\times\Gr(2,6)$. Now, the remaining sections cut $\Gr(2,6)$ as 4 hyperplanes sections, while $\tilde{D}$ is cut only by 3 of them, since $\mZ(\of(1,0))\supset \tilde{D}$. In this way the final picture will be $X=\Bl_{\tilde{D}\cap\mZ(\of(0,1)^{\oplus 3})}X_6^4$ with $C=\tilde{D}\cap\mZ(\of(0,1)^{\oplus 3})\subset\PP^3\times\Gr(2,6)$ a rational curve such that $\mathcal{K}_C=\of(-2,0)_{|C}$, hence $Q_1$.\newline
Since $X$ is birational to $X_6^4$, it is rational.

\end{fano} \vspace{5mm}

\begin{fano}
\fanoid{$K_{529}$}
\label[mystyle]{F.2.69}
$X=\mZ(\mQ_{\Gr(2,5)}(1,0) \oplus \mU^{\vee}_{\Gr(2,5)}(0,1)) \subset \PP^3 \times \Gr(2,5)$
\subsubsection*{Invariants} $h^0(-K)=29, \ (-K)^4=110,\ \chi(T_X)=-19, \ h^{1,1}=2, \ h^{3,1}=0,\ h^{2,2}=17$.
\subsubsection*{Description} $X=\Bl_{\DP_5}Y$, with $Y$ the Fano fourfold \cref{F.1.8}
\subsubsection*{Identification} The projection $\pi$ to $\Gr(2,5)$ gives a birational map between $X$ and $Y=\mZ(\mU^{\vee}_{\Gr(2,5)}(1))\subset\Gr(2,5)$, which is \cref{F.1.8}. Now to understand where the fiber of the projection degenerates we use the usual argument and the combination of \cref{thm:degloc}, \cref{thm:ddl,thm:en} to get that the exceptional locus can be described as $S=\mZ(\mU^{\vee}(1))\subset\Gr(2,4)$ which is a $\DP_5$. 
\end{fano} \vspace{5mm}

\begin{fano}
\fanoid{$K_{714}$}
\label[mystyle]{F.2.70}
$X=\mZ(\of(2,0)\oplus\of(0,2)\oplus\of(0,1)^{\oplus 2})\subset\PP^2\times\Gr(2,5)$
\subsubsection*{Invariants} $h^0(-K)=24, \ (-K)^4=80,\ \chi(T_X)=-19,\ h^{1,1}=2,\ h^{2,1}=10,\ h^{3,1}=0,  \ h^{2,2}=2$.
\subsubsection*{Description} $X=Q_1\times Y$ with $Y$ the Fano threefold \red{1-5} in \cite[Table 1]{DFT}.
\subsubsection*{Identification} Since the bundles cut the two factors separately we have $X=Q_1\times Y$, with $Y=\mZ(\of(1)^{\oplus 2}\oplus\of(2))$, which is the Fano threefold \red{1-5} in \cite[Table 1]{DFT}. 
\end{fano} \vspace{5mm}

\begin{fano}
\fanoid{$K_{426}$}
\label[mystyle]{F.2.71} 
$X=\mZ(\mU^{\vee}_{\Gr(2,4)}\boxtimes\mQ_{\Gr(3,5)}\oplus\mU^{\vee}_{\Gr(2,4)}(0,1))\subset\Gr(2,4)\times\Gr(3,5)$
\subsubsection*{Invariants} $h^0(-K)=36, \ (-K)^4=147,\ \chi(T_X)=-13,\ h^{1,1}=2,\ h^{3,1}=0,  \ h^{2,2}=13$.
\subsubsection*{Description} $X=\Bl_{\Bl_{10pts}}\Gr(2,4)$.
\subsubsection*{Identification} Let us apply \crefpart{flagCor}{cor:flag_1} and \cref{lm:flag2.5}. In this way, we can rewrite $X$ as $\mZ(\mU^{\vee}_{\Gr(2,4)}\boxtimes\of_{\mR}(1))\subset\PP_{\Gr(2,4)}(\mQ_{\Gr(2,4)}(-1)\oplus\of_{\Gr(2,4)}(-1))$, hence we have a $\PP^2$-bundle whose generic fibers are cut by two relative hyperplane sections, so the projection map to $\Gr(2,4)$ gives the birationality between the latter and $X$. To understand where the fiber has different dimensions we can apply \cref{thm:degloc} and study where the map \[\varphi:\mU_{\Gr(2,4)}\to\mQ_{\Gr(2,4)}(-1)\oplus\of_{\Gr(2,4)}(-1)\] degenerates to a rank one map, i.e., $D_1(\varphi)$. By applying \cref{thm:ddl,thm:en}, we get that $D_1(\varphi)$ is a surface $S$ which can be described as $\mZ(\mQ_{\Gr(2,4)}(1,0)^{\oplus 2}\oplus\of(1,1))\subset\PP^3\times\Gr(2,4)$. In particular, we compute $e(S)=13$, $K_S=\of(1,-1)_{S}$, $(-K_S)^2=-1$, $h^{1,0}=h^{2,0}=0$ and $h^0(S,K_S^{\otimes 2})=0$, and see that $S$ is rational, so by classification it is $\Bl_{10pts}\PP^2$.\newline
By description $X$ is rational.
\end{fano} \vspace{5mm}

\begin{fano}
\fanoid{$K_{662}$}
\label[mystyle]{F.2.72}
$\mZ(\mU^{\vee}(1,0)\oplus\of(0,1)\oplus\of(0,2))\subset\PP^2\times\Gr(2,5)$
\subsubsection*{Invariants} $h^0(-K)=22, \ (-K)^4=74,\ \chi(T_X)=-30,\ h^{1,1}=2,\ h^{3,1}=1,  \ h^{2,2}=30$.
\subsubsection*{Description} $X$ is the Fano fourfold \red{GM-18} in \cite[Table 5]{BFMT}.
\end{fano} \vspace{5mm}

\begin{fano}
\fanoid{$K_{671}$}
\label[mystyle]{F.2.73}
$X=\mZ(\mQ_{\PP^2}(0,1)\oplus\mQ_{\PP^2}(0,2))\subset\PP^2\times\PP^6$
\subsubsection*{Invariants} $h^0(-K)=18, \ (-K)^4=55,\ \chi(T_X)=-40,\ h^{1,1}=2,\ h^{3,1}=3,  \ h^{2,2}=47$.
\subsubsection*{Description} $X$ is a small resolution of  $Y$, a fourfold with $e(Y)=51$ and $deg(Y)=7$ singular in 8 points.
\subsubsection*{Identification} $X$ is the zero locus of $\sigma\in V_3\otimes(V_7^{\vee}\oplus\Sym^2V_7^{\vee})$. Let $\pi$ be the projection to $\PP^6$, then studying $\pi$ is equivalent to studying the lines $l\in\PP^2$ such that, fixed a line $\lambda\in\PP^6$, the relation $\sigma(l,\lambda)=0$ holds. Since the relation is generically empty, due to dimension reasons, we get that the image of $\pi$ coincides with the locus where 
\[\varphi:\of_{\PP^6}(-1)\oplus\of_{\PP^6}(-2)\to\of_{\PP^6}^{\oplus 3}\] 
degenerates to a rank-1 map, hence $D_1(\varphi)$. Using \cref{thm:ddl,thm:en} we obtain that the image of $\pi$ is a  fourfold $Y$ with $e(Y)=51$ and $deg(Y)=7$. Moreover, the generic fiber on $Y$ is a point hence the projection gives a birational map between $X$ and $Y$. Since $D_2(\varphi)\subset D_1(\varphi)$ is non-empty, we have that $Y$ is singular along $D_2(\varphi)$, which is, again applying \cref{thm:ddl,thm:en}, a set of 8 points. Hence $X$ is a small resolution of $Y$ with $\PP^2$'s on the singular points.\\
\begin{rmk}
    Note that the fourfold can be also directed described as the complete intersection between $\Bl_{Q_2\cap Q_2'\cap Q_2''}\PP^6\cap\Bl_{\PP^3}\PP^6$.
\end{rmk}

\end{fano} \vspace{5mm}

\begin{fano}
\fanoid{$K_{562}$}
\label[mystyle]{F.2.74}
$X=\mZ(\mQ_{\PP^2}\boxtimes\mU^{\vee}_{\Gr(2,6)}\oplus\mU^{\vee}_{\Gr(2,6)}(0,1))\subset\PP^2\times\Gr(2,6)$
\subsubsection*{Invariants} $h^0(-K)=27, \ (-K)^4=99,\ \chi(T_X)=-25,\ h^{1,1}=2,\ h^{3,1}=1,  \ h^{2,2}=25$
\subsubsection*{Description} $X$ is the Fano fourfold \red{GM-20} in \cite[Table 5]{BFMT}.
\end{fano} \vspace{5mm}

\section{Fano Varieties with Picard Rank 3}\label{sec:pic3}

\renewcommand\thefano{3--\arabic{fano}}
\setcounter{fano}{0}

\begin{fano}
\fanoid{$K_{503}$}
\label[mystyle]{F.3.1}
$X=\mZ(\mQ_{\Gr(2,4)}(0,1,0)\oplus\of(0,1,1)\oplus\of(2,0,0)\oplus\of(0,0,1))\subset\PP^2\times\PP^3\times\Gr(2,4)$
\subsubsection*{Invariants} $h^0(-K)=45, \ (-K)^4=192,\ \chi(T_X)=-2,\ h^{1,1}=3,\ h^{2,1}=1,\ h^{3,1}=0,  \ h^{2,2}=4$.
\subsubsection*{Description} $X$ is $Q_1\times Y$ with $Y$ the Fano threefold \red{2-17}  in \cite[Table 1]{DFT}.
\subsubsection*{Identification} Note that applying \crefpart{lm:blowFlagGr}{item_1} we can rewrite $X$ as $\mZ(\of(2,0,0)\oplus\of(0,1,1)\oplus\of(0,0,1))\subset\PP^2\times\Fl(1,2,4)$. And we immediately obtain $X=Q_1\times Y$, with $Y=\mZ(\of(1,1)\oplus\of(0,1))\subset\Fl(1,2,4)$, which is the Fano threefold \red{2-17}  in \cite[Table 1]{DFT}.
\end{fano} \vspace{5mm}

\begin{fano}
\fanoid{$K_{20}$}
\label[mystyle]{F.3.2}
$X=\mZ(\mQ_{\Gr(2,4)}(1,0,0)\oplus\mQ_{\Gr(2,4)}(0,1,0))\subset\PP^2_1\times\PP^2_2\times\Gr(2,4)$
\subsubsection*{Invariants} $h^0(-K)=76, \ (-K)^4=358,\ \chi(T_X)=9,\ h^{1,1}=3,\ h^{3,1}=0,  \ h^{2,2}=5$.
\subsubsection*{Description} $X=\Bl_{\Bl_{2pts}\PP^2}\Bl_{\PP^2}\Gr(2,4)$.
\subsubsection*{Identification} Let us first consider the projection $\pi_1$ on $\PP_2^2\times\Gr(2,4)$. This gives a birational map between $X$ and $Y=\mZ(\mQ_{\Gr(2,4)}(1,0))\subset\PP_2^2\times\Gr(2,4)$. Using \cref{thm:degloc,thm:ddl} we get that over a surface $S$ the fiber of $\pi_1$ becomes a $\PP^1$, while using \cref{thm:en} we have that $S$ which can be described as $\mZ(\of(1,1))\subset\PP^1\times\PP^2$, i.e., $\Bl_{pt}\PP^2$. Now we can study the projection $\pi_2$ from $Y$ to $\Gr(2,4)$. This is a birational map between $Y$ and $\Gr(2,4)$, using the previous combination of results again, we get that over a surface $S$ the fiber of $\pi_2$ is a $\PP^1$, in particular, with the same argument we have that $S$ is which can be described as $\PP^2$.\newline
The final picture is $X=\Bl_{\Bl_{pt}\PP^2}\Bl_{\PP^2}\Gr(2,4)$, which gives also the rationality.
\end{fano} \vspace{5mm}

\begin{fano}
\fanoid{$K_{521}$}
\label[mystyle]{F.3.3}
$X=\mZ(\of(1,0,1)\oplus\of(0,1,1)\oplus\mQ_{\Gr(2,4)}(0,1,0))\subset\PP^1\times\PP^3\times\Gr(2,4)$
\subsubsection*{Invariants} $h^0(-K)=34, \ (-K)^4=136,\ \chi(T_X)=-12,\ h^{1,1}=3,\ h^{3,1}=0,  \ h^{2,2}=14$.
\subsubsection*{Description} $X=\Bl_{\DP_3}Y$, with $Y$ the Fano fourfold \cref{F.2.36}.
\subsubsection*{Identification} Using \cref{lm:blowPPtX} we have that $X$ is $\Bl_{S}Y$ with $Y=\mZ(\of(1,1)\oplus\mQ_{\Gr(2,4)}(1,0))$, which is the \cref{F.2.36} and $S=\mZ(\of(0,1)^{\oplus 2})\subset Y$. Now, $S$ is such that $e(S)=9$, $\mathcal{K}_{S}=\of(-1,0)_{|S}$, $(-K_S)^2=3$, $h^{1,0}=h^{2,0}=0$ and $h^0(S,2\mathcal{K}_{S})=0$ and is a $\DP_3$.\newline
Hence $X=\Bl_{\DP_3}Y$, which also gives us the rationality of this Fano.
\end{fano} \vspace{5mm}

\begin{fano}
\fanoid{$K_{524}$}
\label[mystyle]{F.3.4}
$X=\mZ(\mQ_{\Gr(2,4)}(0,1,0)\oplus\of(0,0,2)\oplus\of(1,1,0)\subset\PP^1\times\PP^3\times\Gr(2,4)$
\subsubsection*{Invariants} $h^0(-K)=34, \ (-K)^4=136,\ \chi(T_X)=-8,\ h^{1,1}=3,\ h^{2,1}=2,\ h^{3,1}=0,  \ h^{2,2}=8$
\subsubsection*{Description} $X=\Bl_{\DP_4}Y$ with $Y$ the \cref{F.2.48}.
\subsubsection*{Identification} Using \cref{lm:blowPPtX} we have that $X$ is $\Bl_{S}Y$ with $Y=\mZ(\of(0,2)\oplus\mQ_{\Gr(2,4)}(1,0))$, which \cref{F.2.48} and $S=\mZ(\of(1,0)^{\oplus 2})\subset Y$. Now $S$ is such that $e(S)=8$, $\mathcal{K}_{S}=\of(0,-1)_{|S}$, $(-K_S)^2=4$, $h^{1,0}=h^{2,0}=0$ and $h^0(S,2\mathcal{K}_{S})=0$ hence by classification is a $\DP_4$. 
\end{fano} \vspace{5mm}

\begin{fano}
\fanoid{$K_{21}$}
\label[mystyle]{F.3.5}
$X=\mZ(\mQ_{\PP^4}\boxtimes\mU^{\vee}_{\Gr(2,7)}\oplus\mQ_{\Gr(2,7)}(1,0,0)\oplus\mU^{\vee}_{\Gr(2,7)}(0,1,0))\subset\PP^5\times\PP^4\times\Gr(2,7)$
\subsubsection*{Invariants} $h^0(-K)=74, \ (-K)^4=347,\ \chi(T_X)=8,\ h^{1,1}=3,\ h^{3,1}=0,  \ h^{2,2}=5$
\subsubsection*{Description} $X=\Bl_{Q_2}Y$, with $Y$ the fourfold \cref{F.2.62}.
\subsubsection*{Identification} Recall that $Y=\mZ(\mQ_{\PP^4}\boxtimes\mU^{\vee}_{\Gr(2,7)}\oplus\mU^{\vee}_{\Gr(2,7)}(1,0))\subset\PP^4\times\Gr(2,7)$ is the fourfold \cref{F.2.62}, 
hence $X=\mZ(\of(1,0,0)\oplus\mQ_{\Gr(2,7)}(1,0,0))\subset\PP^6\times Y$. Now applying \cref{thm:seq} we can rewrite $X$ as $\mZ(\of_{\mR}(1))\subset\PP_Y(\mU_{\Gr(2,7)|Y})$, hence it is birational to $Y$ via the natural projection $\pi$ on the base of the projective bundle. To understand where the fiber of $\pi$ has greater dimension we can use \cref{thm:ddl,thm:en}, and this gives a surface $S$ where the fiber is a $\PP^1$. Moreover $S$ can be described as $\mZ(\mQ_{\PP^4}\boxtimes\mU^{\vee}_{\Gr(2,6)}\oplus\mU^{\vee}_{\Gr(2,6)}(1,0))\subset\PP^4\times\Gr(2,6)$ with $e(S)=4$, $K_S=\of(-2,0)_{|S}$, $(-K_S)^2=8$, $h^{1,0}=h^{2,0}=0$ and $h^0(S,2\mathcal{K}_{S})=0$ hence by classification $S$ is a $\DP_8$, and since the degree is two it is a $Q_2$.\newline
Since $Y$ is rational, also $X$ is such.
\end{fano} \vspace{5mm}

\begin{fano}
\fanoid{$K_{526}$}
\label[mystyle]{F.3.6}
$X=\mZ(\mQ_{\PP^2_1}(0,1,1)\oplus\of(1,0,1))\subset\PP^2_1\times\PP^2_2\times\PP^3$
\subsubsection*{Invariants} $h^0(-K)=30, \ (-K)^4=116,\ \chi(T_X)=-15,\ h^{1,1}=3,\ h^{3,1}=0,  \ h^{2,2}=17$
\subsubsection*{Description} $X=\Bl_{13pts}\PP^2$.
\subsubsection*{Identification} Note that $X$ is birational to $\PP^2\times\PP^2$ via The projection $\pi$ to $\PP^2\times\PP^2$. To understand what is the exceptional locus we have to consider the first degeneracy locus of the map
\[
    \varphi:V_4\to\mQ_{\PP^2_1}(0,1)\oplus\of(1,0)
\]
Using \cref{thm:degloc}, \cref{thm:ddl,thm:en} we obtain that $D_2(\varphi)$ is $\Bl_{13pts}\PP^2$.
\end{fano} \vspace{5mm}

\begin{fano}
\fanoid{$K_{555}$}
\label[mystyle]{F.3.7}
$X=\mZ(\mQ_{\Gr(2,4)}(0,1,0)\oplus\of(0,0,1)\oplus\of(1,1,1))\subset\PP^1\times\PP^3\times\Gr(2,4)$
\subsubsection*{Invariants}  $h^0(-K)=31, \ (-K)^4=120,\ \chi(T_X)=-18,\ h^{1,1}=3,\ h^{3,1}=1,  \ h^{2,2}=22$
\subsubsection*{Description} $X$ is the Fano fourfold \red{K3-48} in \cite[Table 6]{BFMT}.
\end{fano} \vspace{5mm}

\begin{fano}
\fanoid{$K_{623}$}
\label[mystyle]{F.3.8}
$X=\mZ(\of(0,0,1)^{\oplus 2}\oplus\of(1,0,1)\oplus\of(0,1,1))\subset\PP^1\times\PP^1\times\Gr(2,5)$
\subsubsection*{Invariants}  $h^0(-K)=28, \ (-K)^4=105,\ \chi(T_X)=-16,\ h^{1,1}=3,\ h^{3,1}=0,  \ h^{2,2}=17$
\subsubsection*{Description} $X=\Bl_{\Bl_{9pts}\PP^2}Y$ with $Y$ the fourfold \cref{F.2.55}.
\subsubsection*{Identification} Recall that $Y=\mZ(\of(1,1))\subset\PP^1\times X_5^4$ is the fourfold \cref{F.2.55}, hence $X=\mZ(\of(1,0,1))\subset\PP^1\times Y$. We can apply \cref{lm:blowPPtX} and obtain $X=\Bl_{S}Y$, with $S$ the surface $\mZ(\of(0,1)^{\oplus 2})\subset Y$. In particular $e(S)=12$, $K_S=\of(-1,0)_{|S}$, $(-K_S)^2=0$, $h^{1,0}=h^{2,0}=0$ and $h^0(S,K_S^{\otimes 2})=0$, hence $S$ is rational so is $\Bl_{9pts}\PP^2$.
\end{fano} \vspace{5mm}

\begin{fano}
\fanoid{$K_{294}$}
\label[mystyle]{F.3.9}
$X=\mZ(\mU^{\vee}_{\Gr(2,5)}(1,0,0)\oplus\mU^{\vee}_{\Gr(2,5)}\boxtimes\mQ_{\Gr(2,4)}\oplus\of(0,1,0)\oplus\of(0,0,1))\subset\PP^2\times\Gr(2,5)\times\Gr(2,4)$
\subsubsection*{Invariants} $h^0(-K)=54, \ (-K)^4=241,\ \chi(T_X)=1,\ h^{1,1}=3,\ h^{3,1}=0,  \ h^{2,2}=6$
\subsubsection*{Description} $X=\Bl_{\DP_7}Y$, with $Y$ the fourfold \cref{F.2.28}
\subsubsection*{Identification} Recall that $Y=\mZ(\mU^{\vee}_{\Gr(2,5)}\boxtimes\mQ_{\Gr(2,4)}\oplus\of(1,0)\oplus\of(0,1))\subset\Gr(2,5)\Gr(2,4)$ is the fourfold \cref{F.2.28}, hence $X=\mZ(\mU_{\Gr(2,5)}^{\vee}(1,0,0))\subset\PP^2\times\Gr(2,5)$. Note that The projection $\pi$ to $Y$ gives a birational map between the latter and $X$ and using \cref{thm:degloc} we get that the fiber of $\pi$ degenerates where the map \[\varphi:\mU_{\Gr(2,5)|Y}\to\of_Y^{\oplus 3}\] degenerates to a rank-1 map, hence on $D_1(\varphi)$. By \cref{thm:ddl,thm:en} we obtain that $D_1(\varphi)$ is a surface $S$ that can be described as $\mZ(\mQ_{\Gr(2,5)|Y}(1,0,0))\subset\PP^1\times Y$. In particular $e(S)=5$, $K_S=\of(1,-1,-1)_{|S}$, $(-K_S)^2=7$, $h^{1,0}=h^{2,0}=0$ and $h^0(S,K_S^{\otimes 2})=0$, and we get that $S=\DP_7$.\newline
Since $X=\Bl_{\DP_7}Y$ we have the rationality of this Fano.
\end{fano} \vspace{5mm}

\begin{fano}
\fanoid{$K_{637}$}
\label[mystyle]{F.3.10}
$X=\mZ(\of(0,0,1)^{\oplus 3}\oplus\of(1,1,1))\subset\PP^1\times\PP^1\times\Gr(2,5)$
\subsubsection*{Invariants} $h^0(-K)=27, \ (-K)^4=100,\ \chi(T_X)=-18,\ h^{1,1}=3,\ h^{3,1}=1,  \ h^{2,2}=22$
\subsubsection*{Description} $X$ is the Fano fourfold \red{K3-38} in \cite[Table 6]{BFMT}.
\end{fano} \vspace{5mm}

\begin{fano}
\fanoid{$K_{158}$}
\label[mystyle]{F.3.11}
$X=\mZ(\mQ_{\Gr(2,5)}\boxtimes\mU^{\vee}_{\Gr(2,6)}\oplus\mU^{\vee}_{\Gr(2,5)}\boxtimes\mU^{\vee}_{\Gr(2,6)})\subset\Gr(2,5)\times\Gr(2,6)$
\subsubsection*{Invariants} $h^0(-K)=60, \ (-K)^4=274,\ \chi(T_X)=2,\ h^{1,1}=3,\ h^{3,1}=0,  \ h^{2,2}=8$
\subsubsection*{Description} $X$ is the small contraction of $Y$, a fourfold singular in 4 points with $e(Y)=12$ and $deg(Y)=10$. 
\subsubsection*{Identification} $X$ is the zero locus of $\sigma\in V_6^{\vee}\otimes(V_5\oplus V_5^{\vee})$, hence, if we project to $\Gr(2,5)$ we can describe $X$ as the planes $\Pi_1\in\Gr(2,6)$ such that, fixed a plane $\Pi_2\in\Gr(2,5)$, the relation $\sigma(\Pi_1,\Pi_2)=0$ holds. In this way, by \cref{thm:degloc}, we have that $X$ is birational to the locus where the map \[\varphi:\mQ^{\vee}_{\Gr(2,5)}\oplus\mU_{\Gr(2,5)}\to\of_{\Gr(2,5)}^{\oplus 6}\] degenerates to a rank-4 map, hence $X$ is birational to $Y=D_1(\varphi)$. Using \cref{thm:ddl,thm:en} we have that $Y$ is a fourfold with $e(Y)=12$ and $deg(Y)=10$. Moreover, since there are points in $Y$ where the map $\varphi$ degenerates to a rank-3 map, we have that $Y$ is singular along $D_2(\varphi)$, which is, again by \cref{thm:ddl,thm:en}, a set of 4 points. Over these 4 points, the fiber of the projection is a $\PP^2$, hence $X$ is a small resolution of $Y$ on its singular locus.\\
\end{fano} \vspace{5mm}

\begin{fano}
\fanoid{$K_{398}$}
\label[mystyle]{F.3.12}
$X=\mZ(\mQ_{\PP^2}\boxtimes\mU_{\Gr(2,7)}^{\vee}\oplus\of(0,1)\oplus\Sym^2\mU^{\vee}_{\Gr(2,7)})\subset\PP^2\times\Gr(2,7)$
\subsubsection*{Invariants}  $h^0(-K)=39, \ (-K)^4=164,\ \chi(T_X)=-7,\ h^{1,1}=3,\ h^{3,1}=0,  \ h^{2,2}=12$.
\subsubsection*{Description} $X$ is a small resolution of $Y$, a fourfold singular in 4 points with $e(Y)=16$ and $deg(Y)=36$.
\subsubsection*{Identification} $X$ is the zero locus of $\sigma\in V_3\otimes V_7^{\vee}$ and $\alpha\in\bigwedge^{2}V_7^{\vee}\oplus\Sym^2V_7^{\vee}$. We can first consider $X'=\mZ(\sigma)$ and then restrict ourselves on $X'\cap\mZ(\alpha)$. If we project to $\Gr(2,7)$ then the fiber of the projection are the lines $l\in\PP^2$ such that, fixed a plane $\Pi\in\Gr(2,7)$, the relation $\sigma(l,\Pi)=0$ holds. By \cref{thm:en}, to have a non-empty fiber we have to restrict on planes $\Pi$ in the locus where \[\varphi:\mU_{\Gr(2,7)}\to\of_{\Gr(2,7)}^{\oplus 3}\] degenerates to a rank-1 map, but since $X=X'\cap\mZ(\alpha)$, then we can restrict $\varphi$ on $\mZ(\alpha)$, and get that $X$ is birational to $Y=D_1(\varphi_{|\mZ(\alpha)})$. Using \cref{thm:ddl,thm:en} we have that $Y$ is a fourfold with $e(Y)=16$ and $deg(Y)=36$. Moreover, since there are points in $Y$ where the map $\varphi$ degenerates to a rank-0 map, we have that $Y$ is singular along $D_0(\varphi_{\mZ(\alpha)})$, which is, again by \cref{thm:ddl,thm:en}, a set of 4 points. Over these 4 points, the fiber of the projection is a $\PP^2$, hence $X$ is a small resolution of $Y$.\\
The projection $\pi_1$ to $\PP^2$ is a $Q_2$ fibration.

\end{fano} \vspace{5mm}

\begin{fano}
\fanoid{$K_{431}$}
\label[mystyle]{F.3.13}
$X=\mZ(\mR_1(1,0,0)\oplus\of(0,1,1))\subset\PP^2\times\Fl(1,3,4)$
\subsubsection*{Invariants}  $h^0(-K)=36, \ (-K)^4=148,\ \chi(T_X)=-8,\ h^{1,1}=3,\ h^{3,1}=0,  \ h^{2,2}=12$.
\subsubsection*{Description} $X=\Bl_{\DP_4}Y$ with $Y$ the Fano fourfold obtained as a linear section in $\Fl(1,3,4)$.
\subsubsection*{Identification} Note that The projection $\pi$ to $\Fl(1,3,4)$ gives a birational map between $X$ and $Y=\mZ(\of(1,1))\subset\Fl(1,3,4)$, since the generic fiber is a point. Recall that by \crefpart{lm:blowFlagGr}{item_1} and by dualizing $\Gr(3,4)$ we get that $Y=\mZ(\of(1,1)^{\oplus 2})\subset\PP^3\times\PP^3$, which, by \cite[Lemma 1. 1]{FMMR} is a $\PP^1$-bundle over $\PP^3$, whose fiber jumps to a $\PP^2$ over four points. This fourfold can be also found as \red{$MW_7^4$} in \cite{CGKS}. Now, to study the exceptional locus $\Delta$ of the projection $\pi$, we argue as usual and apply \cref{thm:degloc}, \cref{thm:ddl,thm:en}. In this way, we get that $\Delta$ is the surface $S$ that can be described as $\mZ(\mQ_{\Gr(2,6)}(1,0)\oplus\of(0,1)\oplus\Sym^2\mU^{\vee}_{\Gr(2,6)})\subset\PP^2\times\Gr(2,6)$. In particular, $e(S)=8$, $K_S=\of(1,-1)_{|S}$, $(-K_S)^2=4$, $h^{1,0}=h^{2,0}=0$ and $h^0(S,K_S^{\otimes 2})=0$, hence by classification $S=\DP_4$. This gives us $X=\Bl_{\DP_4}Y$.\\
Since $Y$ is rational, $X$ is as such.\\
The original model of this Fano was $\mZ(\mU^{\vee}_{\Gr(2,6)}(1,0)\oplus\of(0,1)\oplus\Sym^2\mU^{\vee}_{\Gr(2,6)})\subset\PP^2\times\Gr(2,6)$, but we use the identification between $\Fl(1,3,4)$ and $\mZ(\Sym^2\mU^{\vee})\subset\Gr(2,6)$ and get the above description.

\end{fano} \vspace{5mm}

\begin{fano}
\fanoid{$K_{76}$}
\label[mystyle]{F.3.14}
$X=\mZ(\mQ_{\PP^4}\boxtimes\mU_{\Gr(2,7)}\oplus\mU_{\Gr(2,7)}^{\vee}\boxtimes\mU^{\vee}_{\Gr(2,4)}\oplus\mQ_{\Gr(2,4)}(1,0,0))\subset\PP^4\times\Gr(2,7)\times\Gr(2,4)$
\subsubsection*{Invariants} $h^0(-K)=72, \ (-K)^4=337,\ \chi(T_X)=8,\ h^{1,1}=3,\ h^{3,1}=0,  \ h^{2,2}=5$.
\subsubsection*{Description} $X$ is a small resolution of $Y$, a fourfold singular in a point with $e(Y)=12$.
\subsubsection*{Identification} Note that, by \crefpart{flagCor}{cor:flag_1} we can rewrite $X$ as $\mZ(\mU^{\vee}_{\Gr(2,7)}\boxtimes\mQ_{\mR}\oplus\mU^{\vee}_{\Gr(2,7)}\boxtimes\mU_{\Gr(2,4)}^{\vee})\subset\Gr(2,7)\times\PP_{\Gr(2,4)}(\mU_{\Gr(2,4)}\oplus\of_{\Gr(2,4)})$. In this way, by \cref{thm:degloc}, $X$ is birational to the first degeneracy locus of the map \[\varphi:p^*\mU_{\Gr(2,4)}\oplus\mQ^{\vee}_{\mR}\to\of_{\mR}^{\oplus 7}\] with $p$ the natural map associated to $\PP_{\Gr(2,4)}(\mU_{\Gr(2,4)}\oplus\of_{\Gr(2,4)})$, restricted to $Z=\mZ(\mQ_{\Gr(2,4)}\boxtimes\of_{\mR}(1))$. Hence $X$ is birational to $Y=D_3(\varphi_{|Z})$. Using \cref{thm:degloc}, \cref{thm:ddl,thm:en}, and we get that $Y$ is a fourfold with $e(12)$. Moreover, since $D_2(\varphi_{|Z})=\{pt\}$ we have that $Y$ is singular in a point. $X$ is a small resolution of $Y$ with a $\PP^2$. Note that the projection to $\PP^4$ gives a birational map to the latter, hence it is rational. However, even if $X$ is birational to $\PP^4$, we cannot recover it as simply a iterated blow-up of $\PP^4$. 
\end{fano} \vspace{5mm}

\begin{fano}
\fanoid{$K_{116}$}
\label[mystyle]{F.3.15}
$X=\mZ(\mQ_{\Gr(2,4)}(0,0,1)\oplus\mQ_{\Gr(2,4)}\boxtimes\mU^{\vee}_{\Gr(2,5)}\oplus\of(0,1,0)^{\oplus 2})\subset\Gr(2,4)\times\Gr(2,5)\times\PP^2$
\subsubsection*{Invariants}  $h^0(-K)=72, \ (-K)^4=337,\ \chi(T_X)=1,\ h^{1,1}=3,\ h^{3,1}=0,  \ h^{2,2}=5$.
\subsubsection*{Description} $X=\Bl_{\DP_7}Y$ with $Y$ the fourfold \cref{F.2.11}.
\subsubsection*{Identification} Recall that $Y=\mZ(\mQ_{\Gr(2,4)}\boxtimes\mU^{\vee}_{\Gr(2,5)}\oplus\of(0,1)^{\oplus 2})\subset\Gr(2,4)\times\Gr(2,5)$ is the fourfold \cref{F.2.11}, hence $X=\mZ(\mQ_{\Gr(2,4)|Y}(1,0))\subset\PP^2\times Y$. In particular $X$ is birational to $Y$ via The projection $\pi$ to the latter. Note that the invariants of the exceptional locus of $\pi$, $\Delta$, can be computed using \cref{thm:degloc}, \cref{thm:ddl,thm:en}. In this way we get that $\Delta$ is a surface $S$ that can be described as $\mZ(\mU^{\vee}_{\Gr(2,5)}(1,0)\oplus\of(0,1)^{\oplus 2})\subset\PP^2\times\Gr(2,5)$, with $e(S)=5$, $\mathcal{K}_{S}=\of(-1,-2)_{S}$, $(-K_S)^2=7$, $h^{1,0}=h^{2,0}=0$, hence we have that $S=\DP_7$.\newline
Since $X=\Bl_{\DP_7}Y$ and $Y$ is rational, so it is $X$.
\end{fano} \vspace{5mm}

\begin{fano}
\fanoid{$K_{125}$}
\label[mystyle]{F.3.16}
$X=\mZ(\mQ_{\Gr(2,4)}(1,0,0)\oplus\mQ_{\Gr(2,4)}(0,1,0)^{\oplus 2})\subset\PP^2\times\PP^4\times\Gr(2,4)$
\subsubsection*{Invariants} $h^0(-K)=48, \ (-K)^4=211,\ \chi(T_X)=0,\ h^{1,1}=3,\ h^{3,1}=0,  \ h^{2,2}=7$.
\subsubsection*{Description} $X=\Bl_{\DP_6}Y$, with $Y$ the fourfold \cref{F.2.21}
\subsubsection*{Identification} Let us consider The projection $\pi$ to $\PP^4\times\Gr(2,4)$, then we get that $X$ is birational to $Y=\mZ(\mQ_{\Gr(2,4)}(1,0)^{\oplus 2})\subset\PP^4\times\Gr(2,4)$, which is the fourfold \cref{F.2.21}. To understand where the fiber degenerates to a $\PP^1$ we apply \cref{thm:degloc} and we compute the invariants of the exceptional locus $\Delta$ using \cref{thm:ddl,thm:en}. In particular, $\Delta$ is a surface $S$ such that $e(S)=6$, $(-K_S)^2=6$, $h^{1,0}=h^{2,0}=0$ and $h^0(S,K_S^{\otimes 2})=0$, hence $S$ is a $\DP_6$.\newline
Since $X=\Bl_{\DP_6}Y$ with $Y$ rational, then $X$ is as such. 
\end{fano} \vspace{5mm}

\begin{fano}
\fanoid{$K_{48}$}
\label[mystyle]{F.3.17}
$X=\mZ(\mQ_{\PP^2}\boxtimes\mU^{\vee}_{\Gr(2,11)}\oplus(\mQ_{\PP^4}\boxtimes\mU^{\vee}_{\Gr(2,11)})^{\oplus 2})\subset\PP^2\times\PP^4\times\Gr(2,11)$
\subsubsection*{Invariants}  $h^0(-K)=90, \ (-K)^4=433,\ \chi(T_X)=14,\ h^{1,1}=3,\ h^{3,1}=0,  \ h^{2,2}=5$.
\subsubsection*{Description} $X$ is a small resolution of $Y$, a fourfold singular in a point with $e(Y)=12$.
\subsubsection*{Identification}  Let us first consider $\mZ((\mQ_{\PP^4}\boxtimes\mU^{\vee}_{\Gr(2,11)})^{\oplus 2})\subset\PP^4\times\Gr(2,11)$. Now we can rewrite everything as \[Z=\mZ(\mU^{\vee}_{\mR}\boxtimes(\mQ_{\PP^4}^{\oplus 2}))\subset
\bbGr_{\PP^4}(2,\of_{\PP^4}^{\oplus 5}\oplus\of_{\PP^4}^{\oplus 5}\oplus\of_{\PP^4})\]\\
using \cref{thm:seq} backwards we obtain that\\ \[Z=\bbGr_{\PP^4}(2,\of_{\PP^4}(-1)^{\oplus 2}\oplus \of_{\PP^4})\] hence \[X=\mZ(\mQ_{\PP^2}\boxtimes\mU^{\vee}_{\mR})\subset\PP^2\times\bbGr_{\PP^4}(2,\of_{\PP^4}(-1)^{\oplus 2}\oplus\of_{\PP^4}).\]\\ 
We can twist $\of(-1)^{\oplus 2}\oplus\of$ with $\of(1)$, and we obtain\\ \[X=\mZ(\mQ_{\PP^2}\boxtimes\mU^{\vee}_{\mR}(1,0))\subset\PP^2\times\bbGr_{\PP^4}(2,\of_{\PP^4}^{\oplus 2}\oplus\of_{\PP^4}(1)).\]\\ If we dualize the Grassmann bundle we get:\\ \[X=\mZ(\mQ_{\PP^2}\boxtimes\mQ_{\mR}(1,0))\subset\PP^2\times\PP_{\PP^4}(\of(-1)\oplus\of^{\oplus 2})\]\\ Iterating two times \cref{lm:flag2.5} we have the equality
\\ \[X=\mZ(\mQ_{\PP^2}\boxtimes\mQ_{\mR}(1,0)\oplus\mQ_{\Gr(3,7)}^{\oplus 2})\subset\PP^2\times\Gr(3,7)\]\\
and using Borel--Bott--Weil Theorem we finally get
\\ \[\mZ(\mQ_{\PP^4}\boxtimes\mR_1\oplus\mQ_2^{\oplus 2})\subset\PP^2\times\Fl(1,3,7)\]\\
Now, by \cref{thm:degloc}, we can easily describe $X$ as the restriction on $Y'=\mZ(\mQ_2^{\oplus 2})\subset\Fl(1,3,7)$ of the map \[\phi:\mR_1^{\vee}\to\of_{\Fl(1,3,7)}^{\oplus 3}\] which gives us $Y=D_{1}(\phi_{|Y'})$ a fourfold with $e(Y)=12$, and since $D_0(\phi_{|Y'})=\{1 pt\}$, we have that $Y$ is also singular. Now we have the full picture of $X$ which is a small resolution of $Y$.

\end{fano} \vspace{5mm}

\begin{fano}
\fanoid{$K_{188}$}
\label[mystyle]{F.3.18}
$X=\mZ(\of(1,0,1)\oplus\mQ_{\Gr(2,4)}(0,1,0)\oplus\of(0,0,1))\subset\PP^1\times\PP^3\times\Gr(2,4)$
\subsubsection*{Invariants} $h^0(-K)=60, \ (-K)^4=272,\ \chi(T_X)=4,\ h^{1,1}=3,\ h^{3,1}=0,  \ h^{2,2}=4$.
\subsubsection*{Description} $X=\Bl_{Q_2}Y$ with $Y$ the fourfold \cref{F.2.43}.
\subsubsection*{Identification} Note that $Y=\mZ(\mQ_{\Gr(2,4)}(1,0)\oplus\of(0,1))\subset\PP^3\times\Gr(2,4)$ is the fourfold \cref{F.2.43}, hence $X=\mZ(\of(1,0,1))\subset\PP^1\times Y$. We can apply \cref{lm:blowPPtX} and obtain $X=\Bl_{S}Y$, with $S=\mZ(\of(0,1)^{\oplus 2})\subset Y$, which is also $\PP_{Q_3}(\mU_{\Gr(2,4)|C})$, $S=\PP_{Q_3\cap H\cap H'}(\mU_{\Gr(2,4)})$ which is a $\DP_8$. Since $S$ is birational to $\PP^1\times\PP^1$ then it is $Q_2$.\\
Since $Y$ is rational then also $X$ is as such.
\end{fano} \vspace{5mm}

\begin{fano}
\fanoid{$K_{190}$}
\label[mystyle]{F.3.19}
$X=\mZ(\mQ_{\PP^2}\boxtimes\mU^{\vee}_{\Gr(2,6)}\oplus\mQ_{\Gr(2,6)}(1,0,0)\oplus\of(0,0,1)^{\oplus 2})\subset\PP^4\times\PP^2\times\Gr(2,6)$
\subsubsection*{Invariants}  $h^0(-K)=57, \ (-K)^4=256,\ \chi(T_X)=1,\ h^{1,1}=3,\ h^{3,1}=0,  \ h^{2,2}=6$.
\subsubsection*{Description} $X=\Bl_{\Bl_{pt}\PP^2}Y$ with $Y$ the fourfold \cref{F.2.51}.
\subsubsection*{Identification} Recall that $Y=\mZ(\mQ_{\PP^2}\boxtimes\mU^{\vee}_{\Gr(2,6)}\oplus\of(0,1)^{\oplus 2})\subset\PP^2\times\Gr(2,6)$ is the fourfold \cref{F.2.51}, hence $X=\mZ(\mQ_{\Gr(2,6)}(1,0,0))\subset\PP^4\times Y$. The projection $\pi$ to $Y$ gives a birational map between the latter and $X$, and the invariant of the exceptional locus can be computed applying \cref{thm:degloc}, \cref{thm:ddl,thm:en}. In this way, we obtain that $\pi$ has fiber a $\PP^1$ over a surface $S$ which can be described as $\mZ(\mQ_{\PP^2}\boxtimes\mU^{\vee}_{\Gr(2,6)}\oplus\of(0,1)^{\oplus 2})\subset\PP^2\times\Gr(2,5)$. Note that $e(S)=4$, $\mathcal{K}_{S}=\of(-1,-1)_{|S}$ and $(-K_S)^2=8$, hence $S$ is a $\DP_8$. Moreover $S$ is birational to $\PP^2$, hence it is a $\Bl_{pt}\PP^2$. 
\end{fano} \vspace{5mm}

\begin{fano}
\fanoid{$K_{191}$}
\label[mystyle]{F.3.20}
$X=\mZ(\mQ_{\PP^2_1}\boxtimes\mU_{\Gr(2,7)}^{\vee}\oplus\mQ_{\PP^2_2}\boxtimes\mU_{\Gr(2,7)}^{\vee}\oplus\of(0,0,1)^{\oplus 2})\subset\PP^2_1\times\PP^2_2\times\Gr(2,7)$
\subsubsection*{Invariants} $h^0(-K)=54, \ (-K)^4=241,\ \chi(T_X)=-2,\ h^{1,1}=3,\ h^{3,1}=0,  \ h^{2,2}=9$.
\subsubsection*{Description} $X$ is a small resolution of $Y$, a fourfold singular in 2 points with $e(Y)=15$ and $deg(Y)=19$.
\subsubsection*{Identification} We first consider $\mZ(\mQ_{\PP^2_1}\boxtimes\mU^{\vee}_{\Gr(2,7)}\oplus\of(0,1)^{\oplus 2})\subset\PP^2\times\Gr(2,7)$. Using \cref{thm:degloc}, this is a sixfold $Y'$ obtained as $D_1(\varphi_{|Z}')$, with \[\varphi':\mU_{\Gr(2,7)}\to\of_{\Gr(2,7)}^{\oplus 3}\] and $Z=\mZ(\of(1)^{\oplus 2})\subset\Gr(2,7)$. Using \cref{thm:ddl,thm:en} we note that $S=D_0(\varphi_{|Z}')$ is a surface, so $Y'$ is singular along $S$. Now we can also consider the other $\PP^2$, and apply the same argument on $\varphi:\mU_{\Gr(2,7)}\to\of_{\Gr(2,7)}^{\oplus 3}$ restricted to $Y'$. Again, using \cref{thm:ddl,thm:en} we obtain $Y=D_1(\phi_{|Y'})$, which is a singular fourfold with and $e(Y)=15$, $deg(Y)=19$. $Y$ is singular along $D_0(\phi_{|Y'})=\{pt\}$ and along $Y\cap S=\{pt\}$, and generically this 2 points are different. Over these points the fiber of the projection to $\Gr(2,7)$ is a $\PP^2$, hence $X$ is a small resolution of $Y$. 
\end{fano} \vspace{5mm}

\begin{fano}
\fanoid{$K_{173}$}
\label[mystyle]{F.3.21}
$X=\mZ(\mQ_{\PP^2}\boxtimes\mU^{\vee}_{\Gr(2,5)}\oplus\mU^{\vee}_{\Gr(2,6)}\boxtimes\mQ_{\Gr(2,5)}\oplus\of(0,1,0)^{\oplus 2})\subset\PP^2\times\Gr(2,6)\times\Gr(2,5)$
\subsubsection*{Invariants} $h^0(-K)=65, \ (-K)^4=298,\ \chi(T_X)=4,\ h^{1,1}=3,\ h^{3,1}=0,  \ h^{2,2}=5$.
\subsubsection*{Description} $X=\Bl_{\PP^2}Y$, with $Y$ the fourfold \cref{F.2.51}. 
\subsubsection*{Identification} Let us first consider $X'=\mZ(\mU^{\vee}_{\Gr(2,6)}\boxtimes\mQ_{\Gr(2,5)}\oplus\of(1,0)^{\oplus 2})\subset\Gr(2,6)\times\Gr(2,5)$, using 
\crefpart{lm:blowFlagGr}{item_2}
and the fact that the 2 hyperplane sections cut $\Gr(2,6)$ generically, we get that $X'=\Bl_{\PP^2}X_6^6$, with $\PP^2=\mZ(\mQ\oplus\of(1)^{\oplus 2})\subset\Gr(2,6)$. Hence $X$ is $\mZ(\mQ_{\PP^2}\boxtimes\mU^{\vee}_{\Gr(2,5)})\subset\PP^2\times X'$, the section $\mU^{\vee}_{\Gr(2,5)}$ in $X'$ cuts $X_6^6$ and contains $\PP^2$, hence $X=\Bl_{\PP^2}Y$, with $Y=\mZ(\mQ_{\PP^2}\boxtimes\mU^{\vee}_{\Gr(2,6)})\subset\PP^2\times X_6^6$, and this is the fourfold \cref{F.2.51}. 
\end{fano} \vspace{5mm}

\begin{fano}
\fanoid{$K_{178}$}
\label[mystyle]{F.3.22}
$X=\mZ(\mQ_{\Gr(2,4)}(0,1,0)\oplus\of(0,0,1)\oplus\of(1,1,0))\subset\PP^1\times\PP^3\times\Gr(2,4)$
\subsubsection*{Invariants}  $h^0(-K)=66, \ (-K)^4=304,\ \chi(T_X)=6,\ h^{1,1}=3,\ h^{3,1}=0,  \ h^{2,2}=4$.
\subsubsection*{Description} $X=\Bl_{Q_2}Y$, with $Y$ the fourfold \cref{F.2.43}
\subsubsection*{Identification} Recall that $Y=\mZ((\mQ_{\Gr(2,4)}(1,0)\oplus\of(0,1))\subset\PP^3\times\Gr(2,4)$ is the fourfold \cref{F.2.43}, hence $X=\mZ(\of(1,1,0))\subset\PP^1\times Y$, and applying \cref{lm:blowPPtX} we get that $X=\Bl_SY$, with $S=\mZ(\of(1,0)^{\oplus 2})\subset Y$ and this is $Q_1\times\PP^1$, so $Q_2$.\\
Since $Y$ is rational, also $X$ is as such.
\end{fano} \vspace{5mm}

\begin{fano}
\fanoid{$K_{211}$}
\label[mystyle]{F.3.23}
$X=\mZ(\mQ_{\Gr(2,5)}(1,0,0)\oplus\mQ_{\Gr(2,5)}(0,1,0)\oplus\of(0,0,1)^{\oplus 2})\subset\PP^3\times\PP^3\times\Gr(2,5)$
\subsubsection*{Invariants} $h^0(-K)=54, \ (-K)^4=241,\ \chi(T_X)=0,\ h^{1,1}=3,\ h^{3,1}=0,  \ h^{2,2}=7$.
\subsubsection*{Description} $X=\Bl_{\DP_7}Y$ with $Y$ the fourfold \cref{F.2.29}.
\subsubsection*{Identification} Note that $Y=\mZ(\mQ_{\Gr(2,5)}(1,0)\oplus\of(0,1)^{\oplus 2})\subset\PP^3\times\Gr(2,5)$ is the fourfold \cref{F.2.29}, hence $X=\mZ(\mQ_{\Gr(2,5)}(1,0,0))\subset\PP^3\times Y$. This way, $X$ is birational to $Y$ via the projection, $\pi$, to the latter. To understand what is the exceptional locus $\Delta$ of $\pi$ we use \cref{thm:degloc}, \cref{thm:ddl,thm:en} and we get that $\Delta$ is a surface $S$ that can be described as $\mZ(\mQ_{\Gr(2,4)}(1,0)\oplus\of(0,1)^{\oplus 2})\subset\PP^3\times\Gr(2,4)$. Note that $e(S)=5$, $K_S=\of(-2,-2)_{|S}$, $(-K_S)^2=7$, $h^{1,0}=h^{2,0}=0$ and $h^0(S,K_S^{\otimes 2})=0$ hence, by classification, $S$ is a $\DP_7$.\newline
Moreover, since $X=\Bl_{\DP_7}Y$, and $Y$ is rational, we have the rationality of $X$. 

\end{fano} \vspace{5mm}

\begin{fano}
\fanoid{$K_{212}$}
\label[mystyle]{F.3.24}
$X=\mZ(\mQ_{\PP^2}\boxtimes\mU^{\vee}_{\Gr(2,5)}\oplus\mQ_{\Gr(2,5)}(1,0,0)\oplus\mU^{\vee}_{\Gr(2,5)}(1,0,0))\subset\PP^5\times\PP^2\times\Gr(2,5)$
\subsubsection*{Invariants} $h^0(-K)=53, \ (-K)^4=235,\ \chi(T_X)=-1,\ h^{1,1}=3,\ h^{3,1}=0,  \ h^{2,2}=7$.
\subsubsection*{Description} $X=\Bl_{\DP_6}Y$, with $Y$ the Fano \red{6} in \cite[Table 3]{kalashnikov}
\subsubsection*{Identification} 
See \cref{ex:mixQuiv}.

\end{fano} \vspace{5mm}

\begin{fano}
\fanoid{$K_{202}$}
\label[mystyle]{F.3.25}
$X=\mZ(\mQ_{\Gr(2,4)}\boxtimes\mU^{\vee}_{\Gr(2,5)}\oplus\mQ_{\Gr(2,5)}(1,0,0)\oplus\of(0,1,0)\oplus\of(0,0,1))\subset\PP^3\times\Gr(2,4)\times\Gr(2,5)$
\subsubsection*{Invariants} $h^0(-K)=59, \ (-K)^4=267,\ \chi(T_X)=3,\ h^{1,1}=3,\ h^{3,1}=0,  \ h^{2,2}=5$.
\subsubsection*{Description} $X=\Bl_{Q_2}Y$ with $Y$ the fourfold \cref{F.2.28}.
\subsubsection*{Identification} Recall that $Y=\mZ(\mQ_{\Gr(2,4)}\boxtimes\mU^{\vee}_{\Gr(2,5)}\oplus\of(1,0)\oplus\of(0,1))\subset\Gr(2,4)\times\Gr(2,5)$ is the fourfold \cref{F.2.28}, hence $X=\mZ(\mQ_{\Gr(2,5)}(1,0))\subset\PP^3\times Y$, so it is birational to $Y$ via The projection $\pi$ to the latter. Using \cref{thm:degloc}, \cref{thm:ddl,thm:en} we can compute some invariants of the exceptional locus of $\pi$, which is a surface $S$ that can be described as $\mZ(\mU^{\vee}_{\Gr(2,5)})\subset Y$. In particular $e(S)=4$, $K_S=\of(-1,-1)_{|S}$, $(-K_S)^2=8$ and $h^0(S,K_S^{\otimes 2})=0$, so $S$ is a $\DP_8$. Moreover, we have that $S$ is birational to $\PP^1\times\PP^1$, hence it is $Q_2$.\newline
Since $X$ is birational to $Y$, it is rational. 
\end{fano} \vspace{5mm} 

\begin{fano}
\fanoid{$K_{232}$}
\label[mystyle]{F.3.26}
$X=\mZ(\mQ_{\PP^3}(0,1,0)\oplus\mQ_{\Gr(2,4)}(1,0,0)\oplus\mU^{\vee}_{\Gr(2,4)}(0,1,0)\oplus\of(0,0,1))\subset\PP^3\times\PP^5\times\Gr(2,4)$
\subsubsection*{Invariants} $h^0(-K)=52, \ (-K)^4=230,\ \chi(T_X)=0,\ h^{1,1}=3,\ h^{3,1}=0,  \ h^{2,2}=6$.
\subsubsection*{Description} $X=\Bl_{\DP_6}Y$ with $Y$ the fourfold \cref{F.2.43}
\subsubsection*{Identification} Recall that $Y=\mZ(\mQ_{\Gr(2,4)}(1,0)\oplus\of(0,1))\subset\PP^3\times\Gr(2,4)$ is the fourfold \cref{F.2.43}, hence, if we write $Y=\mZ(\of(0,1))\subset\Fl(1,2,4)$, then $X=\mZ(\mU_2(1,0,0)\oplus\mQ_1(1,0,0))\subset\PP^5\times Y$, and in particular The projection $\pi$ to $Y$ gives the birationality between the latter and $X$. To understand where the fiber of $\pi$ degenerates to a $\PP^1$ we use \cref{thm:degloc}, \cref{thm:ddl,thm:en}, in this way we obtain that the exceptional locus of $\pi$ is a surface $S$ that can be described as $\mZ(\mQ_2(1,0,0)\oplus\mU^{\vee}_1(1,0,0))\subset\PP^1\times Y$. Since $e(S)=6$, $K_S=\of(1,-1,-1)_{|S}$, $(-K_S)^2=6$ and $h^0(S,K_S^{\otimes 2})=0$. Hence $S$ is a $\DP_6$ and $X=\Bl_{\DP_6}Y$, which gives also its rationality.\\

\end{fano} \vspace{5mm}

\begin{fano}
\fanoid{$K_{231}$}
\label[mystyle]{F.3.27}
$X=\mZ(\mQ_{\PP^4}\boxtimes\mU^{\vee}_{\Gr(2,7)}\oplus\mU^{\vee}_{\Gr(2,7)}(0,1,0)\oplus\of(1,0,1))\subset \PP^1\times\PP^4\times\Gr(2,7)$
\subsubsection*{Invariants} $h^0(-K)=51, \ (-K)^4=225,\ \chi(T_X)=-3,\ h^{1,1}=3,\ h^{3,1}=0,  \ h^{2,2}=9$.
\subsubsection*{Description} $X=\Bl_{\DP_4}Y$, with $Y$ the fourfold \cref{F.2.62}.
\subsubsection*{Identification} Recall that $Y=\mZ(\mQ_{\PP^4}\boxtimes\mU^{\vee}_{\Gr(2,7)}\oplus\mU^{\vee}_{\Gr(2,7)}(1,0))\subset\PP^4\times\Gr(2,7)$ is the fourfold \cref{F.2.62}, hence $X=\mZ(\of(1,0,1))\subset\PP^1\times Y$. Now we can apply \cref{lm:blowPPtX} and obtain $X=\Bl_SY$, with $S=\mZ(\of(0,1)^{\oplus 2})\subset Y$. Since $e(S)=8$, $K_S=\of(1,-1,-1)_{|S}$, $(-K_S)^2=4$ so $S$ is a $\DP_4$.\newline
$X$ is birational to $Y$, so it is rational. 
\end{fano} \vspace{5mm}

\begin{fano}
\fanoid{$K_{235}$}
\label[mystyle]{F.3.28}
$X=\mZ(\mQ_{\PP^2}(0,1,0)\oplus\mQ_{\PP^2}(0,0,1)\oplus\mQ_{\Gr(2,4)}(0,1,0))\subset \PP^2\times\PP^4\times\Gr(2,4)$
\subsubsection*{Invariants} $h^0(-K)=46, \ (-K)^4=199,\ \chi(T_X)=-3,\ h^{1,1}=3,\ h^{3,1}=0,  \ h^{2,2}=8$.
\subsubsection*{Description} $X=\Bl_{\DP_6}\Bl_{Q_3\cap Q_3'\cap Q_3''}\Gr(2,4)$.
\subsubsection*{Identification}
We consider the projection $\pi_{13}$ to $\PP^2\times\Gr(2,4)$ we get that $\pi_{13}(X)$ is $Y=\mZ(\mQ_{\PP^2}(0,1))\subset\PP^2\times\Gr(2,4)$, which is, by \cref{lm:blowPPtX}, $\Bl_{Q_3\cap Q_3'\cap Q_3''}\Gr(2,4)$. In particular, since the generic fiber of $\pi_{13}$ over its image is one point, we have that $X$ is birational to $Y$. To understand the exceptional locus of $\pi_{13}$ we apply \cref{thm:degloc} and we study the first degeneracy locus of
\[
\varphi: \mQ_{\PP^2|Y}^{\vee}\oplus\mQ_{\Gr(2,4)}^{\vee}\to \of_Y^{\oplus 5}.
\]
Using \cref{thm:ddl,thm:en} we get that $D_3(\varphi)$ is a $\DP_6$. Hence $X=\Bl_{\DP_6}\Bl_{Q_3\cap Q_3'\cap Q_3''}\Gr(2,4)$, which gives also its rationality.
\end{fano} \vspace{5mm}

\begin{fano}
\fanoid{$K_{226}$}
\label[mystyle]{F.3.29}
$X=\mZ(\mQ_{\PP^2}\boxtimes\mU^{\vee}_{\Gr(2,5)}\oplus\mU_{\Gr(2,5)}(0,1,0)^{\oplus 2})\subset \PP^2\times\PP^4\times\Gr(2,5)$
\subsubsection*{Invariants} $h^0(-K)=51, \ (-K)^4=225,\ \chi(T_X)=-4,\ h^{1,1}=3,\ h^{3,1}=0,  \ h^{2,2}=10$.
\subsubsection*{Description} $X=\Bl_{\DP_3}Y$, with $Y$ the Fano fourfold \red{6} in \cite[Table 3]{kalashnikov}
\subsubsection*{Identification} As we have argued in \cref{F.3.24} $Y=\mZ(\mQ_{\PP^2}\boxtimes\mU^{\vee}_{\Gr(2,5)})\subset\PP^2\times\Gr(2,5)$ is the Fano fourfold \red{6} in \cite[Table 3]{kalashnikov}, so $X=\mZ(\mU^{\vee}_{\Gr(2,5)}(1,0,0)^{\oplus 2})\subset\PP^4\times Y$. The projection to $Y$ gives a birational map between the latter and $X$. Using \cref{thm:degloc}, \cref{thm:ddl,thm:en} we get that the fiber of $\pi$ degenerates to a $\PP^1$ on a surface $S$ that can be described as $\mZ(\mQ_{\Gr(2,5)}(1,0,0)^{\oplus 2})\subset\PP^4\times Y$. Note that $e(S)=9$, $K_S=\of(1,-1,-1)_{|S}$, $(-K_S)^2=3$ and $h^0(S,K_S^{\otimes 2})=0$, hence by classification $S$ is a $\DP_3$. This concludes this Fano, obtaining $X=\Bl_{\DP_3}Y$, which gives also the rationality of $X$.
\end{fano} \vspace{5mm}

\begin{fano}
\fanoid{$K_{224}$}
\label[mystyle]{F.3.30}
$X=\mZ(\mQ_{\Gr(2,4)}\boxtimes\mU^{\vee}_{\Gr(2,5)}\oplus\mU^{\vee}_{\Gr(2,4)}(1,0,0)\oplus\of(0,0,1)^{\oplus 2})\subset\PP^2\times\Gr(2,4)\times\Gr(2,5)$
\subsubsection*{Invariants} $h^0(-K)=56, \ (-K)^4=251,\ \chi(T_X)=1,\ h^{1,1}=3,\ h^{3,1}=0,  \ h^{2,2}=6$.
\subsubsection*{Description} $X=\Bl_{\Bl_{pt}\PP^2}Y$ with $Y$ the fourfold \cref{F.2.11}
\subsubsection*{Identification} Recall that $Y=\mZ((\mQ_{\Gr(2,4)}\boxtimes\mU^{\vee}_{\Gr(2,5)}\oplus\of(0,1)^{\oplus 2})\subset\Gr(2,4)\times\Gr(2,5)$ is the fourfold \cref{F.2.11}, and since $X=\mZ(\mU^{\vee}_{\Gr(2,4)}(1,0,0))\subset\PP^2\times Y$, is birational to $Y$ via The projection $\pi$ to the latter. We use \cref{thm:degloc}, \cref{thm:ddl,thm:en} to compute some invariants of the exceptional locus of $\pi$, in particular we get that it is a surface $S$ with $e(S)=4$, $K_S=\of(-1,-1)_{|S}$, $(-K_S)^2=8$, hence is a $\DP_8$. Since $S$ is birational to $\PP^2$, then it is $\Bl_{pt}\PP^2$.
\end{fano} \vspace{5mm}

\begin{fano}
\fanoid{$K_{242}$}
\label[mystyle]{F.3.31}
$X=\mZ(\mQ_{\PP^3}^{\vee}(1,1,0)\oplus\mQ_{\PP^4}(0,0,1)\oplus\of(1,0,1))\subset\PP^3\times\PP^4\times\PP^5$
\subsubsection*{Invariants}  $h^0(-K)=51, \ (-K)^4=224,\ \chi(T_X)=-3,\ h^{1,1}=3,\ h^{3,1}=0,  \ h^{2,2}=8$.
\subsubsection*{Description} $X=\Bl_{\DP_4}Y$, with $Y$ the projectivization of the null correlation bundles over $\PP^3$.
\subsubsection*{Identification} Let us note that $Y=\mZ(\mQ_{\PP^3}^{\vee}(1,1))\subset\PP^3\times\PP^4$ is a Fano fourfold with $\rho(Y)=2$, $\iota(Y)=2$ and $vol(Y)=384$, and, in particular, if we denote $\mE^{\vee}$ the null correlation bundle on $\PP^3$, we have that $Y=\PP_{\PP^3}(\mE)$, as described in \cite{SW}. This fourfold can be also found as $MW_{11}^{4}$ in \cite{CGKS}. Now we can rewrite $X$ as 
$\mZ(\of_{\mR}(1))\subset\PP_Y(\of_{\PP^4}(-1)\oplus\of_{\PP^4})$, which is birational to $Y$ via the natural projection $\pi$. To understand where the fiber degenerates we twist everything with $\of_{\PP^3}(-1)$, getting $X=\PP_Y(\of_{\PP^3\times\PP^4}(-1,-1)\oplus\of_{\PP^3\times\PP^4}(-1,0))$, and we apply \cref{thm:degloc}, \cref{thm:ddl,thm:en} to compute some invariants of the degeneracy locus. In this way, we have that the exceptional locus is a surface $S$ that can be described as $\mZ(\of(1,0)\oplus\of(1,1))\subset Y$ with $e(S)=8$, $K_S=\of(0,-1)_{|S}$, $(-K_S)^2=4$ and $h^0(S,K_S^{\otimes 2})=0$, hence by classification $S$ is a $\DP_4$.\\ At the end $X=\Bl_{\DP_4}Y$, and since $Y$ is rational, so it is $X$.

\end{fano} \vspace{5mm}

\begin{fano}
\fanoid{$K_{244}$}
\label[mystyle]{F.3.32}
$X=\mZ(\mQ_{\PP^3}(0,1,0)\oplus\mQ_{\Gr(2,4)}(1,0,0)\oplus\of(0,1,1)\oplus\of(0,0,1))\subset\PP^3\times\PP^4\times\Gr(2,4)$
\subsubsection*{Invariants} $h^0(-K)=47, \ (-K)^4=203,\ \chi(T_X)=-5,\ h^{1,1}=3,\ h^{3,1}=0,  \ h^{2,2}=9$.
\subsubsection*{Description} $X=\Bl_{\DP_3}Y$,with $Y$ the fourfold \cref{F.2.43}.
\subsubsection*{Identification} Let us recall that $Y=\mZ(\mQ_{\Gr(2,4)}(1,0)\oplus\of(0,1))\subset\PP^3\times\Gr(2,4)$ is the fourfold \cref{F.2.43}, hence $X=\mZ(\mQ_{\PP^3}(1,0,0)\oplus\of(1,0,1))\subset\PP^4\times Y$, which is birational to $Y$ via The projection $\pi$ to the latter. We argue as usual to compute the exceptional locus $\Delta$ of $\pi$, and using \cref{thm:degloc}, \cref{thm:ddl,thm:en} we obtain that $\Delta$ is a surface $S$ that can be described as $\mZ(\of_{\mR}(1)^{\oplus 5})\subset\PP_{Y} (\mQ^{\vee}_{\PP^3|Y}\oplus\of_{\Gr(2,4)}(-1)_{|Y})$, hence, by  $K_S=\of(-1,-1,2)_{|S}$, and by classification $S$ is a $\DP_3$.\\
Since $Y$ is rational, so is $X$ as such.

\end{fano} \vspace{5mm}

\begin{fano}
\fanoid{$K_{450}$}
\label[mystyle]{F.3.33}
$X=\mZ(\mQ_{\Gr(2,4)}(0,1,0)^{\oplus 2}\oplus\of(0,0,1))\subset\PP^1\times\PP^4\times\Gr(2,4)$
\subsubsection*{Invariants} $h^0(-K)=51, \ (-K)^4=224,\ \chi(T_X)=1,\ h^{1,1}=3,\ h^{3,1}=0,  \ h^{2,2}=4$.
\subsubsection*{Description} $X=\PP^1\times Y$, with $Y$ the Fano threefold \red{2-21} in \cite[Table 1]{DFT}
\subsubsection*{Identification} Since the bundles cut separately $\PP^1$ and $\PP^4\times\Gr(2,4)$, we get that $X=\PP^1\times Y$ with $Y=\mZ(\mQ_{\Gr(2,4)}(1,0)^{\oplus 4}\oplus\of(0,1))\subset\PP^4\times\Gr(2,5)$, which is the Fano threefold \red{2-21} in \cite[Table 1]{DFT}.\\ Moreover, since $X$ is product of rational manifolds, it is rational.
\end{fano} \vspace{5mm}

\begin{fano}
\fanoid{$K_{145}$}
\label[mystyle]{F.3.34}
$X=\mZ(\mQ_{\PP^4}\boxtimes\mU^{\vee}_{\Gr(2,7)}\oplus\mQ_{\Gr(2,4)}\boxtimes\mU^{\vee}_{\Gr(2,7)}\oplus\mQ_{\Gr(2,4)}(1,0,0))\subset\PP^4\times\Gr(2,7)\times\Gr(2,4)$
\subsubsection*{Invariants}  $h^0(-K)=75, \ (-K)^4=353,\ \chi(T_X)=8,\ h^{1,1}=3,\ h^{3,1}=0,  \ h^{2,2}=6$.
\subsubsection*{Description} $X$ is the a small resolution of $Y$ a fourfold with $e(Y)=13$ singular in a point. 
\subsubsection*{Identification} Recall that by \crefpart{flagCor}{cor:flag_2} $\mZ(\mQ_{\Gr(2,4)}(1,0))\subset\PP^4\times\Gr(2,4)$ is $\mZ(\mQ_2)\subset\Fl(1,3,5)$, in particular, $X=\mZ(\mU^{\vee}_{\Gr(2,7)}\boxtimes\mQ_1\oplus\mU^{\vee}_{\Gr(2,7)}\boxtimes\mQ_2\oplus\mQ_2)\subset\Gr(2,7)\times\Fl(1,3,5)$. Using \cref{thm:degloc} is clear that $X$ is birational to $Y=D_5(\varphi)$, with \[\varphi:\mQ_{1|Z}^{\vee}\oplus\mQ_{2|Z}^{\vee}\to\of_{Z}^{\oplus 7}\] and $Z=\mZ(\mQ_2)\subset\Fl(1,3,5)$. Moreover, applying \cref{thm:ddl,thm:en}, we get that $Y$ is a fourfold with $e(Y)=13$, and since $D_4(\varphi_{|Z})=\{pt\}$, we can conclude that $Y$ is singular in a point and $X$ is a small resolution of $Y$ along its singularity.
\end{fano} \vspace{5mm}

\begin{fano}
\fanoid{$K_{295}$}
\label[mystyle]{F.3.35}
$X=\mZ(\mQ_{\PP^2}\boxtimes\mU^{\vee}_{\Gr(2,5)}\oplus\mQ_{\Gr(2,5)}(1,0,0)^{\oplus 2})\subset\PP^6\times\PP^2\times\Gr(2,5)$
\subsubsection*{Invariants} $h^0(-K)=55, \ (-K)^4=244,\ \chi(T_X)=0,\ h^{1,1}=3,\ h^{3,1}=0,  \ h^{2,2}=5$.
\subsubsection*{Description} $X=\Bl_{\DP_8}Y$, with $Y$ the Fano \red{6} in \cite[Table 3]{kalashnikov}.
\subsubsection*{Identification} As argued in \cref{F.3.24}, $Y=\mZ(\mQ_{\PP^2}\boxtimes\mU^{\vee}_{\Gr(2,5)})\subset\PP^2\times\Gr(2,5)$ is the Fano \red{6} in \cite[Table 3]{kalashnikov}, hence $X=\mZ(\mQ_{\Gr(2,5)}(1,0,0)^{\oplus 2})\subset\PP^6\times Y$, so the projection $\pi$ to $Y$ gives the birationality between the latter and $X$. Moreover, using \cref{thm:degloc}, \cref{thm:ddl,thm:en} we obtain that the fiber of $\pi$ degenerates to a $\PP^1$ along a surface $S$ that can be described as $\mZ(\mU^{\vee}_{\Gr(2,5)}(1,0,0)^{\oplus 2})\subset\PP^2\times Y$, such that $e(S)$, $K_S=\of(1,-1,-1)_{|S}$, $(-K_S)^2=8$ and $h^0(S,K_S^{\otimes 2})=0$, hence $S$ is a  $\DP_8$. We get that the projection on the second $\PP^2$, is a birational map via a single blow-up. Thus $S$ is $\Bl_{pt}\PP^2$.\\
Since $Y$ is rational, also $X$ is as such.
\end{fano} \vspace{5mm}

\begin{fano}
\fanoid{$K_{320}$}
\label[mystyle]{F.3.36}
$X=\mZ(\mQ_{\PP^2}\boxtimes\mU^{\vee}_{\Gr(2,5)}\oplus\mU^{\vee}_{\Gr(2,5)}(0,1,0)\oplus\of(0,1,1))\subset\PP^2\times\PP^3\times\Gr(2,5)$
\subsubsection*{Invariants} $h^0(-K)=45, \ (-K)^4=193,\ \chi(T_X)=-8,\ h^{1,1}=3,\ h^{3,1}=0,  \ h^{2,2}=12$.
\subsubsection*{Description} $X=\Bl_{\DP_1}Y$, with $Y$ the Fano \red{6} in \cite[Table 3]{kalashnikov}.
\subsubsection*{Identification} As argued in \cref{F.3.24}, $Y=\mZ(\mQ_{\PP^2}\boxtimes\mU^{\vee}_{\Gr(2,5)})\subset\PP^2\times\Gr(2,5)$ is the Fano \red{6} in \cite[Table 3]{kalashnikov}, hence $X=\mZ(\mU^{\vee}_{\Gr(2,5)}(1,0,0)\oplus\of(1,0,1))\subset\PP^3\times Y$, so The projection $\pi$ to $Y$ gives the birationality between the latter and $X$. To understand where the fiber $\pi$ is a $\PP^1$ we use \cref{thm:degloc}, and we study $D_2(\varphi)$, with \[\varphi:\mU_{\Gr(2,5)|Y}\oplus\of_{Y}(-1)\to\of_{Y}^{\oplus 4}\] Using \cref{thm:ddl,thm:en} we get that $D_2(\varphi_{|Y})$ is a surface $S$, which can be described as $\mZ(\of_{\mR}(1)^{\oplus 4})\subset\PP_Y(\mU_{\Gr(2,5)|Y}\oplus\of_{\Gr(2,5)}(-1)_{|Y})$, with $e(S)=11$ and $K_S=(\of_{\mR}(1)\otimes\pi^*\of_Y(-1,-1))_{|S}$, hence by classification it is a $\DP_1$.
\end{fano} \vspace{5mm}

\begin{fano}
\fanoid{$K_{321}$}
\label[mystyle]{F.3.37}
$X=\mZ(\mQ_{\PP^2}\boxtimes\mU^{\vee}_{\Gr(2,5)}\oplus\mQ_{\Gr(2,5)}(0,1,0)\oplus\of(0,1,1))\subset\PP^2\times\PP^4\times\Gr(2,5)$
\subsubsection*{Invariants} $h^0(-K)=47, \ (-K)^4=202,\ \chi(T_X)=-7,\ h^{1,1}=3,\ h^{3,1}=0,  \ h^{2,2}=10$.
\subsubsection*{Description} $X=\Bl_{\DP_3}Y$, with $Y$ the Fano \red{6} in \cite[Table 3]{kalashnikov}. 
\subsubsection*{Identification} As argued in \cref{F.3.24}, $Y=\mZ(\mQ_{\PP^2}\boxtimes\mU^{\vee}_{\Gr(2,5)})\subset\PP^2\times\Gr(2,5)$ is the Fano \red{6} in \cite[Table 3]{kalashnikov}, hence $X=\mZ(\mQ_{\Gr(2,5)}(1,0,0)\oplus\of(1,0,1))\subset\PP^4\times Y$, so The projection $\pi$ to $Y$ gives the birationality between the latter and $X$. To understand where the fiber $\pi$ is a $\PP^1$ we use \cref{thm:degloc} and we study $D_3(\varphi)$, with \[\varphi:\mQ^{\vee}_{\Gr(2,5)|Y}\oplus\of_{\Gr(2,5)|Y}(-1)\to\of_{Y}^{\oplus 5}.\] Using \cref{thm:ddl,thm:en} we get that $D_3(\varphi_{|Y})$ is a surface $S$ that can be described as $\mZ(\of_{\mR}(1)^{\oplus 5})\subset\PP_Y(\mQ^{\vee}_{\Gr(2,5)|Y}\oplus\of_{\Gr(2,5)}(-1)_{|Y})$ with $e(S)=9$ and $K_S=\of(-1,-2,1)_{|S}$, hence a $\DP_3$.\\
Since $X$ is birational to $Y$ we have that is rational.
\end{fano} \vspace{5mm}

\begin{fano}
\fanoid{$K_{322}$}
\label[mystyle]{F.3.38}
$X=\mZ(\mQ_{\PP^2_1}(0,1,1))\subset\PP^2_1\times\PP^2_2\times\PP^2_3$
\subsubsection*{Invariants} $h^0(-K)=46, \ (-K)^4=198,\ \chi(T_X)=-2,\ h^{1,1}=3,\ h^{2,1}=1,\ h^{3,1}=0,  \ h^{2,2}=5$.
\subsubsection*{Description} $X=\Bl_E\PP^2_2\times\PP^2_3$ with $E$ a sextic elliptic curve.
\subsubsection*{Identification} We apply \cref{lm:blowPPtX} and obtain $X=\Bl_{C}\PP^2\times\PP^2$, with $C=\mZ(\of(1,1)^{\oplus 3})$, hence $C$ is a curve, with $\mathcal{K}_C=\of_C$, hence by classification is an elliptic curve $E$ of degree six.\newline
This Fano is rational since it is birational to $\PP^2\times\PP^2$.
\end{fano} \vspace{5mm}

\begin{fano}
\fanoid{$K_{341}$}
\label[mystyle]{F.3.39}
$X=\mZ(\mQ_{\Gr(2,4)}(0,1,0)\oplus\of(0,1,1)\oplus\mQ_{\Gr(2,4)}(1,0,0))\subset\PP^2\times\PP^3\times\Gr(2,4)$
\subsubsection*{Invariants} $h^0(-K)=42, \ (-K)^4=178,\ \chi(T_X)=-7,\ h^{1,1}=3,\ h^{3,1}=0,  \ h^{2,2}=11$.
\subsubsection*{Description} $X=\Bl_{\DP_6}Y$, with $Y$ the fourfold \cref{F.2.36}
\subsubsection*{Identification} Recall that $Y=\mZ(\mQ_{\Gr(2,4)}(0,1,0)\oplus\of(0,1,1))\subset\PP^3\times\Gr(2,4)$ is the fourfold \cref{F.2.36}, hence $X=\mZ(\mQ_{\Gr(2,4)}(1,0,0))\subset\PP^2\times Y$, which gives the birationality between $X$ and $Y$ via The projection $\pi$ to the latter. Using \cref{thm:degloc}, \cref{thm:ddl,thm:en} we have that the fiber of $\pi$ degenerates to a $\PP^1$ on a surface $S$ that can be described as $\mZ(\mU^{\vee}_{\Gr(2,4)})\subset Y$, which is $\mZ(\of(1,1)^{\oplus 2})\subset\PP^2\times\PP^2$ and so it is a $\DP_6$.\newline
Note that since $X=\Bl_{\DP_6}Y$ and $Y$ is rational, so it is $X$.
\end{fano} \vspace{5mm}

\begin{fano}
\fanoid{$K_{342}$}
\label[mystyle]{F.3.40}
$X=\mZ(\mQ_{\Gr(2,4)}(1,0,0)\oplus\mQ_{\Gr(2,4)}(0,1,0)\oplus\of(1,1,0)\oplus\of(0,0,1))\subset\PP^3_1\times\PP^3_2\times\Gr(2,4)$
\subsubsection*{Invariants} $h^0(-K)=44, \ (-K)^4=188,\ \chi(T_X)=-4,\ h^{1,1}=3,\ h^{3,1}=0,  \ h^{2,2}=8$.
\subsubsection*{Description} $X=\Bl_{\DP_4}Y$, with $Y$ the fourfold \cref{F.2.43}
\subsubsection*{Identification} Recall that $Y=\mZ(\mQ_{\Gr(2,4)}(0,1,0)\oplus\of(0,0,1))\subset\PP^3_2\times\Gr(2,4)$ is the fourfold \cref{F.2.43}, hence $X=\mZ(\mQ_{\Gr(2,4)}(1,0,0)\oplus\of(1,1,0))\subset\PP^3_1\times Y$, which gives the birationality between $X$ and $Y$ via The projection $\pi$ to the latter. Using \cref{thm:degloc}, \cref{thm:ddl,thm:en} we have that the fiber of $\pi$ degenerates to a $\PP^1$ on a surface $S$ that can be described as $\mZ(\mU^{\vee}_{\Gr(2,4)}(1,0,0)\oplus\mQ_{\PP^3_2}(1,0,0))\subset\PP^3\times Y$. Now, $e(S)=8$, $K_S=\of(1,1,-1)_{|S}$, $(-K_S)^2=4$ and $h^0(S,K_S^{\otimes 2})=0$, hence $S$ is a $\DP_4$.
\end{fano} \vspace{5mm}

\begin{fano}
\fanoid{$K_{343}$}
\label[mystyle]{F.3.41}
$X=\mZ(\mQ_{\PP^2}(0,1,0)\oplus\mQ_{\Gr(2,4)}(0,1,0)\oplus\of(1,0,1)\oplus\of(0,0,1))\subset\PP^2\times\PP^4\times\Gr(2,4)$
\subsubsection*{Invariants} $h^0(-K)=43, \ (-K)^4=183,\ \chi(T_X)=-5,\ h^{1,1}=3,\ h^{3,1}=0,  \ h^{2,2}=9$.
\subsubsection*{Description} $X=\Bl_{Q_2}Y$, with $Y$ the fourfold \cref{F.2.35}. 
\subsubsection*{Identification} Let us first start with the fivefold $X'=\mZ(\mQ_{\PP^2}(0,1,0)\oplus\mQ_{\Gr(2,4)}(0,1,0)\oplus\of(0,0,1))\subset\PP^2\times\PP^4\times\Gr(2,4)$, then, by \crefpart{flagCor}{cor:flag_2}, $X'=\mZ(\mQ_{\PP^2}(0,1,0)\oplus\of(0,0,1)\oplus\mQ_2)\subset\PP^2\times\Fl(1,3,5)$. Using \cref{lm:blowPPtX} we have that $X'=\Bl_SY'$, with $Y'=\mZ(\mQ_2\oplus\of(0,1))\subset\Fl(1,3,5)$, which, by \cref{lm:flag2.5}, can be written as $\PP_{Q_3}(\mU_{\Gr(2,4)}\oplus\of_{\Gr(2,4)})$. Note that $S=\mZ(\of_{\mR}(1)^{\oplus 3})\subset\PP_{Q_3}(\mU_{\Gr(2,4)}\oplus\of_{\Gr(2,4)})$, and, by \cref{thm:seq} can be rewritten also as $\mZ(\mQ_{\Gr(2,4)}(1,0)\oplus\of(0,1)\oplus\of(1,0)^{\oplus 3})\subset\PP^4\times\Gr(2,4)$. Thus $S=\mZ(\mQ_{\Gr(2,4)|Q_3}(1,0))\subset\PP^1\times Q_3$, and if we project on $\PP^1$ we get that the generic fiber is $\PP^1$. Moreover, $e(S)=4$, $K_S=\of(0,-2)_{|S}$, $(-K_S)^2=8$ and $h^0(S,K_S^{\otimes 2})=0$, hence it is $Q_2$. We rewrite $X$ as $\mZ(\of_{Y'}(1)\boxtimes\of_{\mR}(1))\Bl_{Q_2} Y'$. The bundle $\of_{Y'}(1)\boxtimes\of_{\mR}(1)$ fixes $S$, hence $X=\Bl_SY$, with $Y=\mZ(of_{Y'}(1)\boxtimes\of_{\mR}(1))\subset Y'$, which is the fourfold \cref{F.2.35}.\\
In particular, $Y$ is rational, so $X$ is as such.
\end{fano} \vspace{5mm}

\begin{fano}
\fanoid{$K_{346}$}
\label[mystyle]{F.3.42}
$X=\mZ(\mQ_{\PP^2_1}\boxtimes\mQ_{\PP^2_2}\boxtimes\of_{\PP^4}(1))\subset\PP^2_1\times\PP^2_2\times\PP^4$
\subsubsection*{Invariants} $h^0(-K)=36, \ (-K)^4=150,\ \chi(T_X)=-4,\ h^{1,1}=3,\ h^{3,1}=0,  \ h^{2,2}=10$.
\subsubsection*{Description}  $X=\Bl_{\DP_3}\PP^2\times\PP^2$.
\subsubsection*{Identification} Let us call $Y=\PP_1^2\times\PP^2_2$. If we project everything on $Y$, we have, by \cref{thm:degloc}, that such projection is a birational map which degenerates whenever the map \[\varphi:\mQ^{\vee}_{\PP^2_1}\boxtimes\mQ^{\vee}_{\PP^2_2}\to\of_Y^{\oplus 5}\] degenerates to a rank-3 map. Using \cref{thm:ddl,thm:en} we get that $D_3(\varphi)$ is a surface $S$ that can be described as $\mZ(\mQ^{\vee}_{\PP^2_2}(1,0,1)\oplus\of(0,1,1)^{\oplus 3})\subset\PP^2\times\PP^2\times\PP^3$. In particular $e(S)=6$ and $K_S=\of(-1,-1,1)_{S}$, $(-K_S)^{2}=3$, $h^0(S,K_S^{\otimes 2})=0$ and $h^0(S,(-K_S)^{\otimes 2})=10$. Thus by classification, $S$ is a $\DP_3$.\newline
Since $X$ is birational to $\PP^4$ via the projection to the latter, then it is rational by definition.

\end{fano} \vspace{5mm}

\begin{fano}
\fanoid{$K_{265}$}
\label[mystyle]{F.3.43}
$X=\mZ(\mQ_{\PP^2}\boxtimes\mU^{\vee}_{\Gr(2,6)}\oplus\mQ_{\Gr(2,4)}\boxtimes\mU^{\vee}_{\Gr(2,6)}\oplus\of(0,1,0)\oplus\of(0,0,1))\subset\PP^2\times\Gr(2,4)\times\Gr(2,6)$
\subsubsection*{Invariants} $h^0(-K)=60, \ (-K)^4=273,\ \chi(T_X)=3,\ h^{1,1}=3,\ h^{3,1}=0,  \ h^{2,2}=6$.
\subsubsection*{Description} $\pi:X\to\Gr(2,4)\times\Gr(2,6)$ is a small resolution of $Y$ a fourfold with $e(Y)=13$ singular in a point.
\subsubsection*{Identification} Let us consider $Z=\mZ(\mQ_{\Gr(2,4)}\boxtimes\mU^{\vee}_{\Gr(2,6)}\oplus\of(1,0)\oplus\of(0,1))\subset\Gr(2,6)\times\Gr(2,4)$. By \crefpart{flagCor}{cor:flag_2}, $Z$ is also  $\mZ(\mQ_2^{\oplus 2}\oplus\of(1,0)\oplus\of(0,1))\subset\Fl(2,4,6)$. We rewrite $X$ as $\mZ(\mQ_{\PP^2}\boxtimes\mU^{\vee}_{1})\subset\PP^2\times Z$. Using \cref{thm:degloc}, we get that $X$ is birational to $D_1(\varphi)$, with $\varphi:\mU_{1|Z}\to\of_{\Fl(2,4,6)|Z}^{\oplus 3}$. Using \cref{thm:ddl,thm:en}, $D_1(\varphi)$ corresponds to a fourfold $Y$ with $e(Y)=13$, and since $D_2(\phi)=\{pt\}$ we have that $Y$ is singular in a  point. In particular, since the projection $\pi$ to $\Gr(2,4)\times\Gr(2,6)$ degenerates to a $\PP^2$ on the singular point of $Y$, $X$ is a small resolution of $Y$.\\
\end{fano} \vspace{5mm}

\begin{fano}
\fanoid{$K_{372}$}
\label[mystyle]{F.3.44}
$X=\mZ(\mQ_{\Gr(2,4)}\boxtimes\mU^{\vee}_{\Gr(2,5)}\oplus\of(0,1,0)\oplus\of(0,0,1)^{\oplus 2})\subset\PP^1\times\Gr(2,4)\times\Gr(2,5)$
\subsubsection*{Invariants} $h^0(-K)=60, \ (-K)^4=272,\ \chi(T_X)=4,\ h^{1,1}=3,\ h^{3,1}=0,  \ h^{2,2}=4$. 
\subsubsection*{Description} $X$ is $\PP^1\times Y$, with $Y$ the Fano threefold \red{2-26} in \cite[Table 1]{DFT}.
\subsubsection*{Identification} The bundles cut the factors $\PP^1$ and $\Gr(2,4)\times\Gr(2,5)$ separately, hence $X=\PP^1\times Y$, with $Y=\mZ(\mQ_{\Gr(2,4)}\boxtimes\mU^{\vee}_{\Gr(2,5)}\oplus\of(1,0)\oplus\of(0,1)^{\oplus 2})\subset\Gr(2,4)\times\Gr(2,5)$, which is the Fano threefold \red{2-26} in \cite[Table 1]{DFT}.\\
Since it is a product of rational manifolds, $X$ is rational.
\end{fano} \vspace{5mm}

\begin{fano}
\fanoid{$K_{388}$}
\label[mystyle]{F.3.45}
$X=\mZ(\mQ_{\Gr(2,5)}(0,1,0)\oplus\of(0,0,1)^{\oplus 3})\subset\PP^1\times\PP^3\times\Gr(2,5)$
\subsubsection*{Invariants} $h^0(-K)=54, \ (-K)^4=240,\ \chi(T_X)=2,\ h^{1,1}=3,\ h^{3,1}=0,  \ h^{2,2}=4$.
\subsubsection*{Description} $X$ is $\PP^1\times Y$, with with $Y$ the Fano threefold \red{2-22} in \cite[Table 1]{DFT}.
\subsubsection*{Identification} The bundles cut the factors $\PP^1$ and $\Gr(2,4)\times\Gr(2,5)$ separately, hence $X=\PP^1\times Y$, with $Y=\mZ(\mQ_{\Gr(2,5)}(1,0)\oplus\of(0,1)^{\oplus 3})\subset\PP^3\times\Gr(2,5)$, which is the Fano threefold \red{2-22} in \cite[Table 1]{DFT}.\\
Since it is a product of rational manifolds, $X$ is rational.
\end{fano} \vspace{5mm}

\begin{fano}
\fanoid{$K_{421}$}
\label[mystyle]{F.3.46}
$X=\mZ(\mQ_{\Gr(2,5)}(0,1,0)\oplus\of(1,1,0)\oplus\of(0,0,1)^{\oplus 3})\subset\PP^1\times\PP^4\times\Gr(2,5)$
\subsubsection*{Invariants} $h^0(-K)=43, \ (-K)^4=183,\ \chi(T_X)=-3,\ h^{1,1}=3,\ h^{3,1}=0,  \ h^{2,2}=7$.
\subsubsection*{Description} $X=\Bl_{\DP_5}Y$, with $Y$ the fourfold \cref{F.2.27}
\subsubsection*{Identification} Recall that $Y=\mZ(\mQ_{\Gr(2,5)}(1,0)\oplus\of(0,1)^{\oplus 3})\subset\PP^4\times\Gr(2,5)$ is the fourfold \cref{F.2.27}, hence $X=\mZ(\of(1,1,0))\subset\PP^1\times Y$. We apply \cref{lm:blowPPtX} and obtain that $X=\Bl_{S}Y$, with $S=\mZ(\of(0,1)^{\oplus 3}\oplus\of(1,0)^{\oplus 2})\subset\Fl(1,2,5)$. Note that $e(S)=7$, $K_S=\of(0,-1)_{|S}$, $(-K_S)^2=5$ and $h^0(S,K_S^{\otimes 2})=0$, hence $S$ is a $\DP_5$.\\
Since $Y$ is rational, also $X$ is as such.
\end{fano} \vspace{5mm}

\begin{fano}
\fanoid{$K_{422}$}
\label[mystyle]{F.3.47}
$X=\mZ(\mQ_{\Gr(2,4)}(0,1,0)\oplus\of(0,1,1)\oplus\mU^{\vee}_{\Gr(2,4)}(1,0,0))\subset\PP^2\times\PP^3\times\Gr(2,4)$
\subsubsection*{Invariants} $h^0(-K)=41, \ (-K)^4=172,\ \chi(T_X)=-7,\ h^{1,1}=3,\ h^{3,1}=0,  \ h^{2,2}=10$.
\subsubsection*{Description} $X=\Bl_{\DP_7}Y$, with $Y$ the fourfold \cref{F.2.36}.
\subsubsection*{Identification} Recall that $Y=\mZ((\mQ_{\Gr(2,4)}(1,0)\oplus\of(1,1))\subset\PP^3\times\Gr(2,4)$ is the fourfold \cref{F.2.36}, hence $X=\mZ(\mU^{\vee}_{\Gr(2,4)}(1,0,0))\subset\PP^2\times Y$, so it is birational to $Y$ via the projection to the latter. Note that, using \cref{thm:ddl,thm:en} we can say that the exceptional locus of $\pi$ is a surface $S$ that can be described as $\mZ(\mQ_{\PP^2}(0,1)\oplus\of(1,1))\subset\PP^2\times\PP^3$. In particular $e(S)=5$, $K_S=\of(-1,-1)_{|S}$, $(-K_S)^2=7$ and $h^0(S,K_S^{\otimes 2})=0$, hence $S$ is a $\DP_7$.\\
Since $Y$ is rational, so $X$ is as such.
\end{fano} \vspace{5mm}

\begin{fano}
\fanoid{$K_{17}$}
\label[mystyle]{F.3.48}
$X=\mZ(\mQ_{\PP^2_1}\boxtimes\mU^{\vee}_{\Gr(2,5)}\oplus\mU^{\vee}_{\Gr(2,5)}(0,1,0))\subset\PP_1^2\times\PP_2^2\times\Gr(2,5)$
\subsubsection*{Invariants}  $h^0(-K)=85, \ (-K)^4=406,\ \chi(T_X)=12,\ h^{1,1}=3,\ h^{3,1}=0,  \ h^{2,2}=5$.
\subsubsection*{Description} $X=\Bl_{Q_2}Y$, with $Y$ the Fano fourfold \red{6} in \cite[Table 3]{kalashnikov}.
\subsubsection*{Identification} Recall that $Y=\mZ(\mQ_{\PP^2_1}\boxtimes\mU^{\vee}_{\Gr(2,5)})\subset\PP^2_1\times\Gr(2,5)$ is the Fano fourfold \red{6} in \cite[Table 3]{kalashnikov}, as proved in \cref{F.2.24}. Hence $X=\mZ(\mU^{\vee}_{\Gr(2,5)}(1,0,0))\subset\PP^2\times Y$, and it is birational to $Y$ via The projection $\pi$ to the latter. Using \cref{thm:ddl,thm:en} we have that the fiber of $\pi$ degenerates to a $\PP^1$ along a surface $S$ that can be described as $\mZ(\mQ_{\Gr(2,5)}(1,0,0))\subset\PP^1\times Y$. Moreover, $e(S)=4$, $K_S=\of(1,-1,-2)_{|S}$, $(-K_S)^2=8$ and $h^0(S,K_S^{\otimes 2})=0$, then, by classification, $S$ is a $\DP_8$. Moreover, if we project to $\PP^1$ we get as generic fiber $Y$ cut with a section of $\mQ_{\Gr(2,5)}$, which is a $\PP^1$. Thus $S$ is birational to $\PP^1\times\PP^1$, which means that is $Q_2$.\\
In the end $X=\Bl_{Q_2}Y$, and due to the rationality of $Y$, we have that also $X$ is as such.
\end{fano} \vspace{5mm}
\

\begin{fano}
\fanoid{$K_{472}$}
\label[mystyle]{F.3.49}
$X=\mZ(\mQ_{\Gr(2,4)}(0,1,0)^{\oplus 2}\oplus\of(1,1,0)\oplus\of(0,0,1))\subset\PP^1\times\PP^5\times\Gr(2,4)$
\subsubsection*{Invariants} $h^0(-K)=39, \ (-K)^4=162,\ \chi(T_X)=-7,\ h^{1,1}=3,\ h^{3,1}=0,  \ h^{2,2}=10$.
\subsubsection*{Description} $X=\Bl_{\DP_2}\PP^1\times Q_3$
\subsubsection*{Identification} Note that applying \cref{thm:seq} on $Y=\mZ(\mQ_{\Gr(2,4)}(1,0)^{\oplus 2}\oplus\of(0,1))\subset\PP^5\times\Gr(2,4)$ we get that it is $\mZ(\of_{\mR}(1)^{\oplus 2})\subset\PP_{Q_3}(\mU_{\Gr(2,4)|Q_3}^{\oplus 2})$, hence $X=\mZ(\of(1,1,0))\subset\PP^1\times Y$. Using \cref{lm:blowPPtX} we get that $X=\Bl_SY$, with $S=\mZ(\of(1,0)^{\oplus 2})\subset Y$. Since $e(S)=8$, $K_S=\of(0,-1)_{|S}$, $(-K_S)^2=4$ and $h^0(S,K_S^{\otimes 2})=0$, then, by classification, $S$ is a $\DP_4$. Note that $Y$ is a $\PP^1$-bundle over $Q_3$ and by \cref{thm:ddl,thm:en}, the fiber of the projection to the base degenerates to a $\PP^2$ over 2 points.\\
If we consider the projection $\pi_{13}$ on $\PP^1\times Q_3$, we have a birational map between the latter and $X$. Using \cref{thm:degloc} we can understand the exceptional locus of $\pi_{13}$ studying the first degeneracy locus of the morphism:
\[
\varphi:\of_{\PP^2}(-1)\oplus\mQ^{vee}_{\Gr_{2,4}|Q_3}\to \of_{\PP^2\times Q_3}^{\oplus 4}.
\]
Using \cref{thm:ddl,thm:en} we get that $D_2(\varphi)$ is a $\DP_2$. Hence $X=\Bl_{\DP_2}\PP^1\times Q_3$.\\
Since $X$ is birational to a product of two rational manifolds, it is rational.

\end{fano} \vspace{5mm}

\begin{fano}
\fanoid{$K_{512}$}
\label[mystyle]{F.3.50}
$X=\mZ(\mQ_{\Gr(2,4)}(0,1,0)\oplus\of(1,0,1)\oplus\of(0,2,0)\oplus\of(0,0,1))\subset\PP^1\times\PP^4\times\Gr(2,4)$
\subsubsection*{Invariants} $h^0(-K)=37, \ (-K)^4=152,\ \chi(T_X)=-5,\ h^{1,1}=3,\ h^{3,1}=0,  \ h^{2,2}=8$.
\subsubsection*{Description} $X$ is the blow-up of the base of the conic bundle \cref{F.2.47} along a $Q_1$.
\subsubsection*{Identification}  Let us first consider $Y=\mZ(\mQ_{\Gr(2,4)}(1,0)\oplus\of(2,0)\oplus\of(0,1))\subset\PP^4\times\Gr(2,4)$, which is the Fano \cref{F.2.47}. Thus $X=\mZ(\of(1,0,1))\subset\PP^1\times Y$. Using \cref{lm:blowPPtX} we get $X=\Bl_{C}Y$ with $C=Q_3\cap H\cap H'=Q_1$. In this way, since the fiber of the conic bundle is not involved in the blow-up, we have $X=\PP_{\Bl_{Q_1}Q_3}(\mU_{\Gr(2,4)|Q_3}\oplus\of_{Q_3})$.
\end{fano} \vspace{5mm}

\begin{fano}
\fanoid{$K_{516}$}
\label[mystyle]{F.3.51}
$X=\mZ(\mU^{\vee}_{\Gr(2,5)}(0,1,0)\oplus\of(1,0,1)\oplus\of(0,0,1)^{\oplus 2})\subset\PP^1\times\PP^2\times\Gr(2,5)$
\subsubsection*{Invariants} $h^0(-K)=35, \ (-K)^4=141,\ \chi(T_X)=-10,\ h^{1,1}=3,\ h^{3,1}=0, \ h^{2,2}=12$.
\subsubsection*{Description} $X=\Bl_{\DP_2}Y$, with $Y$ the fourfold \cref{F.2.46}.
\subsubsection*{Identification} Recall that $Y=\mZ(\mU^{\vee}_{\Gr(2,5)}(1,0)\oplus\of(0,1)^{\oplus 2})\subset\PP^2\times\Gr(2,5)$ is the fourfold \cref{F.2.46}, hence $X=\mZ(\of(1,0,1))\subset\PP^1\times Y$. Now using \cref{lm:blowPPtX} we get that $X=\Bl_SY$, with $S=\mZ(\of(0,1)^{\oplus 2})\subset Y$. Since $e(S)=10$, $K_S=\of(-1,0)_{|S}$, $(-K_S)^2=2$ and $h^0(S,K_S^{\otimes 2})=0$, by classification, $S$ is a $\DP_2$.
\newline
Note that $X$ is rational because it is birational to $Y$. 
\end{fano} \vspace{5mm}

\begin{fano}
\fanoid{$K_{519}$}
\label[mystyle]{F.3.52}
$X=\mZ(\mQ_{\Gr(2,4)}(0,1,0)\oplus\of(0,0,1)\oplus\of(0,1,1)\oplus\of(1,1,0))\subset\PP^1\times\PP^4\times\Gr(2,4)$
\subsubsection*{Invariants} $h^0(-K)=35, \ (-K)^4=141,\ \chi(T_X)=-11,\ h^{1,1}=3,\ h^{3,1}=0,  \ h^{2,2}=13$.
\subsubsection*{Description} $X=\Bl_{\DP_4}Y$, with $Y$ the fourfold \cref{F.2.35}.
\subsubsection*{Identification} Recall that $Y=\mZ(\mQ_{\Gr(2,4)}(1,0)\oplus\of(0,1)\oplus\of(1,1))\subset\PP^4\times\Gr(2,4)$ is the fourfold \cref{F.2.35}, hence $X=\mZ(\of(1,0,1))\subset\PP^1\times Y$. Now using \cref{lm:blowPPtX} we get that $X=\Bl_SY$, with $S=\mZ(\of(1,0)^{\oplus 2})\subset Y$. Since $e(S)=8$, $K_S=\of(0,-1)_{|S}$, $(-K_S)^2=4$ and $h^0(S,K_S^{\otimes 2})=0$, by classification, $S$ is a $\DP_4$. \\
Since $Y$ is rational, so $X$ is as such.
\end{fano} \vspace{5mm}

\begin{fano}
\fanoid{$K_{593}$}
\label[mystyle]{F.3.53}
$X=\mZ(\of(1,1,0)\oplus\of(1,0,1)\oplus\of(0,0,1)^{\oplus 3})\subset\PP^{2}\times\PP^1\times\Gr(2,5)$
\subsubsection*{Invariants} $h^0(-K)=33, \ (-K)^4=130,\ \chi(T_X)=-11,\ h^{1,1}=3,\ h^{3,1}=0,  \ h^{2,2}=12$.
\subsubsection*{Description}  $X=\Bl_{\DP_5}Y$ with $Y$ the Fano fourfold \cref{F.2.1}.
\subsubsection*{Identification}  Recall that $Y=\mZ(\of(1,1)\oplus\of(0,1)^{\oplus 3})\subset\PP^2\times\Gr(2,5)$ is the fourfold \cref{F.2.1}, hence $X=\mZ(\of(1,1,0))\subset\PP^1\times Y$. Now using \cref{lm:blowPPtX} we get that $X=\Bl_SY$, with $S=\mZ(\of(1,0)^{\oplus 2})\subset Y$. Since $e(S)=7$, $K_S=\of(0,-1)_{|S}$, $(-K_S)^2=5$ and $h^0(S,K_S^{\otimes 2})=0$, so $S$ is a $\DP_5$.\newline
Note that $X$ is rational because it is birational to $Y$.
\end{fano} \vspace{5mm}

\begin{fano}
\fanoid{$K_{591}$}
\label[mystyle]{F.3.54}
$X=\mZ(\of(0,1,1)\oplus\of(0,0,1)^{\oplus 3})\subset\PP^1_1\times\PP^1_2\times\Gr(2,5)$
\subsubsection*{Invariants} $h^0(-K)=39, \ (-K)^4=160,\ \chi(T_X)=-4,\ h^{1,1}=3, \ h^{2,1}=1, \ h^{3,1}=0,  \ h^{2,2}=4$
\subsubsection*{Description} $X=\PP^1\times Y$, with $Y$ the Fano threefold \red{2-14} in \cite[Table 1]{DFT} 
\subsubsection*{Identification} Since the bundles cut the factors $\PP^1_1$ and $\PP^1_2\times\Gr(2,5)$ separately, we have that $X=\PP^1\times Y$, with $Y=\mZ(\of(1,1)\oplus\of(0,1)^{\oplus 3})\subset\PP^1\times\Gr(2,5)$, which is the Fano threefold \red{2-14} in \cite[Table 1]{DFT}.\\
Note that $X$ is the product of two rational manifolds, hence it is rational.
\end{fano} \vspace{5mm}

\begin{fano}
\fanoid{$K_{668}$}
\label[mystyle]{F.3.55}
$X=\mZ(\of(1,1,0)\oplus\mQ_{\PP^2}(0,0,2))\subset\PP^1\times\PP^2\times\PP^4$
\subsubsection*{Invariants} $h^0(-K)=25, \ (-K)^4=89,\ \chi(T_X)=-14,\ h^{1,1}=3, \ h^{2,1}=5, \ h^{3,1}=0,  \ h^{2,2}=9$.
\subsubsection*{Description} $X=\Bl_{\DP_4}Y$, with $Y$ the fourfold \cref{F.2.22}
\subsubsection*{Identification}  Recall that $Y=\mZ(\mQ_{\PP^2}(0,2))\subset\PP^2\times\PP^4$ is the fourfold \cref{F.2.22}, hence $X=\mZ(\of(1,1,0))\subset\PP^1\times Y$. Now using \cref{lm:blowPPtX} we get that $X=\Bl_SY$, with $S=\mZ(\of(1,0)^{\oplus 2})\subset Y$. Since $e(S)=8$, $K_S=\of(0,-1)_{|S}$, $(-K_S)^2=4$ and $h^0(S,K_S^{\otimes 2})=0$, by classification, $S$ is a $\DP_4$.\newline
Note that $X$ is rational because it is birational to $Y$.
\end{fano} \vspace{5mm}

\begin{fano}
\fanoid{$K_{182}$}
\label[mystyle]{F.3.56}
$X=\mZ(\mQ_{\PP^4}\boxtimes\mU^{\vee}_{\Gr(2,7)}\oplus\mU_{\Gr(2,7)}^{\vee}(1,0,0)\oplus\mU_{\Gr(2,7)}^{\vee}(0,1,0))\subset\PP^2\times\PP^4\times\Gr(2,7)$
\subsubsection*{Invariants} $h^0(-K)=60, \ (-K)^4=273,\ \chi(T_X)=2,\ h^{1,1}=3, \ h^{3,1}=0,  \ h^{2,2}=7$.
\subsubsection*{Description} $X=\Bl_{\DP_6}Y$, with Y the fourfold \cref{F.2.62}
\subsubsection*{Identification} Recall that $Y=\mZ(\mQ_{\PP^4}\boxtimes\mU^{\vee}_{\Gr(2,7)}\oplus\mU_{\Gr(2,7)}^{\vee}(1,0))\subset\PP^4\times\Gr(2,7)$ is the fourfold \cref{F.2.62}, hence $X=\mZ(\mU_{\Gr(2,7)}^{\vee}(1,0,0))\subset\PP^2\times Y$, hence $X$ is birational to $Y$ via The projection $\pi$ to the latter. We use \cref{thm:ddl,thm:en} to get that the fiber of $\pi$ degenerates to a $\PP^1$ on a surface $S$ that can be described as $\mZ(\mQ_{\Gr(2,7)}(1,0,0))\subset\PP^3\times Y$. In particular since $e(S)=6$, $K_S=\of(1,-1,-1)_{|S}$, $(-K_S)^2=6$ and $h^0(S,-K_S)=7$, $h^0(S,K_S^{\otimes 2})=0$, hence $S$ is a $\DP_6$.\newline
Note that $X$ is rational because it is birational to $Y$.
\end{fano} \vspace{5mm}

\begin{fano}
\fanoid{$K_{282}$}
\label[mystyle]{F.3.57}
$X=\mZ(\mQ_{\PP^3}\boxtimes\mU^{\vee}_{\Gr(2,6)}\oplus\mU^{\vee}_{\Gr(2,6)}(1,0,0)\oplus\of(0,1,1))\subset\PP^2\times\PP^3\times\Gr(2,6)$
\subsubsection*{Invariants} $h^0(-K)=51, \ (-K)^4=225,\ \chi(T_X)=-5,\ h^{1,1}=3, \ h^{3,1}=0,  \ h^{2,2}=11$
\subsubsection*{Description} $X=\Bl_{\DP_5}Y$, with $Y$ the Fano fourfold \red{70} in \cite[Table 3]{kalashnikov}.
\subsubsection*{Identification} Note that $Y=\mZ(\mQ_{\PP^3}\boxtimes\mU^{\vee}_{\Gr(2,6)}\oplus\of(1,1))\subset\PP^3\times\Gr(2,6)$ equals $\mZ(\of(1,1))\subset\bbGr_{\PP^3}(2,\of(-1)\oplus\of^{\oplus 2})$ by \cref{thm:seq}. By duality and twisting with $\of_{\PP^3}(-1)$ we get $Y=\mZ(\of(1,1))\subset\PP_{\PP^3}(\of\oplus\of(-1)^{\oplus 2})$. This can be written as zero loci of homogeneous bundles on a quiver flag variety as the variety with adjacency matrix:
\[
\begin{bmatrix}
0 & 1 & 4 \\
0 & 0 & 0 \\
0 & 2 & 0 \\
\end{bmatrix}
\]
dimension vector $[1,1,1]$ and cut by the vector $(\square,\square)$, which is the Fano fourfold \red{70} in \cite[Table 3]{kalashnikov}. Now $X$ can be written as $\mZ(\mU^{\vee}_{\Gr(2,6)}(1,0,0))\subset\PP^2\times Y$, this gives the birationality between $X$ and $Y$ via the projection to the latter. Using \cref{thm:ddl,thm:en} we obtain that the fibers of $\pi$ degenerate to a $\PP^1$ along a surface $S$ that can be described as $\mZ(\mQ_{\Gr(2,6)}(1,0,0))\subset\PP^2\times Y$. Moreover $e(S)=7$, $\mathcal{K}_{S}=\of(1,-1,-1)_{|S}$, $(-K_S)^2=5$ and $h^0(S,2\mathcal{K}_{S})=0$, hence $S$ is a $\DP_5$.\\
Note that $Y$ is birational to a product of rational manifolds, hence it is rational, and also $X$ is as such.
\end{fano} \vspace{5mm}

\begin{fano}
\fanoid{$K_{303}$}
\label[mystyle]{F.3.58}
$X=\mZ(\mQ_{\PP^2_2}\boxtimes\mU^{\vee}_{\Gr(2,6)}\oplus\of(0,0,1)^{\oplus 2}\oplus\mU^{\vee}_{\Gr(2,6)}(1,0,0))\subset\PP^2_1\times\PP^2_2\times\Gr(2,6)$
\subsubsection*{Invariants} $h^0(-K)=48, \ (-K)^4=209,\ \chi(T_X)=-3,\ h^{1,1}=3, \ h^{3,1}=0,  \ h^{2,2}=8$.
\subsubsection*{Description} $X=\Bl_{\DP_6}Y$, with $Y$ the fourfold \cref{F.2.51}.
\subsubsection*{Identification} Let first notice that $Y=\mZ(\mQ_{\PP^2}\boxtimes\mU^{\vee}_{\Gr(2,6)}\oplus\of(0,1)^{\oplus 2})\subset\PP^2\times\Gr(2,6)$ is the Fano fourfold \cref{F.2.51}, hence $X=\mZ(\mU^{\vee}_{\Gr(2,6)})\subset\PP^2\times Y$. The projection to $Y$ induces a birational map between $X$ and $Y$ and by \cref{thm:ddl,thm:en} we obtain that the fiber degenerates to a $\PP^1$ over a surface $S$ that can be described as $\mZ(\mQ_{\Gr(2,6)}(1,0,0))\subset\PP^2\times Y$. Moreover $e(S)=6$, $K_S=\of(1,-1,-1)_{|S}$, $(-K_S)^2=6$ and $h^0(S,K_S^{\otimes 2})=0$, hence $S$ is a $\DP_6$.
\end{fano} \vspace{5mm}

\begin{fano}
\fanoid{$K_{460}$}
\label[mystyle]{F.3.59}
$X=\mZ(\mU^{\vee}_{\Gr(2,5)}(0,1,0)\oplus\of(2,0,0)\oplus\of(0,0,1)^{\oplus 3})\subset\PP^2_1\times\PP^2_2\times\Gr(2,5)$
\subsubsection*{Invariants} $h^0(-K)=48, \ (-K)^4=208,\ \chi(T_X)=0,\ h^{1,1}=3, \ h^{3,1}=0,  \ h^{2,2}=4$
\subsubsection*{Description} $X$ is the product $Q_1\times Y$, with $Y$ the Fano threefold \red{2-20} in \cite[Table 1]{DFT}.
\subsubsection*{Identification} Note that the bundles cut $\PP^2_1$ and $\PP^2_2\times\Gr(2,5)$ separately. Thus $X=Q_1\times Y$, with $Y=\mZ(\mU^{\vee}_{\Gr(2,5)}(1,0)\oplus\of(0,1)^{\oplus 3})\subset\PP^2_2\times\Gr(2,5)$, which is the Fano threefold \red{2-20} in \cite[Table 1]{DFT}.
\end{fano} \vspace{5mm}

\begin{fano}
\fanoid{$K_{318}$}
\label[mystyle]{F.3.60}
$X=\mZ(\mU^{\vee}_{\Gr(2,5)}(1,0,0)\oplus\mQ_{\Gr(2,5)}(0,1,0)\oplus\of(0,0,1)^{\oplus 2})\subset\PP^2\times\PP^3\times\Gr(2,5)$
\subsubsection*{Invariants} $h^0(-K)=48, \ (-K)^4=209,\ \chi(T_X)=-2,\ h^{1,1}=3, \ h^{3,1}=0,  \ h^{2,2}=7$.
\subsubsection*{Description} $X=\Bl_{\DP_7}Y$, with $Y$ the fourfold \cref{F.2.29}.
\subsubsection*{Identification} Recall that $Y=\mZ(\mQ_{\Gr(2,5)}(1,0)\oplus\of(0,1)^{\oplus 2})\subset\PP^3\times\Gr(2,5)$ is the fourfold \cref{F.2.29}, and $X=\mZ(\mU^{\vee}_{\Gr(2,5)}(1,0,0))\subset\PP^2\times Y$. In this way $X$ is birational to $Y$ via The projection $\pi$ to the latter, and, using \cref{thm:ddl,thm:en}, we get that the fibers of $\pi$ degenerate to a $\PP^1$ along a surface $S$ that can be described as $\mZ(\mQ_{\Gr(2,5)}(1,0,0))\subset\PP^1\times Y$. Now, $e(S)=5$, $K_S=\of(1,-1,-1)_{|S}$, $(-K_S)^2=7$ and $h^0(S,K_S^{\otimes 2})=0$, hence, by classification, $S$ is a $\DP_7$.\newline
In the end, since $X=\Bl_{\DP_7}Y$, it is rational because so it is $Y$.
\end{fano} \vspace{5mm}

\begin{fano}
\fanoid{$K_{326}$}
\label[mystyle]{F.3.61}
$X=\mZ(\mQ_{\Gr(2,4)}\boxtimes\mU^{\vee}_{\Gr(2,5)}\oplus\of(1,1,0)\oplus\of(0,0,1)^{\oplus 2})\subset\PP^1\times\Gr(2,4)\times\Gr(2,5)$
\subsubsection*{Invariants} $h^0(-K)=48, \ (-K)^4=209,\ \chi(T_X)=-4,\ h^{1,1}=3, \ h^{3,1}=0,  \ h^{2,2}=9$
\subsubsection*{Description} $X=\Bl_{\DP_5}Y$, with $Y$ the fourfold \cref{F.2.11}
\subsubsection*{Identification} Recall that $Y=\mZ(\mQ_{\Gr(2,4)}\boxtimes\mU^{\vee}_{\Gr(2,5)}\oplus\of(0,0,1)^{\oplus 2})\subset\Gr(2,4)\times\Gr(2,5)$ is the fourfold \cref{F.2.11}, and $X=\mZ(\of(1,1,0))\subset\PP^1\times Y$. Using \cref{lm:blowPPtX} we have that $X=\Bl_SY$, with $S=\mZ(\of(1,0)^{\oplus 2})\subset Y$. Note that $e(S)=7$, $K_S=\of(0,-1)_{|S}$, $(-K_S)^2=5$, $h^0(S,K_S^{\otimes 2})=0$, hence $S$ is a $\DP_5$.\newline
Note that $X$ is rational because it is birational to $Y$.
\end{fano} \vspace{5mm}

\begin{fano}
\fanoid{$K_{327}$}
\label[mystyle]{F.3.62}
$X=\mZ(\mQ_{\Gr(2,4)}(0,1,0)\oplus\mU^{\vee}_{\Gr(2,4)}(0,1,0)\oplus\mU^{\vee}_{\Gr(2,4)}(1,0,0))\subset\PP^2\times\PP^4\times\Gr(2,4)$
\subsubsection*{Invariants} $h^0(-K)=47, \ (-K)^4=204,\ \chi(T_X)=-2,\ h^{1,1}=3, \ h^{3,1}=0,  \ h^{2,2}=7$.
\subsubsection*{Description} $X=\Bl_{\DP_7}Y$, with $Y$ the fourfold \cref{F.2.34}
\subsubsection*{Identification} Recall that $Y=\mZ(\mQ_{\Gr(2,4)}(1,0)\oplus\mU^{\vee}_{\Gr(2,4)}(1,0))\subset\PP^4\times\Gr(2,4)$ is the fourfold \cref{F.2.34}, and $X=\mZ(\mU^{\vee}_{\Gr(2,4)}(1,0,0))\subset\PP^2\times Y$. In this way $X$ is birational to $Y$ via The projection $\pi$ to the latter, and, using \cref{thm:ddl,thm:en}, we get that the fibers of $\pi$ degenerate to a $\PP^1$ along a surface $S$ that can be described as $\mZ(\mQ_{\Gr(2,4)})\subset\times Y$. Since $e(S)=5$, $K_S=\of(-1,-1)_{|S}$, $(-K_S)^{2}$, hence $S$ is a $\DP_7$ by classification.\\
Note that $X$ is rational because $Y$ is as such.
\end{fano} \vspace{5mm}

\begin{fano}
\fanoid{$K_{339}$}
\label[mystyle]{F.3.63}
$X=\mZ(\mQ_{\Gr(2,4)}(1,0,0)^{\oplus 2}\oplus\mU_{\Gr(2,4)}^{\vee}(0,1,0))\subset\PP^4\times\PP^2\times\Gr(2,4)$
\subsubsection*{Invariants} $h^0(-K)=46, \ (-K)^4=199,\ \chi(T_X)=0,\ h^{1,1}=3, \ h^{3,1}=0,  \ h^{2,2}=5$.
\subsubsection*{Description} $X=\Bl_{\Bl_{pt}\PP^2}Y$, with $Y$ the fourfold \cref{F.2.21}.
\subsubsection*{Identification} Note that $Y=\mZ(\mQ_{\Gr(2,4)}(1,0)^{\oplus 2})\subset\PP^4\times\Gr(2,4)$ is \cref{F.2.34}, then $X=\mZ(\mU^{\vee}_{\Gr(2,4)}(1,0,0))\subset\PP^2\times Y$. In this way $X$ is birational to $Y$ via The projection $\pi$ to the latter. Using \cref{thm:degloc}, \cref{thm:ddl,thm:en}, we get that the fibers of $\pi$ degenerate to a $\PP^1$ along a surface $S$ that can be described as $\mZ(\mQ_{\Gr(2,4)})\subset Y$. Now $\mZ(\mQ_{\Gr(2,4)})$ is equal to $\mZ(\mQ_{\PP^2}(1,0)^{\oplus 2})\subset\PP^4\times\PP^2$, applying \cref{thm:seq} becomes $\mZ(\of_{\mR}(1))\subset\PP_{\PP^2}(\of(-1)^{\oplus 2})$. Twisting $\of(-1)^{\oplus 2}$ with $\of(1)$ we have that $S$ can be described as $\mZ(\of(1,1))\subset\PP^1\times\PP^2$, and by \cref{lm:blowPPtX} we have that $S$ is $\Bl_{pt}\PP^2$.
\newline
Since $X$ is birational to $Y$, it is rational.
\end{fano} \vspace{5mm}

\begin{fano}
\fanoid{$K_{344}$}
\label[mystyle]{F.3.64}
$X=\mZ(\of(1,0,1)\oplus\mQ_{\PP^2}(0,1,0)\oplus\mU^{\vee}_{\Gr(2,4)}(1,0,0))\subset\PP^3\times\Gr(2,4)\times\PP^2$
\subsubsection*{Invariants} $h^0(-K)=42, \ (-K)^4=,178\ \chi(T_X)=-5,\ h^{1,1}=3, \ h^{3,1}=0,  \ h^{2,2}=9$.
\subsubsection*{Description} $X=\Bl_{\DP_5}\Bl_{Q_1}\Gr(2,4)$
\subsubsection*{Identification} Let first consider $Y=\mZ(\mQ_{\PP^2}(0,1))\subset\PP^2\times\Gr(2,4)$, then by \cref{lm:blowPPtX} this is $\Bl_{Q_1}\Gr(2,4)$. In this way $X=\mZ(\mU^{\vee}_{\Gr(2,4)}(1,0)\oplus\of(1,1))\subset\PP^3\times Y$, so it is birational to $Y$ via The projection $\pi$ to the latter. Using \cref{thm:ddl,thm:en} we obtain that the fibers of $\pi$ degenerate to a $\PP^1$ over a surface $S$ that can be described as $\mZ(\mQ_{\PP^2}(1,0,0)\oplus\mQ_{\Gr(2,4)}(1,0,0))\subset\PP^2_1\times Y$. Since $e(S)=7$, $K_S=\of(1,-1,-1)_{|S}$, $(-K_S)^2=5$ and $h^0(S,K_S^{\otimes 2})=0$ then $S$ is a $\DP_5$.\newline
Note that $X$ is rational because it is birational to $\Gr(2,4)$.
\end{fano} \vspace{5mm}

\begin{fano}
\fanoid{$K_{413}$}
\label[mystyle]{F.3.65}
$X=\mZ(\mU^{\vee}_{\Gr(2,4)}(1,0,0)\oplus\mQ_{\Gr(2,4)}(0,1,0)\oplus\of(0,0,1)\oplus\of(0,2,0))\subset\PP^2\times\PP^4\times\Gr(2,4)$
\subsubsection*{Invariants} $h^0(-K)=44, \ (-K)^4=188,\ \chi(T_X)=-2,\ h^{1,1}=3, \ h^{3,1}=0,  \ h^{2,2}=6$.
\subsubsection*{Description} $X=\Bl_{\DP_6}Y$, with $Y$ the fourfold \cref{F.2.47}.
\subsubsection*{Identification} Recall that $Y=\mZ(\mQ_{\Gr(2,4)}(1,0)\oplus\of(0,1)\oplus\of(2,0))\subset\PP^4\times\Gr(2,4)$ is the fourfold \cref{F.2.34}, and $X=\mZ(\mU^{\vee}_{\Gr(2,4)}(1,0,0))\subset\PP^2\times Y$. In this way $X$ is birational to $Y$ via the projection $\pi$ to the latter. By \cref{thm:degloc}, \cref{thm:ddl,thm:en}, we get that the fibers of $\pi$ jump to a $\PP^1$ along a surface $S$ that can be described as $\mZ(\mQ_{\Gr(2,4)})\subset Y$. Since $e(S)=6$, $K_S=\of(-1,-1)_{|S}$, $(-K_S)^2=4$ and $h^0(S,K_S^{\otimes 2})=0$ then $S$ is a $\DP_4$.
\end{fano} \vspace{5mm}

\begin{fano}
\fanoid{$K_{423}$}
\label[mystyle]{F.3.66}
$X=\mZ(\mQ_{\Gr(2,5)}(0,1,0)\oplus\of(1,0,1)\oplus\of(0,0,1)^{\oplus 2})\subset\PP^1\times\PP^3\times\Gr(2,5)$
\subsubsection*{Invariants} $h^0(-K)=40, \ (-K)^4=167,\ \chi(T_X)=-8,\ h^{1,1}=3, \ h^{3,1}=0,  \ h^{2,2}=11$.
\subsubsection*{Description} $X=\Bl_{\DP_3}Y$, with $Y$ the fourfold \cref{F.2.29}.
\subsubsection*{Identification} Recall that $Y=\mZ(\mQ_{\Gr(2,5)}(1,0)\oplus\of(0,1)^{\oplus 2})\subset\PP^3\times\Gr(2,5)$ is the fourfold \cref{F.2.29}, and $X=\mZ(\of(1,0,1))\subset\PP^1\times Y$. Now we can apply \cref{lm:blowPPtX} and obtain that $X=\Bl_{S}Y$, with $S=\mZ(\of(0,1)^{\oplus 2})\subset Y$. Note that $e(S)=9$, $K_S=\of(-1,0)_{|S}$, $(-K_S)^2=3$ and $h^0(S,K_S^{\otimes 2})=0$, hence $S$ is a $\DP_3$.\newline
Since $X=\Bl_{\DP_3}Y$, and $Y$ is rational, so it is $X$.
\end{fano} \vspace{5mm}

\begin{fano}
\fanoid{$K_{424}$}
\label[mystyle]{F.3.67}
$X=\mZ(\of(1,0,1)\oplus\mQ_{\PP^2}\boxtimes\mU^{\vee}_{\Gr(2,6)}\oplus\of(0,0,0)^{\oplus 2})\subset\PP^1\times\PP^2\times\Gr(2,6)$
\subsubsection*{Invariants} $h^0(-K)=39, \ (-K)^4=161,\ \chi(T_X)=-10,\ h^{1,1}=3, \ h^{3,1}=0,  \ h^{2,2}=12$.
\subsubsection*{Description} $\pi:X\to\PP^2\times\Gr(2,6)$ is the map associated to $\Bl_{\DP_2}Y$, with $Y$ the fourfold \cref{F.2.61}
\subsubsection*{Identification} Recall that $Y=\mZ(\mQ_{\PP^2}\boxtimes\mU^{\vee}_{\Gr(2,6)}\oplus\of(0,1)^{\oplus 2})\subset\PP^2\times\Gr(2,6)$ is the fourfold \cref{F.2.61}, and $X=\mZ(\of(1,0,1))\subset\PP^1\times Y$. Now we can apply \cref{lm:blowPPtX} and obtain that $X=\Bl_{S}Y$, with $S=\mZ(\of(0,1)^{\oplus 2})\subset Y$. Note that $e(S)=10$, $K_S=\of(-1,0)_{|S}$, $(-K_S)^2=2$ and $h^0(S,K_S^{\otimes 2})=0$, hence $S$ is a $\DP_2$.\newline
Since $X=\Bl_{\DP_2}Y$, and $Y$ is rational, so it is $X$.
\end{fano} \vspace{5mm}

\begin{fano}
\fanoid{$K_{428}$}
\label[mystyle]{F.3.68}
$X=\mZ(\of(1,0,1)\oplus\mQ_{\Gr(2,4)}(0,1,0)^{\oplus 2})\subset\PP^1\times\PP^4\times\Gr(2,4)$
\subsubsection*{Invariants} $h^0(-K)=39, \ (-K)^4=163,\ \chi(T_X)=-5,\ h^{1,1}=3, \ h^{3,1}=0,  \ h^{2,2}=9$.
\subsubsection*{Description} $\pi:X\to\Gr(2,4)$ is the map associated to $\Bl_{\DP_4}Y$ with $Y$ the Fano fourfold \cref{F.2.21}
\subsubsection*{Identification} Recall that $Y=\mZ(\mQ_{\Gr(2,4)}(1,0)^{\oplus 2})\subset\PP^2\times\Gr(2,6)$ is the Fano fourfold \cref{F.2.21}, and $X=\mZ(\of(1,0,1))\subset\PP^1\times Y$. Now we can apply \cref{lm:blowPPtX} and obtain that $X=\Bl_{S}Y$, with $S=\mZ(\of(0,1)^{\oplus 2})\subset Y$. Note that $e(S)=8$, $K_S=\of(-1,0)_{|S}$, $(-K_S)^2=4$ and $h^0(S,K_S^{\otimes 2})=0$, hence, by classification $S$ is a $\DP_2$.\newline
Since $X=\Bl_{\DP_4}Y$, and $Y$ is rational, so it is $X$.
\end{fano} \vspace{5mm}

\begin{fano}
\fanoid{$K_{430}$}
\label[mystyle]{F.3.69}
$X=\mZ(\mQ_{\Gr(2,4)}(0,1,0)\oplus\of(1,0,1)\oplus\of(1,1,0)\oplus\of(0,0,1))\subset\PP^2\times\PP^3\Gr(2,4)$
\subsubsection*{Invariants} $h^0(-K)=39, \ (-K)^4=162,\ \chi(T_X)=-7,\ h^{1,1}=3, \ h^{3,1}=0,  \ h^{2,2}=10$.
\subsubsection*{Description} $\pi:X\to\PP^3\times\Gr(2,4)$ is the map associated to $\Bl_{Q_2}Y$, with $Y$ the fourfold \cref{F.2.43}
\subsubsection*{Identification} Recall that $Y=\mZ(\mQ_{\Gr(2,4)}(1,0)\oplus\of(0,1))\subset\PP^3\times\Gr(2,4)$ is the fourfold \cref{F.2.43}, and $X=\mZ(\of(1,0,1))\subset\PP^1\times Y$. Now we can apply \cref{lm:blowPPtX} and obtain that $X=\Bl_{S}Y$, with $S=\mZ(\of(0,1)^{\oplus 2})\subset Y$. Note that if we project $S$ on $\Gr(2,4)$ we have that this map gives a birationality between $S$ and $Q_2$, and since $e(S)=4$, $\mathcal{K}_{S}=\of(-2,0)_{|S}$, $(-K_S)^2=8$ and $h^0(S,2K_{S})=0$, then $S$ is a $Q_2$.\\
Note that $Y$ is rational so $X$ is as such.
\end{fano} \vspace{5mm}

\begin{fano}
\fanoid{$K_{438}$}
\label[mystyle]{F.3.70}
$X=\mZ(\mQ_{\PP^2_1}(0,0,1)\oplus\mQ_{\PP^2_1}(0,1,1))\subset\PP^2_1\times\PP^2_2\times\PP^4$
\subsubsection*{Invariants} $h^0(-K)=33, \ (-K)^4=131,\ \chi(T_X)=-15,\ h^{1,1}=3, \ h^{3,1}=0,  \ h^{2,2}=17$.
\subsubsection*{Description} $X=\Bl_{\Bl_{13pts}\PP^2}\PP^2\times\PP^2$.
\subsubsection*{Identification} Note that $X$ is birational to $\PP^2_1\times\PP^2_2$ via The projection $\pi$ to the latter. To understand where the fiber of $\pi$ degenerates to a $\PP^1$ we argue as usual and, using \cref{thm:ddl,thm:en}, we obtain that the exceptional locus of $\pi$ is a surface $S$ that can be described as $\mZ(\mQ_{\PP^2_1}(1,0,0)^{\oplus 3}\oplus\of(1,0,1)\oplus\of(1,-1,1))\subset\PP^6\times\PP^2_1\times\PP^2_2$. Moreover since $e(S)=16$, $K_S=\of(1,-1,-1)_{|S}$, $(-K_S)^2=-4$ and $h^0(S,K_S^{\otimes 2})=0$, then $S$ is rational. Thus $S=\Bl_{13pts}\PP^2$ and the final picture of $X$ is $\Bl_{\Bl_{13pts}\PP^2}\PP^2\times\PP^2$. \newline
Note that since $X$ is birational to $\PP^2\times\PP^2$, then it is rational.
\end{fano} \vspace{5mm}

\begin{fano}
\fanoid{$K_{478}$}
\label[mystyle]{F.3.71}
$X=\mZ(\mU^{\vee}_{\Gr(2,5)}(0,1,0)\oplus\of(0,0,1)^{\oplus 3}\oplus\of(1,1,0))\subset\PP^1\times\PP^3\times\Gr(2,5)$
\subsubsection*{Invariants} $h^0(-K)=39, \ (-K)^4=162,\ \chi(T_X)=-5,\ h^{1,1}=3, \ h^{3,1}=0,  \ h^{2,2}=8$.
\subsubsection*{Description} $X=\Bl_{\DP_5}Y$, with $Y$ the fourfold \cref{F.2.45}
\subsubsection*{Identification} Recall that $Y=\mZ(\mU^{\vee}_{\Gr(2,5)}(1,0)\oplus\of(0,1)^{\oplus 3})\subset\PP^3\times\Gr(2,5)$ is the fourfold \cref{F.2.45}, and $X=\mZ(\of(1,1,0))\subset\PP^1\times Y$. Now we can apply \cref{lm:blowPPtX} and obtain that $X=\Bl_{S}Y$, with $S=\mZ(\of(1,0)^{\oplus 2})\subset Y$. Note that $e(S)=7$, $K_S=\of(1,-1)_{|S}$, $(-K_S)^2=5$ and $h^0(S,K_S^{\otimes 2})=0$, hence, by classification $S$ is a $\DP_5$.
\end{fano} \vspace{5mm}

\begin{fano}
\fanoid{$K_{154}$}
\label[mystyle]{F.3.72}
$X=\mZ(\mQ_{\Gr(2,4)}(1,0,0)\oplus\mU^{\vee}_{\Gr(2,4)}(0,1,0)\oplus\of(0,0,1))\subset\PP^2\times\PP^3\times\Gr(2,4)$
\subsubsection*{Invariants} $h^0(-K)=69, \ (-K)^4=320,\ \chi(T_X)=7,\ h^{1,1}=3, \ h^{3,1}=0,  \ h^{2,2}=4$.
\subsubsection*{Description} $X=\Bl_{\Bl_{pt}\PP^2}Y$, with $Y$ the fourfold \cref{F.2.43}.
\subsubsection*{Identification} Recall that $Y=\mZ(\mU^{\vee}_{\Gr(2,4)}(1,0)\oplus\of(0,1))\subset\PP^3\times\Gr(2,4)$ is the fourfold \cref{F.2.43}, and $X=\mZ(\mQ_{\Gr(2,4)}(1,0,0))\subset\PP^2\times Y$. In particular $X$ is birational to $Y$ via The projection $\pi$ to the latter, which, using \cref{thm:ddl,thm:en}, has exceptional locus a surface $S$, which can be described as $\mZ(\mQ_{\Gr(2,4)})\subset Y$. Note that $\mZ(\mQ_{\Gr(2,4)})\subset Y$ is $\mZ(\mQ_2\oplus\of(0,1))\subset\Fl(1,2,4)$, hence by \cref{lm:flag2.5,thm:seq}, is $\mZ(\of(1,1))\subset\PP^1\times\PP^2$, which is a $\Bl_{pt}\PP^2$. In the end $X=\Bl_{\Bl_{pt}\PP^2}Y$.\\
Since $Y$ is rational, so $X$ is as such.
\end{fano} \vspace{5mm}

\begin{fano}
\fanoid{$K_{404}$}
\label[mystyle]{F.3.73}
$X=\mZ(\mQ_{\Gr(2,4)}\boxtimes\mU^{\vee}_{\Gr(2,5)}\oplus\of(0,1,0)\oplus\of(0,0,1)\oplus\of(1,0,1))\subset\PP^1\times\Gr(2,4)\times\Gr(2,5)$
\subsubsection*{Invariants} $h^0(-K)=45, \ (-K)^4=193,\ \chi(T_X)=-5,\ h^{1,1}=, \ h^{3,1}=0,  \ h^{2,2}=9$.
\subsubsection*{Description} $X=\Bl_{\DP_4}Y$, with $Y$ the fourfold \cref{F.2.28}
\subsubsection*{Identification} Recall that $Y=\mZ(\mQ_{\Gr(2,4)}\boxtimes\mU^{\vee}_{\Gr(2,5)}\oplus\of(1,0)\oplus\of(0,1))\subset\Gr(2,4)\times\Gr(2,5)$ is the fourfold \cref{F.2.28}, and $X=\mZ(\of(1,0,1))\subset\PP^1\times Y$. Now we can apply \cref{lm:blowPPtX} and obtain that $X=\Bl_{S}Y$, with $S=\mZ(\of(0,1)^{\oplus 2})\subset Y$. Note that $e(S)=8$, $K_S=\of(-1,0)_{|S}$, $(-K_S)^2=4$ and $h^0(S,K_S^{\otimes 2})=0$, hence $S$ is a $\DP_4$.\newline
Since $X=\Bl_{\DP_4}Y$, and $Y$ is rational, so it is $X$.
\end{fano} \vspace{5mm}

\begin{fano}
\fanoid{$K_{406}$}
\label[mystyle]{F.3.74}
$X=\mZ(\mU^{\vee}_{\Gr(2,5)}(1,0,0)\oplus\mU^{\vee}_{\Gr(2,5)}(0,1,0)\oplus\of(0,0,1)^{\oplus 2})\subset\PP_1^2\times\PP_2^2\times\Gr(2,5)$
\subsubsection*{Invariants} $h^0(-K)=43, \ (-K)^4=183,\ \chi(T_X)=-4,\ h^{1,1}=3, \ h^{3,1}=0,  \ h^{2,2}=8$.
\subsubsection*{Description} $X=\Bl_{\DP_6}Y$, with $Y$ the fourfold \cref{F.2.46}
\subsubsection*{Identification}  Recall that $Y=\mZ(\mU^{\vee}_{\Gr(2,5)}(1,0)\oplus\of(0,1)^{\oplus 2})\subset\PP^2\times\Gr(2,5)$ is the fourfold \cref{F.2.46}, and $X=\mZ(\mU^{\vee}_{\Gr(2,5)}(1,0,0))\subset\PP^2\times Y$. In particular $X$ is birational to $Y$ via The projection $\pi$ to the latter, which, using \cref{thm:ddl,thm:en}, has exceptional locus a surface $S$, which can be described as $\mZ(\mQ_{\Gr(2,5)}(1,0,0))\subset\PP^1\times Y$. Note that $e(S)=6$, $K_S=\of(1,-1,-1)_{|S}$, $(-K_S)^2=6$ and $h^0(S,K_S^{\otimes 2})=0$, hence $S$ is a $\DP_6$.\newline
Since $X=\Bl_{\DP_6}Y$, and $Y$ is rational, so it is $X$.
\end{fano} \vspace{5mm}

\section{Fano Varieties with Picard Rank 4}\label{sec:pic4}
\renewcommand\thefano{4--\arabic{fano}}
\setcounter{fano}{0}

\begin{fano}
\fanoid{$K_{280}$}
\label[mystyle]{F.4.1}
$X=\mZ(\mQ_{\Gr(2,4)}\boxtimes\mU^{\vee}_{\Gr(2,6)}\oplus\of(1,0)\oplus\Sym^2\mU^{\vee}_{\Gr(2,6)})\subset\Gr(2,4)\times\Gr(2,6)$
\subsubsection*{Invariants}  $h^0(-K)=51, \ (-K)^4=228,\ \chi(T_X)=2,\ h^{1,1}=4,\ h^{3,1}=0,  \ h^{2,2}=10$
\subsubsection*{Description}  $\pi:X\to\Gr(2,6)$ is the map associated to $\Bl_{\PP^1\bigsqcup\PP^1}Y$ with $Y$ the Fano fourfold obtained as a linear section in $\Fl(1,3,4)$.
\subsubsection*{Identification} Let $T$ be $\mZ(\mQ_{\Gr(2,4)}\boxtimes\mU^{\vee}_{\Gr(2,6)})\subset\Gr(2,4)\times\Gr(2,6)$. We apply \crefpart{lm:blowFlagGr}{item_3} and we obtain that, outside a point, $T=\Bl_{Z'}\Gr(2,6)$, with $Z'=\mZ(\mQ_{\PP^3}\boxtimes\mU^{\vee}_{\Gr(2,6)})\subset\PP^3\times\Gr(2,6)$. Note that $X=\mZ(\mZ(\of_{\Gr(2,4)}(1)\oplus\Sym^2\mU^{\vee}_{\Gr(2,6)})\subset T$, hence the sections $\of_{\Gr(2,4)}(1)$ and $\Sym^2\mU^{\vee}_{\Gr(2,6)}$ cut both $Z'$ and $\Gr(2,6)$. In this way $X=\Bl_{Z}Y$, with $Z=\mZ(\of_{\Gr(2,6)}(1)\oplus\Sym^2\mU^{\vee}_{\Gr(2,6)})\subset Z'$ and $Y=\mZ(\Sym^2\mU^{\vee}_{\Gr(2,6)}\oplus\of_{\Gr(2,6)}(1))\subset\Gr(2,6)$. Note that $\mZ(\Sym^2\mU_{\Gr(2,6)}^{\vee})\subset\Gr(2,6)$ is $\Fl(1,3,4)$. Thus $Y=\mZ(\of(1,1))\subset\Fl(1,3,4)$, already described in \cref{F.3.13}, and by applying \crefpart{lm:blowFlagGr}{item_1} and dualizing, we get $Y=\mZ(\of(1,1)^{\oplus 2})\subset\PP^3\times\PP^3$, which is exactly the fourfold described in \cref{F.3.23}. $Z$ is simply the disjoint union of 2 $\PP^1s$, hence $X=\Bl_{\PP^1\bigsqcup\PP^1}Y$.

\end{fano} \vspace{5mm}

\begin{fano}
\fanoid{$K_{26}$}
\label[mystyle]{F.4.2}
$X=\mZ(\mQ_{\Gr(2,4)}(1,0,0,0)\oplus\mQ_{\Gr(2,4)}(0,1,0,0)\oplus\mQ_{\Gr(2,4)}(0,0,1,0))\subset\PP^2\times\PP^2\times\PP^2\times\Gr(2,4)$
\subsubsection*{Invariants} $h^0(-K)=63, \ (-K)^4=290,\ \chi(T_X)=6,\ h^{1,1}=4,\ h^{3,1}=0,  \ h^{2,2}=8$.
\subsubsection*{Description} $X=\Bl_{\DP_7}Y$, with $Y$ the fourfold \cref{F.3.2}
\subsubsection*{Identification} Recall that $Y=\mZ(\mQ_{\Gr(2,4)}(1,0,0)\oplus\mQ_{\Gr(2,4)}(0,1,0))\subset\PP^2\times\PP^2\times\Gr(2,4)$ is the fourfold \cref{F.3.2}, so $X=\mZ(\mQ_{\Gr(2,4)}(1,0,0,0))\subset\PP^2\times Y$, and in particular $X$ is birational to $Y$ via the projection to the latter. To understand where the fiber of $\pi$ degenerates to a $\PP^1$ we use \cref{thm:ddl,thm:en} and obtain that the exceptional locus of $\pi$ is a surface $S$ that can be described as $\mZ(\mU^{\vee}_{\Gr(2,4)})\subset Y$. Note that $\mZ(\mU^{\vee}_{\Gr(2,4)})\subset Y$ is $\mZ(\of(1,0,1)\oplus\of(0,1,1))\subset\PP^1\times\PP^1\times\PP^2$, hence a $\DP_7$.\newline
Since $X$ is $\Bl_{\DP_7}Y$, then it is rational, because $Y$ is as such.
\end{fano} \vspace{5mm}

\begin{fano}
\fanoid{$K_{73}$}
\label[mystyle]{F.4.3}
$X=\mZ(\mU^{\vee}_{\Gr(2,4)}(1,0,0,0)\oplus\mQ_{\PP^2_2}\boxtimes\mU^{\vee}_{\Gr(2,5)}\oplus\mQ_{\Gr(2,4)}\boxtimes\mU^{\vee}_{\Gr(2,5)})\subset\PP^2_1\times\PP^2_2\times\Gr(2,4)\times\Gr(2,5))$
\subsubsection*{Invariants}  $h^0(-K)=78, \ (-K)^4=369,\ \chi(T_X)=11,\ h^{1,1}=4,\ h^{3,1}=0,  \ h^{2,2}=7$.
\subsubsection*{Description} $X=\Bl_{Q_2}\Bl_{\PP^1} Y$, with $Y$ the Fano fourfold \red{6} in \cite[Table 3]{kalashnikov}.
\subsubsection*{Identification} Using the same argument as in \cref{F.3.24}, we get that $Y=\mZ(\mQ_{\PP^2_2}\boxtimes\mU^{\vee}_{\Gr(2,5)})\subset\PP^2_2\times\Gr(2,5)$ is the Fano fourfold \red{6} in \cite[Table 3]{kalashnikov}. Moreover, if we call $Z$ the fourfold $\mZ(\mQ_{\Gr(2,4)}\boxtimes\mU^{\vee}_{\Gr(2,5)})\subset\Gr(2,4)\times Y$ we can apply \crefpart{lm:blowFlagGr}{item_2} to obtain $Z=\Bl_{C}Y$, with $C=\mZ(\mQ_{\Gr(2,5)})\subset Y$. In particular $e(C)=2$, $h^{1,0}=0$, hence, by classification, it is a rational curve. Finally $X=\mZ(\mU^{\vee}_{\Gr(2,4)}(1,0,0,0))\subset\PP^2\times Z$, and the projection $\pi'$ on $Z$, gives a birational map between the latter and $X$. Note that, using \cref{thm:ddl,thm:en}, we have that the fiber of $\pi'$ degenerates to a $\PP^1$ along a surface $S$, which can be described as $\mZ(\mQ_{\Gr(2,4)})\subset Z$. Note that $e(S)=4$, $K_S=\of(-1,-1,-1)_{|S}$, $(-K_S)^2=8$, hence $S$ is a $\DP_8$. If we project $S$ to $\PP^2\times\PP^2$ we have, by \cref{thm:degloc}, that $S$ is birational to $\PP^1\times\PP^1$, so it is a $Q_2$\\
Since $Y$ is rational, so $X$ is as such.
\end{fano} \vspace{5mm}

\begin{fano}
\fanoid{$K_{144}$}
\label[mystyle]{F.4.4}
$X=\mZ(\mQ_{\Gr(2,4)}(1,0,0,0)\oplus\mQ_{\PP^2_2}\boxtimes\mU^{\vee}_{\Gr(2,5)}\oplus\mQ_{\Gr(2,4)}\boxtimes\mU^{\vee}_{\Gr(2,5)})\subset\PP^2_1\times\PP^2_2\times\Gr(2,4)\times\Gr(2,5))$
\subsubsection*{Invariants} $h^0(-K)=78, \ (-K)^4=369,\ \chi(T_X)=11,\ h^{1,1}=4,\ h^{3,1}=0,  \ h^{2,2}=7$.
\subsubsection*{Description}  $X=\Bl_{\Bl_{pt}\PP^2}\Bl_{\PP^1} Y$, with $Y$ the Fano fourfold \red{6} in \cite[Table 3]{kalashnikov}.
\subsubsection*{Identification} The description is almost the same as before, except for the fact that $S=\mZ(\mU^{\Gr(2,4)})\subset Z$. Recall that if we project $S$ on $\PP^2_2$ we get that the generic fiber is a point, hence $S=\Bl_{p}\PP^2$.\\
As the previous one, also this Fano fourfold is rational.
\end{fano} \vspace{5mm}

\begin{fano}
\fanoid{$K_{256}$}
\label[mystyle]{F.4.5}
$X=\mZ(\mQ_{\PP^2_1}\boxtimes\mU^{\vee}_{\Gr(2,7)}\oplus \mQ_{\PP^2_2}\boxtimes\mU^{\vee}_{\Gr(2,7)}\oplus \mQ_{\PP^2_3}\boxtimes\mU^{\vee}_{\Gr(2,7)})\subset\PP^2_1\times\PP^2_2\times\PP^2_3\times\Gr(2,7)$
\subsubsection*{Invariants}  $h^0(-K)=81, \ (-K)^4=385,\ \chi(T_X)=12,\ h^{1,1}=4,\ h^{3,1}=0,  \ h^{2,2}=7$.
\subsubsection*{Description} $X$ is a small resolution of $Y$, 4 fold with $e(Y)=16$, $deg(Y)=13$ and singular in a point.
\subsubsection*{Identification} Let us first consider 
$Z=\mZ(\mQ_{\PP^2_2}\boxtimes\mU^{\vee}_{\Gr(2,7)}\oplus \mQ_{\PP^2_3}\boxtimes\mU^{\vee}_{\Gr(2,7)})\subset\PP^2_2\times\PP^2_3\times\Gr(2,7)$, since $\PP^2\times\PP^2$ can be also seen as $\mZ(\mQ_{\Gr(2,6)}^{\vee}(1))\subset\Gr(2,6)$ we can rewrite everything as 
$\mZ(\mQ_{\Gr(2,6)}\boxtimes\mU^{\vee}_{\Gr(2,7)}\oplus\mQ^{\vee}_{\Gr(2,6)}(1))\subset\Gr(2,6)\times\Gr(2,7)$, thus, by \crefpart{flagCor}{cor:flag_2}, $Z$ can be rewritten as $\mZ(\mQ_2\oplus\mQ^{\vee}_2(1))\subset\Fl(2,3,7)$, hence $X=\mZ(\mQ_{\PP^2_1}\boxtimes\mU_1^{\vee})\subset\PP^2_1\times Z$. The fiber of The projection $\pi$ to $Z$ is generically empty, hence, by \cref{thm:degloc}, $X$ is birational to the fourfold $Y=D_1(\varphi)$, with \[\varphi:\mU_{1|Z}\to \of_{\Fl(2,3,7)|Z}^{\oplus 3}.\] Using \cref{thm:ddl,thm:en}, $Y$ is a fourfold with $e(Y)=16$ and $deg(Y)=13$, moreover, since $D_2(\varphi_{|Z})$ consists in a point, then $Y$ is singular there, and the projection to $Z$ is the contraction of a $\PP^2$, or equivalently $X$ is a small resolution of $Y$. If we consider the projection $\pi_{123}$ to $\PP^2_1\times\PP^2_2\times\PP^2_3$ we get, by \cref{thm:degloc}, that the fiber is a point along the $D_5(\varphi)$, with
\[
\varphi:\mQ_{\PP^2_1}^{\vee}\oplus\mQ_{\PP^2_2}^{\vee}\oplus\mQ_{\PP^2_3}^{\vee}\to V_7^{\vee}\otimes\of_{\PP^2_1\times\PP^2_2\times\PP^2_3}.
\]
By \cref{thm:ddl,thm:en} we get that $D_5(\varphi)$ is fourfold $Y'$ with $e(Y')=16$. Moreover, again by \cref{thm:ddl,thm:en}, the fiber of $\pi_{123}$ jumps to a $\PP^2$ over $D_4(\varphi)=\{p\}$. Thus $X$ is a small resolution of $Y'$.
\end{fano} \vspace{5mm}

\begin{fano}
\fanoid{$K_{78}$}
\label[mystyle]{F.4.6}
$X=\mZ(\mQ_{\PP^2_2}\boxtimes\mU^{\vee}_{\Gr(2,5)}\oplus\mQ_{\Gr(2,5)}(1,0,0,0)\oplus\mU^{\vee}_{\Gr(2,5)}(0,1,0,0))\subset\PP^3\times\PP^2_1\times\PP^2_2\times\Gr(2,5)$
\subsubsection*{Invariants}  $h^0(-K)=71, \ (-K)^4=331,\ \chi(T_X)=8,\ h^{1,1}=4, \ h^{3,1}=0,  \ h^{2,2}=7$.
\subsubsection*{Description} $X=\Bl_{\Bl_{pt}\PP^2}Y$, with $Y$ the fourfold \cref{F.3.48}.
\subsubsection*{Identification} Let $Y$ be $\mZ(\mQ_{\PP^2_2}\boxtimes\mU^{\vee}_{\Gr(2,5)}\oplus\mU^{\vee}_{\Gr(2,5)}(1,0,0))\subset\PP^2_1\times\PP^2_2\times\Gr(2,5)$. Note that $Y$ is the fourfold \cref{F.3.48}. Hence $X=\mZ(\mQ_{\Gr(2,5)}(1,0,0,0))\subset\PP^3\times Y$, so it is birational to $Y$ via the projection to the latter. Using \cref{thm:ddl,thm:en} the fiber of $\pi$ degenerates to a $\PP^1$ along a surface $S$ that can be described as $\mZ(\mU^{\vee}_{\Gr(2,5)})\subset Y$. Moreover, $e(S)=4$, $K_S=\of(-1,-1,-1)_{|S}$, $(-K_S)^2=8$ then $S$ is a $\DP_8$. If we project to $\PP^2_2$ we get that the generic fiber is one point. Thus $S$ is birational to $\PP^2$, so it is $\Bl_{pt}\PP^2$.
\end{fano} \vspace{5mm}

\begin{fano}
\fanoid{$K_{104}$}
\label[mystyle]{F.4.7}
$X=\mZ(\mQ_{\PP^2_3}\boxtimes\mU^{\vee}_{\Gr(2,5)}\oplus\mU^{\vee}_{\Gr(2,5)}(1,0,0)\oplus\mU^{\vee}_{\Gr(2,5)}(0,1,0))\subset\PP^2_1\times\PP^2_2\times\PP^2_3\times\Gr(2,5)$
\subsubsection*{Invariants} $h^0(-K)=66, \ (-K)^4=305,\ \chi(T_X)=6,\ h^{1,1}=4, \ h^{3,1}=0,  \ h^{2,2}=8$.
\subsubsection*{Description}  $X=\Bl_{\DP_7}Y$,  with $Y$ the fourfold \cref{F.3.48}.
\subsubsection*{Identification} Note that the argument is the exact same of the \cref{F.4.6}, but the exceptional locus is a surface $S$ that can be described as $\mZ(\mQ_{\Gr(2,5)}(1,0,0,0))\subset\PP^1\times Y$. This gives $e(S)=4$, $K_S=\of(1,-1,-1,-1)_{|S}$, $(-K_S)^2=7$ and $h^0(S,K_S^{\otimes 2})=0$, then $S$ is a $\DP_7$.
\end{fano} \vspace{5mm}

\begin{fano}
\fanoid{$K_{201}$}
\label[mystyle]{F.4.8}
$X=\mZ(\mQ_{\Gr(2,4)}(1,0,0,0)\oplus\mQ_{\Gr(2,4)}(0,1,0,0)\oplus\mU^{\vee}_{\Gr(2,4)}(0,0,1,0))\subset\PP^2_1\times\PP^2_2\times\PP^2_3\times\Gr(2,4)$
\subsubsection*{Invariants} $h^0(-K)=61, \ (-K)^4=278,\ \chi(T_X)=6,\ h^{1,1}=4, \ h^{3,1}=0,  \ h^{2,2}=6$.
\subsubsection*{Description} $X=\Bl_{\PP^2}$, with $Y$ the fourfold \cref{F.3.2}.
\subsubsection*{Identification} Recall that $Y=\mZ((\mQ_{\Gr(2,4)}(1,0,0)\oplus\mQ_{\Gr(2,4)}(0,1,0))\subset\PP^2_1\times\PP^2_2\times\Gr(2,4)$ is the fourfold \cref{F.3.2}, hence $X=\mZ(\mU^{\vee}_{\Gr(2,4)}(1,0,0,0))\subset\PP^2\times Y$, so it is birational to $Y$ via The projection $\pi$ to the latter. Using \cref{thm:ddl,thm:en} we get that the fiber of $\pi$ degenerates to a $\PP^1$ over a surface $S$ that can be described as $\mZ(\mQ_{\Gr(2,4)})\subset Y$. Moreover, $e(S)=3$, $K_S=\of(-1,-1,-1)_{|S}$, $(-K_S)^2=9$, then, $S$ is $\PP^2$.\newline
Note that since $Y$ is rational, so it is $X$.
\end{fano} \vspace{5mm}

\begin{fano}
\fanoid{$K_{204}$}
\label[mystyle]{F.4.9}
$X=\mZ(\mQ_{\PP^2_2}\boxtimes\mU_{\Gr(2,5)}^{\vee}\oplus\mU^{\vee}_{\Gr(2,5)}(0,1,0,0)\oplus\of(1,0,0,1))\subset\PP^1\times\PP^2_1\times\PP^2_2\times\Gr(2,5)$
\subsubsection*{Invariants} $h^0(-K)=57, \ (-K)^4=257,\ \chi(T_X)=2,\ h^{1,1}=4, \ h^{3,1}=0,  \ h^{2,2}=9$.
\subsubsection*{Description} $X=\Bl_{\DP_6}Y$, with $Y$ the fourfold \cref{F.3.48}.
\subsubsection*{Identification} Recall that $Y=\mZ((\mQ_{\PP^2_2}\boxtimes\mU_{\Gr(2,5)}^{\vee}\oplus\mU^{\vee}_{\Gr(2,5)}(1,0,0))\subset\PP^2_1\times\PP^2_2\times\Gr(2,5)$ is, by the same argument in \cref{F.4.6}, the Fano fourfold \red{17} in \cite[Table 3]{kalashnikov}, hence $X=\mZ(\of(1,0,0,1))\subset\PP^1\times Y$. Using \cref{lm:blowPPtX} $X=\Bl_{S}Y$, with $S=\mZ(\of(0,0,1)^{\oplus 2})\subset Y$, in particular $e(S)=6$, $K_S=\of(-1,-1,0)_{|S}$, $(-K_S)^2=6$ and $h^0(S,K_S^{\otimes 2})=0$, so $S$ is a $\DP_6$.
\end{fano} \vspace{5mm}

\begin{fano}
\fanoid{$K_{227}$}
\label[mystyle]{F.4.10}
$X=\mZ(\mQ_{\Gr(2,4)}(0,0,1,0)\oplus\mQ_{\Gr(2,4)}(0,1,0,0)\oplus\of(1,0,0,1))\subset\PP^1\times\PP^2_1\times\PP^2_2\times\Gr(2,4)$
\subsubsection*{Invariants} $h^0(-K)=53, \ (-K)^4=236,\ \chi(T_X)=1,\ h^{1,1}=4, \ h^{3,1}=0,  \ h^{2,2}=9$.
\subsubsection*{Description} $X=\Bl_{\DP_6}$, with $Y$ the fourfold \cref{F.3.2}.
\subsubsection*{Identification} Recall that $Y=\mZ((\mQ_{\Gr(2,4)}(1,0,0)\oplus\mQ_{\Gr(2,4)}(0,1,0))\subset\PP^2_1\times\PP^2_2\times\Gr(2,4)$ is the fourfold \cref{F.3.2}, hence $X=\mZ(\of(1,0,0,1))\subset\PP^1\times Y$. Using \cref{lm:blowPPtX} $X=\Bl_{S}Y$, with $S=\mZ(\of(0,0,1)^{\oplus 2})\subset Y$, in particular $e(S)=6$, $K_S=\of(-1,-1,0)_{|S}$, $(-K_S)^2=6$ and $h^0(S,K_S^{\otimes 2})=0$, so $S$ is a $\DP_4$.\newline
Note that since $Y$ is rational, so it is $X$.
\end{fano} \vspace{5mm}

\begin{fano}
\fanoid{$K_{297}$}
\label[mystyle]{F.4.11}
$X=\mZ(\mQ_{\Gr(2,4)}(0,0,1,0)\oplus\mQ_{\Gr(2,4)}(0,1,0,0)\oplus\of(1,0,1,0)\oplus\of(0,0,0,1))\subset\PP^1\times\PP^2\times\PP^3\times\Gr(2,4)$
\subsubsection*{Invariants} $h^0(-K)=55, \ (-K)^4=246,\ \chi(T_X)=3,\ h^{1,1}=4, \ h^{3,1}=0,  \ h^{2,2}=7$.
\subsubsection*{Description} $X=\Bl_{\DP_7}Y$, with $Y$ the fourfold \cref{F.3.72}. 
\subsubsection*{Identification} Recall that $Y=\mZ(\mQ_{\Gr(2,4)}(0,1,0)\oplus\mQ_{\Gr(2,4)}(1,0,0)\oplus\of(0,0,1))\subset\PP^2\times\PP^3\times\Gr(2,4)$ is the fourfold \cref{F.3.2}, hence $X=\mZ(\of(1,0,1,0))\subset\PP^1\times Y$.  Using \cref{lm:blowPPtX} $X=\Bl_{S}Y$, with $S=\mZ(\of(0,1,0)^{\oplus 2})\subset Y$, in particular $e(S)=5$, $K_S=\of(-1,0,-1)_{|S}$, $(-K_S)^2=7$ and $h^0(S,K_S^{\otimes 2})=0$, so $S$ is a $\DP_7$.
\end{fano} \vspace{5mm}

\begin{fano}
\fanoid{$K_{403}$}
\label[mystyle]{F.4.12}
$X=\mZ(\mQ_{\Gr(2,4)}(0,0,1,0)\oplus\of(1,0,1,0)\oplus\of(0,1,0,1)\oplus\of(0,0,0,1))\subset\PP^1_1\times\PP^1_2\times\PP^3\times\Gr(2,4)$
\subsubsection*{Invariants} $h^0(-K)=47, \ (-K)^4=204,\ \chi(T_X)=0,\ h^{1,1}=4, \ h^{3,1}=0,  \ h^{2,2}=8$.
\subsubsection*{Description} $X=\Bl_{\DP_6}Y$, with $Y$ the fourfold \cref{F.3.18}.
\subsubsection*{Identification} Recall that $Y=\mZ(\mQ_{\Gr(2,4)}(0,1,0)\oplus\of(1,0,1)\oplus\of(0,0,1))\subset\PP^1_2\times\PP^3\times\Gr(2,4)$ is the fourfold \cref{F.3.2}, hence $X=\mZ(\of(1,0,1,0))\subset\PP^1\times Y$. Using \cref{lm:blowPPtX} $X=\Bl_{S}Y$, with $S=\mZ(\of(0,1,0)^{\oplus 2})\subset Y$, in particular $e(S)=6$, $K_S=\of(-1,0,-1)_{|S}$, $(-K_S)^2=6$ and $h^0(S,K_S^{\otimes 2})=0$, so, by classification, $S$ is a $\DP_6$.
\end{fano} \vspace{5mm}

\begin{fano}
\fanoid{$K_{482}$}
\label[mystyle]{F.4.13}
$X=\mZ(\mQ_{\PP^2_1}(0,0,1,1)\oplus\of(1,1,0,0))\subset\PP^1\times\PP^2_1\times\PP^2_2\times\PP^2_3$
\subsubsection*{Invariants} $h^0(-K)=39, \ (-K)^4=162,\ \chi(T_X)=-4,\ h^{1,1}=4,\ h^{2,1}=1, \ h^{3,1}=0,  \ h^{2,2}=9$
\subsubsection*{Description} $X=\Bl_{\DP_6}Y$, with $Y$ the fourfold \cref{F.3.38}.
\subsubsection*{Identification} Recall that $Y=\mZ(\mQ_{\PP^2_1}(0,1,1))\subset\PP^2\times\PP^2_2\times\PP2_3$ is the Fano fourfold \cref{F.3.38}, hence $X=\mZ(\of(1,1,0,0))\subset\PP^1\times Y$. Using \cref{lm:blowPPtX} we get that $X=\Bl_{S}Y$, with $S=\mZ(\of(1,0,0)^{\oplus 2})\subset Y$. Note that $\mZ(\of(1,0,0)^{\oplus 2})\subset Y$ is $\mZ(\of(1,1))\subset\PP^2\times\PP^2$, hence it is a $\DP_6$.
\end{fano} \vspace{5mm}


\section{Summary Table}\label{sec:tab}

In the table, we adopt the following notation:
\begin{itemize}
    \item $K_{n}$ is the family of Fano fourfolds in the $n$-th entry of \cite[Table 3]{kalashnikov};
    \item  $X_n^d$ are the $d$-folds obtained as complete intersections of linear sections in Grassmannians $\Gr(2,n)$;
    \item Str.\ p.\ bdl.\ on $T$, with $T$ a threefold, is a fourfold obtained as a generic $\PP^1$-bundle, whose fiber jumps to a projective space of higher dimension along a certain number of points;
    \item Conic bdl.\ on $T$, is a conic bundle over a threefold $T$;
    \item $F_{\rho-n}$ is the Fano threefold $\rho-n$ in \cite{fanography};
    \item $Q_n$  is the $n$-th dimensional quadrics;
    \item $\DP_n$ is a Del Pezzo surface of volume $n$;
    \item $E$ is an elliptic curve;
    \item $\mathcal{E}$ is an elliptic fibration of dimension two;
    \item Small Res., is a Fano fourfold obtained as a small resolution of a singular fourfold.
\end{itemize}

\begin{landscape}
    
\begin{centering}
\begin{scriptsize}
\setlength\tabcolsep{4pt}
\begin{longtable}{cccccccccccc}

\caption{Fano fourfolds.}\label{tab:4folds}\\
\toprule
X&
Id&
$h^0(-K)$&
$(-K)^4$&
$h^{1,1}$&
$h^{2,1}$& 
$h^{3,1}$&
$ h^{2,2}$&
Amb. & 
$\mathcal{F}$ & 
Desc.&
Rat.\\
\cmidrule(lr){1-1}\cmidrule(lr){2-2}\cmidrule(lr){3-3} \cmidrule(lr){4-4} \cmidrule(lr){5-5} \cmidrule(lr){6-6} \cmidrule(lr){7-7}\cmidrule(lr){8-8}\cmidrule(lr){9-9} \cmidrule(lr){10-10} \cmidrule(lr){11-11} \cmidrule(lr){12-12} 
\endfirsthead
\multicolumn{5}{l}{\vspace{-0.25em}\scriptsize\emph{\tablename\ \thetable{} continued from previous page}}\\
\toprule
X&
Id&
$h^0(-K)$&
$(-K)^4$&
$h^{1,1}$&
$h^{2,1}$& 
$h^{3,1}$&
$ h^{2,2}$&
Amb. & 
$\mathcal{F}$ & 
Desc. &
Rat.\\
\cmidrule(lr){1-1}\cmidrule(lr){2-2}\cmidrule(lr){3-3} \cmidrule(lr){4-4} \cmidrule(lr){5-5} \cmidrule(lr){6-6} \cmidrule(lr){7-7}\cmidrule(lr){8-8}\cmidrule(lr){9-9} \cmidrule(lr){10-10} \cmidrule(lr){11-11} \cmidrule(lr){12-12}
\endhead
\multicolumn{5}{r}{\scriptsize\emph{Continued on next page}}\\
\endfoot
\bottomrule
\endlastfoot

\cref{F.1.1}& $K_{742}$ & 9 & 15 & 1 & 0 & 41 & 232 &\ $\Gr(2,5)$\ & $\of(1)\oplus\of(3)$ &\cite[\red{b1}]{kuchle} & ?\\

\rowcolor{lavender}
\cref{F.1.2}& $K_{740}$ & 10 & 20 & 1 & 0 & 20 & 132 &\ $\Gr(2,5)$\ & $\of(2)^{\oplus 2}$ &\cite[\red{b2}]{kuchle} & ?\\

\cref{F.1.3}& $K_{715}$ & 15 & 42 & 1 & 0 & 6 & 57 &\ $\Gr(2,6)$\ & $\mQ(1)$ &\cite[\red{b3}]{kuchle} & ?\\

\rowcolor{lavender}
\cref{F.1.4}& $K_{730}$ & 13 & 33  & 1 & 0 & 8 & 70 &\ $\Gr(2,6)$\ & $\mU^{\vee}(1)\of(1)^{\oplus 2}$ &\cite[\red{b5}]{kuchle} & ?\\

\cref{F.1.5}& $K_{734}$ & 12 & 28 & 1 & 0 & 15 & 106 &\ $\Gr(2,6)$\ & $\of(1)^{\oplus 3}\oplus\of(2)$ &\cite[\red{b6}]{kuchle} & ?\\ 

\rowcolor{lavender}
\cref{F.1.6}& $K_{24}$ & 85 & 405 & 1 & 0 & 0 & 2 &\ $\Gr(2,5)$\ & $\of(1)^{\oplus 2}$ & $X_5^4$ & +\\

\cref{F.1.7}& $K_{365}$ & 51 & 224 & 1 & 0 & 0 & 8 &\ $\Gr(2,6)$\ & $\of(1)^{\oplus 4}$ & $X_6^4$ & +\\

\rowcolor{lavender}
\cref{F.1.8}& $K_{433}$ & 45 & 192 & 1 & 0 & 0 & 12 &\ $\Gr(2,5)$\ & $\mU^{\vee}(1)$ &\cite[\red{$V_{12}^4$}]{CGKS}& ?\\

\cref{F.1.9}& $K_{496}$ & 39 & 160 & 1 & 0 & 1 & 22 &\ $\Gr(2,5)$\ & $\of(1)\oplus\of(2)$ &\cite[\red{$X_{10}$}]{BFMT} & ?\\

\rowcolor{lavender}
\cref{F.2.1}& $K_{508}$ & 39 & 160 & 2 & 0 & 0 & 7 & $\PP^2\times\Gr(2,5)$ & $\of(1,1)\oplus\of(0,1)^{\oplus 3}$ & Str.\ p.\ bdl.\ on $X_5^3$  &+\\

\cref{F.2.2}& $K_{509}$ & 33 & 130 & 2 & 0 & 1 & 23 & $\Gr(2,4)\times\Gr(2,5)$ & $\mQ_{\Gr(2,4)}\boxtimes\mU^{\vee}_{\Gr(2,5)}\oplus\of(0,1)\oplus\of(1,1)$ &\cite[\red{GM-22}]{BFMT} &+\\

\rowcolor{lavender}
\cref{F.2.3}& $K_{517}$ & 33 & 131 & 2 & 0 & 0 & 12 & $\PP^5\times\Gr(2,5)$ & $\mQ_{\Gr(2,5)}(1,0)\oplus\mU^{\vee}_{\Gr(2,5)}(1,0)\oplus\of(0,1)^{\oplus 2}$ & $\Bl_{\Bl_{9pts}\PP^2}X_5^4$ &+\\

\cref{F.2.4}& $K_{527}$ & 30 & 114 & 2 & 0 & 0 & 17 & $\PP^2\times\Gr(2,7)$ & $\mQ_{\PP^2}\boxtimes\mU^{\vee}_{\Gr(2,7)}\oplus\of(0,1)^{\oplus 4}$ & Small res. &+\\

\rowcolor{lavender}
\cref{F.2.5}& $K_{558}$ & 30 & 114 & 2 & 0 & 1 & 24 & $\Gr(2,4)\times\Gr(2,5$ & $\mQ_{\Gr(2,4)}\boxtimes\mU^{\vee}_{\Gr(2,5)}\oplus\of(1,0)\oplus\of(0,2)$ &\cite[\red{GM-21}]{BFMT}&?\\

\cref{F.2.6}& $K_{559}$ & 29 & 110 & 2 & 0 & 1 & 22 & $\PP^4\times\Gr(2,5)$ & $\mQ_{\Gr(2,5)}(1,0)\oplus\of(0,1)^{\oplus 2}\oplus\of(1,1)$ &\cite[\red{K3-32}]{BFMT}&+\\

\rowcolor{lavender}
\cref{F.2.7}& $K_{566}$ & 26 & 94 & 2 & 0 & 1 & 28 & $\PP^3\times\Gr(2,5)$ & $\mQ_{\Gr(2,5)}(1,0)\oplus\of(0,1)\oplus\of(0,2)$ &\cite[\red{GM-20}]{BFMT}&?\\

\cref{F.2.8}& $K_{547}$ & 30 & 116 & 2 & 0 & 0 & 15 & $\Gr(2,4)\times\PP^6$ & $\mQ_{\Gr(2,4)}(0,1)\oplus\mU^{\vee}_{\Gr(2,4)}(0,1)^{\oplus 2}$ & $\Bl_{\Bl_{12pts}\PP^2}\Gr(2,4)$&+\\

\rowcolor{lavender}
\cref{F.2.9}& $K_{549}$ & 32 & 126 & 2 & 0 & 0 & 10 & $\PP^5\times\Gr(2,5)$ & $\mQ_{\Gr(2,5)}(1,0)\oplus\of(0,1)^{\oplus 3}\oplus\of(2,0)$ & Conic bdl.\ on $X_5^3$ &+\\

\cref{F.2.10}& $K_{552}$ & 30 & 116 & 2 & 0 & 0 & 13 & $\PP^4\times\Gr(2,5)$ & $\of(0,1)^{\oplus 2}\oplus\mU^{\vee}_{\Gr(2,5)}(1,0)^{\oplus 2}$ & $\Bl_{\Bl_{10pts}\PP^2}X_5^4$&+\\

\rowcolor{lavender}
\cref{F.2.11}& $K_{25}$ & 70 & 325 & 2 & 0 & 0 & 4 & $\Gr(2,4) \times \Gr(2,5)$ & $\mQ_{\Gr(2,4)} \boxtimes \mU^{\vee}_{\Gr(2,5)} \oplus \of(0,1)^{\oplus 2}$ & $\Bl_{\PP^1}X_5^4$&+\\

\cref{F.2.12}& $K_{554}$ & 24 & 90 & 2 & 0 & 0 & 12 & $\PP^3 \times \Gr(2,4)$ & $\of(1)\boxtimes Sym^2\mU^{\vee}_{\Gr(2,4)}$ & \cite{Kuz19}&+\\

\rowcolor{lavender}
\cref{F.2.13}& $K_{582}$ & 30 & 116 & 2 & 0 & 0 & 10 & $\PP^4\times\Gr(3,5)$ & $\mQ_{\Gr(3,5)}(1,0)\oplus \of(0,1)^{\oplus 3}\oplus \of(2,0)$ & Double conic bdl.&+\\

\cref{F.2.14}& $K_{583}$ & 28 & 106 & 2 & 0 & 0 & 15 & $\Gr_1(2,4) \times \Gr_2(2,4)$ & $\mU^{\vee}_{\Gr_1(2,4)}(0,1)\oplus\mU^{\vee}_{\Gr_2(2,4)}(1,0)$ & Small res. &+\\

\rowcolor{lavender}
\cref{F.2.15}& $K_{589}$ & 24 & 85 & 2 & 2 & 0 & 22 & $\PP^2 \times \Gr(2,5)$ & $\mQ_{\PP^2}(0,1)^{\oplus 2}$ & Small res. &+\\

\cref{F.2.16}& $K_{579}$ & 27 & 101 & 2 & 0 & 1 & 23 & $\PP^5\times\Gr(2,4)$ & $\mQ_{\Gr(2,4)}(1,0)^{\oplus 2}\oplus\of(1,1)$ &\cite[\red{C-9}]{BFMT}&+\\

\rowcolor{lavender}
\cref{F.2.17}& $K_{597}$ & 27 & 101 & 2 & 0 & 1 & 24 & $\PP^5\times\Gr(2,4)$ & $\mQ_{\Gr(2,4)}(1,0)\oplus\mU^{\vee}_{\Gr(2,4)}(1,0)\oplus\of(1,1)$ &\cite[\red{K3-33}]{BFMT}&+\\

\cref{F.2.18}& $K_{600}$ & 27 & 100 & 2 & 0 & 0 & 18 & $\PP^6 \times \Gr(2,4)$ & $\mU^{\vee}_{\Gr(2,4)}(1,0)\oplus\mQ_{\Gr(2,4)}(1,0)\oplus\of(0,1) \oplus\of(2,0)$ & $\Bl_{\Bl_{9pts}\PP^2}Q_5\cap Q_5'$&?\\

\rowcolor{lavender}
\cref{F.2.19}& $K_{603}$ & 27 & 100 & 2 & 0 & 0 & 16 & $\PP^2 \times \Gr(2,6)$ & $\mU^{\vee}_{\Gr(2,6)}(1,0)\oplus\of(0,1)^{\oplus 4}$ & $\Bl_{\DP_2}X_6^4$&+\\

\cref{F.2.20}& $K_{604}$ & 26 & 95 & 2 & 0 & 1 & 23 & $\PP^3\times\Gr(3,5)$ & $\mQ_{\Gr(3,5)}(1,0)\oplus\of(1,1)\oplus\of(0,1)^{\oplus 2}$ &\cite[\red{K3-34}]{BFMT}&+\\

\rowcolor{lavender}
\cref{F.2.21}& $K_{29}$ & 60 & 273 & 2 & 0 & 0 & 3 & $\PP^4\times\Gr(2,4)$ & $\mQ_{\Gr(2,4)}(1,0)^{\oplus 2}$ & $\Bl_{\PP^2}\Gr(2,4)$&+\\

\cref{F.2.22}& $K_{612}$ & 30 & 113 & 2 & 5 & 0 & 3 & $\PP^2\times\PP^4$ & $\mQ(0,2)$ &$\Bl_{Q_2\cap Q_2'\cap Q_2''}\PP^4$&+\\

\rowcolor{lavender}
\cref{F.2.23}& $K_{614}$ & 24 & 86 & 2 & 0 & 1 & 26 & $\PP^4\times\PP^5$ & $\of(1,1)\oplus\bigwedge^3\mQ_{\PP^4}(0,1)$ &\cite[\red{C-7}]{BFMT}&+\\

\cref{F.2.24}& $K_{616}$ & 26 & 96 & 2 & 2 & 0 & 10 & $\PP^4 \times \Gr(2,4)$ & $\mQ_{\Gr(2,4)}(1,0)\oplus\of(0,2)\oplus\of(2,0)$ & Conic bdl.\ on $\mZ(\of(2))\subset\Gr(2,4)$ &?\\

\rowcolor{lavender}
\cref{F.2.25}& $K_{617}$ & 25 & 90 & 2 & 0 & 0 & 19 & $\PP^2 \times \Gr(2,5)$ & $\mU^{\vee}_{\Gr(2,5)}(0,1) \oplus \mU^{\vee}_{\Gr(2,5)}(1,0)$ & $\Bl_{\DP_3}K_{433}$&+\\

\cref{F.2.26}& $K_{32}$ & 51 & 225 & 2 & 1 & 0 & 3 & $\PP^4_1\times\PP^4_2$ & $\mQ_{\PP^4_1}^{\vee}(1,1)$ & $\Bl_{S}\PP^4$&+\\

\rowcolor{lavender}
\cref{F.2.27}& $K_{109}$ & 60 & 272 & 2 & 0 & 0 & 2 & $\PP^4\times\Gr(2,5)$ & $\mQ_{\Gr(2,5)}(1,0)\oplus\of(0,1)^{\oplus 3}$ & $\Bl_{\nu_2(\PP^2)}\PP^4$&+\\

\cref{F.2.28}& $K_{101}$ & 75 & 352 & 2 & 0 & 0 & 3 & $\Gr(2,4) \times \Gr(2,5)$ & $\mQ_{\Gr(2,4)} \boxtimes \mU^{\vee}_{\Gr(2,5)} \oplus \of(0,1) \oplus \of(1,0)$ & $\Bl_{\PP^2}X_5^4$&+\\

\rowcolor{lavender}
\cref{F.2.29}& $K_{102}$ & 69 & 320 & 2 & 0 & 0 & 4 & $\PP^3 \times \Gr(2,5)$ & $\mQ_{\Gr(2,5)}(1,0) \oplus \of(0,1)^{\oplus 2}$ & $\Bl_{Q_2}X_5^4$&+\\

\cref{F.2.30}& $K_{625}$ & 24 & 85 & 2 & 0 & 2 & 32 & $\PP^4 \times \Gr(2,4)$ & $\of(1,1)^{\oplus 2} \oplus \mU^{\vee}_{\Gr(2,4)}(1,0)$ & $\Bl_{\mathcal{E}}\Gr(2,4)$&+\\

\rowcolor{lavender}
\cref{F.2.31}& $K_{626}$ & 25 & 90 & 2 & 0 & 1 & 24 & $\PP^5\times\Gr(2,4)$ & $\mQ_{\Gr(2,4)}(1,0)\oplus\of(0,1)\oplus\of(1,1)\oplus\of(2,0)$ & \cite[\red{K3-35}]{BFMT}&+\\

\cref{F.2.32}& $K_{628}$ & 23 & 80 & 2 & 0 & 1 & 28 & $\PP^3\times\PP^6$ & $\bigwedge^2\mQ_{\PP^3}(0,1)\oplus\of(1,1)\oplus\of(0,2)$ & \cite[\red{K3-31}]{BFMT}&+\\

\rowcolor{lavender}
\cref{F.2.33}& $K_{629}$ & 23 & 80 & 2 & 0 & 1 & 27 & $\PP^2\times\Gr(2,5)$ & $\mQ_{\PP^2}(0,1)\oplus\of(1,1)\oplus\of(0,1)$ &\cite[\red{GM-19}]{BFMT}&+\\

\cref{F.2.34}& $K_{115}$ & 60 & 272 & 2 & 0 & 0 & 4 & $\PP^4 \times \Gr(2,4)$ & $\mU^{\vee}_{\Gr(2,4)}(1,0)\oplus\mQ_{\Gr(2,4)}(1,0)$ & $\Bl_{Q_2}\Gr(2,4)$&+\\

\rowcolor{lavender}
\cref{F.2.35}& $K_{126}$ & 51 & 224 & 2 & 0 & 0 & 7 & $\PP^4\times\Gr(2,4)$ & $\mQ_{\Gr(2,4)}(1,0)\oplus\of(1,1)\oplus\of(0,1)$ & Str.\ p.\ bdl.\ on $Q_3$&+\\

\cref{F.2.36}& $K_{124}$ & 54 & 240 & 2 & 0 & 0 & 7 & $\PP^3 \times \Gr(2,4)$ & $\mU^{\vee}_{\Gr(2,4)}(1,0)\oplus\of(1,1)$ & $\Bl_{\DP_5}\Gr(2,4)$&+\\

\rowcolor{lavender}
\cref{F.2.37}& $K_{642}$ & 24 & 82 & 2 & 0 & 3 & 42 & $\PP^2\times\Gr(2,6)$ & $\mQ_{\PP^2}\boxtimes\mU^{\vee}_{\Gr(2,6)}\oplus\of(0,1)\oplus\of(0,2)$ & Small res. &?\\

\cref{F.2.38}& $K_{653}$ & 22 & 75 & 2 & 0 & 4 & 32 & $\PP^2 \times \Gr(2,5)$ & $\of(0,1)^{\oplus 2} \oplus \of(1,1)^{\oplus 2}$ & $\Bl_{\mathcal{E}}X_5^4$&+\\

\rowcolor{lavender}
\cref{F.2.39}& $K_{663}$ & 30 & 112 & 2 & 5 & 0 & 2 & $\PP^2 \times \Gr(2,5)$ & $\of(2,0)\oplus\mQ_{\Gr(2,5)}(0,1)$ & $Q_1\times F_{1-7}$&?\\

\cref{F.2.40}& $K_{664}$ & 21 & 69 & 2 & 0 & 5 & 52 & $\PP^4\times\Gr(2,4)$ & $\mQ_{\Gr(2,4)}(1,0)\oplus\of(0,1)\oplus(2,1)$ & Conic bdl.\ on $Q_3$ &+\\

\rowcolor{lavender}
\cref{F.2.41}& $K_{665}$ & 20 & 64 & 2 & 0 & 6 & 59 & $\PP^3 \times \Gr(2,4)$ & $\mQ_{\Gr(2,4)}(1,0) \oplus 
\of(1,2)$ & $\Bl_{S}\Gr(2,4)$&+\\

\cref{F.2.42}& $K_{669}$ & 21 & 70 & 2 & 0 & 3 & 37 & $\PP^2 \times \Gr(2,5)$ & $\of(2,1) \oplus \of(0,1)^{\oplus 3}$ & Conic bdl.\ on $X_5^3$ &+\\

\rowcolor{lavender}
\cref{F.2.43}& $K_{136}$ & 81 & 384 & 2 & 0 & 0 & 2 & $\PP^3\times\Gr(2,4)$ & $\mU^{\vee}_{\Gr(2,4)}(1,0)\oplus\of(0,1)$ & $\PP_{Q_3}(\mU_{\Gr(2,4)|Q_3})$&+\\

\cref{F.2.44}& $K_{723}$ & 16 & 45 & 2 & 0 & 8 & 72 & $\PP^1 \times \Gr(2,5)$ & $\of(0,1)^{\oplus 2} \oplus \of(1,2)$ & $\Bl_{Q_3\cap Q_3'}X_5^4$&+\\

\rowcolor{lavender}
\cref{F.2.45}& $K_{219}$ & 51 & 244 & 2 & 0 & 0 & 3 & $\PP^3\times\Gr(2,5)$ & $\mU^{\vee}_{\Gr(2,5)}(1,0)\oplus\of(0,1)^{\oplus 3}$ & Conic bdl.\ on $\PP^3$ &?\\

\cref{F.2.46}& $K_{195}$ & 63 & 288 & 2 & 0 & 0 & 4 & $\PP^2\times\Gr(2,5)$ & $\mU_{\Gr(2,5)}^{\vee}(1,0)\oplus\of(0,1)^{\oplus 2}$ & $\Bl_{Q_2}X_5^4$&+\\

\rowcolor{lavender}
\cref{F.2.47}& $K_{230}$ & 54 & 240 & 2 & 0 & 0 & 2 & $\PP^4\times\Gr(2,4)$ & $\mQ_{\Gr(2,4)}(1,0)\oplus\of(0,1)\oplus\of(2,0)$ & Conic bdl.\ on $Q_3$ &?\\

\cref{F.2.48}& $K_{238}$ & 45 & 192 & 2 & 2 & 0 & 2 & $\PP^3\times\Gr(2,4)$ & $\mU_{\Gr(2,4)}^{\vee}(1,0)\oplus\of(0,2)$ & $\PP_{F_{1-4}}(\mU_{\Gr(2,4)|F_{1-4}})$&+\\

\rowcolor{lavender}
\cref{F.2.49}& $K_{679}$ & 18 & 56 & 2 & 0 & 4 & 47 & $\PP^2 \times \Gr(2,4)$ & $\mU_{\Gr(2,4)}^{\vee}(1,1)$ & $\Bl_{S}\Gr(2,4)$&+\\

\cref{F.2.50}& $K_{253}$ & 42 & 177 & 2 & 0 & 0 & 11 & $\PP^4\times\Gr(2,8)$ & $\mQ_{\PP^4}\boxtimes\mU^{\vee}_{\Gr(2,8)}\oplus\mU^{\vee}_{\Gr(2,8)}(1,0)\oplus\of(0,1)^{\oplus 2}$ & Small res. &+\\

\rowcolor{lavender}
\cref{F.2.51}& $K_{142}$ & 75 & 352 & 2 & 0 & 0 & 4 & $\PP^2\times\Gr(2,6)$ & $\mQ_{\PP^2}\boxtimes\mU_{\Gr(2,6)}^{\vee}\oplus\of(0,1)^{\oplus 2}$ & Small res. &?\\

\cref{F.2.52}& $K_{686}$ & 27 & 96 & 2 & 7 & 0 & 2 & $\PP^2 \times \Gr(2,5)$ & $\mU^{\vee}_{\Gr(2,5)}(0,1)\oplus\of(0,1)\oplus\of(2,0)$ & $Q_1\times F_{1-6}$&+\\

\rowcolor{lavender}
\cref{F.2.53}& $K_{689}$ & 17 & 51 & 2 & 0 & 2 & 41 & $\PP^2\times\PP^5$ & $\mQ_{\PP^2}(0,2)\oplus\of(1,1)$ &\cite[\red{A-67}]{BFMT} &?\\

\cref{F.2.54}& $K_{699}$ & 19 & 60 & 2 & 0 & 1 & 32 & $\PP^1\times\Gr(2,5)$ & $\mU^{\vee}_{\Gr(2,5)}(0,1)\oplus\of(1,1)$ &\cite[\red{K3-24}]{BFMT}&+\\

\rowcolor{lavender}
\cref{F.2.55}& $K_{288}$ & 54 & 240 & 2 & 0 & 0 & 7 & $\PP^1\times\Gr(2,5)$ & $\of(1,1)\oplus(0,1)^{\oplus 2}$ & $\Bl_{\DP_5}X_5^4$&+\\

\cref{F.2.56}& $K_{464}$ & 39 & 162 & 2 & 0 & 0 & 8 & $\Gr(2,5)\times\Gr(3,5)$ & $\mU^{\vee}_{\Gr(2,5)}\boxtimes\mQ_{\Gr(3,5)}\oplus\of(1,0)^{\oplus 2}\oplus\of(0,1)^{\oplus 2}$ & $\Bl_{\DP_6}X_5^4$&+\\

\rowcolor{lavender}
\cref{F.2.57}& $K_{350}$ & 42 & 177 & 2 & 0 & 0 & 12 & $\Gr(2,4)\times\Gr(3,6)$ & $\mQ_{\Gr(2,4)}\boxtimes\mU^{\vee}_{\Gr(3,6)}\oplus\of(0,1)^{\oplus 3}$ & $\Bl_{\Bl_{9pts} \PP^2}\Gr(2,4)$&+\\

\cref{F.2.58}& $K_{361}$ & 33 & 129 & 2 & 0 & 1 & 23 & $\PP^3\times\Gr(2,7)$ & $\mQ_{\PP^3}\boxtimes\mU^{\vee}_{\Gr(2,7)}\oplus\of(0,1)^{\oplus 2}\oplus\of(1,1)$ & \cite[\red{R-62}]{BFMT}&?\\

\rowcolor{lavender}
\cref{F.2.59}& $K_{364}$ & 69 & 320 & 2 & 0 & 0 & 2 & $\PP^2\times\Gr(2,5)$ & $\of(2,0)\oplus\of(0,1)^{\oplus 3}$ & $Q_1\times F_{1-15}$&+\\

\cref{F.2.60}& $K_{374}$ & 45 & 193 & 2 & 0 & 0 & 15 & $\PP^6\times\Gr(2,9)$ & $\mQ_{\PP^6}\boxtimes\mU^{\vee}_{\Gr(2,9)}\oplus\mU^{\vee}_{\Gr(2,9)}(1,0)^{\oplus 2}$ & Small res. &?\\

\rowcolor{lavender}
\cref{F.2.61}& $K_{385}$ & 45 & 193 & 2 & 0 & 0 & 9 & $\Gr(2,5) \times \Gr(2,6)$ & $mQ_{\Gr(2,5)} \boxtimes \mU^{\vee}_{\Gr(2,6)}\oplus \of(1,0)^{\oplus 2} \oplus \of(0,1)^{\oplus 2}$ & $\Bl_{\PP^2}X_6^{4}$&+\\

\cref{F.2.62}& $K_{15}$ & 90 & 433 & 2 & 0 & 0 & 3 & $\PP^4\times\Gr(2,7)$ & $\mQ_{\PP^4}\boxtimes\mU^{\vee}_{\Gr(2,7)}\oplus\mU^{\vee}_{\Gr(2,7)}(1,0)$ & $\Bl_{Q_1}\PP^4$&+\\

\rowcolor{lavender}
\cref{F.2.63}& $K_{407}$ & 39 & 161 & 2 & 0 & 1 & 23 & $\PP^5\times\Gr(2,8)$ & $\mQ_{\PP^5}\boxtimes\mU^{\vee}_{\Gr(2,8)}\oplus\mU_{\Gr(2,5)^{\vee}(1,0)}\oplus\of(1,1)$ &\cite[\red{C-5}]{BFMT}&?\\

\cref{F.2.64}& $K_{473}$ & 36 & 146 & 2 & 0 & 0 & 14 & $\Gr(2,4)\times\Gr(2,5)$ & $\mQ_{\Gr(2,4)}\boxtimes\mU^{\vee}_{\Gr(2,5)}\oplus\mU^{\vee}_{\Gr(2,4)}(0,1)$ & $\Bl_{\PP^2}K_{433}$&?\\

\rowcolor{lavender}
\cref{F.2.65}& $K_{439}$ & 35 & 141 & 2 & 0 & 0 & 13 & $\PP^4 \times \Gr(2,6)$ & $\mQ_{\Gr(2,6)}(1,0)\oplus\of(0,1)^{\oplus 4}$ & $\Bl_{\DP_5}X_6^4$&+\\

\cref{F.2.66}& $K_{717}$ & 17 & 50 & 2 & 0 & 2 & 42 & $\PP^1 \times \Gr(2,5)$ & $\of(0,1) \oplus \of(0,2)  \oplus \of(1,1)$ &\cite[\red{A-69}]{BFMT}&?\\

\rowcolor{lavender}
\cref{F.2.67}& $K_{474}$ & 40 & 165 & 2 & 1 & 0 & 4 & $\PP^2\times\Gr(2,5)$ & $\mQ_{\PP^2}(0,1)\oplus\of(0,1)^{\oplus 2}$ & $\Bl_EX_5^4$&+\\

\cref{F.2.68}& $K_{483}$ & 36 & 146 & 2 & 0 & 0 & 10 & $\Gr(2,4)\times\Gr(2,6)$ & $\mQ_{\Gr(2,4)}\boxtimes\mU_{\Gr(2,6)}^{\vee}\oplus\of(1,0)\oplus\of(0,1)^{\oplus 3}$ & $\Bl_{Q_1}X_6^4$&+\\

\rowcolor{lavender}
\cref{F.2.69}& $K_{529}$ & 29 & 110 & 2 & 0 & 0 & 17 & $\PP^3 \times \Gr(2,5)$ & $\mQ_{\Gr(2,5)}(1,0) \oplus \mU^{\vee}_{\Gr(2,5)}(0,1)$ & $\Bl_{\DP_5}K_{433}$&?\\

\cref{F.2.70}& $K_{714}$ & 24 & 80 & 2 & 10 & 0 & 2 & $\PP^2\times\Gr(2,5)$ & $\of(2,0)\oplus\of(0,2)\oplus\of(0,1)^{\oplus 2}$ & $Q_1\times F_{1-5}$&?\\

\rowcolor{lavender}
\cref{F.2.71}& $K_{426}$ & 36 & 147 & 2 & 0 & 0 & 13 & $\Gr(2,4)\times\Gr(3,5)$ & $\mU^{\vee}_{\Gr(2,4)}\boxtimes\mQ_{\Gr(3,5)}\oplus\mU^{\vee}_{\Gr(2,4)}(0,1)$ & $\Bl_{\Bl_{10pts}}\Gr(2,4)$&+\\

\cref{F.2.72}& $K_{662}$ & 22 & 74 & 2 & 0 & 1 & 30 & $\PP^2\times\Gr(2,5)$ & $\mU^{\vee}(1,0)\oplus\of(0,1)\oplus\of(0,2)$ &\cite[\red{GM-18}]{BFMT}&?\\

\rowcolor{lavender}
\cref{F.2.73}& $K_{671}$ & 18 & 55 & 2 & 0 & 3 & 47 & $\PP^2\times\PP^6$ & $\mQ_{\PP^2}(0,1)\oplus\mQ_{\PP^2}(0,2)$ & Small res. &?\\

\cref{F.2.74}& $K_{562}$ & 27 & 99 & 2 & 0 & 1 & 25 & $\PP^2\times\Gr(2,6)$ & $\mQ_{\PP^2}\boxtimes\mU^{\vee}_{\Gr(2,6)}\oplus\mU^{\vee}_{\Gr(2,6)}(0,1)$ &\cite[\red{GM-20}]{BFMT}&?\\

\rowcolor{lavender}
\cref{F.3.1}& $K_{503}$ & 45 & 192 & 3 & 1 & 0 & 4 & $\PP^2\times\PP^3\times\Gr(2,4)$ & $\mQ_{\Gr(2,4)}(0,1,0)\oplus\of(0,1,1)\oplus\of(2,0,0)\oplus\of(0,0,1)$ & $Q_1\times F_{2-17}$&+\\

\cref{F.3.2}& $K_{20}$ & 76 & 358 & 3 & 0 & 0 & 5 & $\PP^2_1\times\PP^2_2\times\Gr(2,4)$ & $\mQ_{\Gr(2,4)}(1,0,0)\oplus\mQ_{\Gr(2,4)}(0,1,0)$ & $\Bl_{\DP_8}\Bl_{\PP^2}\Gr(2,4)$&+\\

\rowcolor{lavender}
\cref{F.3.3}& $K_{521}$ & 34 & 136 & 3 & 0 & 0 & 14 & $\PP^1\times\PP^3\times\Gr(2,4)$ & $\of(1,0,1)\oplus\of(0,1,1)\oplus\mQ_{\Gr(2,4)}(0,1,0)$ & $\Bl_{\DP_3}K_{124}$&+\\

\cref{F.3.4}& $K_{524}$ & 34 & 136 & 3 & 2 & 0 & 8 & $\PP^1\times\PP^3\times\Gr(2,4)$ & $\mQ_{\Gr(2,4)}(0,1,0)\oplus\of(0,0,2)\oplus\of(1,1,0)$ & $\Bl_{\DP_4}K_{238}$&+\\

\rowcolor{lavender}
\cref{F.3.5}& $K_{21}$ & 74 & 347 & 3 & 0 & 0 & 5 & $\PP^5\times\PP^4\times\Gr(2,7)$ & $\mQ_{\PP^4}\boxtimes\mU^{\vee}_{\Gr(2,7)}\oplus\mQ_{\Gr(2,7)}(1,0,0)\oplus\mU^{\vee}_{\Gr(2,7)}(0,1,0)$ & $\Bl_SK_{15}$&+\\

\cref{F.3.6}& $K_{526}$ & 30 & 116 & 3 & 0 & 0 & 17 & $\PP^2_1\times\PP^2_2\times\PP^3$ & $\mQ_{\PP^2_1}(0,1,1)\oplus\of(1,0,1)$ & $\Bl_{\Bl_{13pts}\PP^2}\PP^2\times\PP^2$ &+\\

\rowcolor{lavender}
\cref{F.3.7}& $K_{555}$ & 31 & 120 & 3 & 0 & 1 & 22 & $\PP^1\times\PP^3\times\Gr(2,4)$ & $\mQ_{\Gr(2,4)}(0,1,0)\oplus\of(0,0,1)\oplus\of(1,1,1)$ & \cite[\red{K3-48}]{BFMT}&+\\

\cref{F.3.8}& $K_{623}$ & 28 & 105 & 3 & 0 & 0 & 17 & $\PP^1\times\PP^1\times\Gr(2,5)$ & $\of(0,0,1)^{\oplus 2}\oplus\of(1,0,1)\oplus\of(0,1,1)$ & $\Bl_{\Bl_{9pts}\PP^2}K_{288}$&+\\

\rowcolor{lavender}
\cref{F.3.9}& $K_{294}$ & 54 & 241 & 3 & 0 & 0 & 6 & $\PP^2\times\Gr(2,5)\Gr(2,4)$ & $\mU^{\vee}_{\Gr(2,5)}(1,0,0)\oplus\mU^{\vee}_{\Gr(2,5)}\boxtimes\mQ_{\Gr(2,4)}\oplus\of(0,1,0)\oplus\of(0,0,1)$ & $\Bl_{\DP_7}K_{101}$&+\\

\cref{F.3.10}& $K_{637}$ & 27 & 100 & 3 & 0 & 1 & 22 & $\PP^1\times\PP^1\times\Gr(2,5)$ & $\of(0,0,1)^{\oplus 3}\oplus\of(1,1,1)$ &\cite[\red{K3-38}]{BFMT}&+\\

\rowcolor{lavender}
\cref{F.3.11}& $K_{158}$ & 60 & 274 & 3 & 0 & 0 & 8 & $\Gr(2,5)\times\Gr(2,6)$ & $\mQ_{\Gr(2,5)}\boxtimes\mU^{\vee}_{\Gr(2,6)}\oplus\mU^{\vee}_{\Gr(2,5)}\boxtimes\mU^{\vee}_{\Gr(2,6)}$ & Small res. &?\\

\cref{F.3.12}& $K_{398}$ & 39 & 164 & 3 & 0 & 0 & 12 & $\PP^2\times\Gr(2,7)$ & $\mQ_{\PP^2}\boxtimes\mU_{\Gr(2,7)}^{\vee}\oplus\of(0,1)\oplus\Sym^2\mU^{\vee}_{\Gr(2,7)}$ & Small res.&?\\

\rowcolor{lavender}
\cref{F.3.13}& $K_{431}$ & 36 & 148 & 3 & 0 & 0 & 12 & $\PP^2\times\Fl(1,3,4)$ & $\mR_1(1,0,0)\oplus\of(0,1,1)$ & $\Bl_{\DP_4}Y$&+\\

\cref{F.3.14}& $K_{76}$ & 72 & 337 & 3 & 0 & 0 & 5 & $\PP^4\times\Gr(2,7)\times\Gr(2,4)$ & $\mQ_{\PP^4}\boxtimes\mU_{\Gr(2,7)}\oplus\mU_{\Gr(2,7)}^{\vee}\boxtimes\mU^{\vee}_{\Gr(2,4)}\oplus\mQ_{\Gr(2,4)}(1,0,0)$ & Small res.&+\\

\rowcolor{lavender}
\cref{F.3.15}& $K_{116}$ & 57 & 257 & 3 & 0 & 0 & 7 & $\Gr(2,4)\times\Gr(2,5)\times\PP^2$ & $\mQ_{\Gr(2,4)}(0,0,1)\oplus\mQ_{\Gr(2,4)}\boxtimes\mU^{\vee}_{\Gr(2,5)}\oplus\of(0,1,0)^{\oplus 2}$ & $\Bl_{\DP_7}K_{25}$&+\\

\cref{F.3.16}& $K_{125}$ & 48 & 211 & 3 & 0 & 0 & 7 & $\PP^2\times\PP^4\times\Gr(2,4)$ & $\mQ_{\Gr(2,4)}(1,0,0)\oplus\mQ_{\Gr(2,4)}(0,1,0)^{\oplus 2}$ & $\Bl_{\DP_6}K_{29}$&+\\

\rowcolor{lavender}
\cref{F.3.17}& $K_{48}$ & 90 & 433 & 3 & 0 & 0 & 5 & $\PP^2\times\PP^4\times\Gr(2,11)$ & $\mQ_{\PP^2}\boxtimes\mU^{\vee}_{\Gr(2,11)}\oplus(\mQ_{\PP^4}\boxtimes\mU^{\vee}_{\Gr(2,11)})^{\oplus 2}$ & Small res. &?\\

\cref{F.3.18}& $K_{188}$ & 60 & 272 & 3 & 0 & 0 & 4 & $\PP^1\times\PP^3\times\Gr(2,4)$ & $\of(1,0,1)\oplus\mQ_{\Gr(2,4)}(0,1,0)\oplus\of(0,0,1)$ & $\Bl_{Q_1}K_{136}$&+\\

\rowcolor{lavender}
\cref{F.3.19}& $K_{190}$ & 57 & 256 & 3 & 0 & 0 & 6 & $\PP^4\times\PP^2\times\Gr(2,6)$ & $\mQ_{\PP^2}\boxtimes\mU^{\vee}_{\Gr(2,6)}\oplus\mQ_{\Gr(2,6)}(1,0,0)\oplus\of(0,0,1)^{\oplus 2}$ & $\Bl_{\Bl_{pt}\PP^2}K_{142}$&?\\

\cref{F.3.20}& $K_{191}$ & 54 & 241 & 3 & 0 & 0 & 9 & $\PP^2_1\times\PP^2_2\times\Gr(2,7)$ & $\mQ_{\PP^2_1}\boxtimes\mU_{\Gr(2,7)}^{\vee}\oplus\mQ_{\PP^2_2}\boxtimes\mU_{\Gr(2,7)}^{\vee}\oplus\of(0,0,1)^{\oplus 2}$ & Small res. &?\\

\rowcolor{lavender}
\cref{F.3.21}& $K_{173}$ & 65 & 298 & 3 & 0 & 0 & 5 & $\PP^2\times\Gr(2,6)\times\Gr(2,5)$ & $\mQ_{\PP^2}\boxtimes\mU^{\vee}_{\Gr(2,5)}\oplus\mU^{\vee}_{\Gr(2,6)}\boxtimes\mQ_{\Gr(2,5)}\oplus\of(0,1,0)^{\oplus 2}$ & $\Bl_{\PP^2}K_{142}$&?\\

\cref{F.3.22}& $K_{178}$ & 66 & 304 & 3 & 0 & 0 & 4 & $\PP^1\times\PP^3\times\Gr(2,4)$ & $\mQ_{\Gr(2,4)}(0,1,0)\oplus\of(0,0,1)\oplus\of(1,1,0)$ &  $\Bl_{Q_2}K_{136}$&+\\

\rowcolor{lavender}
\cref{F.3.23}& $K_{211}$ & 54 & 241 & 3 & 0 & 0 & 7 & $\PP^3\times\PP^3\times\Gr(2,5)$ & $\mQ_{\Gr(2,5)}(1,0,0)\oplus\mQ_{\Gr(2,5)}(0,1,0)\oplus\of(0,0,1)^{\oplus 2}$ & $\Bl_{\DP_7}K_{102}$&+\\

\cref{F.3.24}& $K_{212}$ & 53 & 235 & 3 & 0 & 0 & 7 & $\PP^5\times\PP^2\times\Gr(2,5)$ & $\mQ_{\PP^2}\boxtimes\mU^{\vee}_{\Gr(2,5)}\oplus\mQ_{\Gr(2,5)}(1,0,0)\oplus\mU^{\vee}_{\Gr(2,5)}(1,0,0)$ & $\Bl_{\DP_6}K_6$&+\\

\rowcolor{lavender}
\cref{F.3.25}& $K_{202}$ & 59 & 267 & 3 & 0 & 0 & 5 & $\PP^3\times\Gr(2,4)\times\Gr(2,5)$ & $\mQ_{\Gr(2,4)}\boxtimes\mU^{\vee}_{\Gr(2,5)}\oplus\mQ_{\Gr(2,5)}(1,0,0)\oplus\of(0,1,0)\oplus\of(0,0,1)$ & $\Bl_{Q_2}K_{101}$&+\\

\cref{F.3.26}& $K_{232}$ & 52 & 230 & 3 & 0 & 0 & 6 & $\PP^3\times\PP^5\times\Gr(2,4)$ & $\mQ_{\PP^3}(0,1,0)\oplus\mQ_{\Gr(2,4)}(1,0,0)\oplus\mU^{\vee}_{\Gr(2,4)}(0,1,0)\oplus\of(0,0,1)$ & $\Bl_{\DP_6}K_{136}$&+\\

\rowcolor{lavender}
\cref{F.3.27}& $K_{231}$ & 51 & 225 & 3 & 0 & 0 & 9 & $\PP^1\times\PP^4\times\Gr(2,7)$ & $\mQ_{\PP^4}\boxtimes\mU^{\vee}_{\Gr(2,7)}\oplus\mU^{\vee}_{\Gr(2,7)}(0,1,0)\oplus\of(1,0,1)$ & $\Bl_{\DP_4}K_{15}$&+\\

\cref{F.3.28}& $K_{235}$ & 46 & 199 & 3 & 0 & 0 & 8 & $\PP^2\times\PP^4\times\Gr(2,4)$ & $\mQ_{\PP^2}(0,1,0)\oplus\mQ_{\PP^2}(0,0,1)\oplus\mQ_{\Gr(2,4)}(0,1,0)$ & Small res. &+\\

\rowcolor{lavender}
\cref{F.3.29}& $K_{226}$ & 51 & 225 & 3 & 0 & 0 & 10 & $\PP^2\times\PP^4\times\Gr(2,5)$ & $\mQ_{\PP^2}\boxtimes\mU^{\vee}_{\Gr(2,5)}\oplus\mU_{\Gr(2,5)}(0,1,0)^{\oplus 2}$ & $\Bl_{\DP_3}K_6$&+\\

\cref{F.3.30}& $K_{224}$ & 56 & 251 & 3 & 0 & 0 & 6 & $\PP^2\times\Gr(2,4)\times\Gr(2,5)$ & $\mQ_{\Gr(2,4)}\boxtimes\mU^{\vee}_{\Gr(2,5)}\oplus\mU^{\vee}_{\Gr(2,4)}(1,0,0)\oplus\of(0,0,1)^{\oplus 2}$ & $\Bl_{\Bl_{pt}\PP^2}K_{25}$&+\\

\rowcolor{lavender}
\cref{F.3.31}& $K_{242}$ & 51 & 224 & 3 & 0 & 0 & 8 & $\PP^3\times\PP^4\times\PP^5$ & $\mQ_{\PP^3}^{\vee}(1,1,0)\oplus\mQ_{\PP^4}(0,0,1)\oplus\of(1,0,1)$ & $\Bl_{\DP_4}Y$ &+\\

\cref{F.3.32}& $K_{244}$ & 47 & 203 & 3 & 0 & 0 & 9 & $\PP^3\times\PP^4\times\Gr(2,4)$ & $\mQ_{\PP^3}(0,1,0)\oplus\mQ_{\Gr(2,4)}(1,0,0)\oplus\of(0,1,1)\oplus\of(0,0,1)$ & $\Bl_{\DP_3}K_{136}$&+\\

\rowcolor{lavender}
\cref{F.3.33}& $K_{450}$ & 51 & 224 & 3 & 0 & 0 & 4 & $\PP^1\times\PP^4\times\Gr(2,4)$ & $\mQ_{\Gr(2,4)}(0,1,0)^{\oplus 2}\oplus\of(0,0,1)$ & $\PP^1\times F_{2-21}$&+\\

\cref{F.3.34}& $K_{145}$ & 75 & 353 & 3 & 0 & 0 & 6 & $\PP^4\times\Gr(2,7)\times\Gr(2,4)$ & $\mQ_{\PP^4}\boxtimes\mU^{\vee}_{\Gr(2,7)}\oplus\mQ_{\Gr(2,4)}\boxtimes\mU^{\vee}_{\Gr(2,7)}\oplus\mQ_{\Gr(2,4)}(1,0,0)$ & Small res.&?\\

\rowcolor{lavender}
\cref{F.3.35}& $K_{295}$ & 55 & 244 & 3 & 0 & 0 & 5 & $\PP^6\times\PP^2\times\Gr(2,5)$ & $\mQ_{\PP^2}\boxtimes\mU^{\vee}_{\Gr(2,5)}\oplus\mQ_{\Gr(2,5)}(1,0,0)^{\oplus 2}$ & $\Bl_{\DP_8}K_6$&+\\

\cref{F.3.36}& $K_{320}$ & 45 & 193 & 3 & 0 & 0 & 12 & $\PP^2\times\PP^3\times\Gr(2,5)$ & $\mQ_{\PP^2}\boxtimes\mU^{\vee}_{\Gr(2,5)}\oplus\mU^{\vee}_{\Gr(2,5)}(0,1,0)\oplus\of(0,1,1)$ & $\Bl_{\DP_1}K_6$&+\\

\rowcolor{lavender}
\cref{F.3.37}& $K_{321}$ & 47 & 202 & 3 & 0 & 0 & 10 & $\PP^2\times\PP^4\times\Gr(2,5)$ & $\mQ_{\PP^2}\boxtimes\mU^{\vee}_{\Gr(2,5)}\oplus\mQ_{\Gr(2,5)}(0,1,0)\oplus\of(0,1,1)$ & $\Bl_{\DP_3}K_6$&+\\

\cref{F.3.38}& $K_{322}$ & 46 & 198 & 3 & 1 & 0 & 5 & $\PP^2_1\times\PP^2_2\times\PP^2_3$ & $\mQ_{\PP^2_1}(0,1,1)$ & $\Bl_E\PP^2_2\times\PP^2_3$&+\\

\rowcolor{lavender}
\cref{F.3.39}& $K_{341}$ & 42 & 178 & 3 & 0 & 0 & 11 & $\PP^2\times\PP^3\times\Gr(2,4)$ & $\mQ_{\Gr(2,4)}(0,1,0)\oplus\of(0,1,1)\oplus\mQ_{\Gr(2,4)}(1,0,0)$ & $\Bl_{\DP_6}K_{124}$&+\\

\cref{F.3.40}& $K_{342}$ & 44 & 188 & 3 & 0 & 0 & 8 & $\PP^3_1\times\PP^3_2\times\Gr(2,4)$ & $\mQ_{\Gr(2,4)}(1,0,0)\oplus\mQ_{\Gr(2,4)}(0,1,0)\oplus\of(1,1,0)\oplus\of(0,0,1)$ & $\Bl_{\DP_4}K_{136}$&+\\

\rowcolor{lavender}
\cref{F.3.41}& $K_{343}$ & 43 & 183 & 3 & 0 & 0 & 9 & $\PP^2\times\PP^4\times\Gr(2,4)$ & $\mQ_{\PP^2}(0,1,0)\oplus\mQ_{\Gr(2,4)}(0,1,0)\oplus\of(1,0,1)\oplus\of(0,0,1)$ & $\Bl_{Q_2}K_{126}$&+\\

\cref{F.3.42}& $K_{346}$ & 36 & 150 & 3 & 0 & 0 & 10 & $\PP^2_1\times\PP^2_2\times\PP^4$ & $\mQ_{\PP^2_1}\boxtimes\mQ_{\PP^2_2}\boxtimes\of_{\PP^4}(1)$ & $\Bl_{\DP_3}\PP^2\times\PP^2$&+\\

\rowcolor{lavender}
\cref{F.3.43}& $K_{265}$ & 60 & 273 & 3 & 0 & 0 & 6 & $\PP^2\times\Gr(2,4)\times\Gr(2,6)$ & $\mQ_{\PP^2}\boxtimes\mU^{\vee}_{\Gr(2,6)}\oplus\mQ_{\Gr(2,4)}\boxtimes\mU^{\vee}_{\Gr(2,6)}\oplus\of(0,1,0)\oplus\of(0,0,1)$ & Small res. &?\\

\cref{F.3.44}& $K_{372}$ & 60 & 272 & 3 & 0 & 0 & 4 & $\PP^1\times\Gr(2,4)\times\Gr(2,5)$ & $\mQ_{\Gr(2,4)}\boxtimes\mU^{\vee}_{\Gr(2,5)}\oplus\of(0,1,0)\oplus\of(0,0,1)^{\oplus 2}$ & $\PP^1\times F_{2-26}$&+\\

\rowcolor{lavender}
\cref{F.3.45}& $K_{388}$ & 54 & 240 & 3 & 0 & 0 & 4 & $\PP^1\times\PP^3\times\Gr(2,5)$ & $\mQ_{\Gr(2,5)}(0,1,0)\oplus\of(0,0,1)^{\oplus 3}$ & $\PP^1\times F_{2-22}$&+\\

\cref{F.3.46}& $K_{421}$ & 43 & 183 & 3 & 0 & 0 & 7 & $\PP^1\times\PP^4\times\Gr(2,5)$ & $\mQ_{\Gr(2,5)}(0,1,0)\oplus\of(1,1,0)\oplus\of(0,0,1)^{\oplus 3}$ & $\Bl_{\DP_5}\Bl_{\nu_2(\PP^2)}\PP^4$&+\\

\rowcolor{lavender}
\cref{F.3.47}& $K_{422}$ & 41 & 172 & 3 & 0 & 0 & 10 & $\PP^2\times\PP^3\times\Gr(2,4)$ & $\mQ_{\Gr(2,4)}(0,1,0)\oplus\of(0,1,1)\oplus\mQ_{\Gr(2,4)}(1,0,0)$ & $\Bl_{\DP_7}K_{124}$&+\\

\cref{F.3.48}& $K_{17}$ & 85 & 406 & 3 & 0 & 0 & 5 & $\PP_1^2\times\PP_2^2\times\Gr(2,5)$ & $\mQ_{\PP^2_1}\boxtimes\mU^{\vee}_{\Gr(2,5)}\oplus\mU^{\vee}_{\Gr(2,5)}(0,1,0)$ & $\Bl_{\DP_8}K_6$&+\\

\rowcolor{lavender}
\cref{F.3.49}& $K_{472}$ & 39 & 162 & 3 & 0 & 0 & 10 & $\PP^1\times\PP^5\times\Gr(2,4)$ & $\mQ_{\Gr(2,4)}(0,1,0)^{\oplus 2}\oplus\of(1,1,0)\oplus\of(0,0,1)$ & $\Bl_{\DP_2}\PP^1\times Q_3$&+\\

\cref{F.3.50}& $K_{512}$ & 37 & 152 & 3 & 0 & 0 & 8 & $\PP^1\times\PP^4\times\Gr(2,4)$ & $\mQ_{\Gr(2,4)}(0,1,0)\oplus\of(1,0,1)\oplus\of(0,2,0)\oplus\of(0,0,1)$ & $\PP_{\Bl_{Q_1}Q_3}(\mU_{\Gr(2,4)|Q_3}\oplus\of_{Q_3})$ &?\\

\rowcolor{lavender}
\cref{F.3.51}& $K_{516}$ & 35 & 141 & 3 & 0 & 0 & 12 & $\PP^1\times\PP^2\times\Gr(2,5)$ & $\mU^{\vee}_{\Gr(2,5)}(0,1,0)\oplus\of(1,0,1)\oplus\of(0,0,1)^{\oplus 2}$ & $\Bl_{\DP_2}K_{195}$&+\\

\cref{F.3.52}& $K_{519}$ & 35 & 141 & 3 & 0 & 0 & 13 & $\PP^1\times\PP^4\times\Gr(2,4)$ & $\mQ_{\Gr(2,4)}(0,1,0)\oplus\of(0,0,1)\oplus\of(0,1,1)\oplus\of(1,1,0)$ & $\Bl_{\DP_4}K_{126}$&+\\

\rowcolor{lavender}
\cref{F.3.53}& $K_{593}$ & 33 & 130 & 3 & 0 & 0 & 12 & $\PP^{2}\times\PP^1\times\Gr(2,5)$ & $\of(1,1,0)\oplus\of(1,0,1)\oplus\of(0,0,1)^{\oplus 3}$ & $\Bl_{\DP_5}K_{508}$&+\\

\cref{F.3.54}& $K_{591}$ & 39 & 160 & 3 & 1 & 0 & 4 & $\PP^1_1\times\PP^1_2\times\Gr(2,5)$ & $\of(0,1,1)\oplus\of(0,0,1)^{\oplus 3}$ & $\PP^1\times F_{2-14}$&+\\

\rowcolor{lavender}
\cref{F.3.55}& $K_{668}$ & 25 & 89 & 3 & 5 & 0 & 9 & $\PP^1\times\PP^2\times\PP^4$ & $\of(1,1,0)\oplus\mQ_{\PP^2}(0,0,2)$ & $\Bl_{\DP_4}K_{612}$&+\\

\cref{F.3.56}& $K_{182}$ & 60 & 273 & 3 & 0 & 0 & 7 & $\PP^2\times\PP^4\times\Gr(2,7)$ & $\mQ_{\PP^4}\boxtimes\mU^{\vee}_{\Gr(2,7)}\oplus\mU_{\Gr(2,7)}^{\vee}(1,0,0)\oplus\mU_{\Gr(2,7)}^{\vee}(0,1,0)$ & $\Bl_{\DP_6}K_{15}$&+\\

\rowcolor{lavender}
\cref{F.3.57}& $K_{282}$ & 51 & 225 & 3 & 0 & 0 & 11 & $\PP^2\times\PP^3\times\Gr(2,6)$ & $\mQ_{\PP^3}\boxtimes\mU^{\vee}_{\Gr(2,6)}\oplus\mU^{\vee}_{\Gr(2,6)}(1,0,0)\oplus\of(0,1,1)$ & $\Bl_{\DP_5}K_{70}$&+\\

\cref{F.3.58}& $K_{303}$ & 48 & 209 & 3 & 0 & 0 & 8 & $\PP^2_1\times\PP^2_2\times\Gr(2,6)$ & $\mQ_{\PP^2_2}\boxtimes\mU^{\vee}_{\Gr(2,6)}\oplus\of(0,0,1)^{\oplus 2}\oplus\mU^{\vee}_{\Gr(2,6)}(1,0,0)$ & $\Bl_{\DP_6}K_{142}$&?\\

\rowcolor{lavender}
\cref{F.3.59}& $K_{460}$ & 48 & 209 & 3 & 0 & 0 & 8 & $\PP^2_1\times\PP^2_2\times\Gr(2,5)$ & $\mU^{\vee}_{\Gr(2,5)}(0,1,0)\oplus\of(2,0,0)\oplus\of(0,0,1)^{\oplus 3}$ & $Q_1\times F_{2-20}$&+\\

\cref{F.3.60}& $K_{318}$ & 48 & 209 & 3 & 0 & 0 & 7 & $\PP^2\times\PP^3\times\Gr(2,5)$ & $\mU^{\vee}_{\Gr(2,5)}(1,0,0)\oplus\mQ_{\Gr(2,5)}(0,1,0)\oplus\of(0,0,1)^{\oplus 2}$ & $\Bl_{\DP_7}K_{102}$&+\\

\rowcolor{lavender}
\cref{F.3.61}& $K_{326}$ & 48 & 209 & 3 & 0 & 0 & 9 & $\PP^1\times\Gr(2,4)\times\Gr(2,5)$ & $\mQ_{\Gr(2,4)}\boxtimes\mU^{\vee}_{\Gr(2,5)}\oplus\of(1,1,0)\oplus\of(0,0,1)^{\oplus 2}$ & $\Bl_{\DP_5}K_{25}$&+\\

\cref{F.3.62}& $K_{327}$ & 47 & 204 & 3 & 0 & 0 & 7 & $\PP^2\times\PP^4\times\Gr(2,4)$ & $\mQ_{\Gr(2,4)}(0,1,0)\oplus\mU^{\vee}_{\Gr(2,4)}(0,1,0)\oplus\mU^{\vee}_{\Gr(2,4)}(1,0,0)$ & $\Bl_{\DP_7}K_{115}$&+\\

\rowcolor{lavender}
\cref{F.3.63}& $K_{339}$ & 46 & 199 & 3 & 0 & 0 & 5 & $\PP^4\times\PP^2\times\Gr(2,4)$ & $\mQ_{\Gr(2,4)}(1,0,0)^{\oplus 2}\oplus\mU_{\Gr(2,4)}^{\vee}(0,1,0)$ & $\Bl_{\DP_8}K_{29}$&+\\

\cref{F.3.64}& $K_{344}$ & 42 & 178 & 3 & 0 & 0 & 9 & $\PP^3\times\Gr(2,4)\times\PP^2$ & $\of(1,0,1)\oplus\mQ_{\PP^2}(0,1,0)\oplus\mU^{\vee}_{\Gr(2,4)}(1,0,0)$ & $\Bl_{\DP_5}\Bl_{Q_1}\Gr(2,4)$&+\\

\rowcolor{lavender}
\cref{F.3.65}& $K_{413}$ & 44 & 188 & 3 & 0 & 0 & 6 & $\PP^2\times\PP^4\times\Gr(2,4)$ & $\mU^{\vee}_{\Gr(2,4)}(1,0,0)\oplus\mQ_{\Gr(2,4)}(0,1,0)\oplus\of(0,0,1)\oplus\of(0,2,0)$ & $\Bl_{\DP_6}K_{230}$&?\\

\cref{F.3.66}& $K_{423}$ & 40 & 167 & 3 & 0 & 0 & 11 & $\PP^1\times\PP^3\times\Gr(2,5)$ & $\mQ_{\Gr(2,5)}(0,1,0)\oplus\of(1,0,1)\oplus\of(0,0,1)^{\oplus 2}$ & $\Bl_{\DP_3}K_{102}$&+\\

\rowcolor{lavender}
\cref{F.3.67}& $K_{424}$ & 39 & 161 & 3 & 0 & 0 & 12 & $\PP^1\times\PP^2\times\Gr(2,6)$ & $\of(1,0,1)\oplus\mQ_{\PP^2}\boxtimes\mU^{\vee}_{\Gr(2,6)}\oplus\of(0,0,1)^{\oplus 2}$ & $\Bl_{\DP_2}K_{385}$&+\\

\cref{F.3.68}& $K_{428}$ & 39 & 163 & 3 & 0 & 0 & 9 & $\PP^1\times\PP^4\times\Gr(2,4)$ & $\of(1,0,1)\oplus\mQ_{\Gr(2,4)}(0,1,0)^{\oplus 2}$ & $\Bl_{\DP_4}K_{29}$&+\\

\rowcolor{lavender}
\cref{F.3.69}& $K_{430}$ & 39 & 162 & 3 & 0 & 0 & 10 & $\PP^2\times\PP^3\times\Gr(2,4)$ & $\mQ_{\Gr(2,4)}(0,1,0)\oplus\of(1,0,1)\oplus\of(1,1,0)\oplus\of(0,0,1)$ & $\Bl_{\DP_8}K_{136}$&+\\

\cref{F.3.70}& $K_{438}$ & 33 & 131 & 3 & 0 & 0 & 17 & $\PP^2_1\times\PP^2_2\times\PP^4$ & $\mQ_{\PP^2_1}(0,0,1)\oplus\mQ_{\PP^2_1}(0,1,1)$ & $\Bl_{\Bl_{13pts}\PP^2}\PP^2\times\PP^2$&+\\

\rowcolor{lavender}
\cref{F.3.71}& $K_{478}$ & 39 & 162 & 3 & 0 & 0 & 8 & $\PP^1\times\PP^3\times\Gr(2,5)$ & $\mU^{\vee}_{\Gr(2,5)}(0,1,0)\oplus\of(0,0,1)^{\oplus 3}\oplus\of(1,1,0)$ & $\Bl_{\DP_5}K_{219}$&?\\

\cref{F.3.72}& $K_{154}$ & 69 & 320 & 3 & 0 & 0 & 4 & $\PP^2\times\PP^3\times\Gr(2,4)$ & $\mQ_{\Gr(2,4)}(1,0,0)\oplus\mU^{\vee}_{\Gr(2,4)}(0,1,0)\oplus\of(0,0,1)$ & $\Bl_{\DP_8}K_{136}$&+\\

\rowcolor{lavender}
\cref{F.3.73}& $K_{404}$ & 45 & 193 & 3 & 0 & 0 & 9 & $\PP^1\times\Gr(2,4)\times\Gr(2,5)$ & $\mQ_{\Gr(2,4)}\boxtimes\mU^{\vee}_{\Gr(2,5)}\oplus\of(0,1,0)\oplus\of(0,0,1)\oplus\of(1,0,1)$ & $\Bl_{\DP_4}K_{101}$&+\\

\cref{F.3.74}& $K_{406}$ & 43 & 183 & 3 & 0 & 0 & 8 & $\PP_1^2\times\PP_2^2\times\Gr(2,5)$ & $\mU^{\vee}_{\Gr(2,5)}(1,0,0)\oplus\mU^{\vee}_{\Gr(2,5)}(0,1,0)\oplus\of(0,0,1)^{\oplus 2}$ & $\Bl_{\DP_6}K_{195}$&+\\

\rowcolor{lavender}
\cref{F.4.1}& $K_{280}$ & 51 & 228 & 4 & 0 & 0 & 10 & $\Gr(2,4)\times\Gr(2,6)$ & $\mQ_{\Gr(2,4)}\boxtimes\mU^{\vee}_{\Gr(2,6)}\oplus\of(1,0)\oplus\Sym^2\mU^{\vee}_{\Gr(2,6)}$ & $\Bl_{\PP^1\bigsqcup\PP^1}Y$&?\\

\cref{F.4.2}& $K_{26}$ & 63 & 290 & 4 & 0 & 0 & 8 & $\PP^2\times\PP^2\times\PP^2\times\Gr(2,4)$ & $\mQ_{\Gr(2,4)}(1,0,0,0)\oplus\mQ_{\Gr(2,4)}(0,1,0,0)\oplus\mQ_{\Gr(2,4)}(0,0,1,0)$ & $\Bl_{\DP_7}K_{20}$&+\\

\rowcolor{lavender}
\cref{F.4.3}& $K_{73}$ & 78 & 369 & 4 & 0 & 0 & 7 & $\PP^2_1\times\PP^2_2\times\Gr(2,4)\times\Gr(2,5)$ & $\mU^{\vee}_{\Gr(2,4)}(1,0,0,0)\oplus\mQ_{\PP^2_2}\boxtimes\mU^{\vee}_{\Gr(2,5)}\oplus\mQ_{\Gr(2,4)}\boxtimes\mU^{\vee}_{\Gr(2,5)}$ & $\Bl_{\DP_8}\Bl_{\PP^1} K_6$&+\\

\cref{F.4.4}& $K_{144}$ & 78 & 369 & 4 & 0 & 0 & 7 & $\PP^2_1\times\PP^2_2\times\Gr(2,4)\times\Gr(2,5)$ & $\mQ_{\Gr(2,4)}(1,0,0,0)\oplus\mQ_{\PP^2_2}\boxtimes\mU^{\vee}_{\Gr(2,5)}\oplus\mQ_{\Gr(2,4)}\boxtimes\mU^{\vee}_{\Gr(2,5)}$ & $\Bl_{Q_2}\Bl_{\PP^1} K_6$&+\\
\rowcolor{lavender}
\cref{F.4.5}& $K_{256}$ & 81 & 385 & 4 & 0 & 0 & 7 & $\PP^2_1\times\PP^2_2\times\PP^2_3\times\Gr(2,7)$ & $\mQ_{\PP^2_1}\boxtimes\mU^{\vee}_{\Gr(2,7)}\oplus \mQ_{\PP^2_2}\boxtimes\mU^{\vee}_{\Gr(2,7)}\oplus \mQ_{\PP^2_3}\boxtimes\mU^{\vee}_{\Gr(2,7)}$ & Small res.&?\\

\cref{F.4.6}& $K_{78}$ & 71 & 331 & 4 & 0 & 0 & 7 & $\PP^3\times\PP^2_1\times\PP^2_2\times\Gr(2,5)$ & $\mQ_{\PP^2_2}\boxtimes\mU^{\vee}_{\Gr(2,5)}\oplus\mQ_{\Gr(2,5)}(1,0,0,0)\oplus\mU^{\vee}_{\Gr(2,5)}(0,1,0,0)$ & $\Bl_{\DP_8}K_{17}$&+\\

\rowcolor{lavender}
\cref{F.4.7}& $K_{104}$ & 66 & 305 & 4 & 0 & 0 & 8 & $\PP^2_1\times\PP^2_2\times\PP^2_3\times\Gr(2,5)$ & $\mQ_{\PP^2_3}\boxtimes\mU^{\vee}_{\Gr(2,5)}\oplus\mU^{\vee}_{\Gr(2,5)}(1,0,0)\oplus\mU^{\vee}_{\Gr(2,5)}(0,1,0)$ & $\Bl_{\DP_7}K_{17}$&+\\

\cref{F.4.8}& $K_{201}$ & 61 & 278 & 4 & 0 & 0 & 6 & $\PP^2_1\times\PP^2_2\times\PP^2_3\times\Gr(2,4)$ & $\mQ_{\Gr(2,4)}(1,0,0,0)\oplus\mQ_{\Gr(2,4)}(0,1,0,0)\oplus\mU^{\vee}_{\Gr(2,4)}(0,0,1,0)$ & $\Bl_{\PP^2}K_{20}$&+\\

\rowcolor{lavender}
\cref{F.4.9}& $K_{204}$ & 57 & 257 & 4 & 0 & 0 & 9 & $\PP^1\times\PP^2_1\times\PP^2_2\times\Gr(2,5)$ & $\mQ_{\PP^2_2}\boxtimes\mU_{\Gr(2,5)}^{\vee}\oplus\mU^{\vee}_{\Gr(2,5)}(0,1,0,0)\oplus\of(1,0,0,1)$ & $\Bl_{\DP_6}K_{17}$&+\\

\cref{F.4.10}& $K_{227}$ & 53 & 236 & 4 & 0 & 0 & 9 & $\PP^1\times\PP^2_1\times\PP^2_2\times\Gr(2,4)$ & $\mQ_{\Gr(2,4)}(0,0,1,0)\oplus\mQ_{\Gr(2,4)}(0,1,0,0)\oplus\of(1,0,0,1)$ & $\Bl_{\DP_6}K_{20}$&+\\

\rowcolor{lavender}
\cref{F.4.11}& $K_{297}$ & 55 & 246 & 4 & 0 & 0 & 7 & $\PP^1\times\PP^2\times\PP^3\times\Gr(2,4)$ & $\mQ_{\Gr(2,4)}(0,0,1,0)\oplus\mQ_{\Gr(2,4)}(0,1,0,0)\oplus\of(1,0,1,0)\oplus\of(0,0,0,1)$ & $\Bl_{\DP_7}K_{154}$&+\\

\cref{F.4.12}& $K_{403}$ & 47 & 204 & 4 & 0 & 0 & 8 & $\PP^1_1\times\PP^1_2\times\PP^3\times\Gr(2,4)$ & $\mQ_{\Gr(2,4)}(0,0,1,0)\oplus\of(1,0,1,0)\oplus\of(0,1,0,1)\oplus\of(0,0,0,1)$ & $\Bl_{\DP_6}K_{188}$&+\\

\rowcolor{lavender}
\cref{F.4.13}& $K_{482}$ & 39 & 162 & 4 & 1 & 0 & 9 & $\PP^1\times\PP^2_1\times\PP^2_2\times\PP^2_3$ & $\mQ_{\PP^2_1}(0,0,1,1)\oplus\of(1,1,0,0)$ & $\Bl_{\DP_6}K_{322}$&+\\

\end{longtable}

\end{scriptsize}
\end{centering}
\end{landscape}

\appendix
\renewcommand{\thesection}{A} 
\section{From zero loci of quiver flag varieties to zero loci of products of flag manifolds}
\label[app]{sec:ZlqToZlg}
\begin{center} \small\textsc{E.\ Kalashnikov, F.\ Tufo} \end{center}

\vspace{1ex}
In this appendix, we describe the method to translate quiver flag zero loci into zero loci of generic global sections of irreducible globally generated homogeneous vector bundles in the product of Grassmannians.\\

Let us briefly recall what a quiver flag variety, for more details we refer to \cite{CRAW11}.\\
To define such a variety we have to fix some notation.
\begin{itemize}
   \item Let $Q$ be a finite, acyclic, with only one source, namely 0, quiver. 
   \item Let $Q_0=\{0,1,2,...,\rho\}$ be its set of vertices and $Q_1=\{a_{i,j}|i,j\in Q_0\}$ its set of arrows from the vertex $i$ to the vertex $j$. 
    be the tail and head map such that $t(a_{i,j})=i$ and $h(a_{i,j})=j$. 
    \item Let us call $r=[1,r_1,...,r_{\rho}]\in\mathbb{Z}^{\rho+1}$ the dimension vector.
    \item Let $W$ be a representations of the quiver $Q$, i.e.,\ a collection $((W_i)_{i\in Q_0},(\omega_a)_{a\in Q_1})$, with $W_i$ a $\mathbb{C}$-vector space of dimension $r_i$ and $\omega_{a}$ a linear map between $W_{t(a)}$ and $W_{h(a)}$. 
    \item Let $\theta\in \mathbb{Q}^{\rho+1}$ be the  special weight $[-\sum_{i=1}^{\rho}r_i,1,...,1]$. Let us define $\theta(W)$ as $\sum_{i\in Q_0}\theta_ir_i$.
    \item We say that a representation $W$ is $\theta$-semistable if $\theta(W)=0$ and for each subrepresentation $W'\subset W$, $\theta(W')\geq 0$. It is stable if it is semistable and $\theta(W')= 0$ if and only if $W'$ is $0$ or $W$.
\end{itemize}

Now we can finally define the quiver flag varieties.

\begin{definition}
    A quiver flag variety $\mathcal{M}_{\theta}(Q,r)$ is the moduli space of the $\theta$-stable representation of a quiver $Q$, with $Q$, $\theta$ and $r$ as above.
\end{definition}

The advantage of writing these manifolds in the product of Grassmannians is computational. Suppose we work with zero loci of generic global sections of irreducible globally generated homogeneous vector bundles in the product of Grassmannians. In that case, we have all the tools described in \cref{workingTool} hence we can compute their Hodge numbers and we can even give a biregular description.\\
In \cite{kalashnikov} they considered varieties obtained as zero loci of direct sums of tensor products between nef line bundles and Schur powers of the dual of the universal bundles that come along with the quiver flag varieties themselves, hence by \cite[Theorem 3.3]{CRAW11} they are towers of Grassmannians bundles cut by direct sums of relative tautological bundles. By iterating \cref{thm:seq} we can write those towers as products of Grassmannians cut by direct sums of box products between tautological and quotient bundles and Schur powers of tautological bundles, which is the same language used in \cite{BFMT}, or in \cite{DFT}.\\
Let us give a couple of examples of how it is possible to translate a quiver flag zero loci into the zero loci of a product of Grassmannians:

\begin{ex} We want to show that $K_1$ is $\PP^4$. Indeed $K_1$ is described in \cite[Table 3]{kalashnikov} as the datum of:
\[
    A=\begin{bmatrix}
        0 & 5 \\
        0 & 0
    \end{bmatrix},\ \ r=[1,1]
\]
with $A$ the adjacency matrix, and $r$ is the dimension vector. Since $A$ is a $2\times2$ matrix and $r$ is a 2-dimensional vector, we get that $K_1$ is a quiver flag variety obtained from a quiver with only 2 vertices. Now, the arrows between the two vertices of the quiver are given by $A$, in particular, the $a_{i,j}$ entry of the matrix gives the number of arrows from the i-th vertex to the j-th vertex, while the entry $r_i$ of the dimension vector gives the dimension of the i-th vector space associated to the i-th vertex in the representation of the quiver.\\
In this way, we can rewrite $K_1$ as the following representation of quiver:
\[
 \begin{tikzcd}
1 \arrow[r] \arrow[r, shift left=2] \arrow[r, shift right=2] \arrow[r, shift right=4] \arrow[r, shift left=4] & 1
\end{tikzcd},\ or\
\begin{tikzcd}
1 \arrow[r, "5"] & 1
\end{tikzcd}
\]
where the vertices are represented by the dimension of the associated vector space in the representation of the quiver and the numbers on the arrows indicate the number of the arrows of the same type.\\
Now, this means that the quiver $\mathbf{Q}$ associated with $K_1$ is the Kronecker quiver with 2 vertices and $5$ arrows from the first one to the second one, hence, given $r$ as dimension vector, its representation is $Rep(\mathbf{Q},r)=Mat(1\times5,\mathbb{C})$, and by \cite[Lemma 2.1 and Prop. 2.2]{CRAW11}, the quiver flag variety $K_1=\mathcal{M}_{\vartheta}(\mathbf{Q},r)$ is $\PP^4$.\\

\end{ex}

\begin{ex} $K_2$ is $\mZ(\of(1,1))\subset\PP^1\times\PP^4$. In this case, we have to use \cite[Theorem 3.3]{CRAW11} to obtain the previous equivalence. In \cite[Table 3]{kalashnikov} $K_2$ is described as the datum of:
\[
    A=\begin{bmatrix}
          0 & 0 & 5\\
          0 & 0 & 0\\
          0 & 3 & 0
    \end{bmatrix},\ \ r=[1,1,1],\ \ E=(\square,\blacksquare)\oplus(\square,\varnothing)
\]
with $A$ the adjacency matrix and $r$ the dimension vector associated to the quiver flag varieties $M=\mathcal{M}_{\vartheta}(\mathbf{Q},r)$ and $E$ the vector bundle on $M$ such that $K_2=\mZ(E)\subset M$, the white squares are the Young tableaux associated to the Schur power of the universal bundle while the black ones are their duals. The quiver $\mathcal{Q}$ has 3 vertices with the following representation:
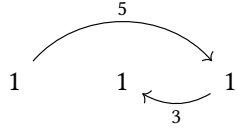
\[
\begin{tikzcd}
1 \arrow[rr, "5", bend left=49] & 1 & 1 \arrow[l, "3", bend left]
\end{tikzcd}
\]
which can be rearranged by the permutation of the second and the third vertex as 
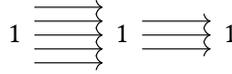
\[
\begin{tikzcd}
1 \arrow[r] \arrow[r, shift left=2] \arrow[r, shift right=2] \arrow[r, shift left=4] \arrow[r, shift right=4] & 1 \arrow[r] \arrow[r, shift left=2] \arrow[r, shift right=2] & 1
\end{tikzcd}
\]
In this way, we can rewrite $K_2$ as the datum of:
\[
    A=\begin{bmatrix}
        0 & 5 & 0\\
        0 & 0 & 3\\
        0 & 0 & 0
    \end{bmatrix},\ \ r=[1,1,1],\ \ E=(\blacksquare,\square)\oplus(\varnothing,\square)
\]
Now if we denote $\mathbf{Q}(i)$ the quiver obtained from $\mathbf{Q}$ looking only up to the i-th vertex, and $r(i)$ the dimension vector of the vertices in $\mathbf{Q}(i)$, then we can define $Y_i=\mathcal{M}_{\vartheta}(\mathbf{Q}(i),r(i))$. In our case we get $Y_2=\PP^4$ for the previous example, and $Y_3=M$, but we can apply \cite[Theorem 3.3]{CRAW11} and obtain $Y_3=\PP_{\PP^4}(\of_{\PP^4}(-1)^{\oplus 3})$. Note that $E$ is clearly $\of_{\PP^4}(-1)\boxtimes\of_{\mR}(1)\oplus\of_{\mR}(1)$, hence $K_2=\mZ(\of_{\PP^4}(-1)\boxtimes\of_{\mR}(1)\oplus\of_{\mR}(1))\subset\PP_{\PP^4}(\of_{\PP^4}(-1)^{\oplus 3})$, so if we twist $\of_{\PP^4}(-1)^{\oplus 3}$ by $\of_{\PP^4}(1)$ then $K_2=\mZ(\of(1,1))\subset\PP^4\times\PP^1$, which is $\Bl_{\PP^2}\PP^4$.
\end{ex}

\begin{ex} $K_{109}$ is $\mZ(\mQ_{\Gr(2,5)}(0,1)\oplus\of(1,0)^{\oplus 3})\subset\Gr(2,5)\times\PP^4$. In \cite[Table 3]{kalashnikov} $K_{109}$ is given by the followings:
\[
    A=\begin{bmatrix}
        0 & 0 & 5\\
        0 & 0 & 0\\
        0 & 1 & 0
    \end{bmatrix},\ \ r=[1,1,2],\ \ E=(\varnothing,\bigwedge^2\square)^{\oplus 3}
\]
Now, we can rearrange and obtain:
\[
    A=\begin{bmatrix}
        0 & 5 & 0\\
        0 & 0 & 1\\
        0 & 0 & 0
    \end{bmatrix},\ \ r=[1,2,1],\ \ E=(\bigwedge^2\square,\varnothing)^{\oplus 3}
\]  
in this way, we have a quiver $\mathbf{Q}$ with 3 vertices, whose representation is 
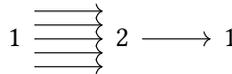
\[
\begin{tikzcd}
1 \arrow[r] \arrow[r, shift left=2] \arrow[r, shift right=2] \arrow[r, shift left=4] \arrow[r, shift right=4] & 2 \arrow[r] & 1
\end{tikzcd}
\]
Now if we consider $Y_2$, this is, by \cite[Lemma 2.1 and Prop. 2.2]{CRAW11}, $\Gr(2,5)$, while, using \cite[Theorem 3.3]{CRAW11}, $Y_3=\PP_{\Gr(2,5)}(\mU_{\Gr(2,5)})$. Note that $(\bigwedge^2\square,\varnothing)$ can be translated in $\of_{\Gr(2,5)}(1)=\bigwedge^2\mU^{\vee}_{\Gr(2,5)}$, hence $K_{109}=\mZ(\of_{\Gr(2,5)}(1)^{\oplus 3})\subset\PP_{\Gr(2,5)}(\mU_{\Gr(2,5)})$, then, by applying \cref{thm:seq}, we obtain $K_{109}=\mZ(\of(1,0)^{\oplus 3}\oplus\mQ_{\Gr(2,5)}(0,1))\subset\Gr(2,5)\times\PP^4$.
\end{ex}

\begin{ex} $K_{127}$ is $\mZ(\mQ_{\PP^4}(0,1)\oplus\of(3,0))\subset\PP^4\times\PP^5$. Note that $K_{127}$ is the datum of:
\[
    A=\begin{bmatrix}
        0 & 1 & 5\\
        0 & 0 & 0\\
        0 & 1 & 0
    \end{bmatrix},\ \ r=[1,1,1],\ \ E=(\varnothing,\Sym^3\square)
\]
If we permute the last two vertices, then we can obtain $K_{127}$ also from the triple
\[
    A=\begin{bmatrix}
        0 & 5 & 1\\
        0 & 0 & 1\\
        0 & 0 & 0
    \end{bmatrix},\ \ r=[1,1,1],\ \ E=(\Sym^3\square,\varnothing)
\]
hence, again, we have a 3-vertices quiver $\mathbf{Q}$ that whose representation is:
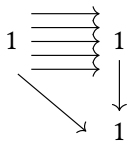
\[
\begin{tikzcd}
1 \arrow[r] \arrow[r, shift left=2] \arrow[r, shift right=2] \arrow[r, shift left=4] \arrow[r, shift right=4] \arrow[rd, shift right=3] & 1 \arrow[d] \\
                                                                                                                                        & 1          
\end{tikzcd}
\]
Note that we can divide again $\mathbf{Q}$ into $\mathbf{Q}(2)$ and $\mathbf{Q}(3)$, and it is clear that $Y_2$ is $\PP^4$, but $Y_3$ is harder. Note that the vertical arrow gives a projective bundle of 1 copy of the tautological bundle of $\PP^4$, while the transverse one gives a copy of the tautological bundle over $Y_1=\Spec\mathbb{C}$, which is $\of_{Y_1}$, hence for \cite[Theorem 3.3]{CRAW11}, we have $M=Y_3=\PP_{\PP^4}(\of_{\PP^4}\oplus\of_{\PP^4}(-1))$. Now $E$ is simply $\of_{\PP^4}(3)$, so $K_{127}=\mZ(\of_{\PP^4}(3))\subset\PP_{\PP^4}(\of\oplus\of(-1))$, and by applying \cref{thm:seq}, we get $K_{127}=\mZ(\of(3,0)\oplus\mQ_{\PP^4}(0,1))\subset\PP^4\times\PP^5$.
\end{ex}

In particular, a useful trick to make quick computations is the following combination of the methods described in the previous examples.
\begin{alg}\label[alg]{QuivToGrass}
Given the datum $(A,r,E)$ as above, fix the source, i.e., the unique vertex which is always the tail of arrows and never the head of the arrows that involve it:
\begin{itemize}
\item if a vertex is a head only of $n_i$ arrows coming from the source, and its space associated has dimension $r_i$, then the zero loci of the product of Grassmannians corresponding to the subquiver obtained from the source and that vertex is $\Gr(r_i,n_i)$;
\item if a vertex is the head of $n_a$ arrows coming from the source and $n_b$ arrows coming from another vertex of the previous type, and its space associated has dimension $r_j$, then the zero locus of product of Grassmannians corresponding to the subquiver obtained from the vertices involved is $\mZ((\mQ_{\Gr(r_i,n_i)}\boxtimes\mU^{\vee}_{\Gr(r_j,n_i\times n_b+n_a)})^{\oplus n_b})\subset\Gr(r_i,n_i)\times\Gr(r_j,n_i\times n_b+n_a)$;
\item the bundle $E$ is rewritten in terms of duals of the corresponding Schur powers of tautological vector bundles defined on the Grassmannians.
\end{itemize}
\end{alg}

\frenchspacing

\newcommand{\etalchar}[1]{$^{#1}$}

\end{document}